\documentclass[14pt]{amsart}
\usepackage{amsmath}
\usepackage{amsthm}
\usepackage{amssymb}
\usepackage{amscd}
\usepackage{amsfonts}
\usepackage{amsbsy}
\usepackage{epsfig,afterpage}
\usepackage[dvips]{psfrag}
\usepackage{subfigure}
\usepackage{epsfig}
\usepackage{amsmath}
\usepackage{srcltx}
\usepackage{amsfonts}
\usepackage{amssymb}
\usepackage{psfrag}
\usepackage{verbatim}
\usepackage{graphicx}
\usepackage{amsmath}
\usepackage{amsthm}
\usepackage{amssymb}
\usepackage{amscd}
\usepackage{amsfonts}
\usepackage{amsbsy}
\usepackage{epsfig}
\usepackage[dvips]{psfrag}
\usepackage{multirow}
\usepackage{graphics}
\usepackage{epstopdf}
\usepackage[abs]{overpic}
\usepackage{float}
\usepackage{enumerate}
\usepackage{multicol}

\usepackage{graphicx}

\def\e{\varepsilon}

\newcommand{\s}{\ensuremath{\mathbb{S}}}

\newtheorem {theorem} {Theorem}%[section]
\newtheorem {proposition} [theorem]{Proposition}

\newtheorem {lemma}  [theorem]{Lemma}
\newtheorem {example} [theorem]{Example}
\newtheorem {remark} [theorem]{Remark}

\newcommand{\R}{\mathbb{R}}
\newcommand{\C}{\mathbb{C}}

\newcommand{\N}{\mathbb{N}}

\newcommand{\D}{\ensuremath{\mathbb{D}}}

\usepackage{float}
\usepackage{tikz, wrapfig}
\tikzset{node distance=3cm, auto}
\allowdisplaybreaks

\begin{document}

\title[Holomorphic  Systems]
{On Planar Holomorphic  Systems}

\author[ L. F. S. Gouveia,  G. A. Rond\'on and P. R. da Silva]
{Luiz F. S. Gouveia, Paulo R. da Silva and Gabriel A. Rond\'on}

\address{  S\~{a}o Paulo State University (Unesp), Institute of Biociences, Humanities and
	Exact Sciences. Rua C. Colombo, 2265, CEP 15054--000. S. J. Rio Preto, S\~ao Paulo,
	Brazil.}

\email{paulo.r.silva@unesp.br; fernando.gouveia@unesp.br; garv202020@gmail.com}


\maketitle

\begin{abstract}
Planar holomorphic systems  $\dot{x}=u(x,y)$, $\dot{y}=v(x,y)$
are those that  $u=\operatorname{Re}(f)$ and $v=\operatorname{Im}(f)$
for some holomorphic function $f(z)$.
They have important dynamical properties, highlighting, for example, the fact that they do not have limit cycles and that center-focus problem is trivial. 
In particular, the hypothesis that a polynomial system is holomorphic reduces the number of parameters of the system. 
Although a polynomial system of degree $n$ depends on $n^2 +3n+2$ parameters,  a polynomial holomorphic depends only on $2n + 2$ parameters.
In this work, in addition to making a general overview of the theory of holomorphic systems, we classify all the possible global phase portraits, on the Poincar\'{e} disk, of  systems $\dot{z}=f(z)$ and $\dot{z}=1/f(z)$, where $f(z)$ is a 
polynomial of degree $2$, $3$ and $4$ in the variable $z\in \C$. 
We also classify all the possible global phase portraits of Moebius systems $\dot{z}=\frac{Az+B}{Cz+D}$, where $A,B,C,D\in\C, AD-BC\neq0$.
Finally, we obtain explicit expressions of first integrals of holomorphic systems and of conjugated holomorphic systems, which have important applications in the study of fluid dynamics.	
	
\end{abstract}

\section{Introduction}
The understanding of the phase portrait of planar differential systems involves some questions: 

\begin{itemize}
\item How is the local behavior around the equilibrium points?  A serious problem consists in distinguishing between a focus and a center.
\item After equilibrium  points the main subjects  are limit cycles, i.e., periodic orbits that are isolated in the set of all periodic orbits of a differential system. How many limit cycles are there in the phase portrait?
\item A first integral completely determines its phase portrait.  Given a vector field on $\R^2$, how can one determine if this vector field has a first integral? 
\end{itemize}

These problems are unsolved in general. However for a specific class of systems we get very satisfactory answers. The class we are referring to is the one formed by holomorphic systems.  

Planar holomorphic systems  $\dot{x}=u(x,y)$, $\dot{y}=v(x,y)$
are those that  $u=\operatorname{Re}(f)$ and $v=\operatorname{Im}(f)$
for some holomorphic function $f(z)$.

Holomorphic systems have surprising properties: 
\begin{itemize}
\item The equilibriums are isolated and the topological classification of the local phase portraits is fully known.
\item They do not have limit cycles.
\item The center--focus problem is totally solved.
\item There is a first integral $H(x,y)$ defined in a subset of total measure in $\R^2$.
\end{itemize}

 A  \textit{holomorphic function}  $f$
is a complex-valued function  defined in a domain $\mathcal{V}\subseteq\C$ and satisfying that
\begin{itemize}
	\item $u=\operatorname{Re}(f)$ and $v=\operatorname{Im}(f)$  are continuous;
	\item  there exist the partial derivatives $u_x,u_y,v_x,v_y$ in $\mathcal{V}$ and 
	\item  the partial derivatives satisfy the Cauchy--Riemann equations, see \cite{Morris}, \[u_x=v_y,\quad u_y=-v_x,\quad \forall z=x+iy\in\mathcal{V}.\]  
\end{itemize}
We remark that \textit{Looman--Menchoff's Theorem} states that the above conditions are sufficient 
to guarantee the analyticality of  $f$. It means that for any $z_0\in\mathcal{V}$ 
\begin{equation}
f(z)=A_0+A_1(z-z_0)+A_2(z-z_0)^2+...,\quad A_k=a_k+ib_k=\dfrac{f^{(k)}(z_0)}{k!}
\label{analF}
\end{equation}
for $z\in D(z_0,R_{z_0})\subseteq\mathcal{V}$ where $D(z_0,R_{z_0})$ is the largest possible $z_0$--centered disk contained in $\mathcal{V}.$ 
Unless a translation we can always assume that $z_0 = 0$.

If $f$ is holomorphic in a punctured disc $D(z_0,R)\setminus\{z_0\}$ and it is not derivable at $z_0$ we say that
$z_0$ is a singularity of $f$. In this case $f(z)$ is equal to Laurent's series in $D(z_0,R)\setminus\{z_0\}$

\begin{equation} 
	f(z)=\sum_{k=1}^{\infty}\dfrac{B_k}{(z-z_0)^k}+ \sum_{k=0}^{\infty}A_k(z-z_0)^k, \label{laurent}
\end{equation}
where $B_k=\dfrac{1}{2\pi i}\int_{C_{\e}}f(z)(z-z_0)^{k-1}dz,\quad A_k=\dfrac{1}{2\pi i}\int_{C_{\e}}\dfrac{f(z)}{(z-z_0)^{k+1}}dz$
with $C_{\e}$ parameterized by $z(t)=\e e^{it}, \e\sim0$. \\

If $B_k\neq0$  for an infinite set of indices $k$ we say that $z_0$ is an \textit{essential singularity} and if
there exists $n \geq1$ such that $B_n \neq0$ and $B_k = 0$ for every $k>n$ then we say that $z_0$ is a \textit{pole of order n}. 
Moreover $B_1$ is called \textit{residue} of $f$ at $z_0$ and it is denoted by $B_1 = \operatorname{res} (f, z_0)$.

Let  $f:D(0,R)\setminus\{0\}\rightarrow\C$ be a holomorphic function as \eqref{laurent} with $z_0=0, B_k=c_k+id_k$ and $A_k=a_k+ib_k$.
Consider the ordinary differential equation 
\begin{equation}\label{hde}
\dot{z}(t)=f(z(t)),\quad t\in\R.
\end{equation}
The solution of \eqref{hde} passing through $z\in D(0,R)\setminus\{0\}$  at 
	$t=0$ is denoted by $\varphi_f(t,z)=x(t)+iy(t).$

\begin{theorem}  The real and imaginary parts of $\varphi_f(t,z)$ must satisfy the following system
	\begin{equation}\left\{\begin{array}{ll}
	\dot{x}&= \displaystyle\sum_{k=1}^{\infty}\left(c_k\dfrac{p_k}{(x^2+y^2)^k}+d_k\dfrac{q_k}{(x^2+y^2)^k}\right)+a_0+
	\displaystyle\sum_{k=1}^{\infty}\left(a_kp_k-b_kq_k\right)\\
	\dot{y}&= \displaystyle\sum_{k=1}^{\infty}\left(d_k\dfrac{p_k}{(x^2+y^2)^k}-c_k\dfrac{q_k}{(x^2+y^2)^k}\right)+b_0+
	\displaystyle\sum_{k=1}^{\infty}\left(b_kp_k+a_kq_k\right)
	\end{array}
	\right.\label{hvf}
	\end{equation}
	with $p_k,q_k$ given in Table \eqref{Tpq}.
\end{theorem}
\begin{equation}
\begin{array}{llll}
&k	&p_k &q_k \\
\hline\\
&1	& x & y\\
\hline\\
&2	& x^2-y^2&2x y\\
\hline\\
&3	& x^3-3xy^2 & 3x^2y-y^3\\
\hline\\
&4	& x^4-6x^2y^2+y^4 & 4x^3y-4xy^3\\
\hline\\
&5	& x^5-10x^3y^2+5xy^4 & 5x^4y-10x^2y^3+y^5\\
\hline
&...&...&...

\end{array}
\label{Tpq}
\end{equation}

\begin{proof}
We have 
$f(z)=\displaystyle\sum_{k=1}^{\infty}\dfrac{c_k+id_k}{z^k}+ \sum_{k=0}^{\infty}(a_k+ib_k)z^k.$
A direct calculation using Newton's binomial formula gives us
\[z^k=(x+iy)^k=p_k+iq_k\]
with $p_k$ and $q_k$ as in the table \eqref{Tpq}.
Thus
\[(a_k+ib_k)z^k=(a_kp_k-b_kq_k)+i(b_kp_k+a_kq_k)\]
and 
\[ \dfrac{c_k+id_k}{z^k}=(c_k+id_k)\dfrac{\bar{z}^k}{|z|^{2k}}= \dfrac{(c_k+id_k)}{(x^2+y^2)^k}(p_k-iq_k)=\]
\[=\dfrac{1}{(x^2+y^2)^k} \left((c_kp_k+d_kq_k)+i(d_kp_k-c_kq_k)\right).\]

Since $\dot{x}=\operatorname{Re} (f(z))$ and $\dot{y}=\operatorname{Im}(f(z))$ we conclude the proof.
\end{proof}

We refer to system \eqref{hvf} as a \textit{holomorphic system} if $f$ is holomorphic in $D(0,R)$ and as 
\textit{meromorphic system} if $f$ is holomorphic in $D(0,R)\setminus{0}$ and $0$ is a singularity of the kind pole. 
If $f$ is holomorphic then the coefficients $c_k,d_k$  are zero due to Cauchy's Theorem.\\

\noindent\textbf{Remark.} An easy way to find the polynomials $p_k$ and $q_k$ that appear in table \eqref{Tpq} 
is to consider the triangle below. The numbers in bold refer to the coefficients of $ p_k $ and the others refer to  $ q_k$. 
The monomials of the $k$ line are $x^k,x^{k-1}y, ...,xy^{k-1},y^k$.
\[
\begin{array}{llllllllll}
\bf{1}& 1&&&&&&&&\\
\bf{1}&2&\bf{-1}&&&&&&&\\
\bf{1}&3&\bf{-3}&-1&&&&&&\\
\bf{1}&4&\bf{-6}&-4&\bf{1}&&&&&\\
\bf{1}&5&\bf{-10}&-10&\bf{5}&1&&&&\\
\bf{1}&6&\bf{-15}&-20&\bf{15}&6&\bf{-1}&&&\\
\bf{1}&7&\bf{-21}&-35&\bf{35}&21&\bf{-7}&-1&&\\
\bf{1}&8&\bf{-28}&-56&\bf{70}&56&\bf{-28}&-8&\bf{1}&\\
\bf{1}&9&\bf{-36}&-84&\bf{126}&126&\bf{-84}&-36&\bf{9}&1
\end{array}
\]

\begin{proposition}$E=(0,0)$ is an equilibrium point of 
\eqref{hvf} if and only if  $f(E)=f(0)=0$.  Moreover, if $f'(0)= a+ib$ then the 
linear part of \eqref{hvf} at $E$  has jacobian matrix given by

\[  J(E) =\left[ \begin{array}{lr}
a&-b\\ b&a
\end{array}      \right]  . \]
\end{proposition}

\begin{proof}
	This follows directly from the fact that $f'(0)=u_{x}(0,0)+iv_{x}(0,0)=a+ib$ and from the Cauchy Riemann equations $u_x=v_y, u_y=-v_x$.
\end{proof}

\begin{proposition} If $f=u+iv$ is holomorphic in $D(0,R)\setminus\{0\}$ and it  is not identically null 
	then system \eqref{hvf} has a finite number of equilibrium points
	and all of them are isolated.
\end{proposition}

\begin{proof}
If there exists a sequence of distinct equilibria $(x_n,y_n)$ of \eqref{hvf}  
then the sequence $z_n=x_n+iy_n$ will be formed by zeros of $f$. Taking  $\overline{D(0,R)}$ if necessary, we can 
assume that $z_n$ admits  a convergent subsequence $z_{n_k}$. In this case $f$ is identically null  in a set that 
has an accumulation point.  From the principle of identity of analytic functions it follows that $f \equiv 0$. 
\end{proof}  

Holomorphic systems have only three kinds of simple equilibrium points, 
all of them have index +1 (see \cite{DLA}), they are foci, centers or nodes (see Theorem 2.1 of \cite{AlvGasPro}). Moreover, this class of system  do not have limit cycles, see \cite{Ben, Bro, Oto1, Oto2,GXG, NeeKing, Sverdlove}. 

There are many motivations for study holomorphic vector fields. In addition to those already mentioned, we can also cite \cite { GAP,BraDias,DiasEnu,Dias}. We can highlight the study of parabolic bifurcations, see \cite{BufTan,Shi}, the study of simultaneous bifurcation of limit cycles \cite{GAG1}, the study of dynamical fluids through conjugate systems, see \cite{avila, bat, mars, mey}, the study of integrability of holomorphic vector fields, more specifically, study the problem of upper bounds of bifurcation of limit cycles relating to Hilbert's 16th Problem, see \cite{ Benzinger,Sverdlove}. For more details about this problem, we recommend also \cite{GinGouTor2020,GouTor2020}.\\

In this work, we study the phase portraits of the equations
\begin{equation}\label{eqhol}
\dot{z}=f(z),\quad\quad \dot{z}=\dfrac{1}{f(z)} \quad\mbox{and}\quad \dot{z}=T(z)
\end{equation}
where $f$ is a complex polynomial of degree $2,3,4$ and $T(z)=\frac{Az+B}{Cz+D}$ is a Moebius transformation, 
that is, $A,B,C,D\in\C$ with $AD-BC\neq0$.\\

We provide the following possible phase portraits on the Poincar{\'e} disk, unless topological equivalence:
\begin{itemize}
\item If $f$ is a complex polynomial of degree $2$ (resp $3$, $4$) then there are $3$ (resp $9$, $22$) possible phase portraits of  $\dot{z}=f(z)$.
See Theorem \ref{teoquad} (resp \ref{teocubphase}, \ref{teoquartphase}.)
\item If $f$ is a complex polynomial of degree $2$ (resp $3$, $4$)  then there are $3$ (resp $11$, $8$) possible phase portraits  of  $\dot{z}=\frac{1}{f(z)}$. See Theorem \ref{teoquadinv} (resp \ref{teocubinv}, \ref{teoquartinv}).
\item If $T$ is a Moebius transformation  then there are $9$  possible phase portraits  of  $\dot{z}=T(z)$. See Theorem \ref{proptransmoebius}.\end{itemize}

The paper is organized as follows. In Section \ref{sec2} we present some preliminary results. In Section \ref{sec7} we discuss some aspects of fluid dynamics and integrability. We show that the phase portrait of $\dot{z}=\dfrac{1}{f(z)}$ is equal to the phase portrait of $\dot{z}=\overline{f(z)}$, where $f(z)\neq0$. In addition, we show that the complex potential of the conjugate holomorphic system is a primitive of $f(z)$. In Section \ref{sectionlocal}, we present the local phase portraits for the system $\dot{z}=f(z)$, where $f(z)$ is a polynomial of degree $n=2,3,4$. This section will help us to show all the phase portrait on Poincar\'e disk. In Section \ref{sec3}, \ref{sec4}, \ref{sec5}, \ref{sec6} we state and prove Theorems \ref{teoquad}, \ref{teocubphase}, \ref{teoquartphase}, \ref{teoquadinv}, \ref{teocubinv}, and \ref{teoquartinv}. In Section \ref{sec9}, we study the local dynamics of  polynomial conjugate systems $\overline{f_{j}(z)}$, $j=2,3,$ and $4$. Finally in Section \ref{sec10} we state and prove Theorem  \ref{proptransmoebius}.

\section{Preliminaries}\label{sec2}
This section is devoted to state some classic results that will help us to classify the local phase portraits of holomorphic systems. 
In order to do this, we will start by introducing the 
concept of conformal conjugation.

Let $f$ and $g$ be holomorphic functions defined in some punctured neighborhood of $0\in\C$. We say that $f$ and $g$ are \textbf{ \textit{$0$--conformally conjugated} } if there exist $R>0$ and  a conformal map $\Phi:D(0,R)\rightarrow D(0,R)$ such that  $\Phi(0)=0$ and $ \Phi(\varphi_f(t,z)) =\varphi_g(t,\Phi(z))$,  for any $z\in D(0,R)\setminus\{0\}$ and all $t $ for which the above expressions  are well defined and the corresponding points are in $ D(0,R)$.

Let $f$ and $g$ be holomorphic functions defined in some punctured neighborhoods of $z_1\in\C$ and $z_2\in\C$, respectively. 
We say that  $f $ and $g$ are \textbf{\textit{$z_1z_2$--conformally conjugated }} if $ f \circ (z-z_1)$ and $g\circ  (z-z_2 )$ 
are conformally conjugated at $0$.\\

If $f$ and $g$ are holomorphic in  $D(0,R)$ then we have:
\begin{itemize}
	\item If $f(0)\neq0$ and $g(0)\neq0$ then $f$ and $g$ are $0$--conformally conjugated;
	\item If $f(0)\neq0$ and $g(0)=0$  then $f$ and $g$ are not $0$--conformally conjugated;
	\item If $f(0)=0$ and $g(0)=0$ and $f,g$ are non constant then 
	\[  \Phi(\varphi_f(t,z)) =\varphi_g(t,\Phi(z))\Leftrightarrow \Phi'(z)f(z)=g(\Phi(z)),\]
	for $ |z|$ sufficiently small.
\end{itemize}

Conformal conjugation classes are known in the literature. See for instance \cite{BLT} and \cite{gx}.
If $f$ is a holomorphic function defined in some punctured neighborhood of $z_0\in\C$ we have:
	\begin{itemize}
		\item [(a)] If $f(z_0)\neq0$  then $f$ and $g(z)\equiv 1$   are $z_00$--conformally conjugated.
		\item [(b)] If $f(z_0)=0$ and  $f'(z_0)\neq0$ then $f$ and $g(z)\equiv f'(z_0)z$ are $z_00$--conformally conjugated. 
		\item [(c)] If $f(z_0)=0$, $z_0$  is a zero of $f$ of order $n>1$ and  $\operatorname{Res}(1/f,z_0)=1/\gamma$ then 
		$f$ and $g(z)\equiv \gamma z^n/(1+z^{n-1})$   are $z_00$--conformally conjugated. 
		\item[(d)] If $f(z_0)=0$, $z_0$  is a zero of $f$ of order $n>1$ and  $\operatorname{Res}(1/f,z_0)=0$ then 
		$f$ and $g(z)\equiv z^n$   are $z_00$--conformally conjugated. 
		\item[(e)] If $z_0$  is a pole of $f$ of order $n$  then 
		$f$ and $g(z)\equiv \dfrac{1}{z^n} $   are $z_00$--conformally conjugated. 
	\end{itemize}

If $z_0$ is an essential singularity of $f$ then for any direction $w_0$,
	 there exists $z$, arbitrarily close to $z_0$,  whose flow $\varphi_f(0,z)$ follows the direction $w$, 
	 with $w$ being arbitrarily close to $w_0$. More precisely we have the following theorem.

\begin{theorem} Let $\dot{z}=f(z)$ be a holomorphic system defined in some punctured neighborhood of 
	an essential singularity $z_0\in\C$. For $\varepsilon, \delta>0$  sufficiently small and an arbitrary direction
	$w_0\in\C$ there exist $z,w\in\C$ such that $|z-z_0|<\delta$, $|w-w_0|<\varepsilon$
	and $\dfrac{d}{dt}\varphi_f(0,z)=w$.
	\end{theorem}
	
\begin{proof} Let $\varepsilon, \delta, z_0$ and $w_0$ be as in the statement. 
	The Casorati-Weierstrass theorem states that the image of $D(z_0,\delta)\setminus\{z_0\}$ is dense in $\C$. 
	Thus there exists $w\in f(D(z_0,\delta)\setminus\{z_0\})$ such that $|w- w_0|<\varepsilon$.
	 So $w=f(z)$ for some $z\in D(z_0,\delta)\setminus\{z_0\}$. 
	 Since $\dfrac{d}{dt}\varphi_f(0,z))=f(z)=w$, we conclude the proof.
\end{proof}

\begin{center}
	\begin{figure}[h]
		\begin{overpic}[scale=0.5]
			{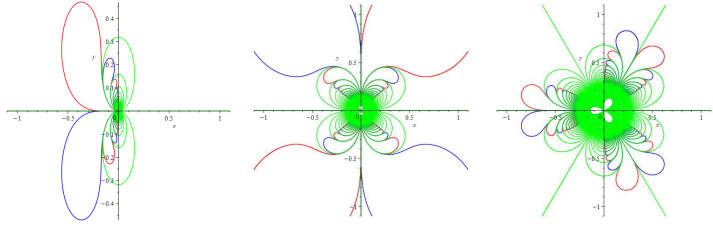} 
		\end{overpic}
		\caption{Local dynamics for $\dot{z}=z^{m}exp(1/z^{n})$ for $(n,m)=(1,2), (n,m)=(2,3)$, and $(n,m)=(3,4)$, resp.}\label{zexp}
	\end{figure}
\end{center} 

In the Figure \ref{zexp}, using the first integrals $H_{1,2}=\exp(-x/(x^2 + y^2))\sin(y/(x^2 + y^2))$, $H_{2,3}=\exp(-(x - y)(x + y)/(x^2 + y^2)^2)\sin(2xy/(x^2 + y^2)^2)/2$, and $H_{3,4}=\exp(-x(x^2 - 3y^2)/(x^2 + y^2)^3)\sin(y(3x^2 - y^2)/(x^2 + y^2)^3)/3$, we obtain the local dynamics around the essential singularity of  $\dot{z}=z^{m}exp(1/z^{n})$ for $(n,m)=(1,2), (n,m)=(2,3)$, and $(n,m)=(3,4)$.

Next proposition, whose proof can be found in \cite{AlvGasPro}, gives us important information about 
how the configuration of the equilibrium points of a polynomial holomorphic system can be if they are all simple.

\begin{proposition}\label{teoGas1}
Consider equation $\dot{z}=f(z)$ where $f(z)$ is a complex polynomial of degree $n$. 
Assume that all equilibrium points $z_{k}$, $k=1,\cdots,n$ are simple. Then
\begin{enumerate}[a)]
	\item If $z_{1},\cdots,z_{n-1}$ are centers, then $z_{n}$ is also a center.
	\item If $z_{1},\cdots,z_{n-1}$ are nodes, then $z_{n}$ is also a node.
	\item If not all the equilibrium points are centers, then there exist at least two of them that have different stability.
\end{enumerate}	
\end{proposition}

As it can be seen in \cite{AlvGasPro}, there are many other 
results that help us to study the phase portraits. We can cite here some results. Considering 
the equation \eqref{eqhol}, with $f(z)$ a complex polynomial of degree $n$ and assuming 
that all their equilibrium points are simple, so we have that if all the equilibrium points are foci, then any geometrical distribution in $\mathbb{C}$ can be achieved. Moreover, not all the 
equilibrium points have same stability. Besides we can check that if all the 
equilibrium points are collinear, then all of them are of the same type and if these points are 
not center, then they have alternated stability. In the same way, as a direct consequence, 
we have that $n$ aligned equilibrium points have alternating stability. Moreover, $n-2$ aligned
 equilibrium points have alternating stability and the two symmetric with respect to this 
 line and sharing stability. 
 
Another important  information to obtain the phase portrait of any system is to study 
the dynamics at infinity. Consider the equation \eqref{eqhol} with $\operatorname{deg}(f)=n$. Then, it 
has exactly $n-1$ equilibrium points at infinity, all of them of saddle type. Moreover,  the points 
at infinity in the Poincar\'e compactification, have exactly $n-1$ 
pairs of saddle points, see \cite{AlvGasPro}, Theorem 5.1.

\begin{proposition}\label{sver1}
	Consider equation $\dot{z}=f(z)$ where $f(z)$ is a complex polynomial.
	Then every equilibrium point  has a positive index $n$, 
	where $n$ is the order of the zero of $f$. If $n=1$, the 
	point is a source, sink or center, depending on the sign of the real part of 
	$f'(z)$. If $n>1$, the point is of purely elliptic type with $2n-2$ elliptic sectors.
\end{proposition}

For a proof see \cite{Sverdlove}.

Regarding periodic orbits in holomorphic systems, it is possible to calculate the time needed of an orbit to leave a point $z_{0}$ 
and reach a point $z_{1}$. 

\begin{proposition} Let $f$  be a complex polynomial function.
	If $z=\varphi_f(t,w)$  and $c=\operatorname{Res}(1/f,0)$ then
	\begin{itemize}
		\item [(a)] $\operatorname{exp}\left(\displaystyle\int_w^z\dfrac{1}{cf(s)}ds\right)=\operatorname{exp}\left(\dfrac{t}{c}\right)$ if $c\neq0$.
		\item[(b)] $\displaystyle\int_w^z\dfrac{1}{f(s)}ds=t$ if $c=0$.
	\end{itemize}
\end{proposition}

See \cite{GXG} for a proof.

\begin{example}
	Consider $\dot{z}=(-1+i)z$. The holomorphic system is given by
\[\dot{x}=-x-y,\quad \dot{y}=x-y.\]
The equilibrium point $(0,0)$ is an attracting focus. The solution that passes through $(1,0)$ at $t=0$
will intersect the $x$ axis again at the point $(-e^{-\pi}, 0)$ and this will occur after the time $t = \pi$.
Indeed, 
$\dfrac{1}{f(z)}=\dfrac{1}{(-1+i)z}$ implies $\operatorname{res}(1/f,0)=\dfrac{1}{-1+i}, $
and thus if $w=\varphi(t,z)$ we have
$ \exp\left(\int_z^w \dfrac{ds}{\frac{1}{-1+i}(-1+i)s}\right)=\exp\left(\dfrac{t}{\frac{1}{-1+i}}\right) , $ and taking $t=\pi, z=1$ we get $w=-e^{-\pi}.$
\end{example}

For the beauty of the argument used in \cite {GXG} to prove the next theorem, let is reproduce its proof here.

\begin{theorem}
	Let $f$ be a holomorphic function defined in a domain $\mathcal{V}\subseteq\C$. The phase portrait of
	$\dot{z}=f(z)$ has no limit cycle.
\end{theorem}

\begin{proof}
Suppose $ \gamma $ is a periodic orbit of $ \dot {z} = f (z) $ with period $ T $, that is, 
$ \varphi_f (z, T) = z $ whatever $ z \in\gamma $. Let us fix any point in $ \gamma $ and consider the transition
 function given $ \xi (z) = \varphi_f (z, T)$. The transition function is analytic and it is equal to identity at all 
 points that are in $ \gamma $. So this function coincides with the identity in a neighborhood of $ z $. 
 This means that the periodic orbit belongs to a continuum of periodic orbits, all with the same period $T$.
 \end{proof}

\section{Fluid Dynamics and Integrability}\label{sec7}

Consider a perfect, homogeneous and incompressible fluid. In this context, being perfect means that the force is due to pressure only, perpendicular to the separation surface between the parts of the fluid. Being homogeneous and incompressible means that the mass density $\rho$ is constant at all points of motion. Furthermore, let is assume that the velocity remains parallel to the plane $xy$ regardless of the third spatial coordinate $z$ such that the motion is the same in all planes parallel to the plane $xy$.\\

Let $X(x,y)=(X_1(x,y),X_2(x,y))$  be the velocity vector of this movement at the point $(x,y)$. Components $X_1$ and $X_2$ are smooth functions in $\R^2$ and must satisfy the following equations
\begin{equation} 
	\operatorname{div} X=\dfrac{\partial X_1}{\partial x}+\dfrac{\partial X_2}{\partial y}=0.\label{condDIV} \end{equation}

Let is also add the hypothesis that the motion is irrotational. We start with the line integral
\[\Gamma= \int_C X.t ds\]
where $C$ is a simple closed path and $t$ is the unit tangent vector to $C$. Note that $X.t$ represents the scalar value of the tangential velocity
and $\Gamma$ represents a measure of how much particles tend to circulate along the $C$ circuit. A fundamental theorem due to Lord Kelvin states that circulation remains constant over time. Since motion originates from rest, we conclude that circulation is zero for all time. Thus

\[ 0=\int_C X.t ds=\int_C X_1dx+X_2dy=\iint_R\left(\dfrac{\partial X_2}{\partial x}-\dfrac{\partial X_1}{\partial y}\right)dxdy,\]
where is the region inside the circuit $C$ and the last equation is due to the Green theorem. Since the integrand is a continuous function and $R$ is an arbitrary region
components $X_1$ and $X_2$ must also satisfy the following equations
\begin{equation}  
	\dfrac{\partial X_2}{\partial x}-\dfrac{\partial X_1}{\partial y}=0. \label{condIRRO} \end{equation}
Note that the equations \eqref{condDIV} and \eqref{condIRRO} are exactly the Cauchy-Riemann equations for the functions $X_1$ and $-X_2$.\\

Based on the above, it becomes of our interest to study systems 
\begin{equation} \label{sisCONJ}  
	\dot{x}=u(x,y), \\ \dot{y}=-v(x,y). 
\end{equation}
where $f(z)=u(x,y)+iv(x,y)$ is a holomorphic function.  Let is refer to the above system as the \textbf{conjugate system} $\overline{f(z)}$.\\

The conjugate system is a \textbf{gradient system}, that is,
\[\dot{x}=u=\dfrac{\partial \phi}{\partial x},\quad \dot{y}=-v=\dfrac{\partial \phi}{\partial y}.\]

The function $\phi$ satisfies the Laplace equation
\[\Delta\phi=\dfrac{\partial^2\phi}{\partial x^2}+\dfrac{\partial^2\phi}{\partial y^2}=0\]
and therefore it is a harmonic function. This implies that $\phi$ is the real part of a holomorphic function $F(z)=\phi+\psi i$.\\

The orthogonality of the level curves of the real and imaginary parts of $F$ implies that the function $\psi$ is in fact a first integral of the conjugate system.\\

\noindent\textbf{Definition} $F(z)=\phi+\psi i$ is called the complex potential function of motion.\\

\begin{proposition} \label{intconj} Let f(z) be a holomorphic function. A first integral of the conjugate holomorphic system 
	$\dot{z}=\overline{f(z)}$ is given by $\psi(x,y)=\Im F(z)$ where $F'(z)= f(z)$. In other words, the complex potential of 
	the conjugate holomorphic system is a primitive of $f(z)$.
\end{proposition}
\noindent\textit{Proof.} Indeed, for $f=u+iv$ the holomorphic conjugated system is 
\[\dot{x}=u,\quad \dot{y}=-v\]
which is gradient and thus it exists $\phi$ such that we obtain $u=\phi_x$ and $v=-\phi_y$.
The complex potential $F=\phi+\psi i$ is holomorphic and $F'(z)=\phi_x+i\psi_x=u+iv=f(z).$

\begin{proposition} Let f(z) be a holomorphic function. The equilibrium points of the conjugate holomorphic system $\dot{z}=\overline{f(z)}$
	are the critical points of $\psi$ and  those that are not degenerate are of the saddle type.
\end{proposition}
\noindent\textit{Proof.} Since $\dot{x}=u=\phi_ x, \quad \dot{y}=-v= \phi_y$
the claim about equilibria being the critics of $\psi$ holds true. Indeed , 
$u=-v=0$ implies $\psi_x=-\phi_y=0,\quad \psi_y=\phi_x=0 .$

The function $\psi$ is harmonic and $\psi\in C^{\infty}$, 
thus $\psi_{xx}+\psi_{yy}=0$
and, for the Schwartz theorem $\psi_{xy}-\psi_{yx}=0.$
Consequently the determinant of the Hessian matrix $\psi_{xx}.\psi_{yy}-\psi_{xy}\psi_{yx}=-\psi_{xx}^2-\psi_{xy}^2\leq 0.$

\noindent

\textbf {Example.} Let us consider $f(z)=z^2$. So the conjugate system is
\[ \dot{x}=x^2-y^2, \quad \dot{y}=-2xy, \]
and the first integral is given by $\psi(x,y)=x^2y-\dfrac{y^3}{3}$.

\begin{center}
	\begin{figure}[h]
		\begin{overpic}[scale=0.30]
			{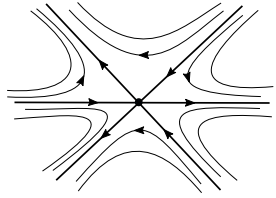} 
		\end{overpic}
		\caption{Level curves $x^2y-\dfrac{y^3}{3}=k$.}
	\end{figure}
\end{center}

\textbf {Example.} Let us consider $f(z)=\dfrac{z^2}{1+z}$. The complex potential is given by
$F(z)=\dfrac{z^2}{2}-z+\log(z+1)$. So the conjugate system is
\[ \dot{x}=\dfrac{x^3+x^2+xy^2-y^2}{(1+x)^2+y^2}, \quad
\dot{y}=-\dfrac{x^2y+2xy+y^3}{(1+x)^2+y^2}. \]
and the first integral is given by $\psi(x,y)=(x-1)y+\arctan \dfrac{y}{x+1}.$

\begin{center}
	\begin{figure}[h]
		\begin{overpic}[scale=0.30]
			{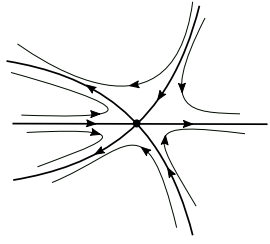} 
		\end{overpic}
		\caption{Level curves $(x-1)y+\arctan \dfrac{y}{x+1}=k$.}
	\end{figure}
\end{center}

\subsection{On the integrability of holomorphic systems}
Note that the phase portrait of $\dot{z}=\dfrac{1}{f(z)}$ is equal to the phase portrait of $\dot{z}=\overline{f(z)}$  where $f(z)\neq0$.
In fact, it follows from the fact that $\dfrac{1}{f(z)}=\dfrac{\overline{f(z)}}{|f(z)|^2}$ and $\dfrac{1}{|f(z)|^2}>0.$ 
In particular, the classification of the possible phase portraits of $\dot{z}=\dfrac{1}{p(z)}$ with $\partial p(z)\leq 4$, is the same classification of the possible phase portraits of the conjugate systems $\dot{z}=\overline{p(z)}$ with $\partial p(z)\leq 4$. The only difference is that the equilibrium points of $\dot{z}=\overline{p(z)}$ are polo-type singularities of $\dot{z}=\dfrac{1}{p(z)}.$\\

Now consider $\dot{z}=f(z)$ with $z\in A\subseteq\C$. Consider $g(z)=\dfrac{1}{f(z)}$. As before, the phase portrait of $\dot{z}=f(z)$
is equal to the phase portrait of $\dot{z}=\overline{g(z)}$ and the equilibrium points of $\dot{z}=f(z)$ are singularities of $g(z)$.  According Proposition \ref{intconj} 
a first integral of $\dot{z}=f(z)$ is $H(x,y)=\Im G(z)$ where $G'(z) =g(z).$  In short we have the following theorem.

\begin{theorem} Let $\dot{z}=f(z)$ be a holomorphic system defined on the open set $A\subseteq\C$. Thus its
	trajectories are contained in the level curves of $H(x,y)=\Im G(z)$ where $G'(z) =\dfrac{1}{f(z)}.$
\end{theorem}
\noindent\textit{Proof.} The result follows from the above considerations. However, a very simple way to obtain the same result is to consider the technique of separating variables. In fact, being $\dot{z}=f(z)$ it follows that
\[\dfrac{dz}{dt}=f(z)\implies \dfrac{dz}{f(z)}=dt\]
and integrating both sides of the equation we get
\[G(z)=t+a+bi\] where $G(z)$ is a primitive of $\dfrac{1}{f(z)}$ and $a+bi$ is a complex constant. 
So the imaginary part on the right side must be equal to the imaginary part on the left side and therefore
\[\Im{G(z)}=b.\]

\noindent\textbf{Example.} Consider $\dot{z}=z$. Thus $g(z)=\dfrac{1}{z}$, $G(z)=\operatorname{log}(z)$
and 
\[H(x,y)=\Im(\operatorname{log}(z))=\operatorname{artan}\dfrac{y}{x}.\]

\noindent\textbf{Example.} Consider $\dot{z}=z^2$. Thus $g(z)=\dfrac{1}{z^2}$, $G(z)=-\dfrac{1}{z}$
and 
\[H(x,y)=\Im(-\dfrac{1}{z})=\dfrac{y}{x^2+y^2}.\]

\noindent\textbf{Example.} Consider $\dot{z}=(1+i)z$. Thus $g(z)=\dfrac{1}{(1+i)z}$, $G(z)=\dfrac{1}{1+i}\operatorname{log}(z)$
and 
\[H(x,y)=\Im(\dfrac{1}{1+i}\operatorname{log}(z))=\operatorname{artan}\dfrac{y}{x}-\dfrac{1}{2}\operatorname{ log}(x^2+y^2).\]
See Figure \eqref{figex2}.

\begin{center}
	\begin{figure}[h]
		\begin{overpic}[scale=0.3]{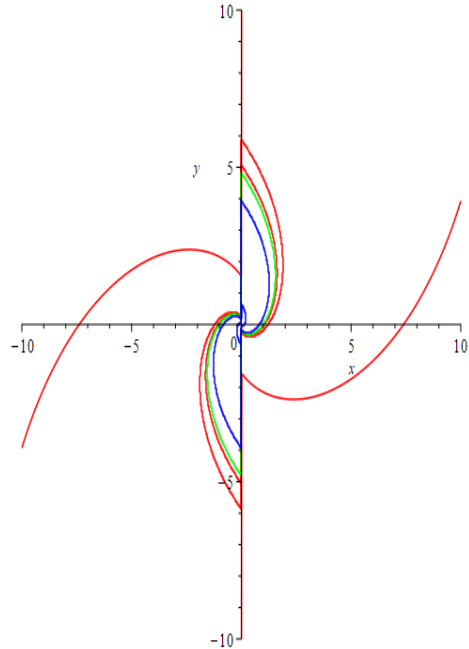}
		\end{overpic}
		\caption{Level curves $\operatorname{arctan}\dfrac{y}{x}-\dfrac{1}{2}\operatorname{log}(x^2+y^2)=k$,  $k=0$ green, $k>0$ blue, $k<0$ red.}
		\label{figex2}
	\end{figure}
\end{center}

\noindent\textbf{Example.} Consider $\dot{z}=z^2\exp(1/z)$. Thus $g(z)=\dfrac{1}{z^2}\exp(-1/z)$, $G(z)=\exp(-1/z)$
and 
\[H(x,y)=\Im(\exp(-1/z))=\exp({\dfrac{-x}{x^2+y^2}})\sin(\dfrac{y}{x^2+y^2}).\]
See Figure \eqref{figex3}.

\begin{center}
	\begin{figure}[h]
		\begin{overpic}[scale=0.4]{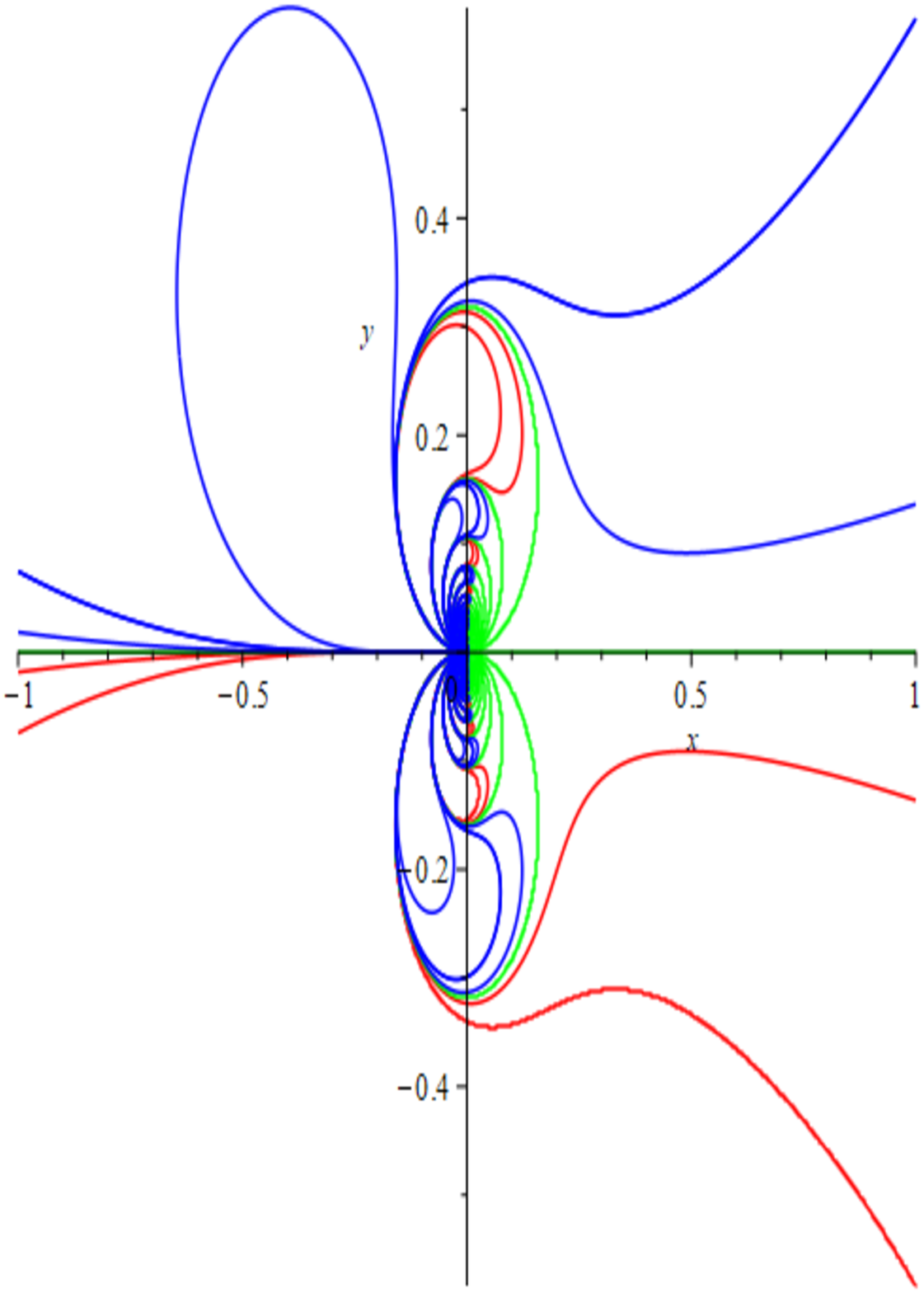}
		\end{overpic}
		\caption{Level curves $\exp\big({\dfrac{-x}{x^2+y^2}}\big)\sin(\dfrac{y}{x^2+y^2})=k$, $k=0$ green, $k>0$ blue, $k<0$ red.}
		\label{figex3}
	\end{figure}
\end{center}

\section{Local phase portraits and center-focus problem}\label{sectionlocal}

This section is devote to study the local phase portrait of the systems  $ \dot {z} = f (z) $ with f$(z)$ holomorphic polynomial of degree $n=2,3,4$. Moreover, we will present the triviality of center-focus problem.

% In order to simplify reading this work, in all figures of this work, blue points are attracting equilibria and red points are repelling equilibria.

%In order to simplify reading this study, in all figures of this section, blue points are attracting equilibria and red points are repelling equilibria.

\subsection{Center-Focus Problem}

Consider $\dot{z}=f(z)$, $f$ polynomial and $f(z_{0})=0$. Let us suppose  $f'(z_{0})\neq0$. We know that $f$ and $f'(z_0)z$ are $z_00$--conformally conjugated. So, $z_{0}$ is a center if $\operatorname{Re}(f'(z_{0}))=0$ and $z_{0}$ is a focus/node if  $\operatorname{Re}(f'(z_{0}))\neq 0$. Moreover, if  
$f'(z_{0})=0$, then there exists $n$ such that $f$ is $z_00$--conformally conjugated to $z^{n}$ or $(\gamma z^{n})/(1+z^{n})$. In both cases, the equilibrium point is neither a center nor a focus. See more details in \cite{BLT}. Therefore, the center-problem for holomorphic systems is trivial. It is enough analyzing $f'(z_{0})$.

\subsection{Quartic Polynomial Holomorphic Systems}

Let $p(z)$ be a non constant  polynomial holomorphic systems with degree $\partial p\leq 4$.  Without loss  of generality we assume that $z_0=0$ is an equilibrium point.
\[p(z)=A_1z+A_2z^2+A_3z^3+A_4z^4.\]
\begin{itemize}
	\item If $A_1\neq0$ then $z_0$ is a simple zero. In this case $p(z)$ and $g(z)=A_1z$ are  conformally conjugated.
	\item If $A_1=0$,  $A_2\neq 0$ and $A_3\neq0$ then $z_0$ is a zero of order $2$
	and $\operatorname{res}\left(\dfrac{1}{p},0\right)=\dfrac{1}{\lambda}=-\dfrac{A_3}{A_2^2}$. Thus follow that 
	$p(z)$ and $g(z)=\dfrac{\lambda z^2}{1+z}$ are  conformally conjugated.
	\item If $A_1=0$,  $A_2\neq 0$ and $A_3=0$ then $z_0$ is a zero of order $2$
	and $\operatorname{res}\left(\dfrac{1}{p},0\right)=0$. Thus follow that 
	$p(z)$ and $g(z)=A_2z^2$ are  conformally conjugated.
	\item If $A_1=A_2=0$  and $A_3\neq0$ then $z_0$ is a zero of order $3$. 
	Since $\operatorname{res}\left(\dfrac{1}{p},0\right)=\dfrac{1}{\alpha}=\dfrac{A_4^2}{A_3^3}$ follow that 
	$p(z)$ and $g(z)=\dfrac{\alpha z^3}{1+z^2}$ are  conformally conjugated.
%	\begin{center}
%		\begin{overpic}[scale=0.4]
%			{quad_local6.eps} 
%			\put(40,-20){$\dot{z}=\dfrac{z^3}{1+z^2}$ }
%		\end{overpic}
%	\end{center}
	
	\item If $A_1=A_2=A_3=0$  then  $p(z)=A_4z^4$.
%	\begin{center}
%		\begin{overpic}[scale=0.4]
%			{quad_local7.eps} 
%			\put(25,-20){$\dot{z}=A_4z^4$ }
%		\end{overpic}
%	\end{center}
\end{itemize}

For a  general equilibrium point $z_0$, $p(z_0)=0$, we have:
\begin{itemize}
	\item If $p'(z_0)\neq0$ then  $p(z)$ and $g(z)=p'(z_0)z$ are  $z_00-$conformally conjugated.
	\item If $p'(z_0)=0$, $p''(z_0)\neq 0$ and $p'''(z_0)\neq0$ then 
	$p(z)$ and $g(z)=\dfrac{\lambda z^2}{1+z}$ are  $z_00-$conformally conjugated where $\dfrac{1}{\lambda}=-\dfrac{2p'''(z_0)}{3p''(z_0)^2}$.
	\item If $p'(z_0)=0$, $p''(z_0)\neq0$  and $p'''(z_0)=0$ then $p(z)$ and $g(z)=\dfrac{p''(z_0)}{2}z^2$ are  $z_00-$conformally conjugated.
	\item If $p'(z_0)=p''(z_0)=0$  and $p'''(z_0)\neq0$ then  $p(z)$ and $g(z)=\dfrac{\alpha z^3}{1+z^2}$ are  $z_00-$conformally conjugated, 
	where $\dfrac{1}{\alpha}=\dfrac{3p^{(4)}(z_0)^2}{8p'''(z_0)^3}$.
	\item If $p'(z_0)=p''(z_0)=p''(z_0)=0$  then $p(z)$ and $g(z)=\dfrac{p^{(4)}(z_0)}{24}z^4$ are  $z_00-$conformally conjugated.
\end{itemize}

%\begin{center}
%\begin{figure}[h]
%\begin{overpic}[scale=0.35]
%	{figuras_locais2.eps} 
%\end{overpic}
%\caption{Local dynamics:  $\dot{z}=(a_1+ib_1)z$,  $\dot{z}=ib_1z$, , $A_{3}z^3$, $\dot{z}=a_1z^2$ and $\dot{z}=A_4z^4$, respectively.}
%\end{figure}
%\end{center}

\begin{center}
	\begin{figure}[h]
		\begin{overpic}[scale=0.6]
			{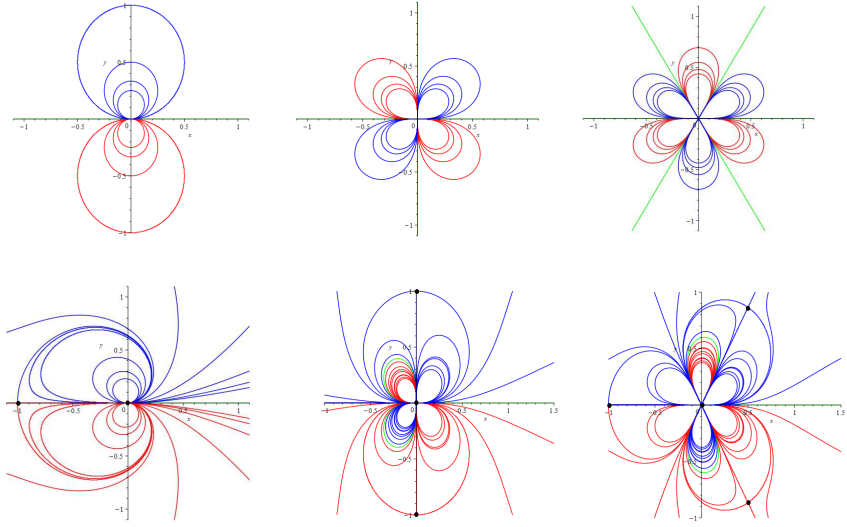} 
		\end{overpic}
		\caption{Local dynamics: $\dot{z}=z^2$, $\dot{z}=z^3$, $\dot{z}=z^4$, $\dot{z}=\dfrac{z^2}{1+z}$, $\dot{z}=\dfrac{z^3}{1+z^2}$ and $\dot{z}=\dfrac{z^4}{1+z^3}$, respectively.}
	\end{figure}
\end{center}
\newpage

\section{Quadratic  Polynomial Holomorphic Systems}\label{sec3}

Consider a quadratic polynomial holomorphic function

\begin{equation}
f(z)=A_0+A_1z+A_2z^2, \quad A_k=a_k+ib_k,\quad  A_2\neq 0.\label{hol-q-vec}
\end{equation}	
 
Its trajectory are the solutions of the differential system
\begin{equation}\left\{\begin{array}{ll}
\dot{x}&=a_0+a_1x-b_1y+ a_2(x^2-y^2)-b_2(2xy),\\
\dot{y}&=b_0+b_1x+a_1y+b_2(x^2-y^2)+a_2(2xy).
\end{array}
\right.\label{qhvf}
\end{equation}
 
\subsection{Finite equilibrium points.}

\begin{proposition}
	If $z(t)$ is a periodic orbit of the system \eqref{qhvf} then $z(t)$ intersects the straight-line $a_1+2a_2x-2b_2y=0$.
\end{proposition}

\begin{proof}
The divergence of \eqref{qhvf}  is 
\[u_x+v_y=2(a_1+2a_2x-2b_2y) .\]
Thus the divergence has constant sign for $(x,y)$ out of the straight-line $a_1+2a_2x-2b_2y=0$.
Applying Bendixson's criteria  we conclude the proof. 	
\end{proof}

	\begin{proposition} 
		Let $f$ be a quadratic holomorphic polynomial function as in \eqref{hol-q-vec} with $a_0=b_0=0$.
		Then system \eqref{qhvf} has at most two equilibrium points.
		Moreover
		\begin{itemize}
			\item [(a)] If $a_1\neq0$ then there exist two equilibrium points, one of them being a repelling focus or node and
			the other being an attracting focus or node.
			\item [(b)] If $a_1=0$  and $b_1\neq0$ then both equilibria are centers.
			\item [(c)] If $a_1=b_1=0$ and $a_2\neq0$ then there exists only one equilibrium, non-hyperbolic, and the straight-line 
			$y=-\frac{b_2}{a_2}x$ is invariant.
			\item [(d)] If $a_1=b_1=a_2=0$ and $b_2\neq0$ then there exists only one equilibrium, non-hyperbolic, and the straight-line 
			$x=0$ is invariant.
		\end{itemize}
	\end{proposition}
	
\begin{proof}

Considering a translation $x\mapsto x-a_0$, $y\mapsto y-b_0$ if necessary, we can assume that 
	$a_0=b_0=0.$  We have  
\[f(z)=A_1z+A_2z^2=0 \Leftrightarrow z=0\quad\mbox{or}\quad z=-\frac{A_1}{A_2}.\]	
So the equilibria are
$E_1=(0,0)$ and $E_2=\left( \frac{-a_1a_2-b_1b_2}{a_2^2+b_2^2},\frac{a_1b_2-a_2b_1}{a_2^2+b_2^2}\right).$  
Since $f'(0)=A_1$ and $f'(-\frac{A_1}{A_2})=-A_1$ the jacobian matrices at the equilibrium points of system \eqref{qhvf} are
\[  J(E_1) =\left[ \begin{array}{lr}
a_1&-b_1\\ b_1&a_1
\end{array}      \right]   \quad\mbox{and}   \quad J(E_2) =\left[ \begin{array}{lr}
-a_1&b_1\\ -b_1&-a_1
\end{array}      \right] . \]
The determinant $D$ and the trace $T$ are 
\[D(E_1)=D(E_2)=a_1^2+b_1^2\quad\mbox{and}\quad T(E_1)=-T(E_2)=2a_1.\]
We conclude that if $a_1> 0$ then $E_1$ is a repelling focus and $E_2$ is an attracting focus and if $a_1 <0$ then $E_1$
 is an attracting focus and $E_2$ is a repelling focus. This concludes the proof of (a).
 
The Lyapunov coefficients (see the appendix) of the equilibrium $E_1$ are $V_1=a_1$, $V_2=0$
and $V_3=\frac{\pi a_1(a_2^2+b_2^2)}{b_1^3}.$
Thus, if $a_1=0$ and $b_1\neq0$  then $E_1$ is a center. 
\footnote{We remember that if an equilibrium of a quadratic system has $V_1=V_2=V_3=0$ then it is a center. Moreover, for holomorphic systems,  the calculation of Lyapunov Constants are not necessary. This is not true for general systems. This fact shows one more special property of holomorphic systems. Above, we present the calculation for the reader to compare the results and has more materiality. } \\
	
Similarly we prove that if $a_1 = 0$ and $b_1\neq0$  then $ E_2 $ will also be a center. 
To do this we move the singularity to the origin considering 
	\[ x\rightarrow x+\frac{b_1b_2}{a_2^2+b_2^2} \quad y\rightarrow y+\frac{a_2b_1}{a_2^2+b_2^2}\]
and we follow the same steps.  Note that in this case the equilibria $E_1$ and $E_2$ are 
contained in the straight line $a_1+2a_2x-2b_2y=0$. This concludes the proof of (b). 

If $a_1=b_1=0$ we have 
$f(z)=[a_2(x^2-y^2)-2b_2xy]+i[b_2(x^2-y^2)+2a_2xy]$
and if $a_2\neq0,$ 
$ f\left(x-i\frac{b_2}{a_2}x\right)  = \left(a_2+\frac{b_2^2}{a_2}\right)x^2-i\left(b_2+\frac{b_2^3}{a_2^2}\right)x^2 . $

Since 
$\frac{\operatorname{Im} f(z)}{\operatorname{Re} f(z)}=-\frac{b_2}{a_2}$
the straight line 
$y=-\frac{b_2}{a_2}x$ is invariant. This concludes the proof of (c).

If $a_1=b_1=a_2=0$ we have 
$f(z)=-2b_2xy+ib_2(x^2-y^2)$
and 
$ f(-iy)  = -ib_2y^2 . $
then the straight line 
$x=0$ is invariant. This concludes the proof of (d).
\end{proof}

\subsection{Infinite equilibrium points} The formulas we will use in compactification are presented in the appendix.

	\begin{proposition}
			Let $f$ be a quadratic holomorphic polynomial function as in \eqref{hol-q-vec} with $a_0=b_0=0$.
		\begin{itemize}
			\item [(a)] If $a_1\neq0$ then there exist two saddle points on the infinite $\s^1$. One of them is the $\omega-$limit of 
			a repelling focus (or node) and the other one is the $\alpha-$limit of an attracting focus (or node).
			\item [(b)] If $a_1=0$ and $b_1\neq0$ then there exist two saddle points on the infinite $\s^1$
			which are connected by a finite orbit.
			\item [(c)] If $a_1=b_1=0$ and $a_2\neq0$ then there exist two saddle points on the infinite $\s^1$ with finite separatrix contained
		    in the straight line $y=-\frac{b_2}{a_2}x$.
			\item [(d)] If $a_1=b_1=a_2=0$ and $b_2\neq0$ then there exist two saddle points on the infinite $\s^1$ with finite separatrix contained
			in the straight line $x=0$.
		\end{itemize}
	\end{proposition}

\begin{proof}  The phase portrait on $U_1$ is the central projection of the phase
	portrait of the system
	\begin{equation}
	\left\{\begin{array}{ll}
	\dot{s}&=b_2+b_1w+a_2s+b_2s^2+b_1s^2w+a_2s^3,\\
	\dot{w}&=-a_2w-a_1w^2+2b_2sw+b_1sw^2+a_2ws^2.
	\end{array}
	\right.
	\end{equation}
	The equilibrium points in $\s^1$ are determined by
	\[w=0,\quad b_2+a_2s+b_2s^2+a_2s^3=(b_2+a_2s)(1+s^2)=0.\]
	Thus if $a_2\neq0$ the equilibrium is $(-\frac{b_2}{a_2},0)$ and if $a_2=0$ either
	all points are equilibrium points or none point is equilibrium.
	Since the jacobian matrix at $(-\frac{b_2}{a_2},0)$  is
	\[J\left(-\frac{b_2}{a_2},0\right)=\left[ \begin{array}{rl}
(a_2+\frac{b_2^2}{a_2}) &(b_1+b_1\frac{b_2^2}{a_2})\\
&\\
0&-(a_2+\frac{b_2^2}{a_2})
	\end{array}   \right]
	\]
we conclude that $(-\frac{b_2}{a_2},0)$ is a saddle because  $\det\left(J\left(-\frac{b_2}{a_2},0\right)\right)<0.$

The phase portrait on $U_2$ is the central projection of the phase
portrait of the system
\begin{equation}
\left\{\begin{array}{ll}
\dot{s}&=-a_2-b_2s-b_1w-a_2s^2-b_2s^3-b_1s^2w,\\
\dot{w}&=-b_1sw^2-a_1w^2-b_2s^2w+b_2w-2a_2ws.
\end{array}
\right.
\end{equation}
The equilibrium points in $\s^1$ are determined by
$w=0,$ $-a_2-b_2s-a_2s^2-b_2s^3=(a_2+b_2s)(-1-s^2)=0.$
Thus if $b_2\neq0$ the equilibrium is $(-\frac{a_2}{b_2},0)$ and if $b_2=0$ either
all points are equilibrium points or none point is equilibrium.
Since the jacobian matrix at $(-\frac{a_2}{b_2},0)$  is
\[J\left(-\frac{a_2}{b_2},0\right)=\left[ \begin{array}{rl}
-(b_2+\frac{a_2^2}{b_2}) &-(b_1+b_1\frac{a_2^2}{b_2})\\
&\\
0&(b_2+\frac{a_2^2}{b_2})
\end{array}   \right]
\]
we conclude that $(-\frac{a_2}{b_2},0)$ is a saddle because  $\det\left(J\left(-\frac{a_2}{b_2},0\right)\right)<0.$
\end{proof}

\begin{theorem} \label{teoquad} Any quadratic holomorphic polynomial system  $\dot{z}=z(z-(a+ib))$ is topologically equivalent to one of the 3 phase portraits presented in next figure.
\end{theorem}

%\begin{figure}
%	\begin{overpic}[width=12cm]{poin123n}
%		%\begin{overpic}[grid,tics=10,width=7cm]{poin123n}
%	\end{overpic}
%	\caption{ (A): $a_0=b_0=a_1=b_1=0$, (B): $a_0=b_0=a_1=0,\quad b_1\neq0$ and (C): $a_0=b_0=0,\quad a_1\neq0$.}
%	\label{poin123}
%\end{figure}

\begin{center}
	\begin{overpic}[scale=0.2]
	{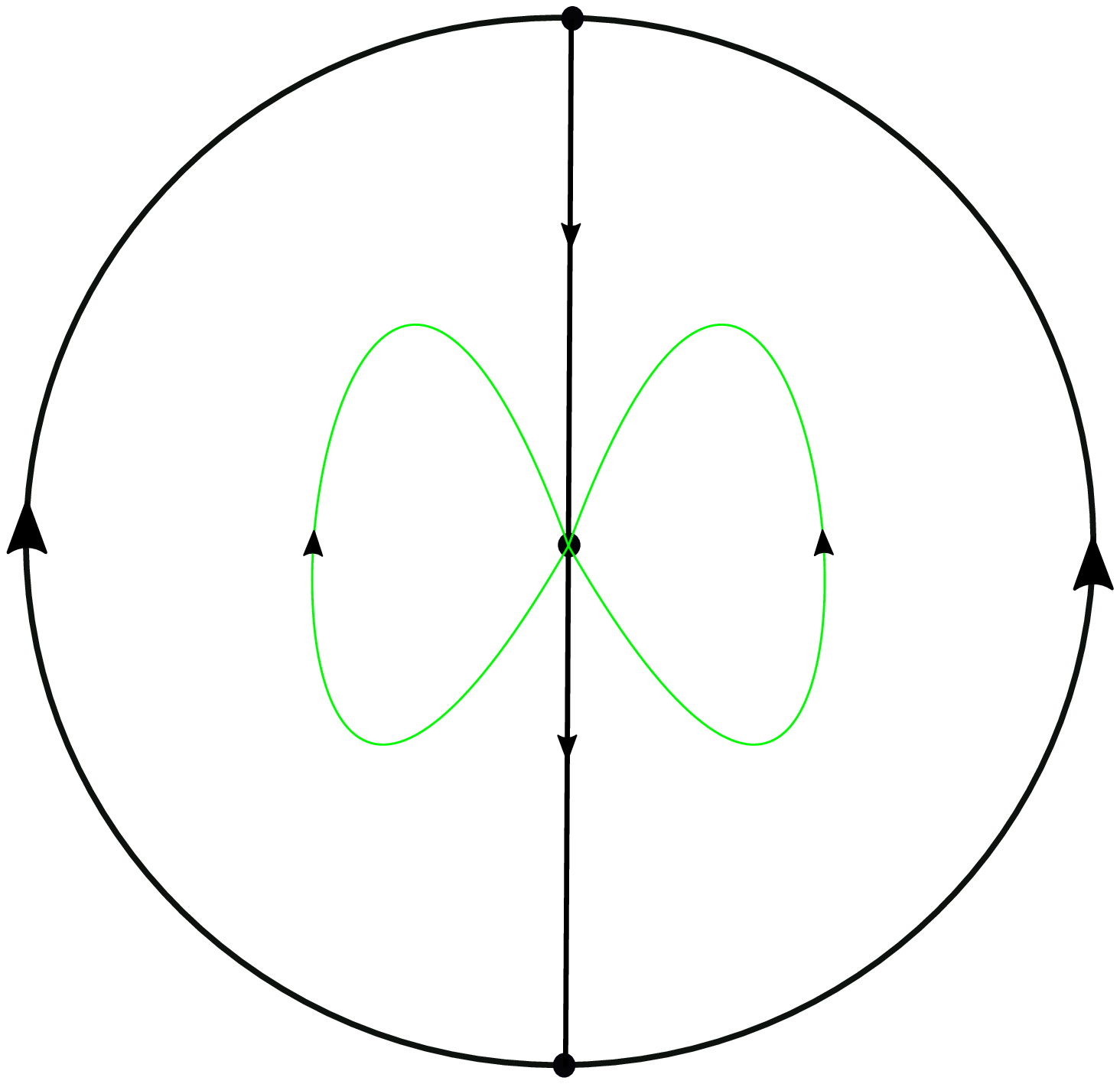} 
\put(0,-20){Double Point - DD}
	\end{overpic}
\hspace{1cm}
\begin{overpic}[scale=0.2]
{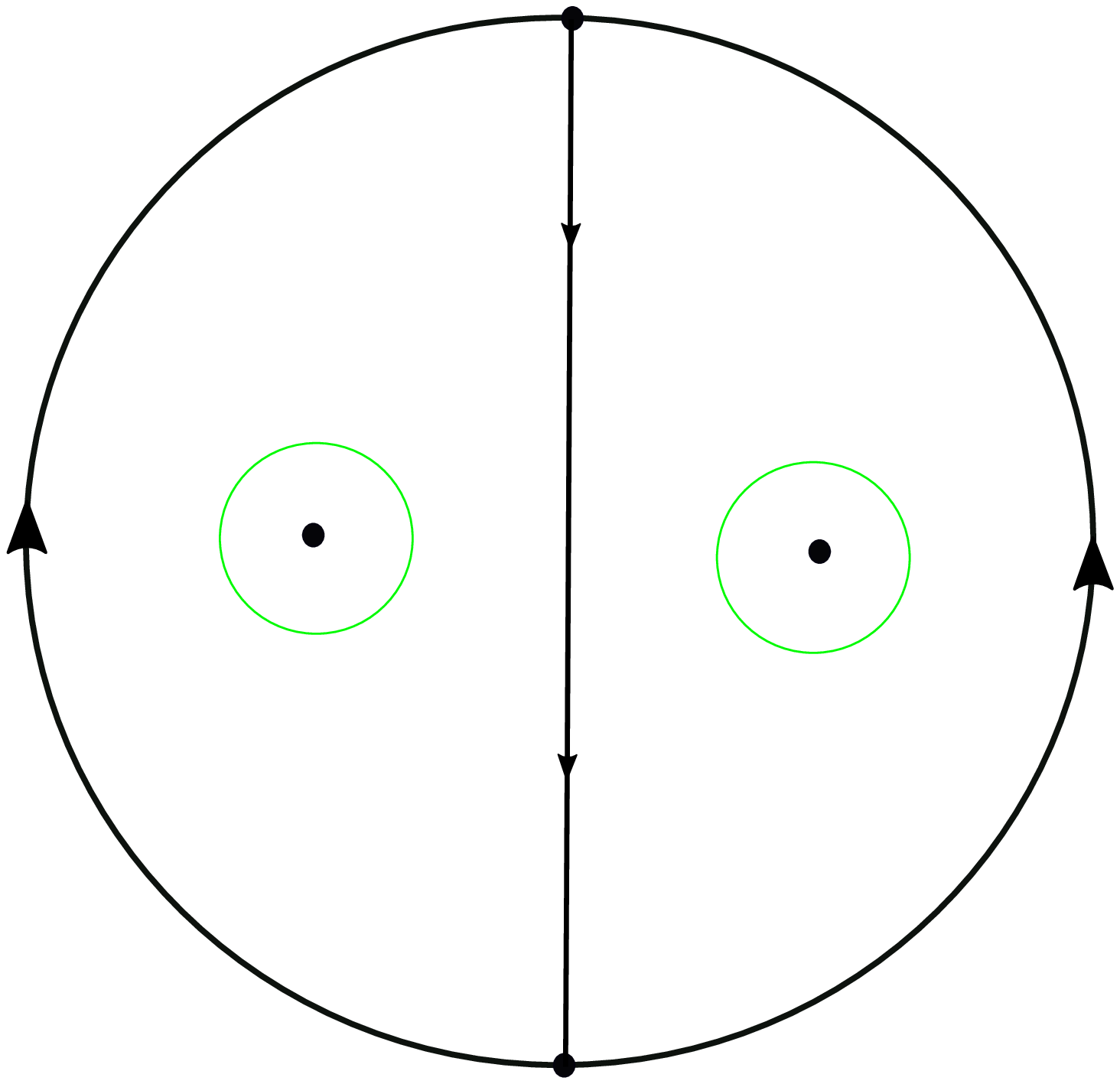}	
\put(0,-20){Two centers - CC}
\end{overpic}
\hspace{1cm}
\begin{overpic}[scale=0.2]
{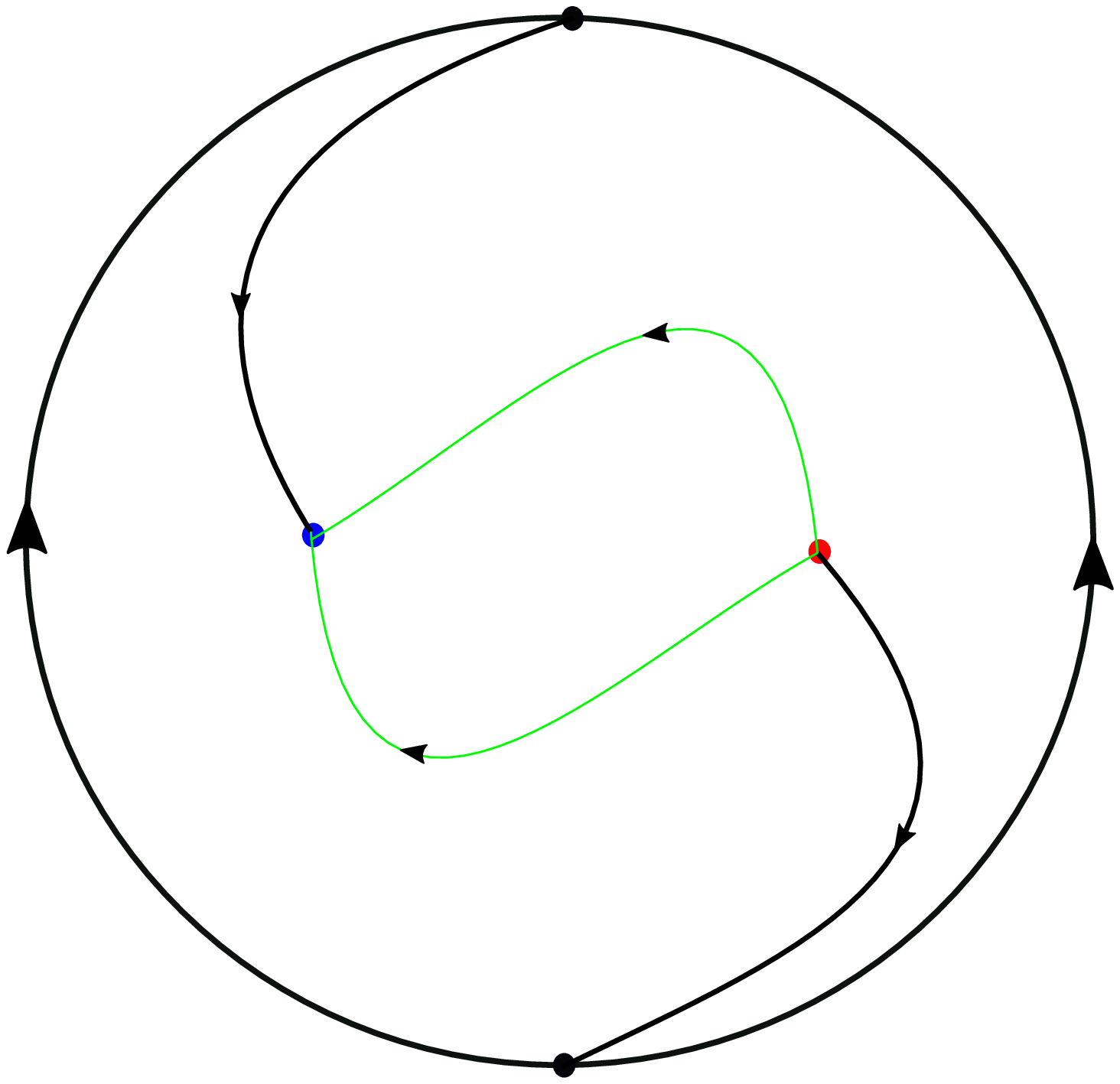}
\put(10,-20){Two foci - FF}
\end{overpic}
\end{center}

\vspace{0.5cm}

\section{Cubic Holomorphic Polynomial Systems}\label{sec4}

In this section, we are interested in the study all the possibilities of phase portrait of a cubic
holomorphic polynomial system. We know that the equilibrium points can be 
foci, centers, nodes or of purely elliptic type with $2n-2$ elliptic sectors. So, consider the function.

\begin{equation}\label{holcubvec}
f(z)=A_0+A_1z+A_2z^2+A_{3}z^3, \quad A_k=a_k+ib_k,\quad  A_3\neq 0.
\end{equation}

\begin{remark}\label{remark1}
As we know, by Fundamental Algebra Theorem we can rewrite $f(z)$ as $f(z)=\alpha(z-A_{1})(z-A_{2})(z-A_{3})$, $\alpha \in \mathbb{C}$. 
Making the change of variables in the form $z=z+A_{1}$, and making a rescheduling of time, we obtain 
\[\tilde{f}=e^{i\theta}z(z-B_{1})(z-B_{2}).\] We denote $B_k=a_k+ib_k$.	
\end{remark}

\begin{lemma}
	Let $v_{j}$ be an eigenvalue of the system $\dot{z}=f(z)$. Then $e^{i \theta}v_{j}$ is an eigenvalues of $\dot{z}=e^{i \theta}f(z)$.
\end{lemma}

\begin{proof}
	Let $p_{j}$ be a complex number  $j =1,\cdots,n$ such that $f(p_{j})=0$. As we know, the eigenvalues of $\dot{z}=f(z)$ are given by $f'(p_{j})=v_{j}$. It is clear that $p_{j}$ is also equilibrium point of $\dot{z}=e^{i \theta}f(z)$. Then, their respective eigenvalues are given by  $e^{i \theta}f'(p_{j})=e^{i \theta}v_{j}$. 	
\end{proof}

\begin{remark}
Note that the phase portrait of the systems $\dot{z}=f(z)$ and $\dot{z}=e^{i \theta}f(z)$   will not necessarily be  the same. For example, take  $\dot{z}=iz$ and  $\dot{z}=e^{i \pi/2}iz$. For the first system we have a center in the origin. However, taking $\theta=\pi/2$ we obtain an attracting node in the origin. On the other hand, as already mentioned, holomorphic systems have only three kinds of simple equilibrium points that are foci, centers and nodes (see Theorem 2.1 of \cite{AlvGasPro}). Therefore, in order to simplify the study we will consider $\theta=0$.
\end{remark}

%Considering $\theta=0$, its trajectory are the solutions of the differential system
%\begin{equation}\left\{\begin{array}{ll}\label{eqc}
%\dot{x}&=a_{1}a_{2}x - a_{1}b_{2}y - a_{1}x^2 + a_{1}y^2 - a_{2}b_{1}y - a_{2}x^2 + a_{2}y^2 - b_{1}b_{2}x + 2b_{1}xy\\
%&\phantom{=} + 2b_{2}xy + x^3 - 3xy^2,\\
%\dot{y}&=a_{1}a_{2}y + a_{1}b_{2}x - 2a_{1}xy + a_{2}b_{1}x - 2a_{2}xy - b_{1}b_{2}y - b_{1}x^2 + b_{1}y^2 - b_{2}x^2\\
%&\phantom{=} +b_{2}y^2 + 3x^2y - y^3.
%\end{array}
%\right.
%\end{equation}

To get eigenvalues of $\eqref{holcubvec}$, it is enough to calculate $f'(z)$ at the equilibrium points $0$, $B_{1}$ and $B_{2}$, 
that is, the eigenvalues are given by $f'(0)$, $f'(B_{1})$ and $f'(B_{2})$ and their respective conjugates, see \cite{Bro}. 
Therefore, the eigenvalues are given by $v_{0}=B_{1}B_{2}$ and $\overline{v_{0}}=\overline{B_{1}B_{2}}$, 
$v_{B_{1}}=B_{1}^2-B_{1}B_{2}$ and $\overline{v_{B_{1}}}=\overline{B_{1}^2-B_{1}B_{2}}$, and
 $v_{B_{2}}=B_{2}^2-B_{1}B_{2}$ and $\overline{v_{B_{2}}}=\overline{B_{2}^2-B_{1}B_{2}}$. 
 
% In terms of $\alpha_{i}$ and $\beta_{i}$, we have
%\begin{equation*}
%\begin{aligned}
% v_{0}^{1}&= -\beta_{1}\beta_{2} + \alpha_{1}\alpha_{2}+i(\alpha_{1}\beta_{2}+\alpha_{2}\beta_{1}),\\
% v_{0}^{2}&= -\beta_{1}\beta_{2} + \alpha_{1}\alpha_{2}-i(\alpha_{1}\beta_{2}+\alpha_{2}\beta_{1})
%\end{aligned}
%\end{equation*}
%
%     
%\begin{equation*}
%\begin{aligned}
%v_{\beta_{1}}^{1}&=-\beta_{1}^2+\beta_{1}\beta_{2}+\alpha_{1}^2 \alpha_{1}\alpha_{2}+i(2\alpha_{1}\beta_{1}-\alpha_{1}\beta_{2}-\alpha_{2}\beta_{1}),\\
%v_{\beta_{1}}^{2}&=-\beta_{1}^2+\beta_{1}\beta_{2}+\alpha_{1}^2 \alpha_{1}\alpha_{2}-i(2\alpha_{1}\beta_{1}-\alpha_{1}\beta_{2}-\alpha_{2}\beta_{1})
%\end{aligned}
%\end{equation*}
%
%
%
%\begin{equation*}
%\begin{aligned}
%v_{\beta_{2}}^{1}&=\beta_{1}\beta_{2}-\beta_{2}^2-\alpha_{1}\alpha_{2}+\alpha_{2}^2+i(-\alpha_{1}\beta_{2}-\alpha_{2}\beta_{1}+2\alpha_{2}\beta_{2}),\\
%v_{\beta_{2}}^{2}&=\beta_{1}\beta_{2}-\beta_{2}^2-\alpha_{1}\alpha_{2}+\alpha_{2}^2-i(-\alpha_{1}\beta_{2}-\alpha_{2}b_{1}+2\alpha_{2}\beta_{2})
%\end{aligned}
%\end{equation*}

To study the infinite equilibrium points, we  use the  Poincar\'e Compactification. The expression of  Poincar\'e Compactification in the chart $U_{1}$ is given by

\begin{equation*}
\dot{s}=w^{3}\left(-su\left(\frac{1}{w},\frac{s}{w}\right)+v\left(\frac{1}{w},\frac{s}{w}\right)\right), \dot{w}=-w^{4}u\left(\frac{1}{w},\frac{s}{w}\right).
\end{equation*}
And the expression in $U_{2}$ is given by 

\begin{equation*}
\dot{s}=w^{3}\left(-u\left(\frac{s}{w},\frac{1}{w}\right)-sv\left(\frac{s}{w},\frac{1}{w}\right)\right),\dot{w}=-w^{4}v\left(\frac{s}{w},\frac{1}{w}\right).
\end{equation*}
Note that, in this charts, the point $(s,w)$ at infinity has its coordinate in $(s,0)$. So, for the chart $U_{1}$  we must study the system $\dot{s}=2s(s^{2}+1),\dot{w}=0.$

Note that $s=0$ is a saddle  of the system above. For the other charts, following the same steps, we also found a saddle point. We call the four points at infinity as $I_{1}, I_{2}, I_{3}$ and $I_{4}$. 

\begin{center}
	\begin{figure}[h]
\begin{overpic}[abs,unit=1mm,scale=.3]
	{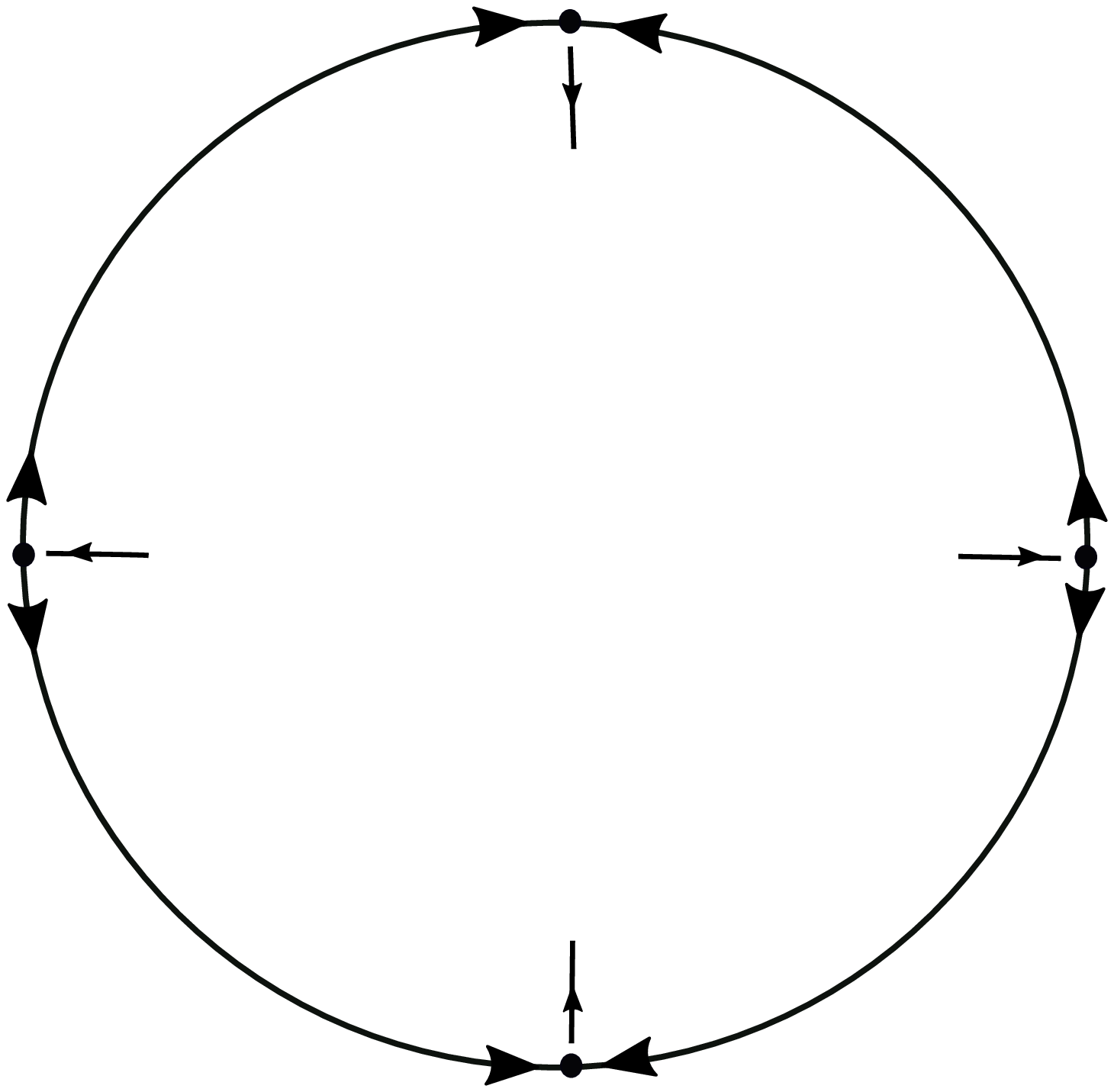}
\put(21,44){$I_{1}$}
\put(45,20){$I_{2}$}
\put(20,-5){$I_{3}$}
\put(-5,20){$I_{4}$}
\end{overpic}
\caption{Saddle points at infinity.}
\end{figure}
\end{center}

\begin{remark}
We would like to emphasize, as already mentioned in \cite{GXG}, that conformal conjugacy is stronger than topological conjugacy or topological equivalence. Moreover, the angles in the tangent space are preserved.
\end{remark}

\begin{theorem}\label{teocubphase}
	If we distinguish nodes and foci, there are only nine different topologically phase portraits of the cubic holomorphic system. Without this distinction, there are only six phase portraits.
\end{theorem}

\begin{overpic}[scale=0.25]
{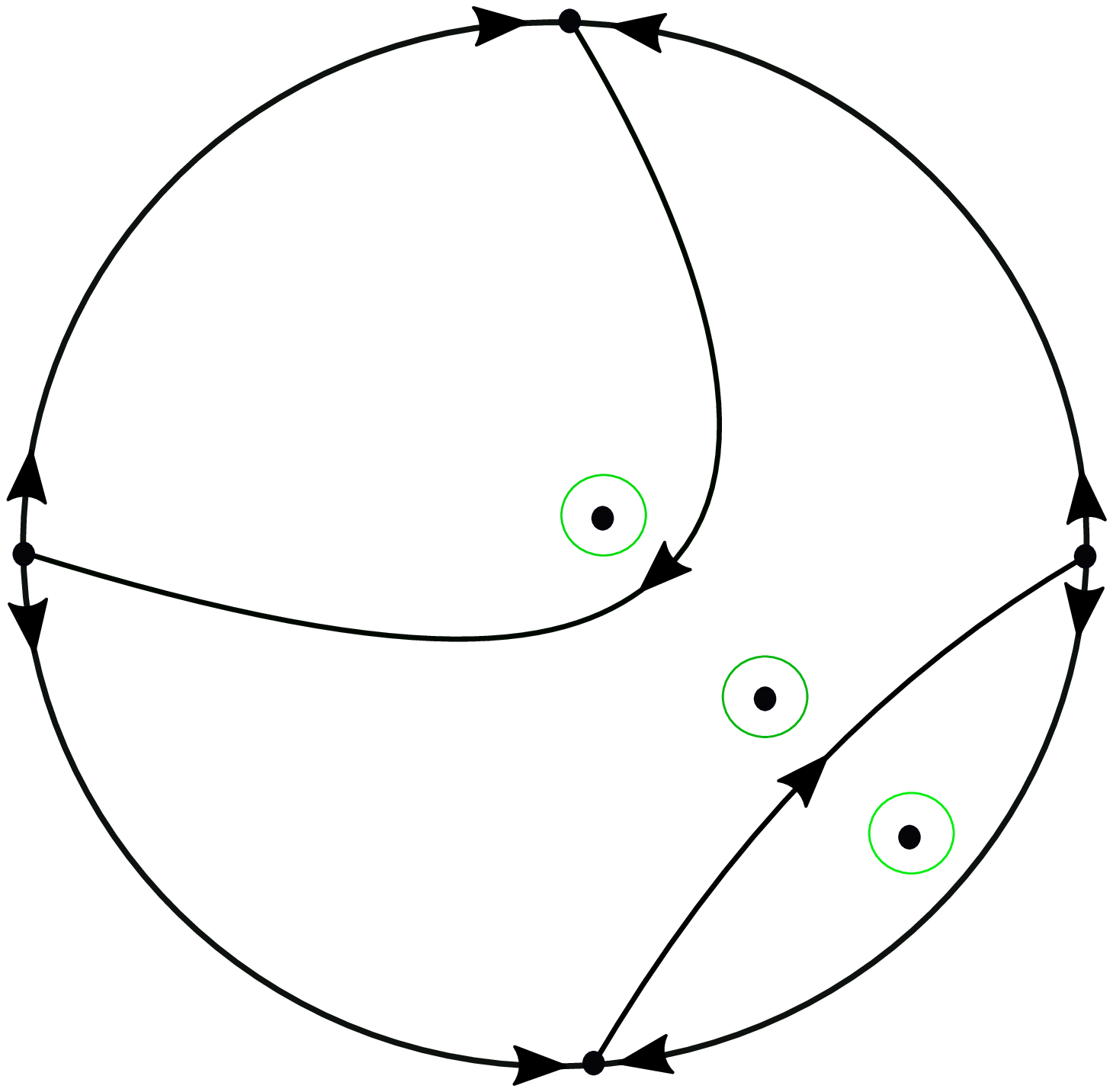} 
\put(50,-20){\textbf{c1)}}
\put(40,50){$p_{1}$}
\put(55,35){$p_{2}$}
\put(75,15){$p_{3}$}
\end{overpic}
\begin{overpic}[scale=0.25]
	{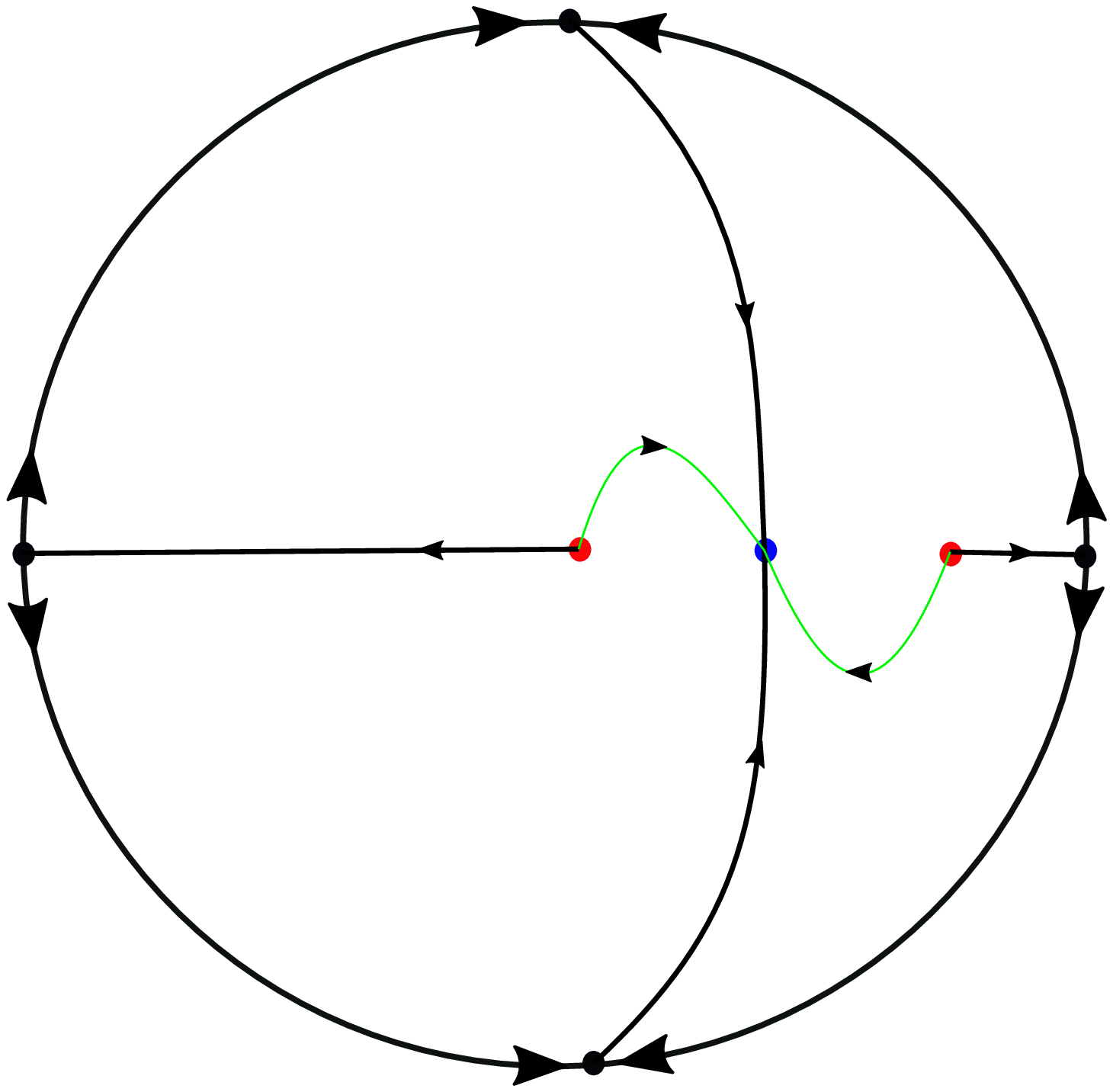} 
	\put(50,-20){\textbf{c2)}}
	\put(45,43){$p_{1}$}
	\put(72,55){$p_{2}$}
	\put(88,43){$p_{3}$}
\end{overpic}
\begin{overpic}[scale=0.25]
	{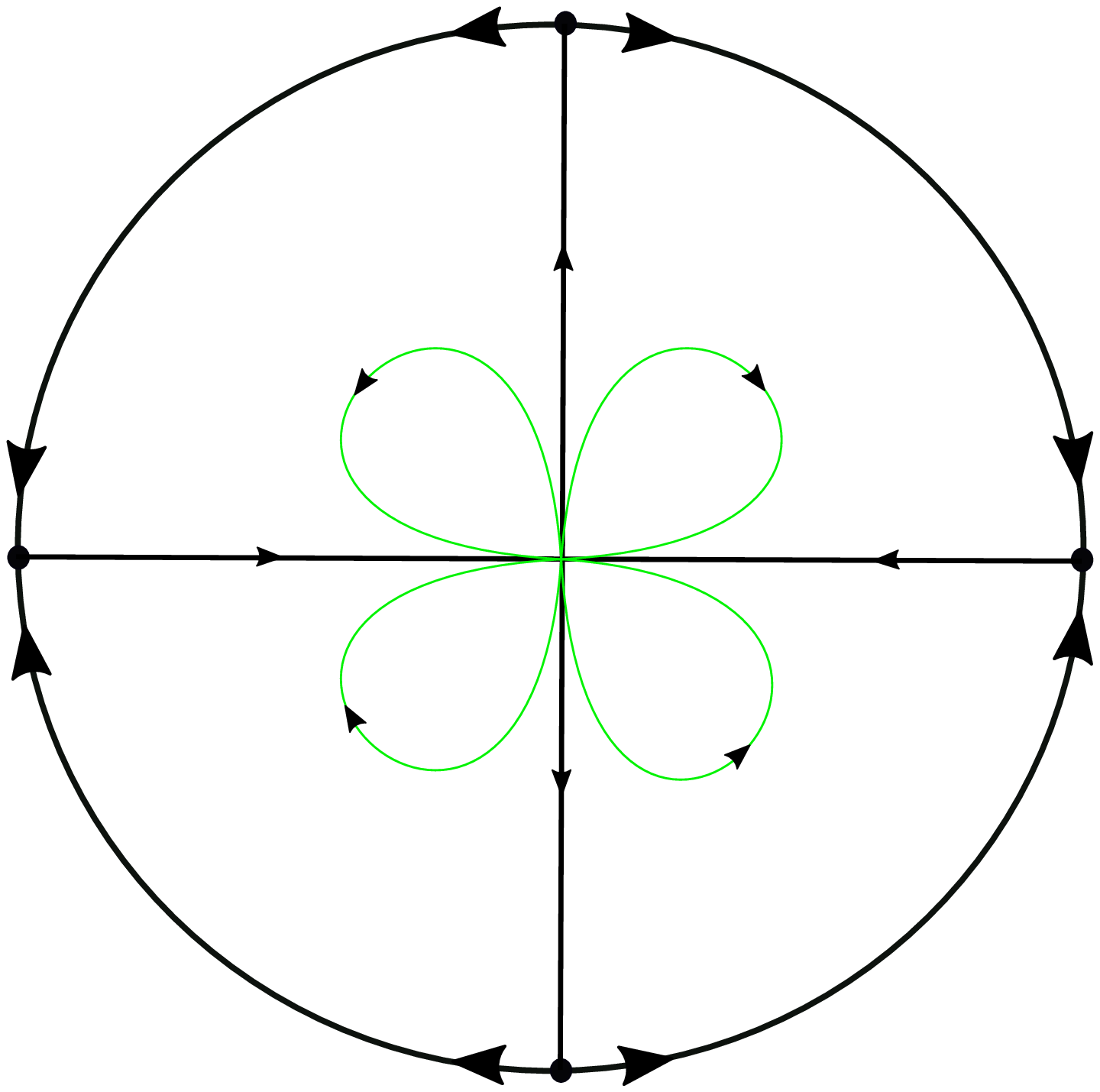} 
	\put(50,-20){\textbf{c3)}}
\end{overpic}

\vspace{1cm}

\begin{overpic}[scale=0.25]
	{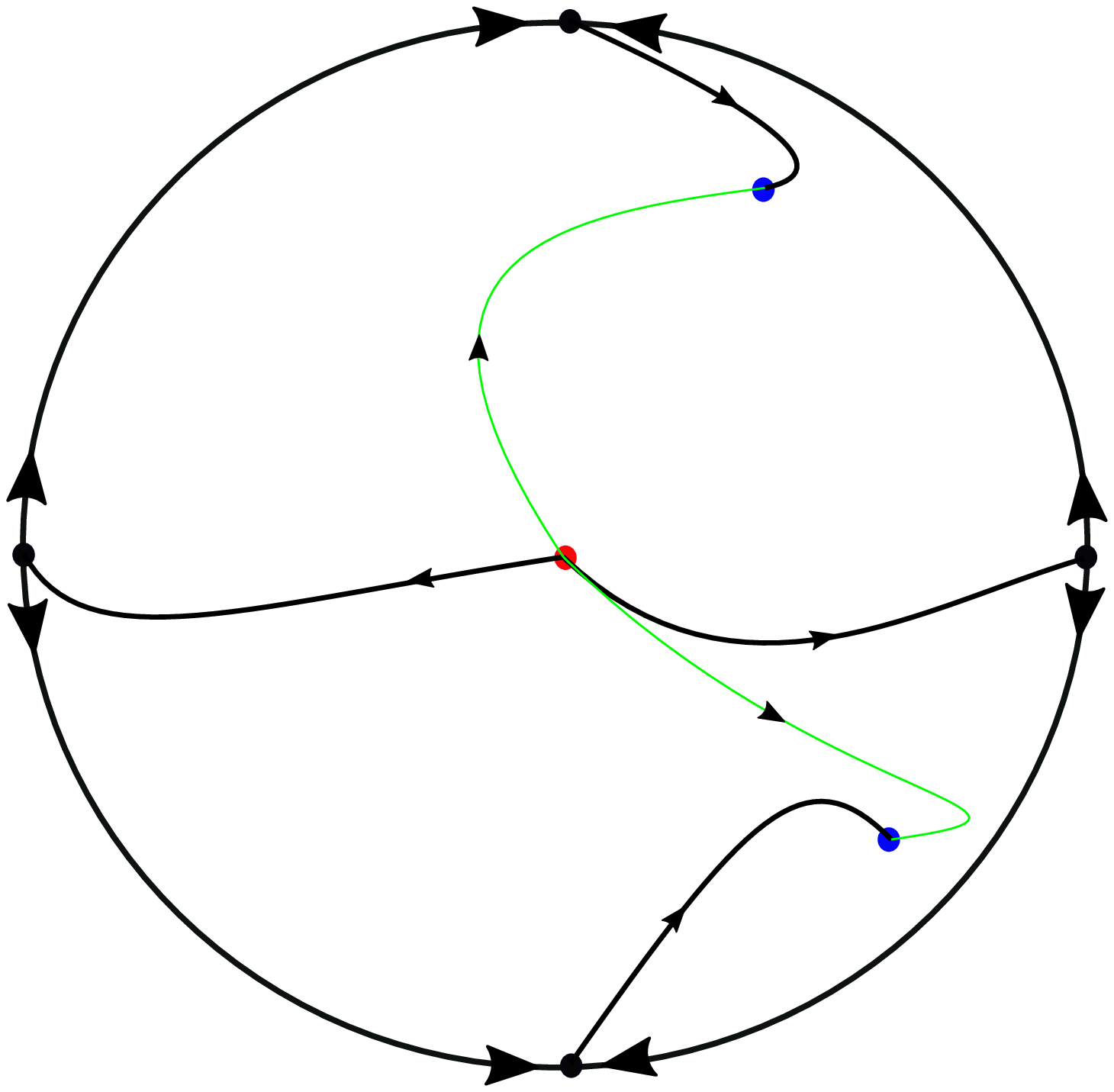} 
	\put(50,-20){\textbf{c4)}}
	\put(45,40){$p_{1}$}
	\put(72,79){$p_{2}$}
	\put(75,18){$p_{3}$}
\end{overpic}
\begin{overpic}[scale=0.25]
	{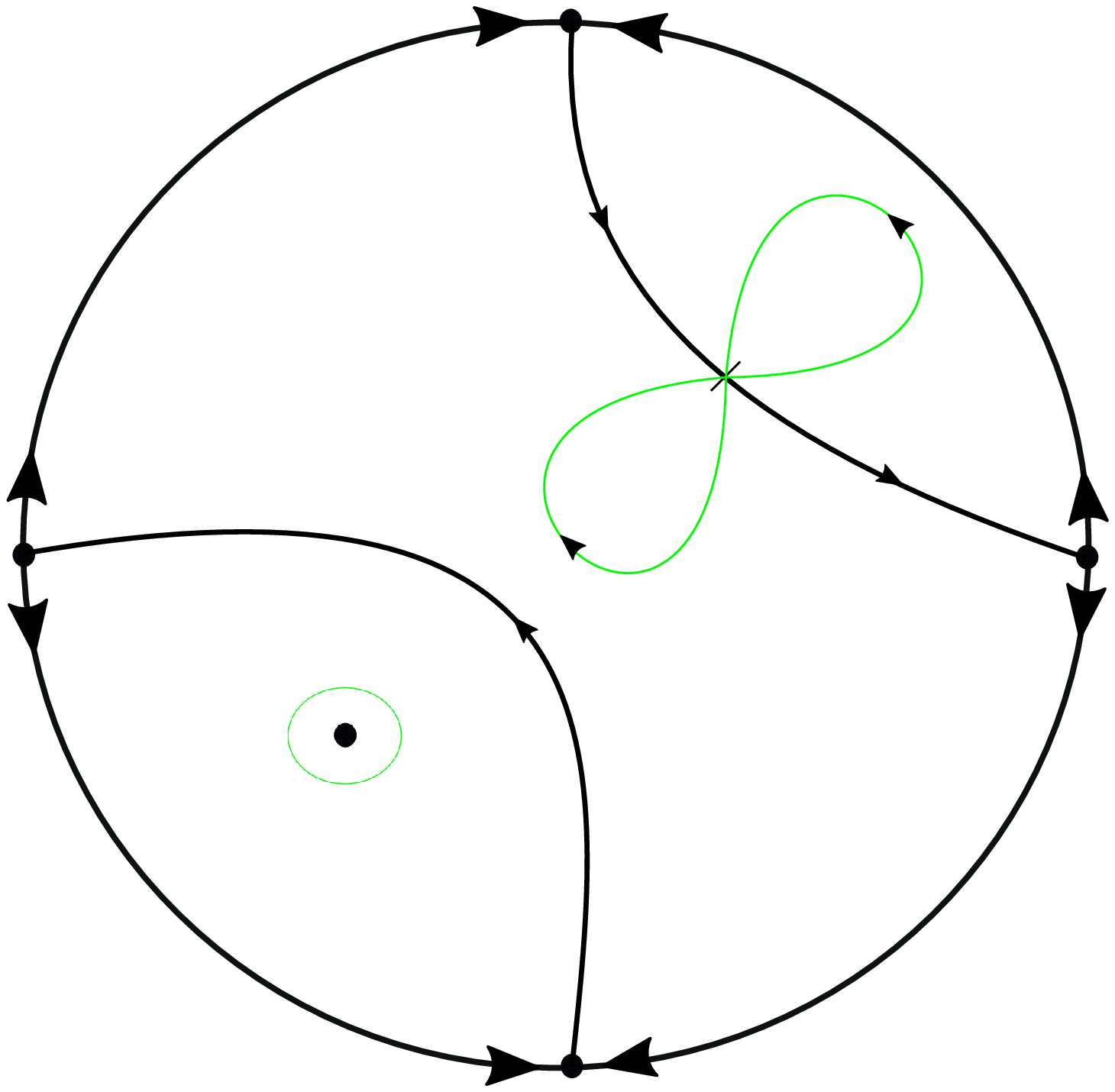} 
	\put(50,-20){\textbf{c5)}}
	\put(30,20){$p_{1}$}
\end{overpic}
\begin{overpic}[scale=0.25]
	{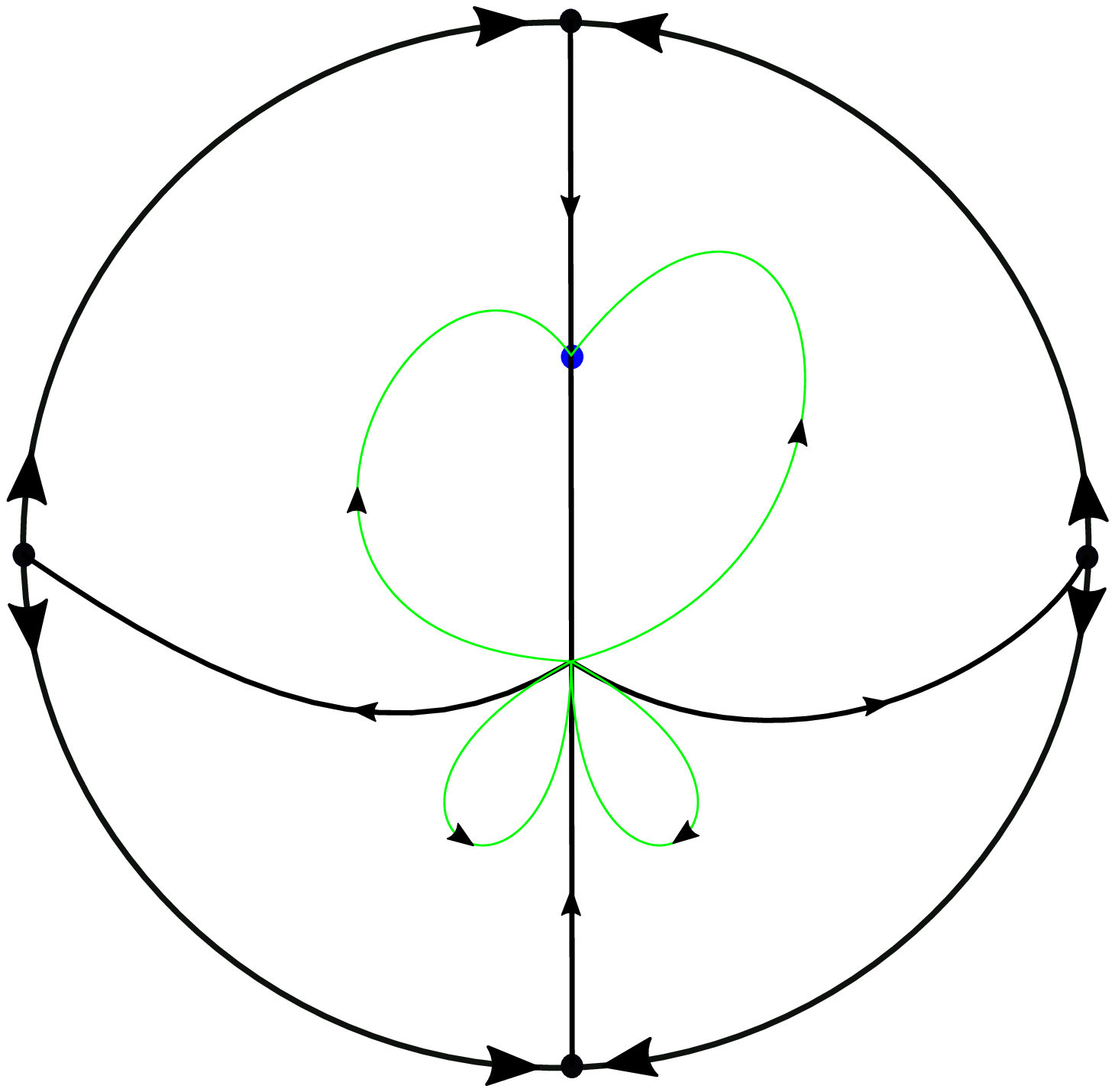} 
	\put(50,-20){\textbf{c6)}}
	\put(42,65){$p_{1}$}
\end{overpic}

\vspace{1cm}

\begin{overpic}[scale=0.25]
	{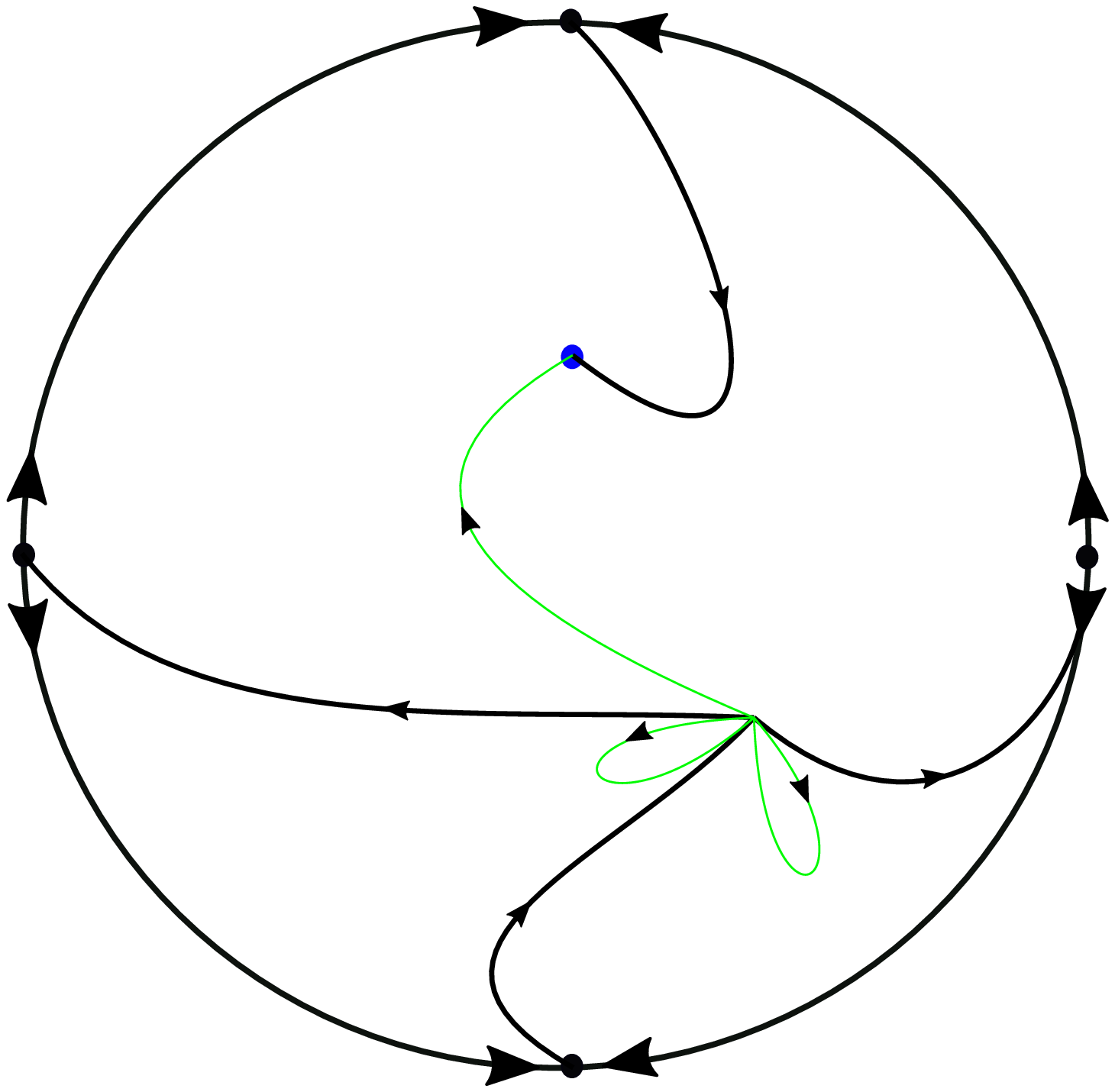} 
	\put(50,-20){\textbf{c7)}}
	\put(42,70){$p_{1}$}
\end{overpic}
\begin{overpic}[scale=0.25]
	{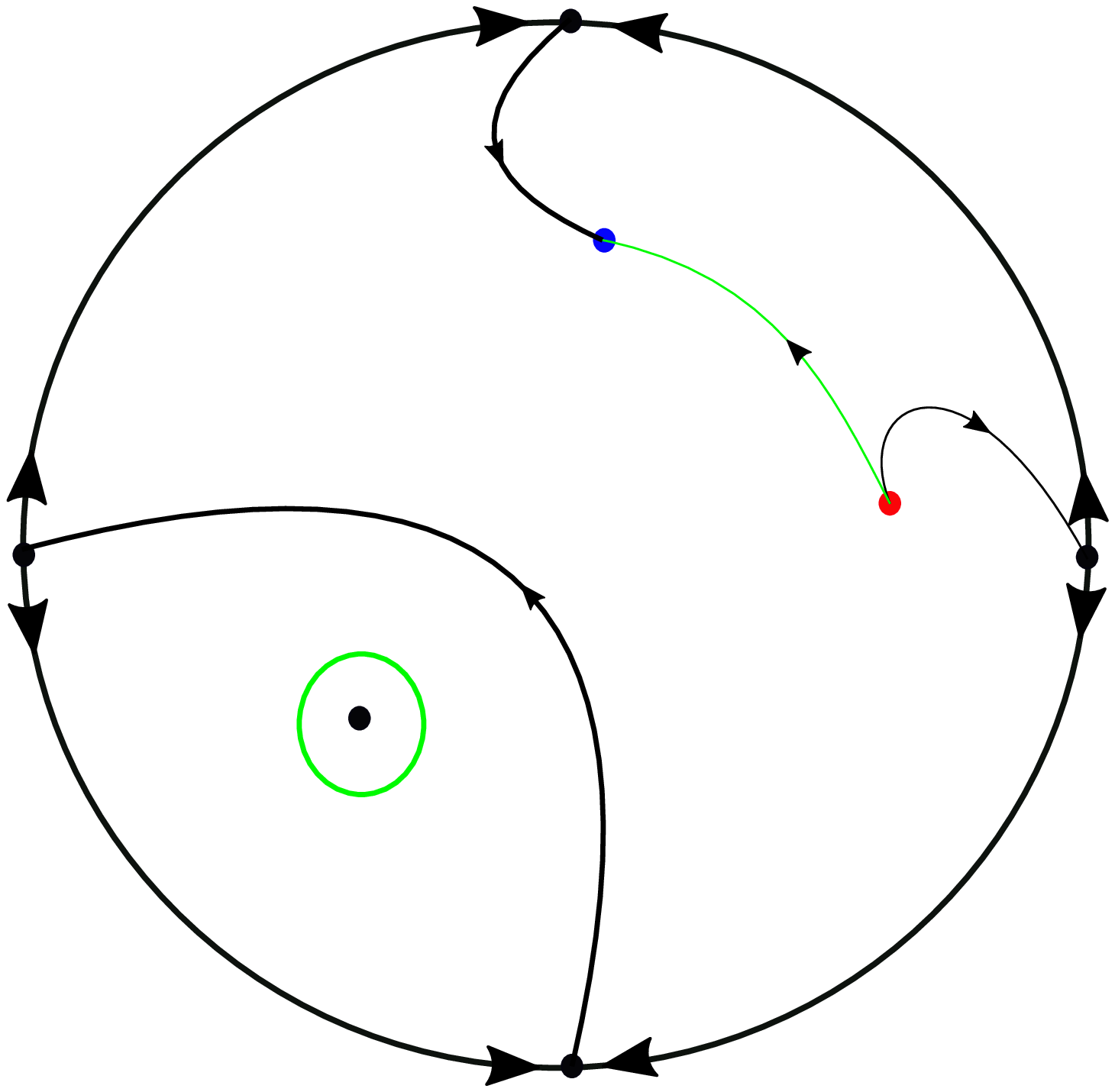} 
	\put(50,-20){\textbf{c8)}}
	\put(25,20){$p_{1}$}
	\put(72,55){$p_{2}$}
	\put(50,72){$p_{3}$}
	\end{overpic}
\begin{overpic}[scale=0.25]
	{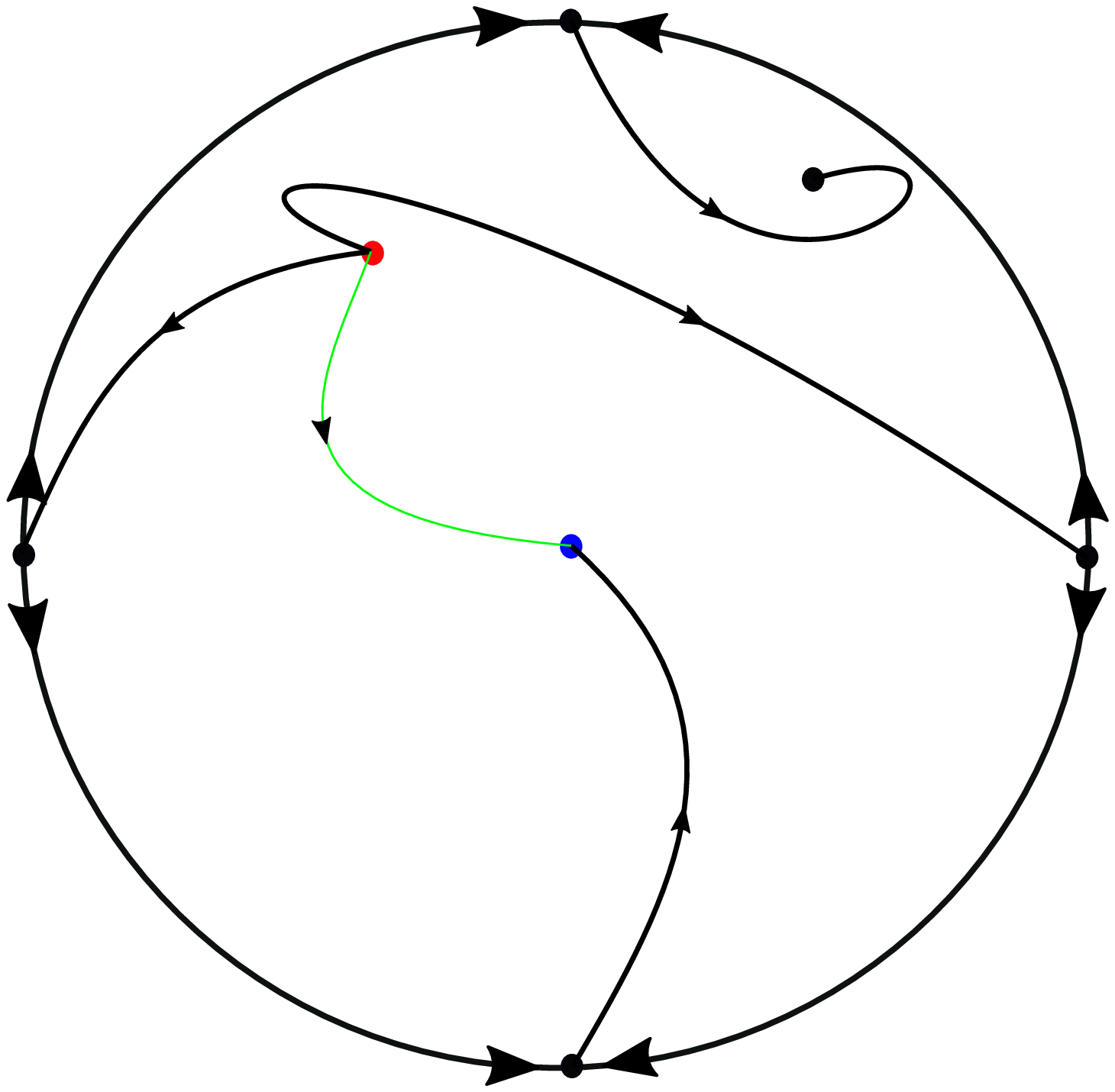} 
	\put(50,-20){\textbf{c9)}}
	\put(45,43){$p_{1}$}
	\put(68,90){$p_{2}$}
	\put(35,70){$p_{3}$}
\end{overpic}

\vspace{1cm}

\begin{remark}
	Without distinction between nodes and foci, the phase portraits represented by the Figures $\textbf{c2)}$, $\textbf{c4)}$ and $\textbf{c9)}$, are topologically equivalent. The phase portraits $\textbf{c6)}$ and $\textbf{c7)}$ are also topologically equivalent.
\end{remark}

\begin{proof}
To prove this result is enough to show that among all configurations possible, some of them are no-realizable and the remaining is obtained. All possible configurations are given by

\begin{multicols}{2}
\begin{enumerate}[(a)]
\item 2 centers and 1 node/focus,	
\item 2 nodes and 1 center/focus,
\item 3 centers,
\item 3 nodes,
\item 1 triple,
\item 3 foci,
\item 1 center and 1 double, 
\item 1 node and 1 double,
\item 1 focus and 1 double,
\item 1 center and 2 foci, 
\item 1 node and 2 foci,
\item 1 center, 1 node and 1 focus.
\end{enumerate}
\end{multicols}

The items $(a)$ and $(b)$ are no-realizable. From Proposition \ref{teoGas1} it is impossible to obtain two centers and one focus/node. In the same way, it is impossible to obtain two nodes and one focus/center. The remaining  cases, all of them are realizable and we are going to show examples.

According \cite{GAP}, there exists only possibility for three centers, see the phase portrait $\textbf{c1)}$. Using Proposition \ref{teoGas1}, we have Figure $\textbf{c2)}$. Using Proposition \ref{sver1}, we obtain Figure $\textbf{c3)}$. For the remaining cases, it is sufficient to analyze the possibilities of the separatrices. Before we continue the proof, let us analyze  Figure $\textbf{c2)}$. 

\begin{center}
	\begin{figure}[h]
\begin{overpic}[scale=0.25]
	{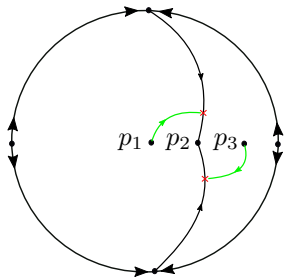} 
	\put(40,-20){}
	\put(42,50){$p_{1}$}
	\put(60,50){$p_{2}$}
	\put(78,50){$p_{3}$}
\end{overpic}
\caption{Impossible orbit paths.}
\end{figure}
\end{center}

Since $p_{2}$ is an attracting node, it  is the $\omega-limit$ of $I_{1}$ and $I_{3}$. Then, there is no orbit coming out from $p_{1}$ and arriving at  $I_{2}$. Analogously, the same conclusion is true for any orbit coming out from $p_{3}$ and arriving at $I_{4}$. Therefore, the only phase portrait possible is given by $\textbf{c2)}$.

In the case $(f)$ we have that $p_{1}$ is a repelling focus and $p_{2}$ and $p_{3}$  are attracting foci. The orbits coming from $p_{1}$ have $I_{2}$, $I_{4}$, $p_{2}$, and $p_{3}$ as $\omega-limit$ . The orbits coming from $I_{1}$ and $I_{3}$ have $p_{2}$ and $p_{3}$ as $\omega-limit$. See, Figure $\textbf{c4)}$. 

In the case $(g)$, we have one double point formed by two elliptic sectors separated by parabolic sectors and one center. See Figure $\textbf{c5)}$.

In the case $(h)$, we have one node and one double point. The double point is formed by two elliptic sectors separated by parabolic sectors. The attracting node is the $\omega-limit$ of the orbits coming from the infinity and from the double point.  See $\textbf{c6)}$. The same happens for the case $(i)$. See Figure $\textbf{c7)}$. 

In the case $(j)$, we have one center  $p_{1}$, and two foci. The point  $p_{2}$ is repelling, and the point $p_{3}$ is attracting. Since $p_{3}$ is attracting, the orbits coming from $I_{1}$ and $p_{2}$ have $p_{3}$ as  $\omega-limit$. We also have that $I_{2}$ is the  $\omega-limit$ of the orbits coming from $p_{2}$. See Figure $\textbf{c8)}$.

In the case $(k)$, we have that $p_{1}$ is an attracting node, $p_{2}$ is an attracting focus, and $p_{3}$ is a repelling focus. The orbits coming from $p_{3}$ have, $I_{2}$, $I_{4}$, $p_{1}$, and $p_{2}$ as $\omega-limit$. The orbits coming from $I_{3}$ and $I_{1}$ have $p_{1}$ and $p_{2}$ as $\omega-limit$. See Figure $\textbf{c9)}$.

Finally, in the case $(l)$, we have a center $p_{1}$,  a repelling node $p_{2}$  and an attracting focus  $p_{3}$. The point $p_{3}$ is the $\omega-limit$ of the orbits coming from $I_{1}$ and $p_{2}$. We also have  that $I_{2}$ is  $\omega-limit$ of the orbit coming from $p_{2}$. See Figure $\textbf{c8)}$.
\end{proof}

\begin{remark}
	This result is similar to the result showed in \cite{LliValls}, where the authors study the phase portrait of Abel polynomial systems, that is, $F(z)=(z-A)(z-B)(z-C)$, where $A$, $B$ and $C$ are complex numbers. However, using remark \eqref{remark1}, to studying $F(z)=(z-A)(z-B)(z-C)$ is the same to as studying $F(z)=z(z-\tilde{A})(z-\tilde{B})$. In \cite{LliValls} the authors say that are not possible the case $l$. Since we work with two less parameters, we can find conditions to exist 1 center, 1 node and 1 focus. Moreover, here we use $F'(0), F'(\tilde{A})$ and $F'(\tilde{B})$ to find the eigenvalues.  
\end{remark}

\begin{remark}\label{remarke1}
	In order to obtain the phase portrait in Figures  $\textbf{1)}-\textbf{9)}$, we take the following parameter values:
	
%	$a_{1}=1, a_{2}=3, b_{1}=-1,b_{2}=-3.$
%	\item 3 nodes - $a_{1}=1, a_{2}=2, b_{1}=0,b_{2}=0.$
%	\item 1 triple - $a_{1}=0, a_{2}=0, b_{1}=0,b_{2}=0.$
%	\item 3 focus - $a_{1}=1, a_{2}=2, b_{1}=3,b_{2}=-4.$
%	\item 1 center and 1 double - $a_{1}=1, a_{2}=1, b_{1}=1,b_{2}=1.$
%	\item 1 node and 1 double - $a_{1}=0, a_{2}=0, b_{1}=-1,b_{2}=-1.$ 
%	\item 1 focus and 1 double - $a_{1}=1, a_{2}=1, b_{1}=-3,b_{2}=-3.$
%	\item 1 center and 2 foci - $a_{1}=3/2, a_{2}=2, b_{1}=1,b_{2}=3.$
%	\item 1 node and 2 foci - $a_{1}=-2/3, a_{2}=2, b_{1}=1,b_{2}=3.$
%		\item 1 center, 1 node and 1 focus - $z(z-(1-i)(z-(-1-3i)))$$a_{1}=2, a_{2}=36/25, b_{1}=3/2,b_{2}=48/25.$
	
		\begin{enumerate}[$(a)$]
		\item 2 centers and 1 node/focus: no-realizable.	
		\item 2 nodes and 1 center/focus: no-realizable.
		\item 3 centers: $\dot{z}=z(z-(1+i))(z-(3+3i)).$ 
		\item 3 nodes:  $\dot{z}=z(z-1))(z-2).$
		\item 1 triple: $\dot{z}=z^3.$
		\item 3 foci: $\dot{z}=z(z-(1-3i))(z-(2+4i)).$
		\item 1 center and 1 double:  $\dot{z}=z(z-(1-i))(z-(1-i)).$
		\item 1 node and 1 double: $\dot{z}=z(z+i)(z+i).$ 
		\item 1 focus and 1 double: $\dot{z}=z(z-(1+3i))(z-(1+3i)).$
		\item 1 center and 2 foci:  $\dot{z}=z(z-(3/2-i))(z-(2-3i)).$
		\item 1 node and 2 foci:  $\dot{z}=z(z-(-2/3-i))(z-(2-3i)).$
		\item 1 center, 1 node and 1 focus: $\dot{z}=z(z-(2-3/2i))(z-(36/25-48/25i)).$
	\end{enumerate}
\end{remark}

\section{Quartic Holomorphic Polynomial System }\label{sec5}
Consider a quartic holomorphic polynomial function

\begin{equation}
F(z)=A_0+A_1z+A_2z^2+A_{3}z^3+A_{4}z^{4}, \quad A_k=a_k+ib_k,\quad  A_4\neq 0\label{holquarvec}.
\end{equation}
By the Remark \ref{remark1}, it is enough study the system $F(z)=z(z-B_{1})(z-B_{2})(z-B_{3})$. 
%Its trajectory are the solutions of the differential system
%\begin{equation}\left\{\begin{array}{ll}\label{eq2}
%\dot{x}&=-a_{1}a_{2}a_{3}x+a_{1}a_{2}b_{3}y+a_{1}a_{2}x^2-a_{1}a_{2}y^2+a_{1}a_{3}b_{2}y+a_{1}a_{3}x^2-a_{1}a_{3}y^2\\
%&\phantom{=}+a_{1}b_{2}b_{3}x-2a_{1}b_{2}xy-2a_{1}b_{3}xy-a_{1}x^3+3a_{1}xy^2+a_{2}a_{3}b_{1}y+a_{2}a_{3}x^2\\
%&\phantom{=}-a_{2}a_{3}y^2+a_{2}b_{1}b_{3}x-2a_{2}b_{1}xy-2a_{2}b_{3}xy-a_{2}x^3+3a_{2}xy^2+a_{3}b_{1}b_{2}x\\
%&\phantom{=}-2a_{3}b_{1}xy -2a_{3}b_{2}xy-a_{3}x^3+3a_{3}xy^2-b_{1}b_{2}b_{3}y-b_{1}b_{2}x^2+b_{1}b_{2}y^2\\
%&\phantom{=}-b_{1}b_{3}x^2+b_{1}b_{3}y^2+3b_{1}x^2y-b_{1}y^3-b_{2}b_{3}x^2+b_{2}b_{3}y^2+3b_{2}x^2y-b_{2}y^3\\
%&\phantom{=}+3b_{3}x^2y-b_{3}y^3+x^4-6x^2y^2+y^4,\\
%\dot{y}&=-a_{1}a_{2}a_{3}y-a_{1}a_{2}b_{3}x+2a_{1}a_{2}xy-a_{1}a_{3}b_{2}x+2a_{1}a_{3}xy+a_{1}b_{2}b_{3}y+a_{1}b_{2}x^2\\
%&\phantom{=}-a_{1}b_{2}y^2+a_{1}b_{3}x^2-a_{1}b_{3}y^2-3a_{1}x^2y+a_{1}y^3-a_{2}a_{3}b_{1}x+2a_{2}a_{3}xy+a_{2}b_{1}b_{3}y\\
%&\phantom{=}+a_{2}b_{1}x^2-a_{2}b_{1}y^2+a_{2}b_{3}x^2-a_{2}b_{3}y^2-3a_{2}x^2y+a_{2}y^3+a_{3}b_{1}b_{2}y+a_{3}b_{1}x^2\\
%&\phantom{=}-a_{3}b_{1}y^2+a_{3}b_{2}x^2-a_{3}b_{2}y^2-3a_{3}x^2y+a_{3}y^3+b_{1}b_{2}b_{3}x-2b_{1}b_{2}xy-2b_{1}b_{3}xy\\
%&\phantom{=}-b_{1}x^3+3b_{1}xy^2-2b_{2}b_{3}xy-b_{2}x^3+3b_{2}xy^2-b_{3}x^3+3b_{3}xy^2+4x^3y-4xy^3.
%\end{array}
%\right.
%\end{equation}

It is clear that the roots of $F(z)=z(z-B_{1})(z-B_{2})(z-B_{3})$ are $0$, $B_{1}$, $B_{2}$ and $B_{3}$ and the associated eigenvalues  are $v_{0}=-B_{1}B_{2}B_{3}$, $v_{B_{1}}=B_{1}^3 - B_{1}^2B_{2} - B_{1}^2B_{3}+B_{1}B_{2}B_{3}$,  $v_{B_{2}}=-B_{1}B_{2}^2 + B_{1}B_{2}B_{3} + B_{2}^3 - B_{2}^2B_{3}$ and $v_{B_{3}}=B_{1}B_{2}B_{3} - B_{1}B_{3}^2 - B_{2}B_{3}^2 + B_{3}^3$ and their respective conjugates.

\begin{remark}
As in Section \ref{sec4}, it is easy to find six saddle points at infinity.
\end{remark}
\begin{center}
	\begin{figure}[h]
\begin{overpic}[scale=0.2]
	{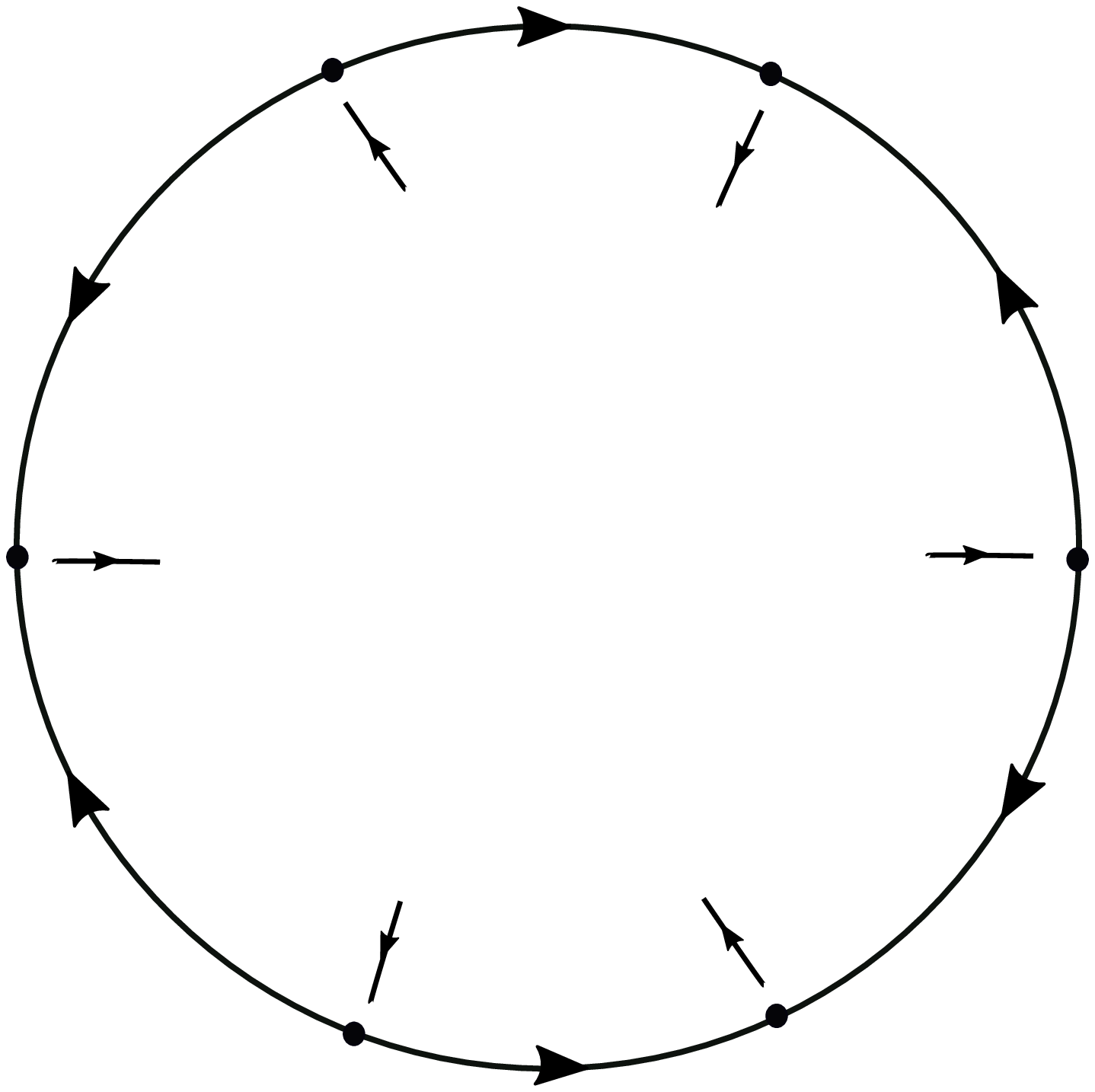} 
	\put(15,80){$I_{1}$}
	\put(60,80){$I_{2}$}
	\put(85,35){$I_{3}$}
	\put(60,-5){$I_{4}$}
	\put(10,-5){$I_{5}$}
	\put(-10,40){$I_{6}$}
\end{overpic}
\caption{Saddle points at infinity for a quartic system.}
\end{figure}
\end{center}

\begin{remark}\label{remakfoci}
	As we can see in \cite{AlvGasPro}, when we have four equilibrium points of type focus, the points can take four different geometrical distributions.
	
	\begin{itemize}
		\item Collinear: all of them are aligned.
		\item Triangle: three on the vertices of a triangle and the other one inside.
		\item Border: three aligned and the other one not.
		\item Quadrilateral: on the vertices of a quadrilateral.
	\end{itemize} 
	
\begin{figure}
\begin{center}
	\begin{overpic}[scale=0.15]
		{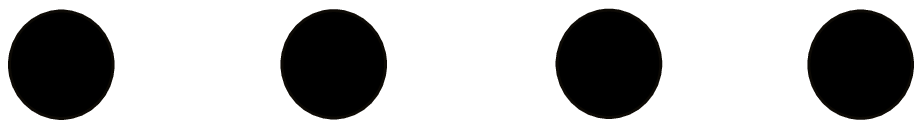} 
		\put(-5,-20){\textbf{Collinear}}
	\end{overpic}
	\hspace{0.5cm}
	\begin{overpic}[scale=0.15]
		{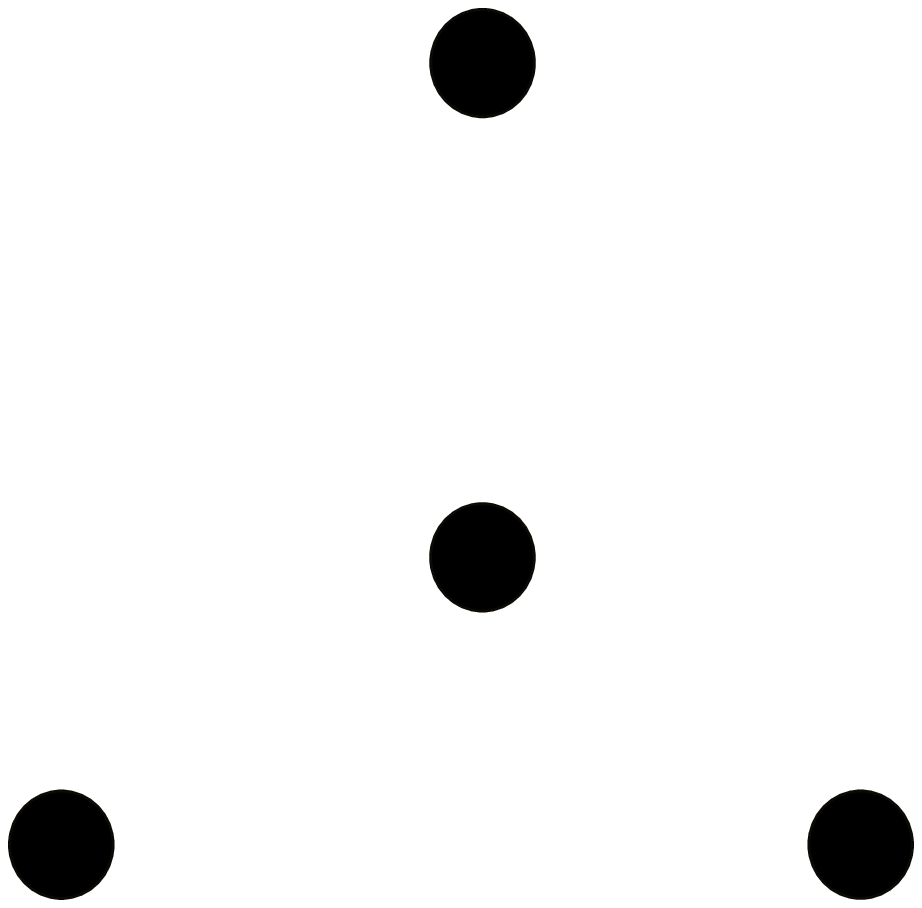} 
		\put(0,-20){\textbf{Triangle}}
	\end{overpic}
	\hspace{0.5cm}
	\begin{overpic}[scale=0.15]
		{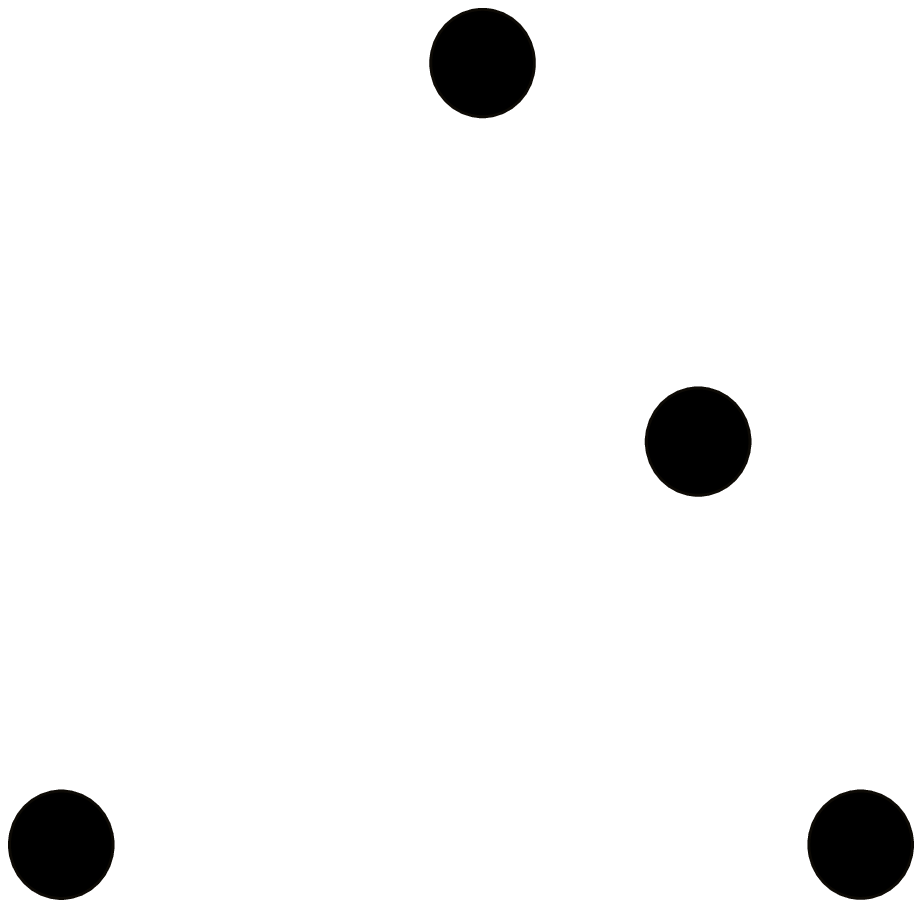} 
		\put(0,-20){\textbf{Border}}
	\end{overpic}
	\hspace{0.5cm}
	\begin{overpic}[scale=0.15]
		{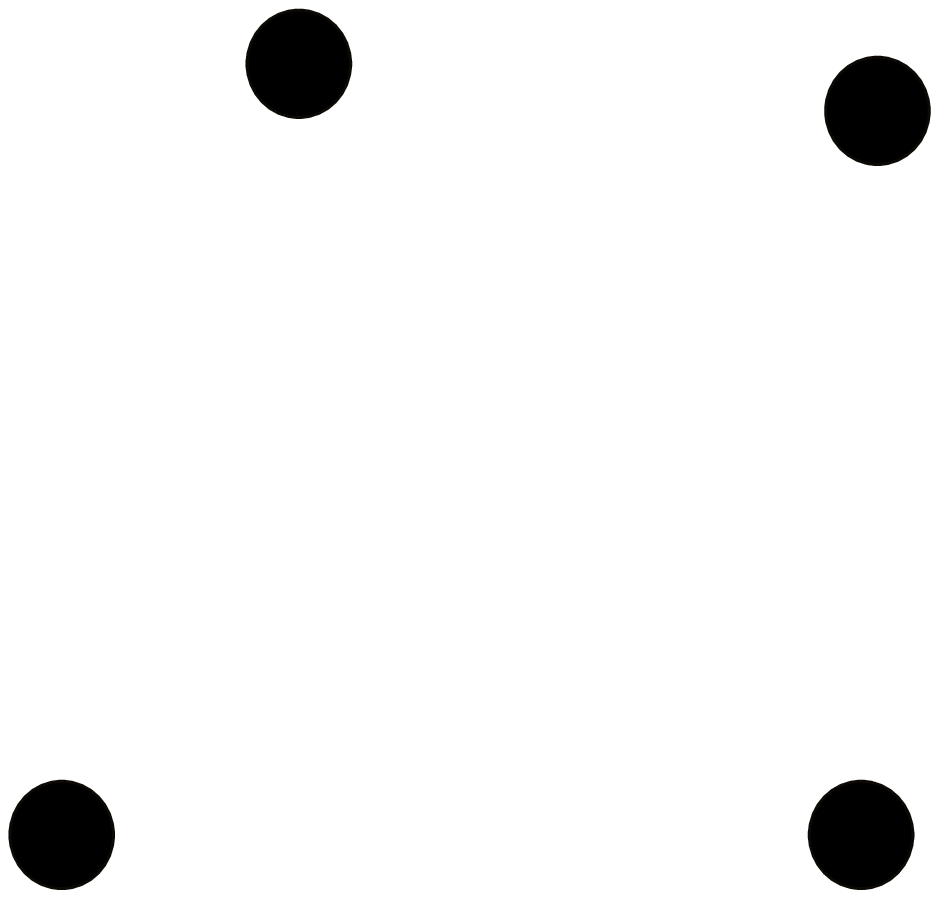} 
		\put(-10,-20){\textbf{Quadrilateral}}
	\end{overpic}
\end{center}
\vspace{0.5cm}
\caption{Geometrical distributions.}
\end{figure}
\end{remark}

\vspace{0.5cm}

\begin{proposition}If a quartic system has four equilibrium points, all of them of focus type, then  its phase portrait is given one of the figures \textbf{Q22)}-\textbf{Q27)}.
\end{proposition}
See \cite{AlvGasPro} for a proof.\\

\begin{theorem}\label{teoquartphase}
	If we distinguish nodes and foci, there are  twenty-nine topologically different phase portraits of the quartic holomorphic system. Without this distinction, there are twenty two phase portraits. See figures below.
\end{theorem}

\begin{overpic}[scale=0.25]
	{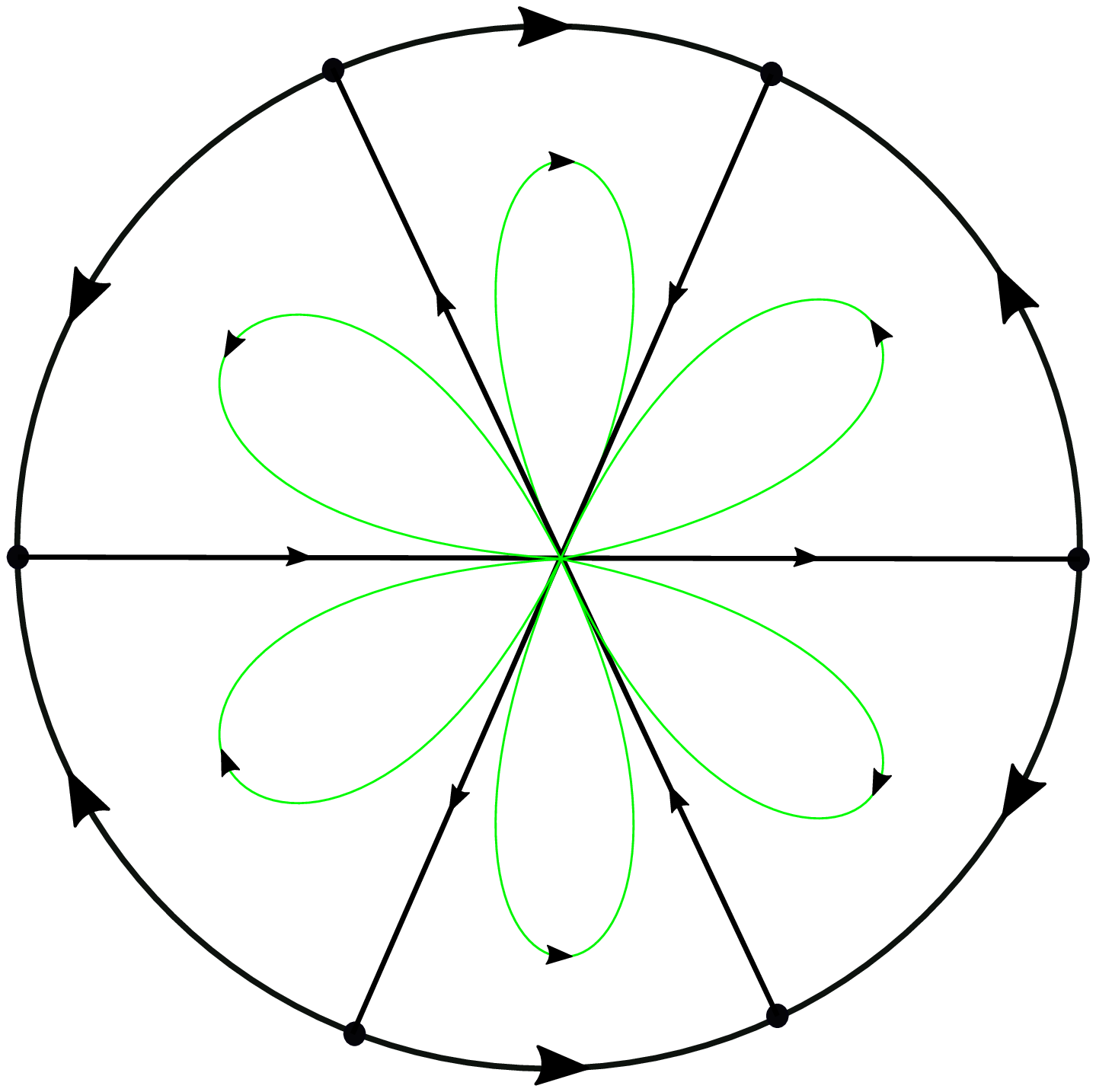} 
	\put(50,-20){\textbf{Q1)}}
\end{overpic}
\begin{overpic}[scale=0.25]
	{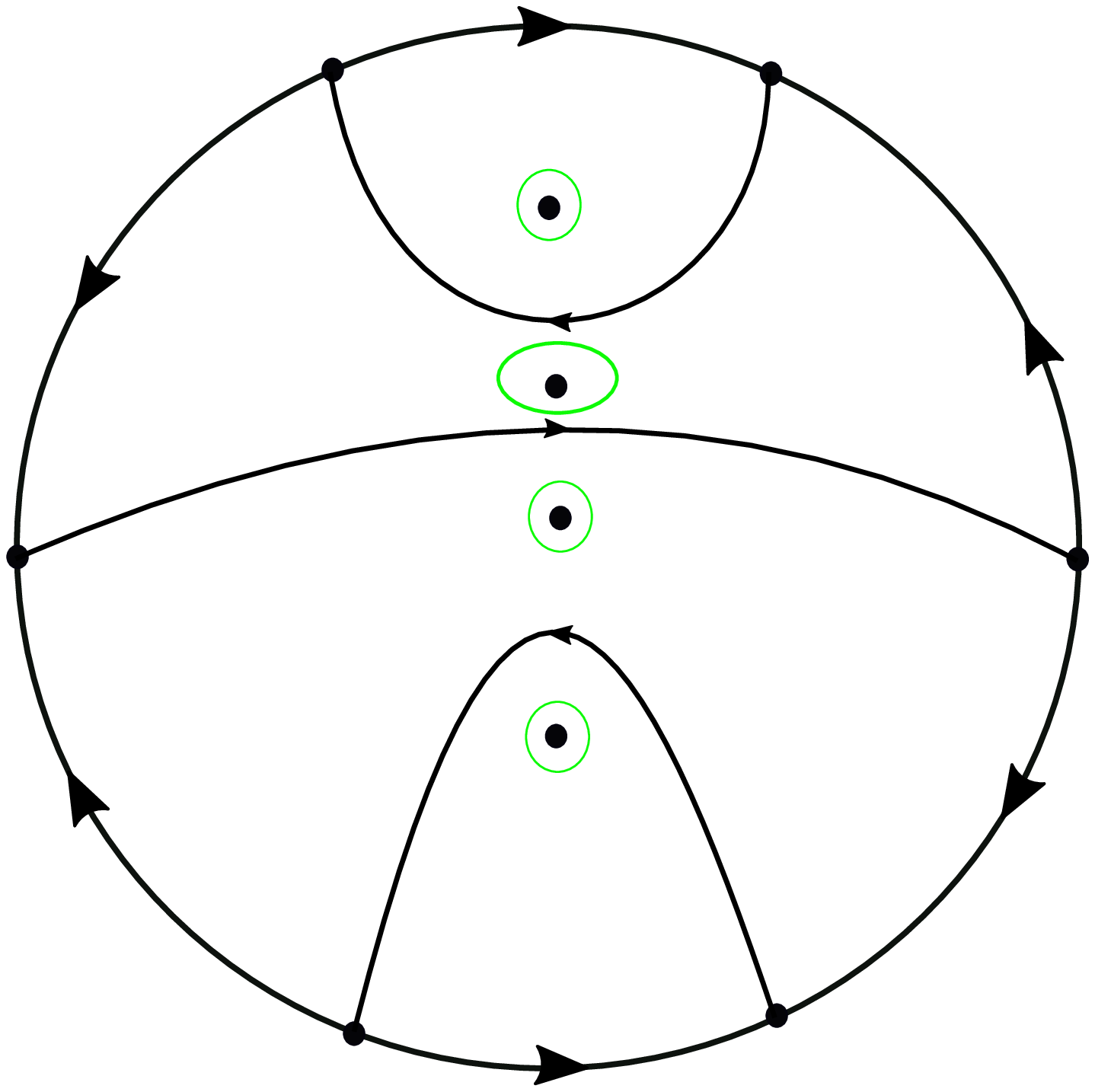} 
	\put(50,-20){\textbf{Q2)}}
	\put(50,24){$p_{1}$}
	\put(40,50){$p_{2}$}
	\put(35,65){$p_{3}$}
	\put(40,80){$p_{4}$}
\end{overpic}
\begin{overpic}[scale=0.25]
	{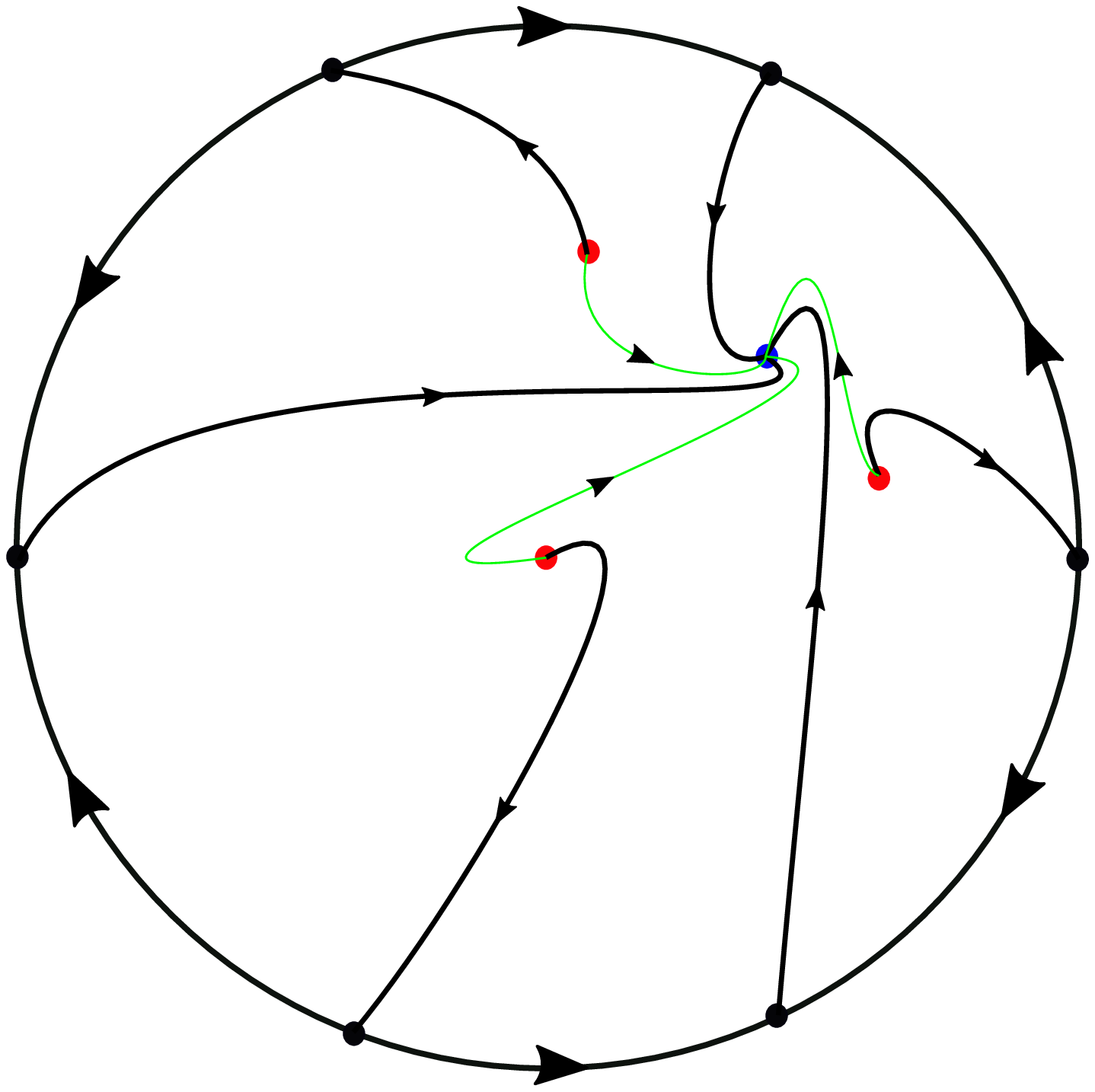} 
	\put(50,-20){\textbf{Q3)}}
	\put(40,44){$p_{1}$}
	\put(85,50){$p_{2}$}
	\put(70,80){$p_{3}$}
	\put(40,80){$p_{4}$}
\end{overpic}

\vspace{1cm}

\begin{overpic}[scale=0.25]
	{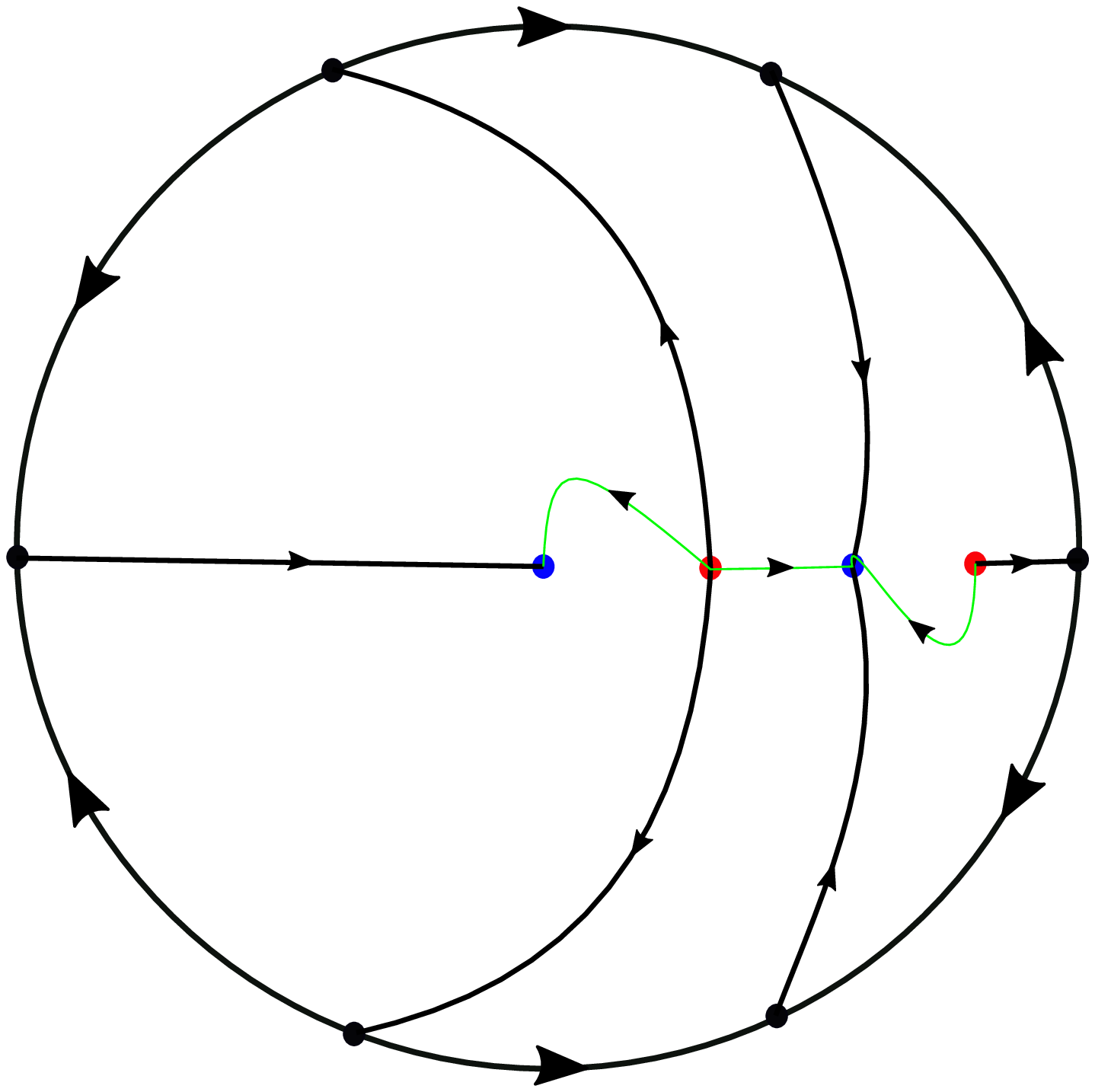} 
	\put(50,-20){\textbf{Q4)}}
	\put(40,40){$p_{1}$}
	\put(55,40){$p_{2}$}
	\put(70,40){$p_{3}$}
	\put(85,55){$p_{4}$}
\end{overpic}
\begin{overpic}[scale=0.25]
	{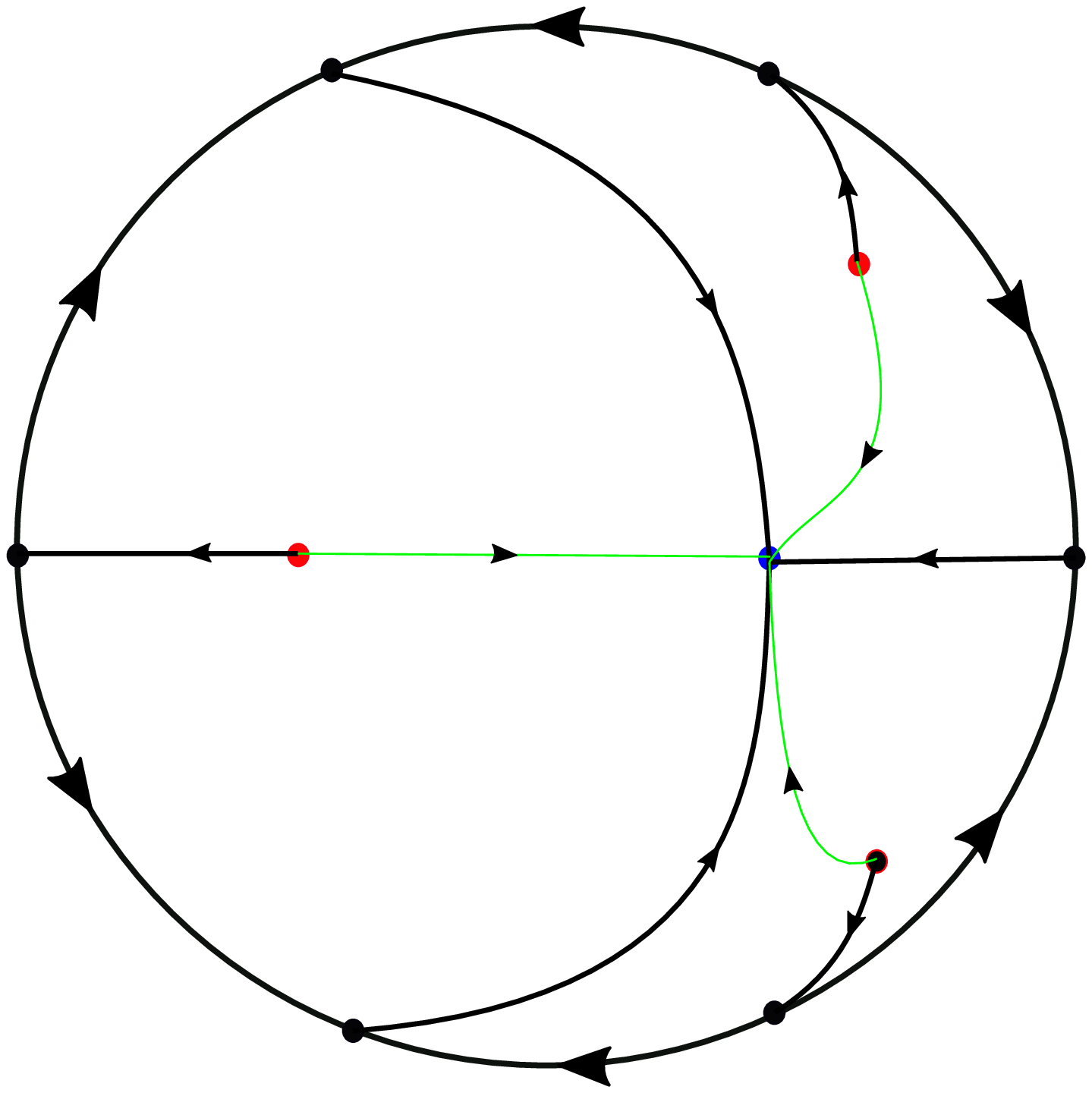} 
	\put(50,-20){\textbf{Q5)}}
	\put(60,50){$p_{1}$}
	\put(20,40){$p_{2}$}
	\put(70,15){$p_{3}$}
	\put(85,70){$p_{4}$}
\end{overpic}
\begin{overpic}[scale=0.25]
	{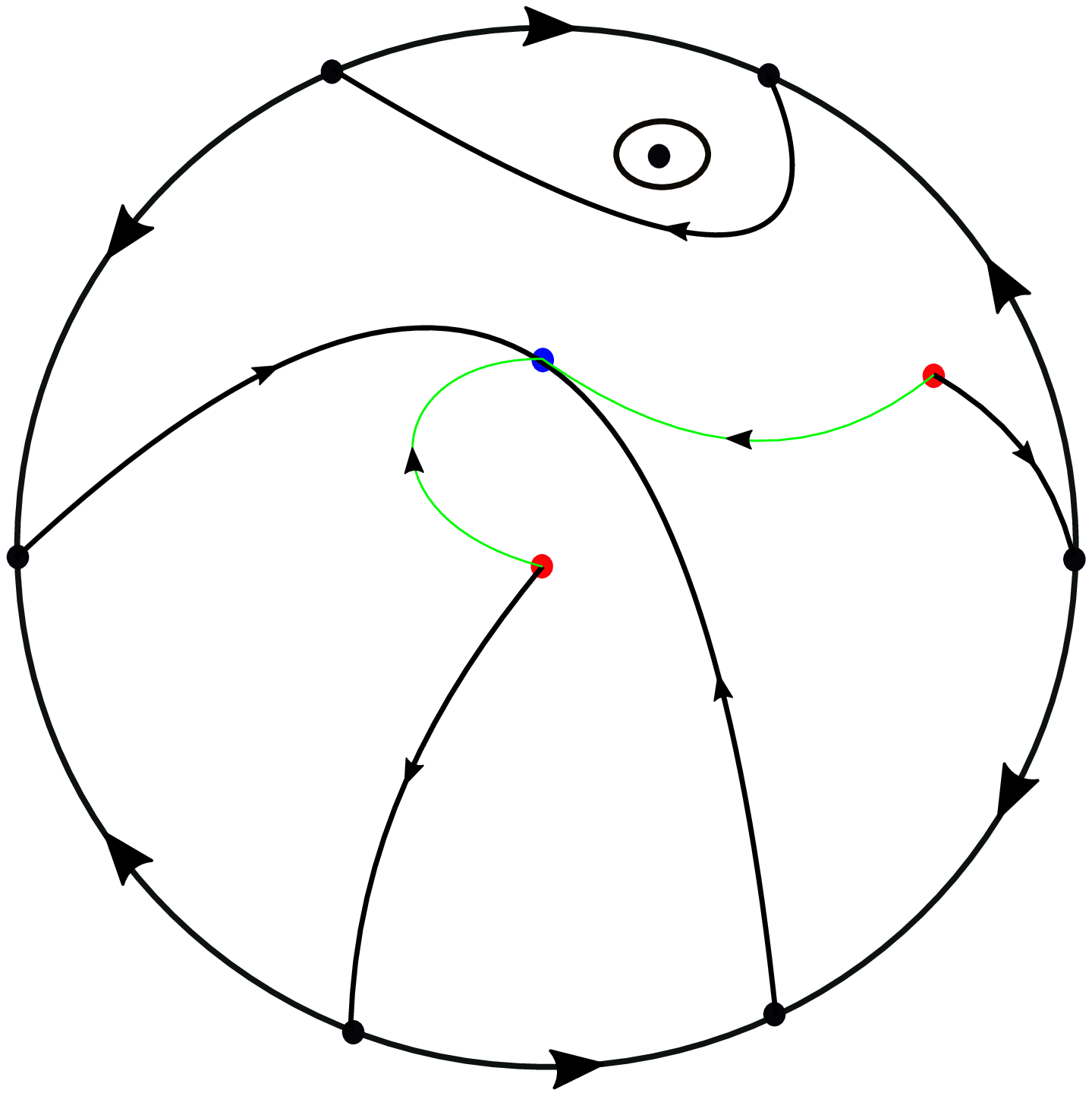} 
	\put(50,-20){\textbf{Q6)}}
	\put(50,40){$p_{1}$}
	\put(50,75){$p_{2}$}
	\put(50,95){$p_{3}$}
	\put(85,70){$p_{4}$}
\end{overpic}

\vspace{1cm}

\begin{overpic}[scale=0.25]
	{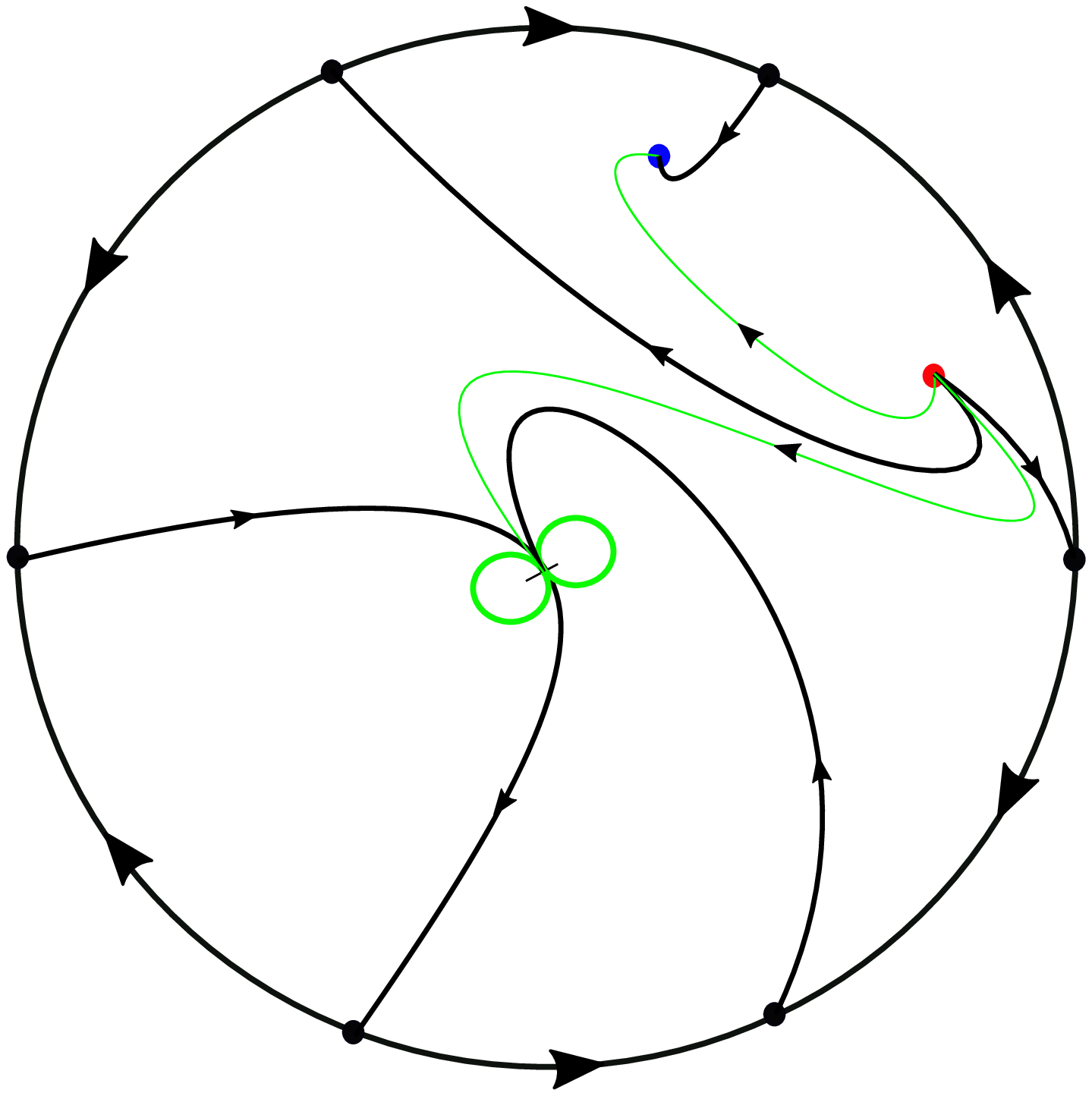} 
	\put(50,-20){\textbf{Q7)}}
	\put(80,75){$p_{1}$}
	\put(50,92){$p_{2}$}
\end{overpic}
\begin{overpic}[scale=0.25]
	{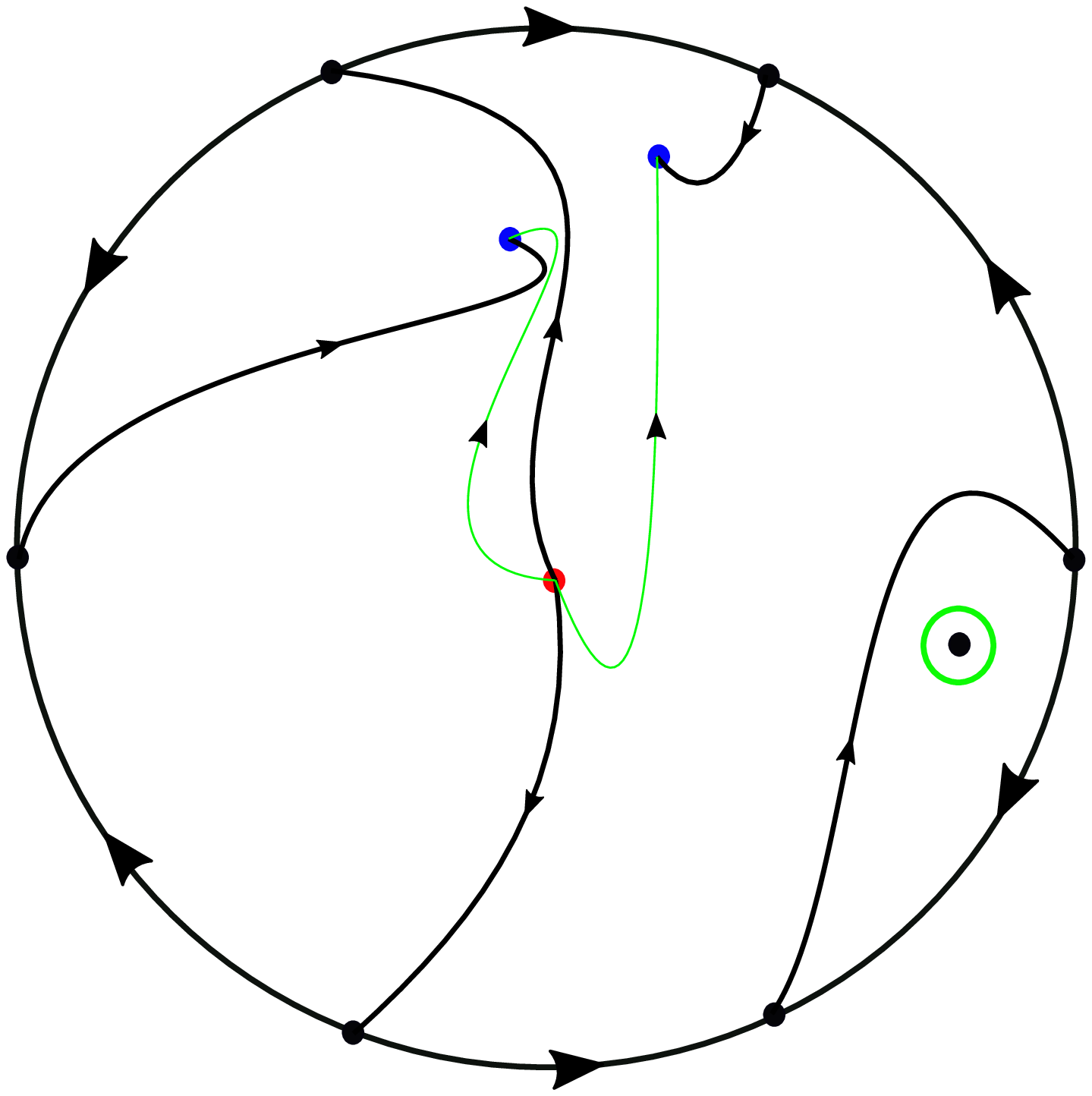} 
	\put(50,-20){\textbf{Q8)}}
	\put(40,40){$p_{1}$}
	\put(35,80){$p_{2}$}
	\put(52,93){$p_{3}$}
	\put(85,30){$p_{4}$}
\end{overpic}
\begin{overpic}[scale=0.25]
	{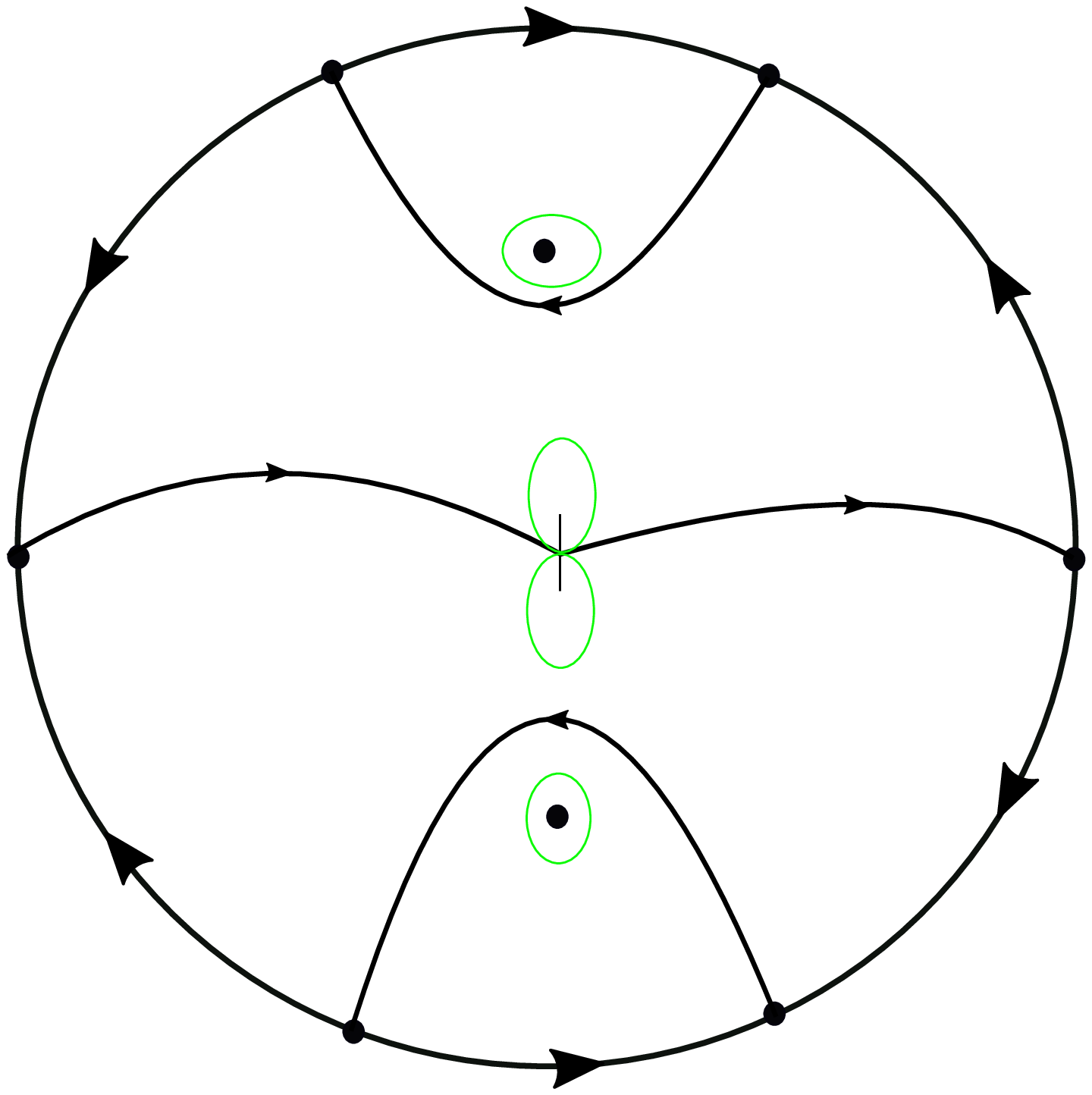} 
	\put(50,-20){\textbf{Q9)}}
	\put(45,15){$p_{1}$}
	\put(45,86){$p_{2}$}
	\end{overpic}

\vspace{1cm}

\begin{overpic}[scale=0.25]
	{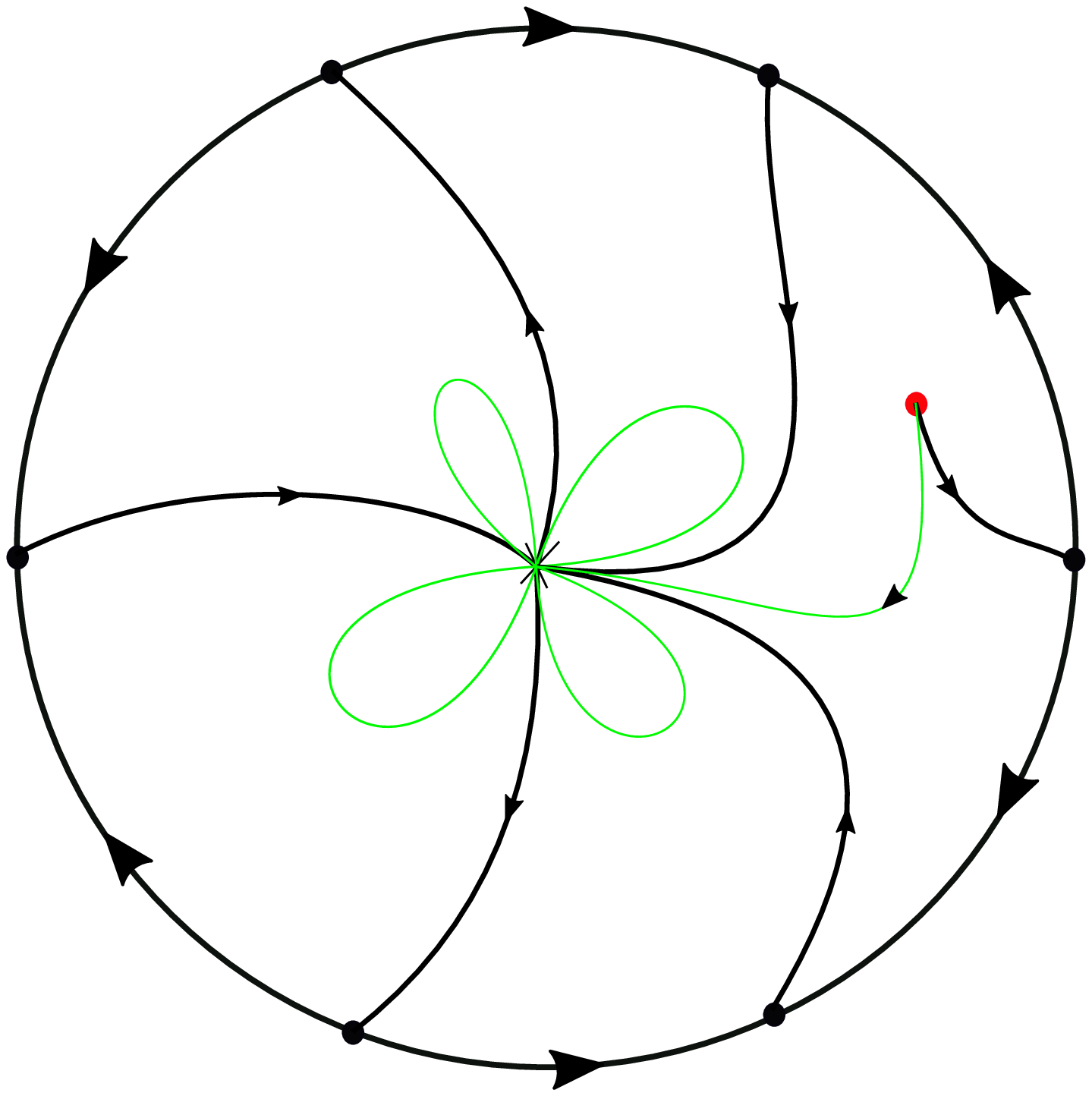} 
	\put(50,-20){\textbf{Q10)}}
    \put(78,70){$p_{1}$}
\end{overpic}\begin{overpic}[scale=0.25]
	{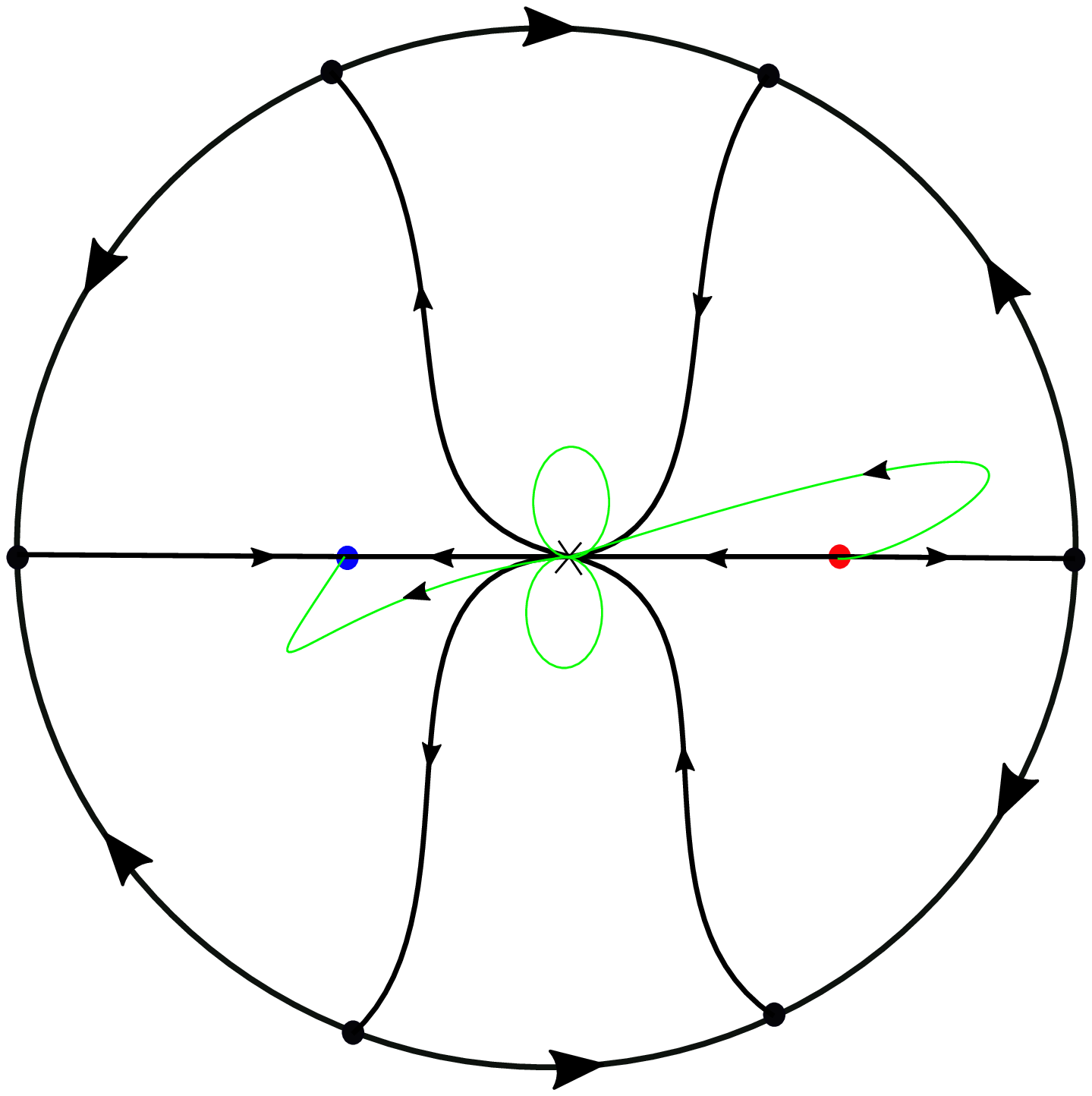} 
	\put(50,-20){\textbf{Q11)}}
	\put(30,55){$p_{1}$}
	\put(75,43){$p_{2}$}
\end{overpic}
\begin{overpic}[scale=0.25]
	{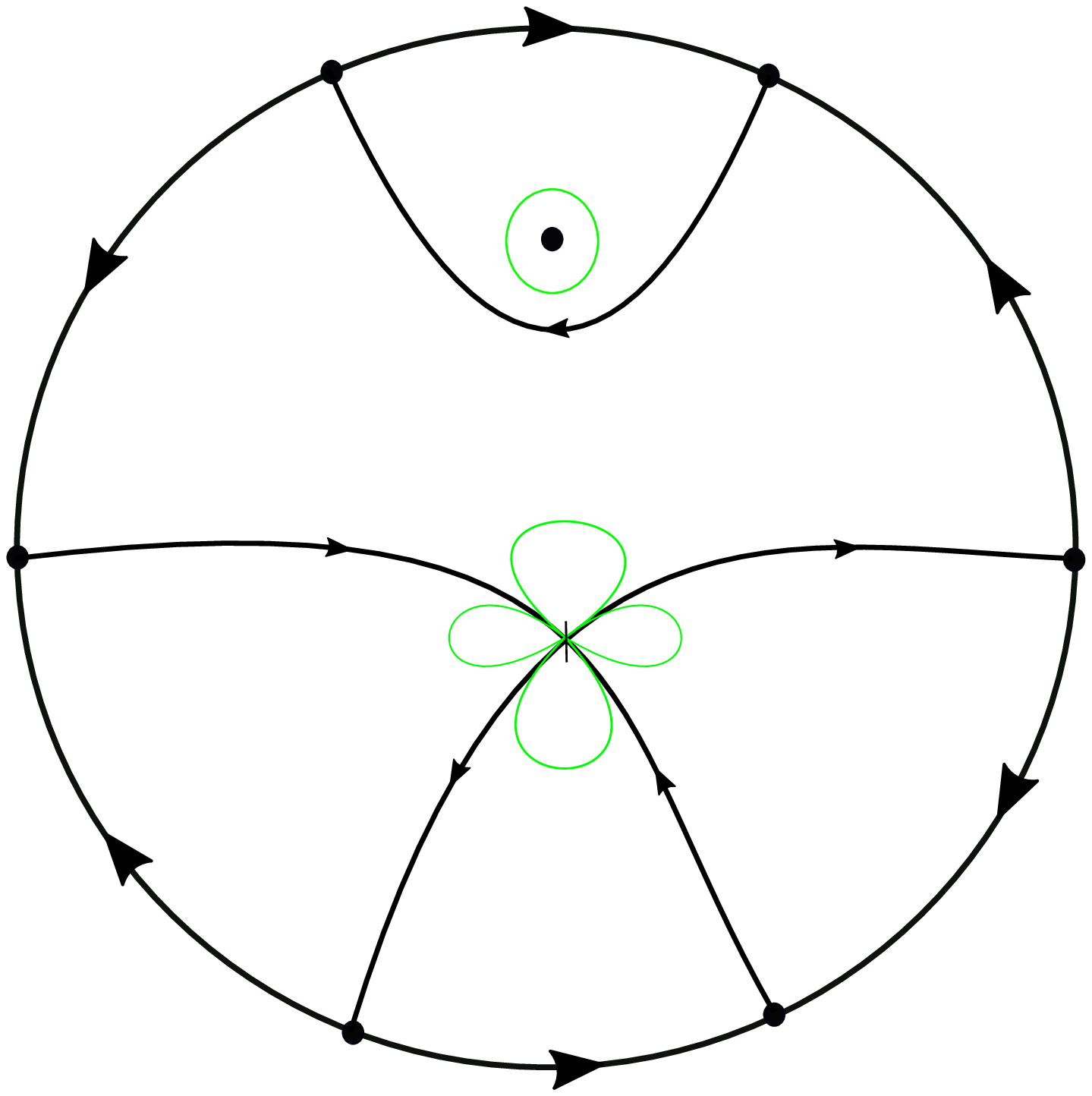} 
	\put(50,-20){\textbf{Q12)}}
	\put(45,90){$p_{1}$}
\end{overpic}

\vspace{1cm}

\begin{overpic}[scale=0.25]
	{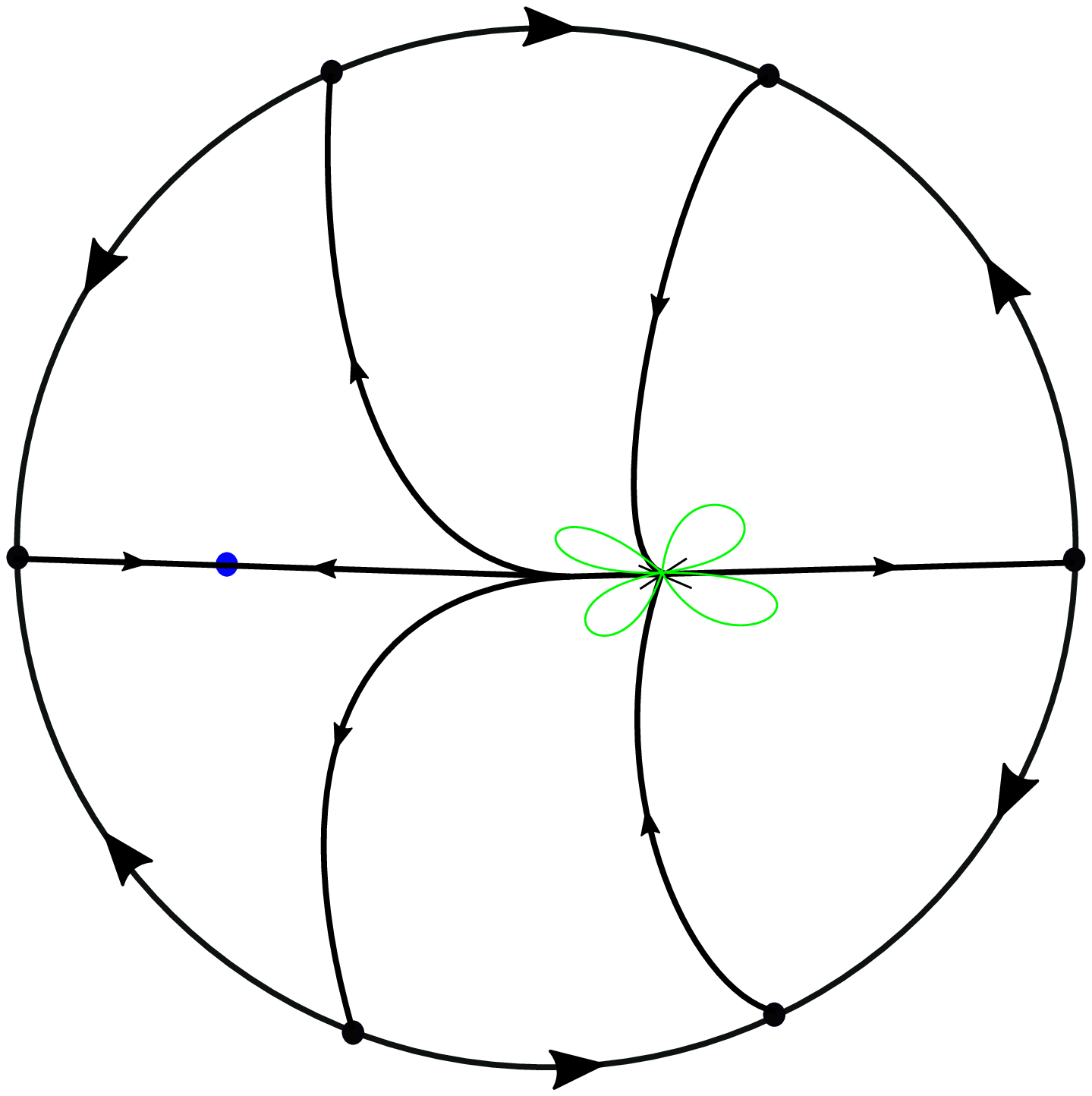} 
	\put(50,-20){\textbf{Q13)}}
	\put(15,55){$p_{1}$}
\end{overpic}
\begin{overpic}[scale=0.25]
	{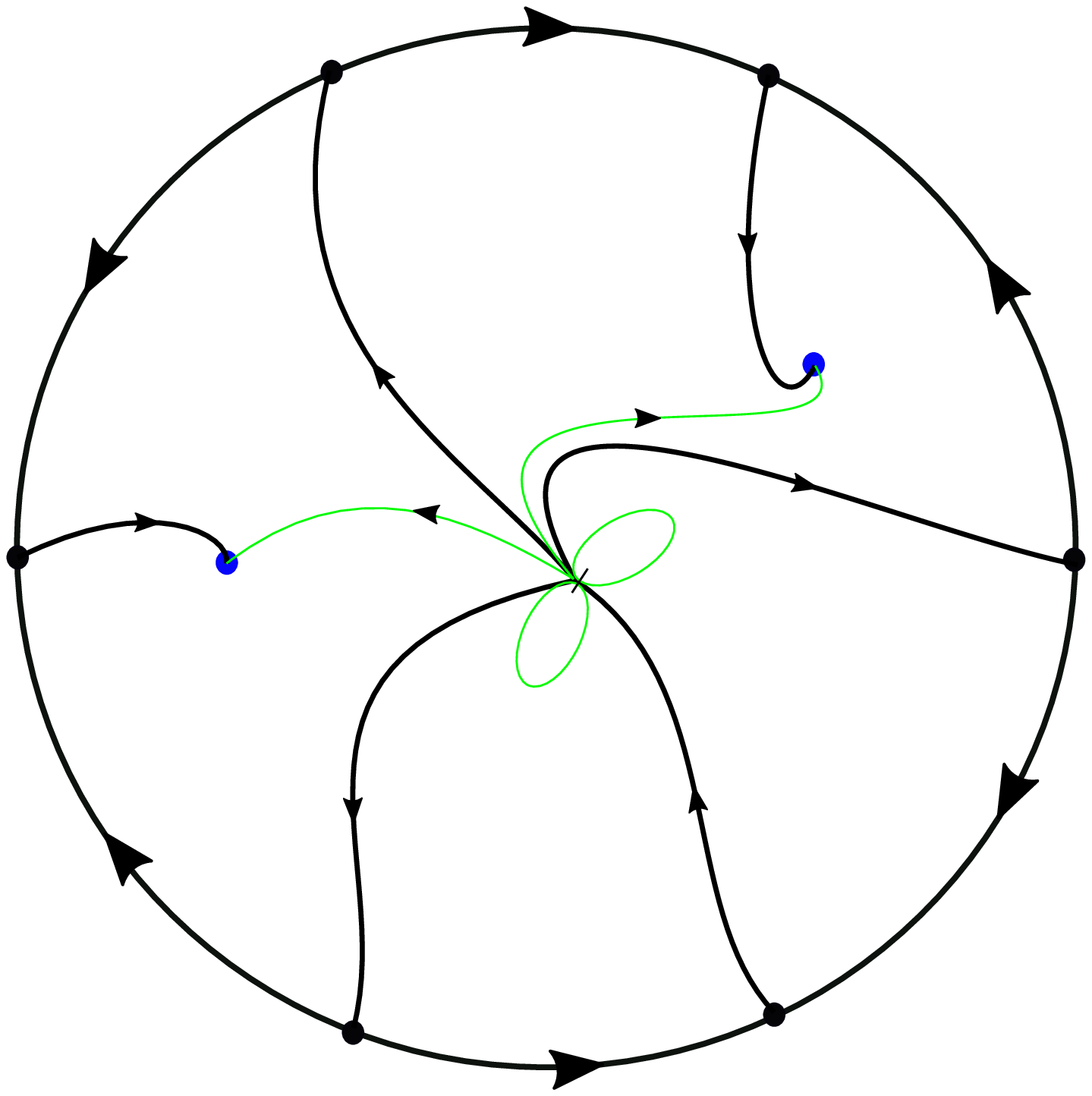} 
	\put(50,-20){\textbf{Q14)}}
	\put(15,45){$p_{1}$}
	\put(80,70){$p_{2}$}
\end{overpic}
\begin{overpic}[scale=0.25]
	{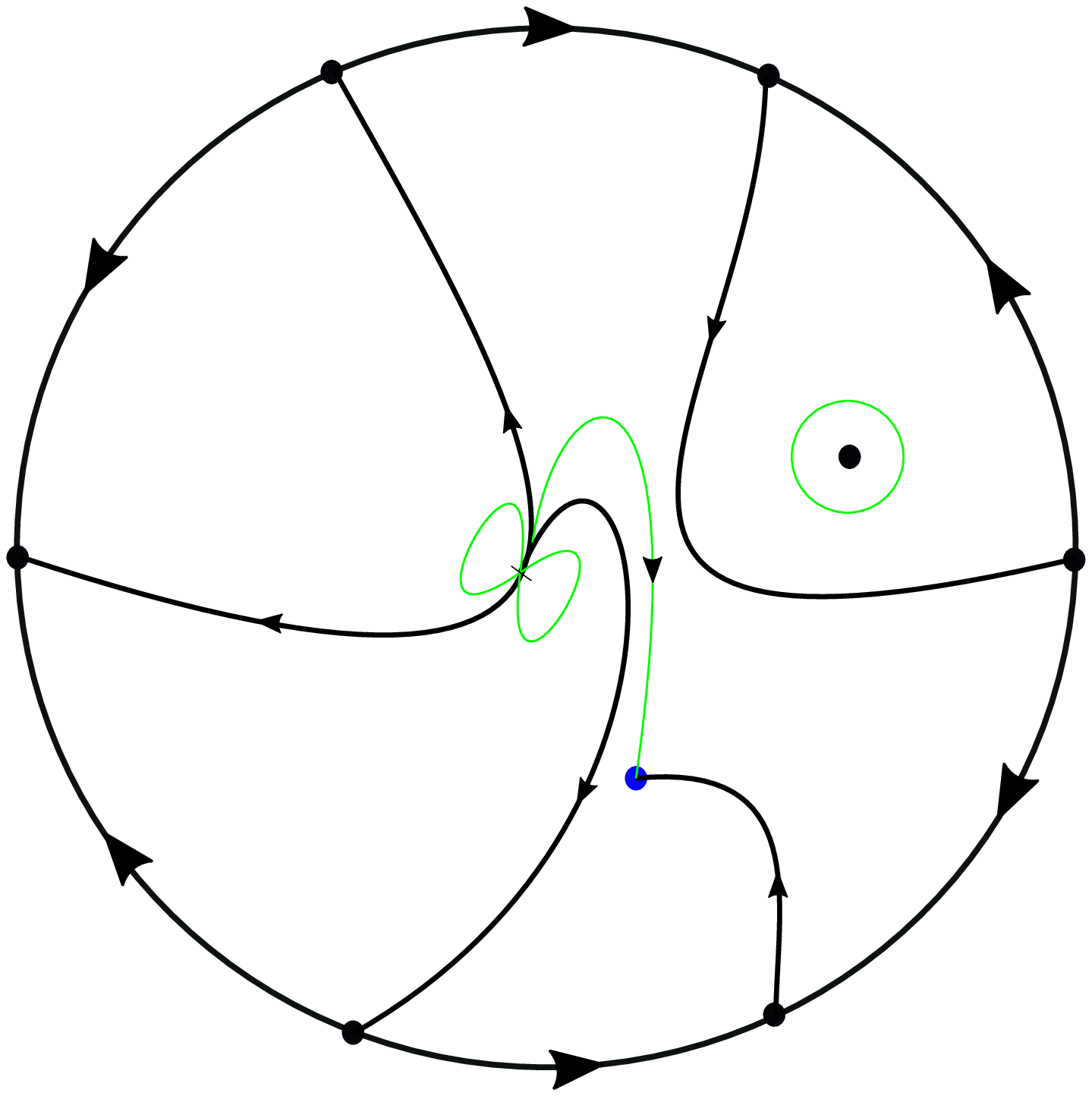} 
	\put(50,-20){\textbf{Q15)}}
	\put(60,20){$p_{1}$}
	\put(80,70){$p_{2}$}
\end{overpic}

\vspace{1cm}

\begin{overpic}[scale=0.25]
	{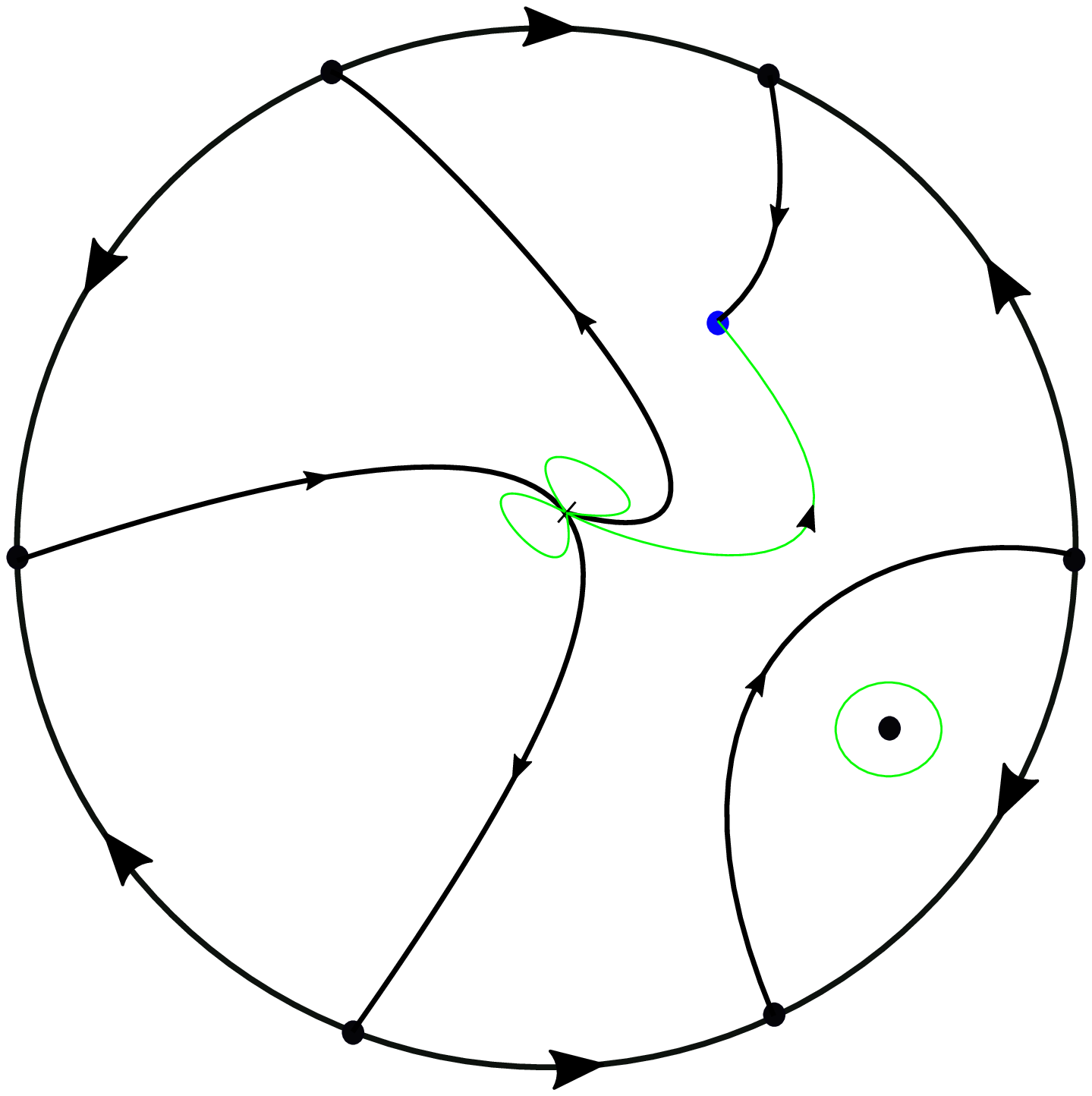} 
	\put(50,-20){\textbf{Q16)}}
	\put(60,80){$p_{2}$}
	\put(80,23){$p_{1}$}
\end{overpic}
\begin{overpic}[scale=0.25]
	{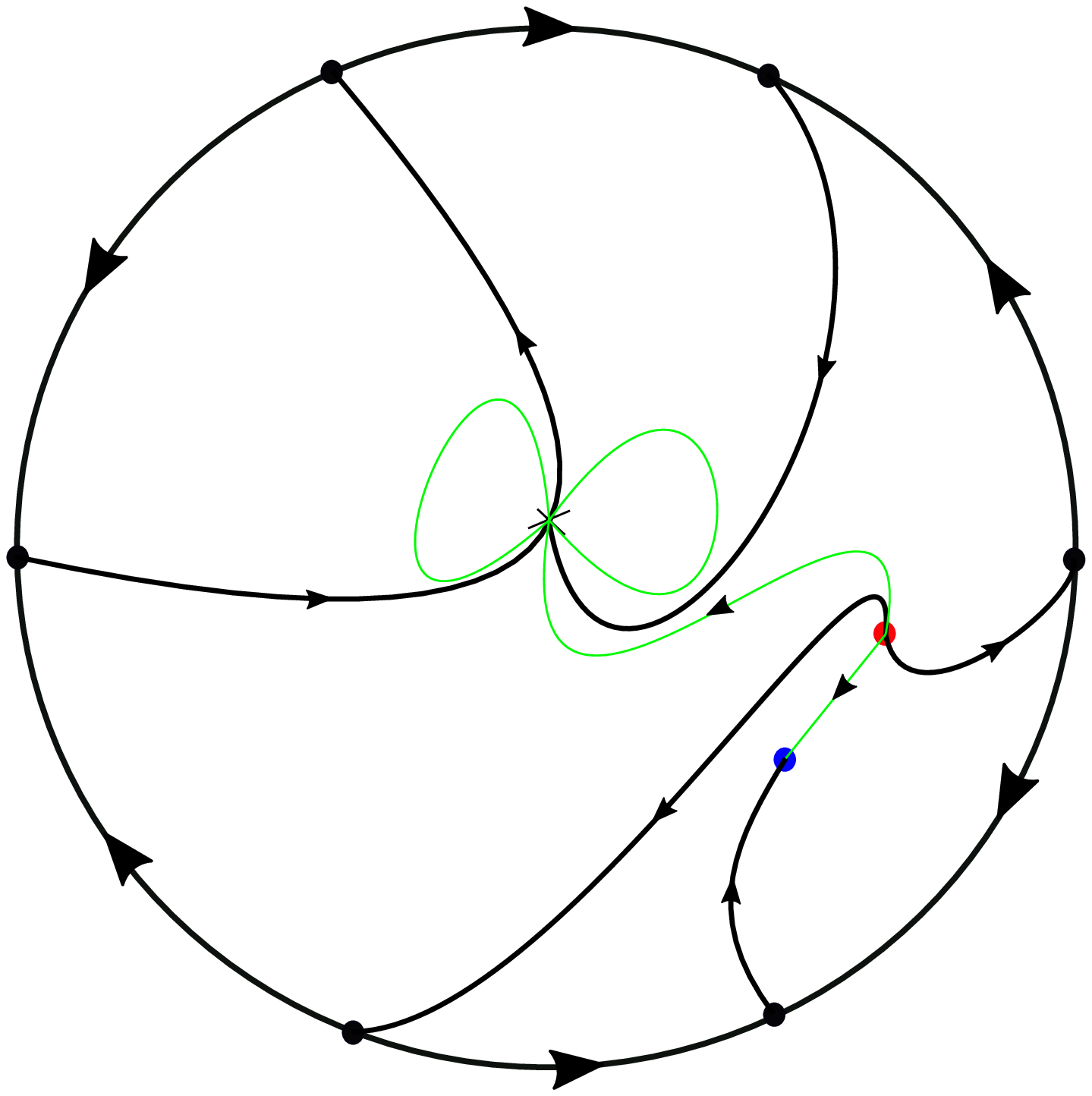} 
	\put(50,-20){\textbf{Q17)}}
	\put(85,50){$p_{1}$}
	\put(75,25){$p_{2}$}
\end{overpic}
\begin{overpic}[scale=0.25]
	{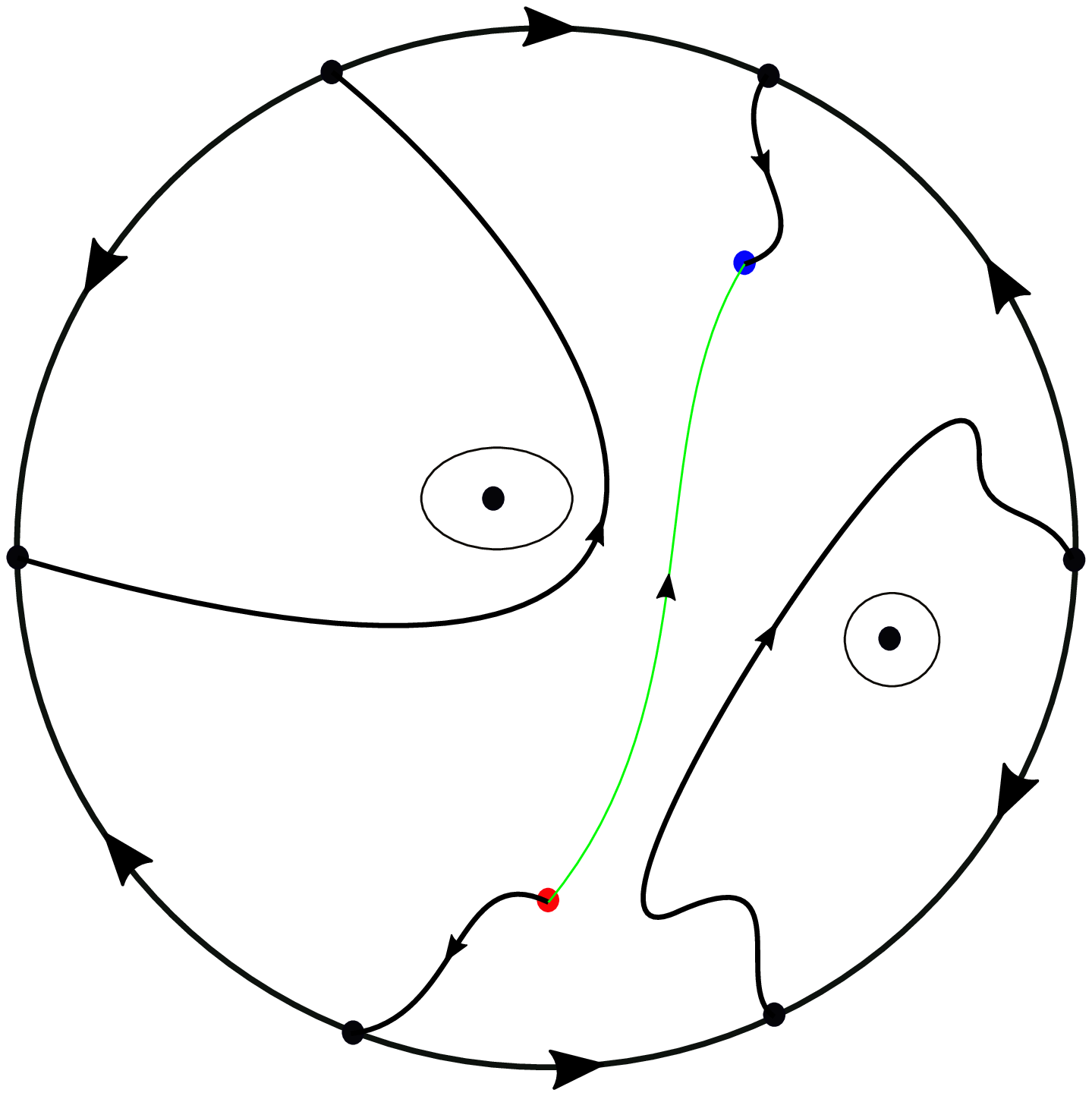} 
	\put(50,-20){\textbf{Q18)}}
	\put(85,30){$p_{1}$}
	\put(35,65){$p_{2}$}
	\put(60,85){$p_{3}$}
	\put(40,22){$p_{4}$}
\end{overpic}

\vspace{1cm}

\begin{overpic}[scale=0.25]
	{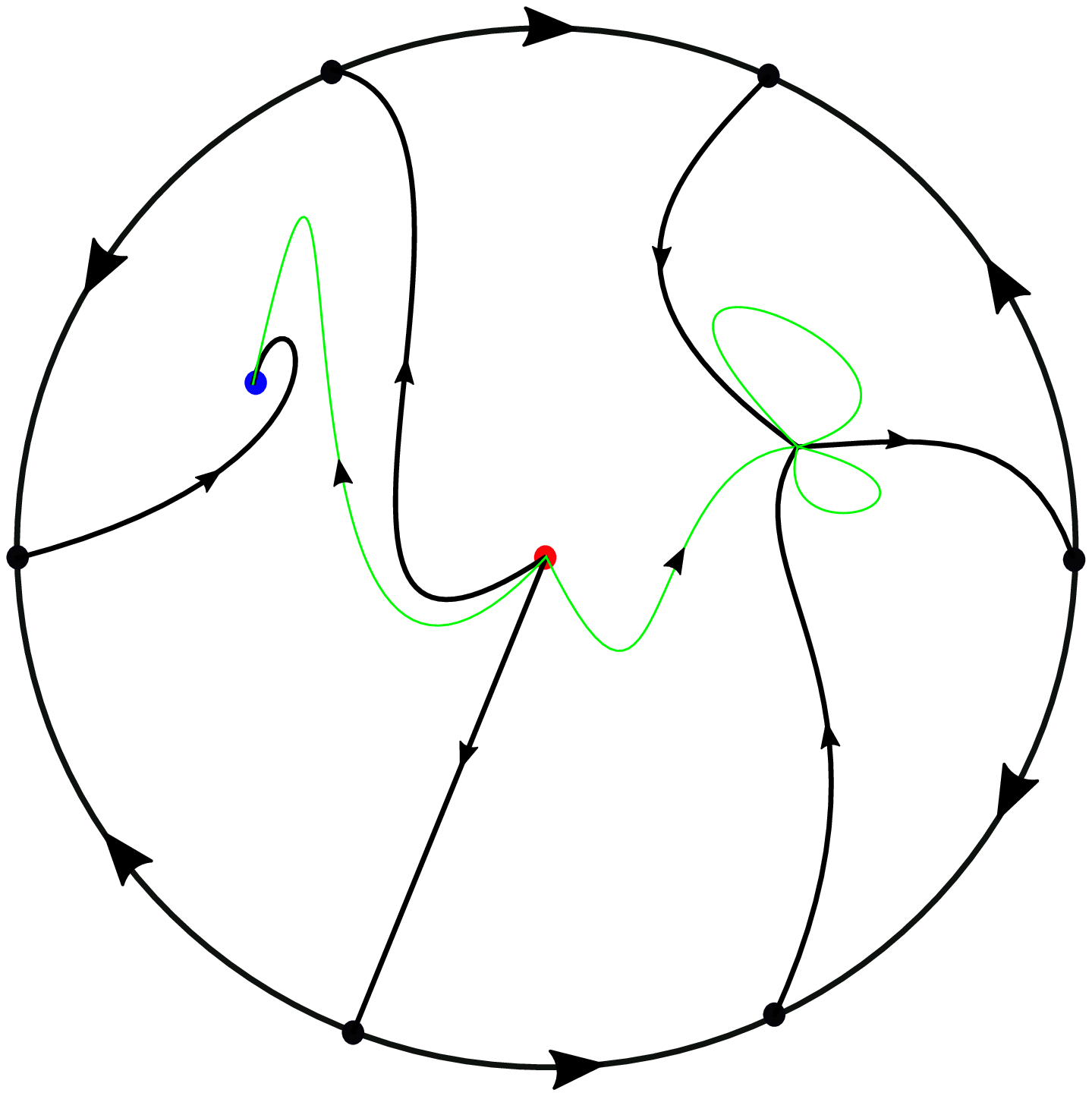} 
	\put(50,-20){\textbf{Q19)}}
	\put(50,55){$p_{1}$}
	\put(10,65){$p_{2}$}
\end{overpic}
\begin{overpic}[scale=0.25]
	{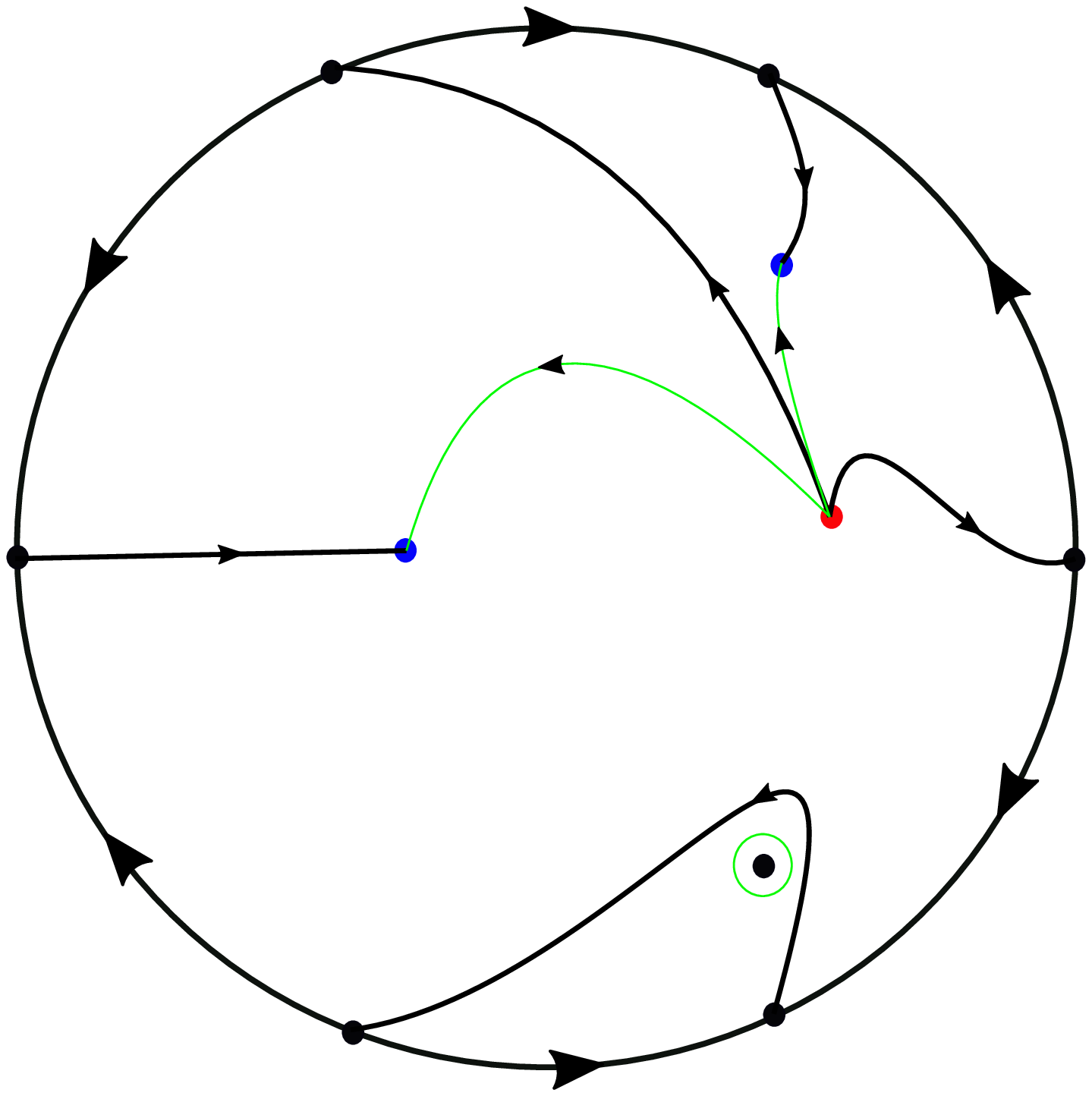} 
	\put(50,-20){\textbf{Q20)}}
	\put(60,12){$p_{1}$}
	\put(30,60){$p_{2}$}
	\put(65,85){$p_{3}$}
	\put(70,45){$p_{4}$}
\end{overpic}
\begin{overpic}[scale=0.25]
	{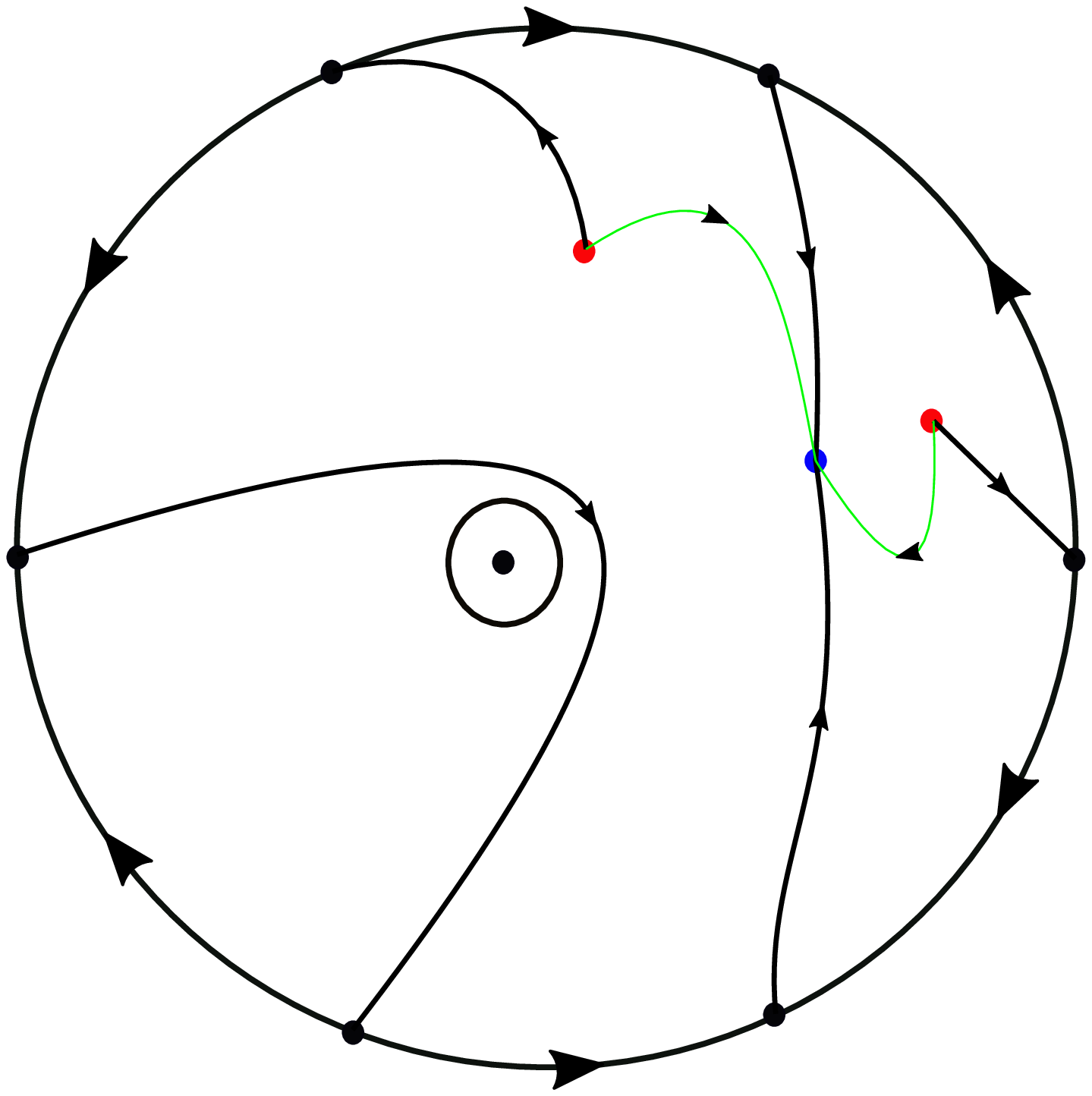} 
	\put(50,-20){\textbf{Q21)}}
	\put(40,35){$p_{1}$}
	\put(65,60){$p_{2}$}
	\put(60,85){$p_{3}$}
	\put(86,70){$p_{4}$}
\end{overpic}

\vspace{1cm}

\begin{overpic}[scale=0.25]
	{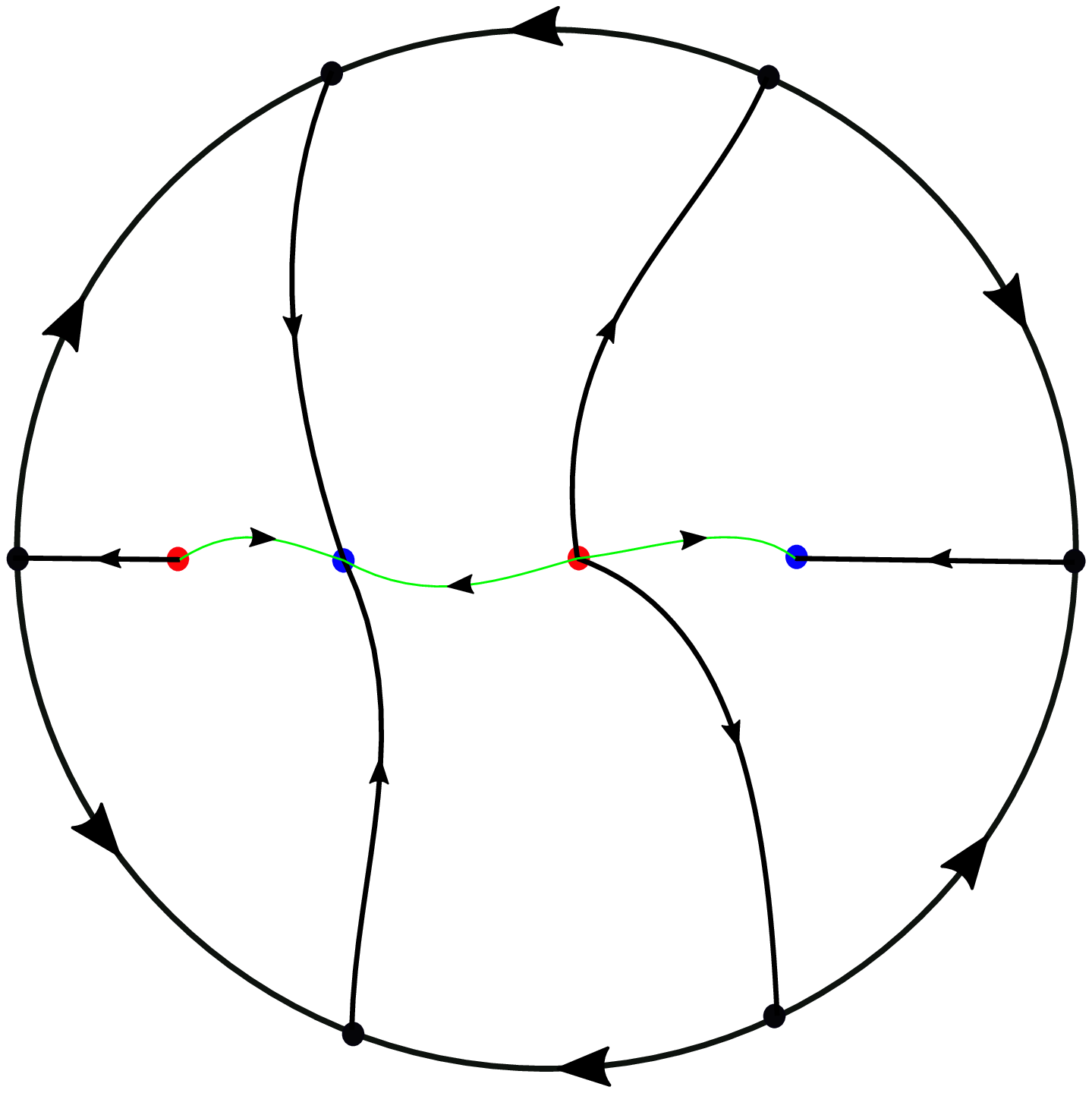} 
	\put(50,-20){\textbf{Q22)}}
	\put(10,42){$p_{1}$}
	\put(32,58){$p_{2}$}
	\put(58,58){$p_{3}$}
	\put(70,40){$p_{4}$}
\end{overpic}
\begin{overpic}[scale=0.25]
	{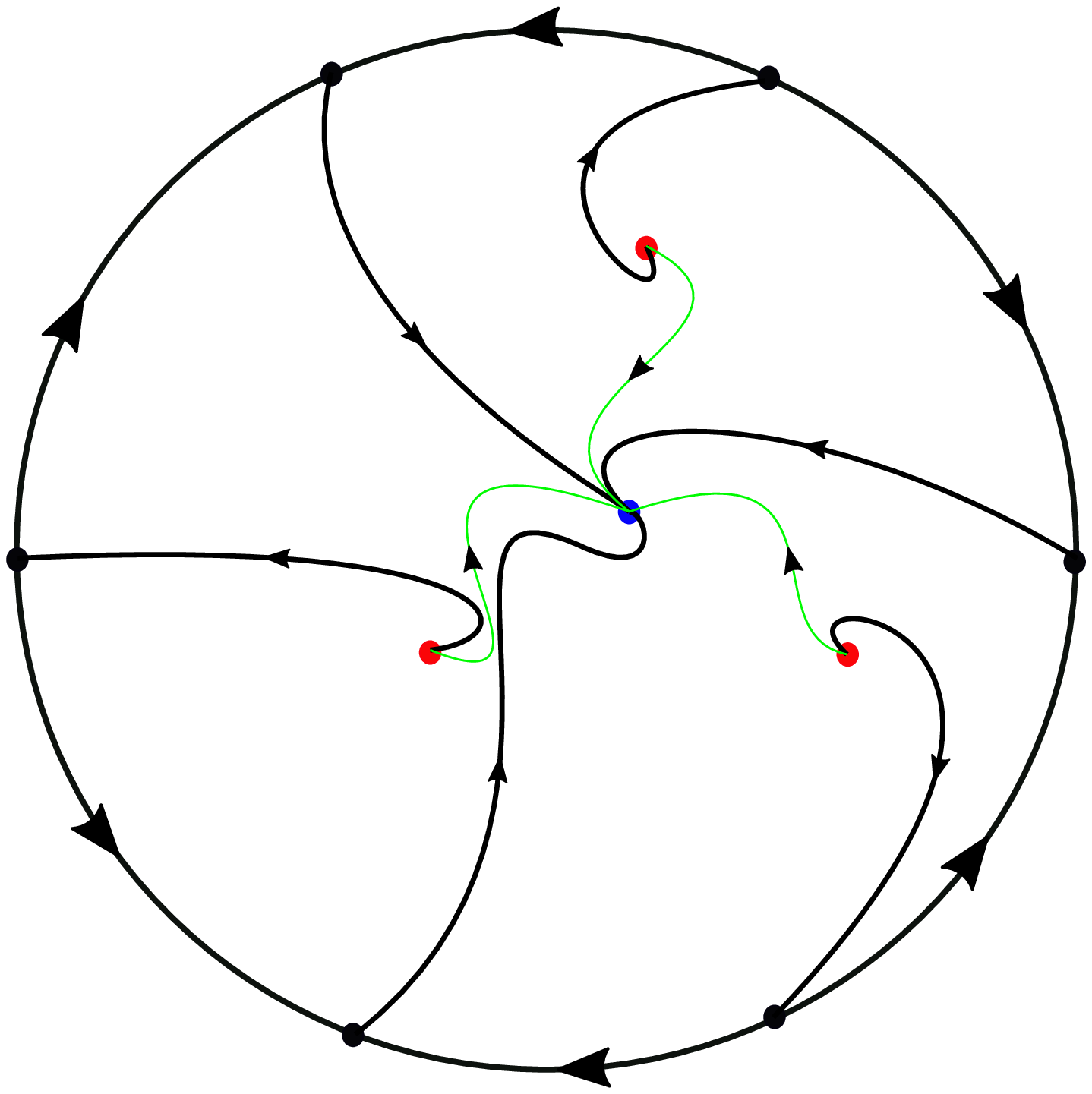} 
	\put(50,-20){\textbf{Q23)}}
	\put(35,35){$p_{1}$}
	\put(62,50){$p_{2}$}
	\put(60,85){$p_{3}$}
	\put(72,35){$p_{4}$}
\end{overpic}
\begin{overpic}[scale=0.25]
	{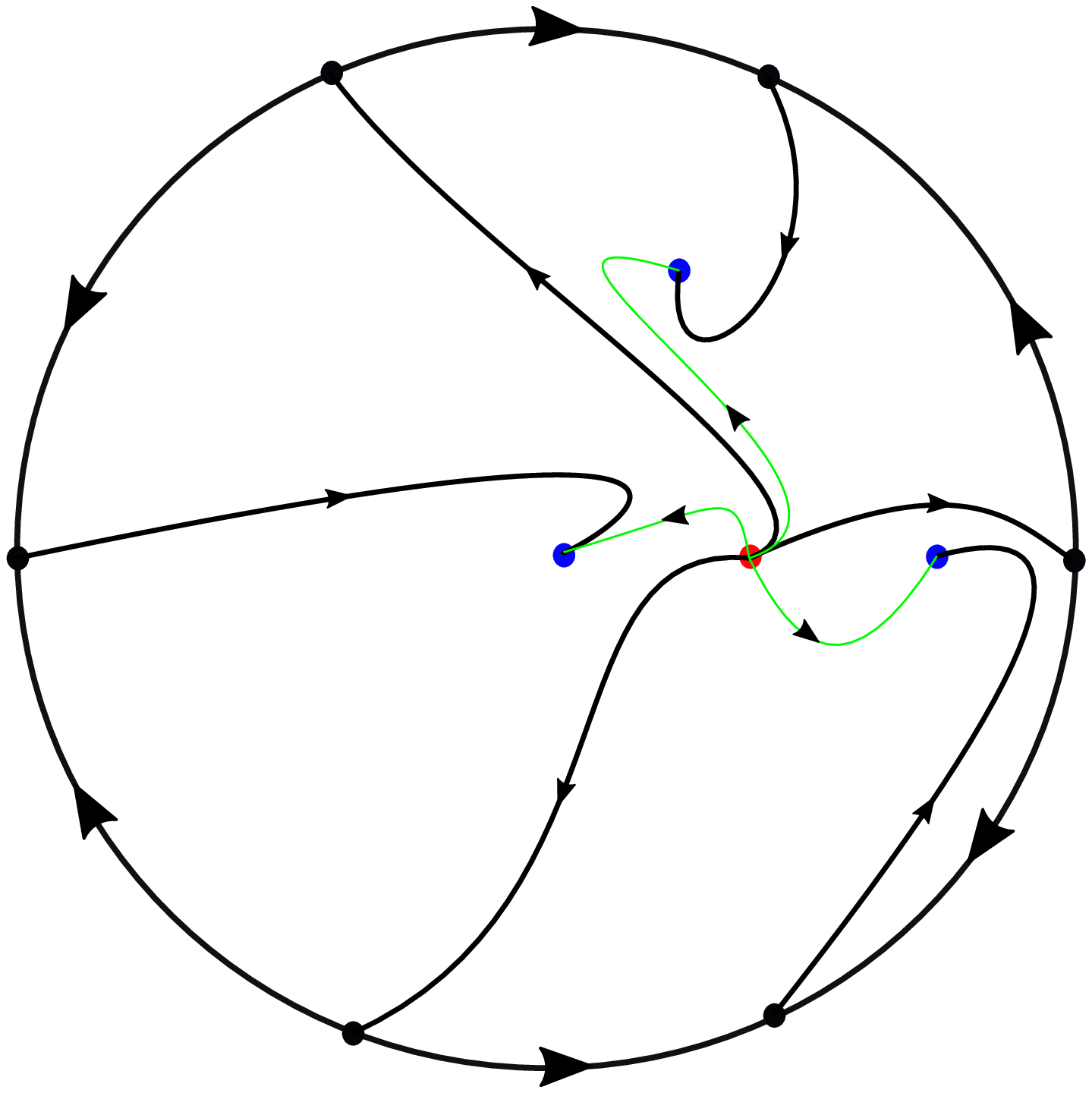} 
	\put(50,-20){\textbf{Q24)}}
	\put(40,50){$p_{1}$}
	\put(65,40){$p_{2}$}
	\put(60,85){$p_{3}$}
	\put(86,40){$p_{4}$}
\end{overpic}

\vspace{1cm}

\begin{overpic}[scale=0.25]
	{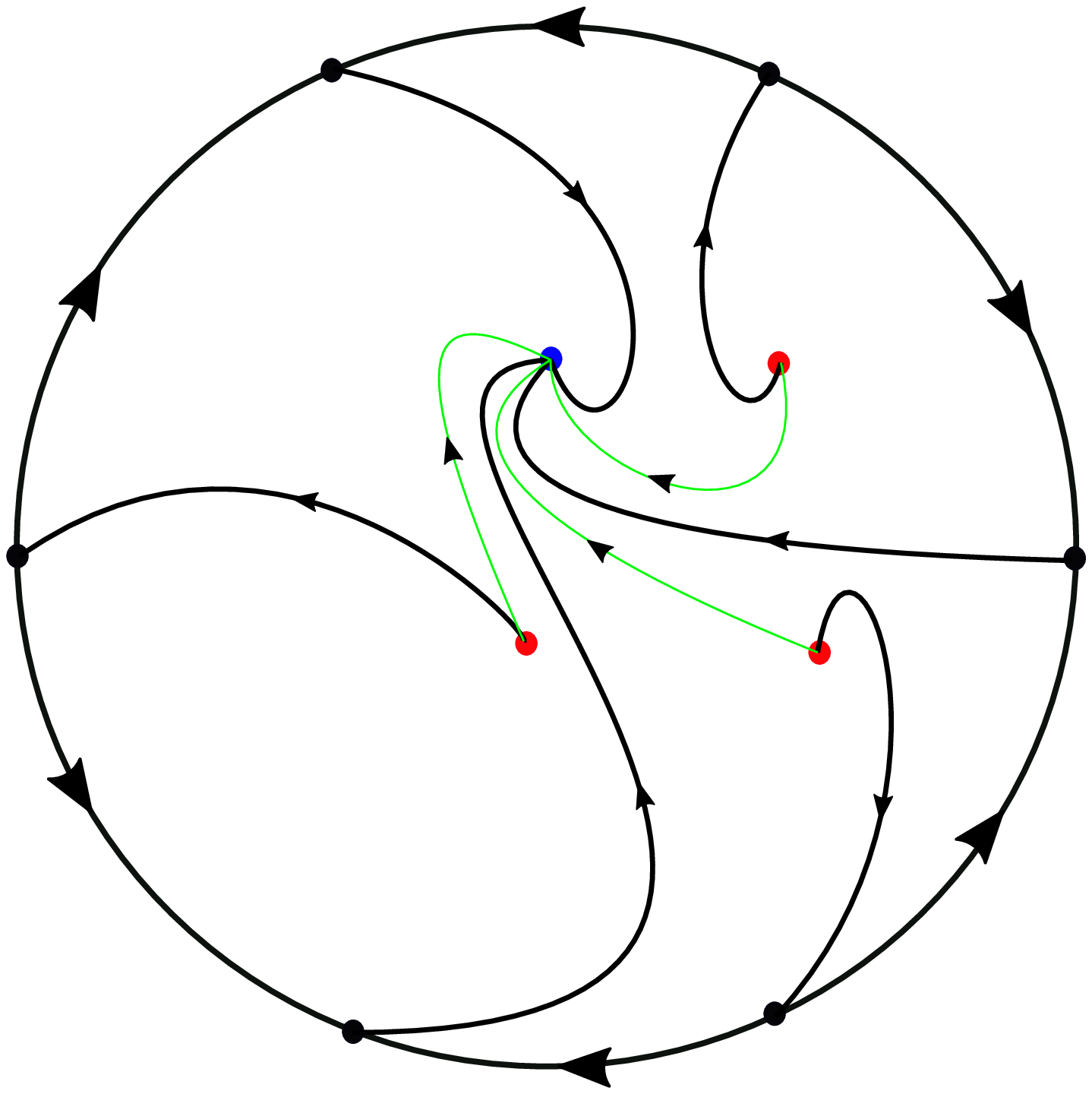} 
	\put(50,-20){\textbf{Q25)}}
	\put(40,35){$p_{1}$}
	\put(45,75){$p_{2}$}
	\put(72,72){$p_{3}$}
	\put(70,32){$p_{4}$}
\end{overpic}
\begin{overpic}[scale=0.25]
	{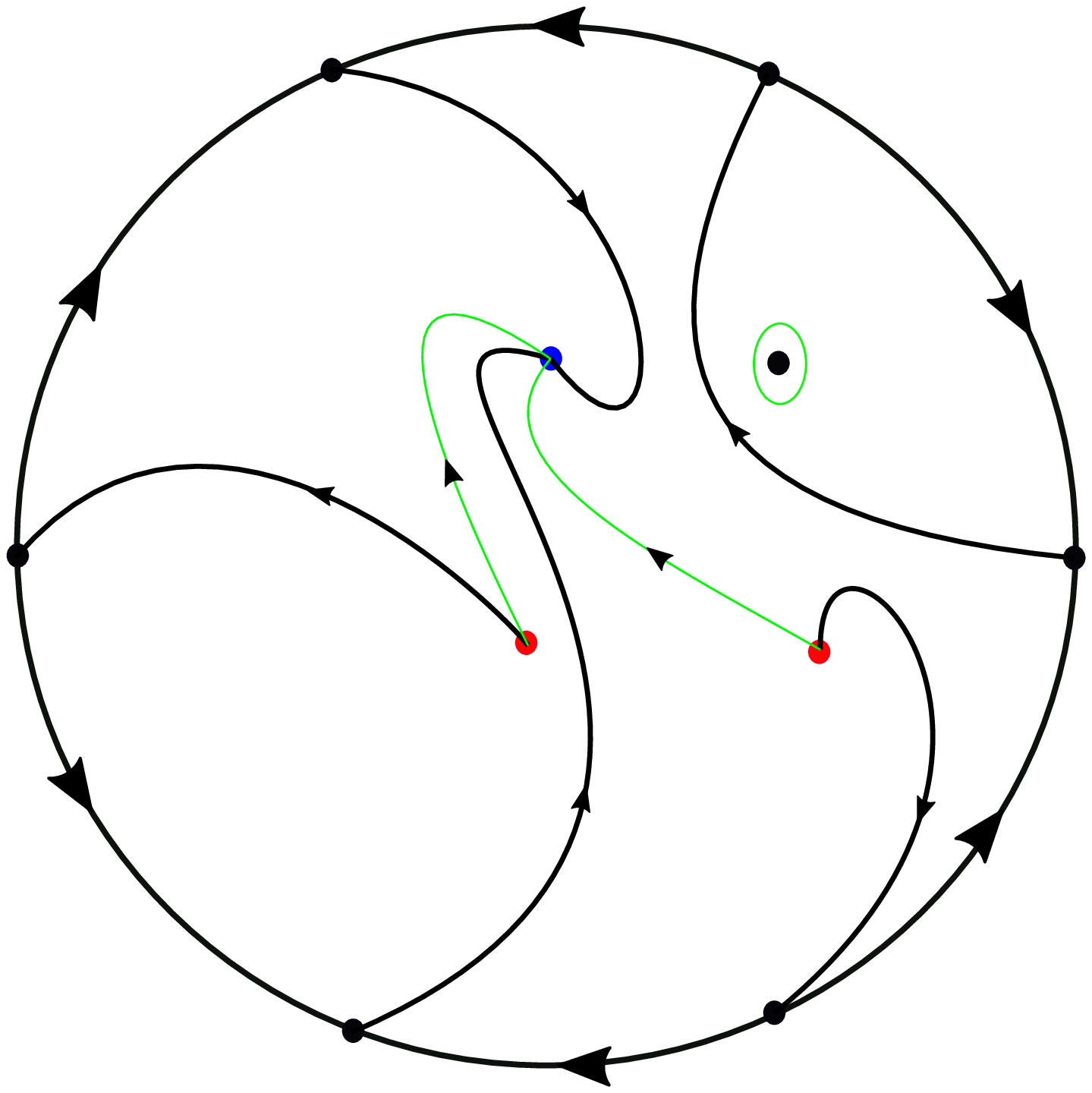} 
	\put(50,-20){\textbf{Q26)}}
	\put(40,35){$p_{1}$}
	\put(75,75){$p_{2}$}
	\put(45,75){$p_{3}$}
	\put(70,32){$p_{4}$}
\end{overpic}
\begin{overpic}[scale=0.25]
	{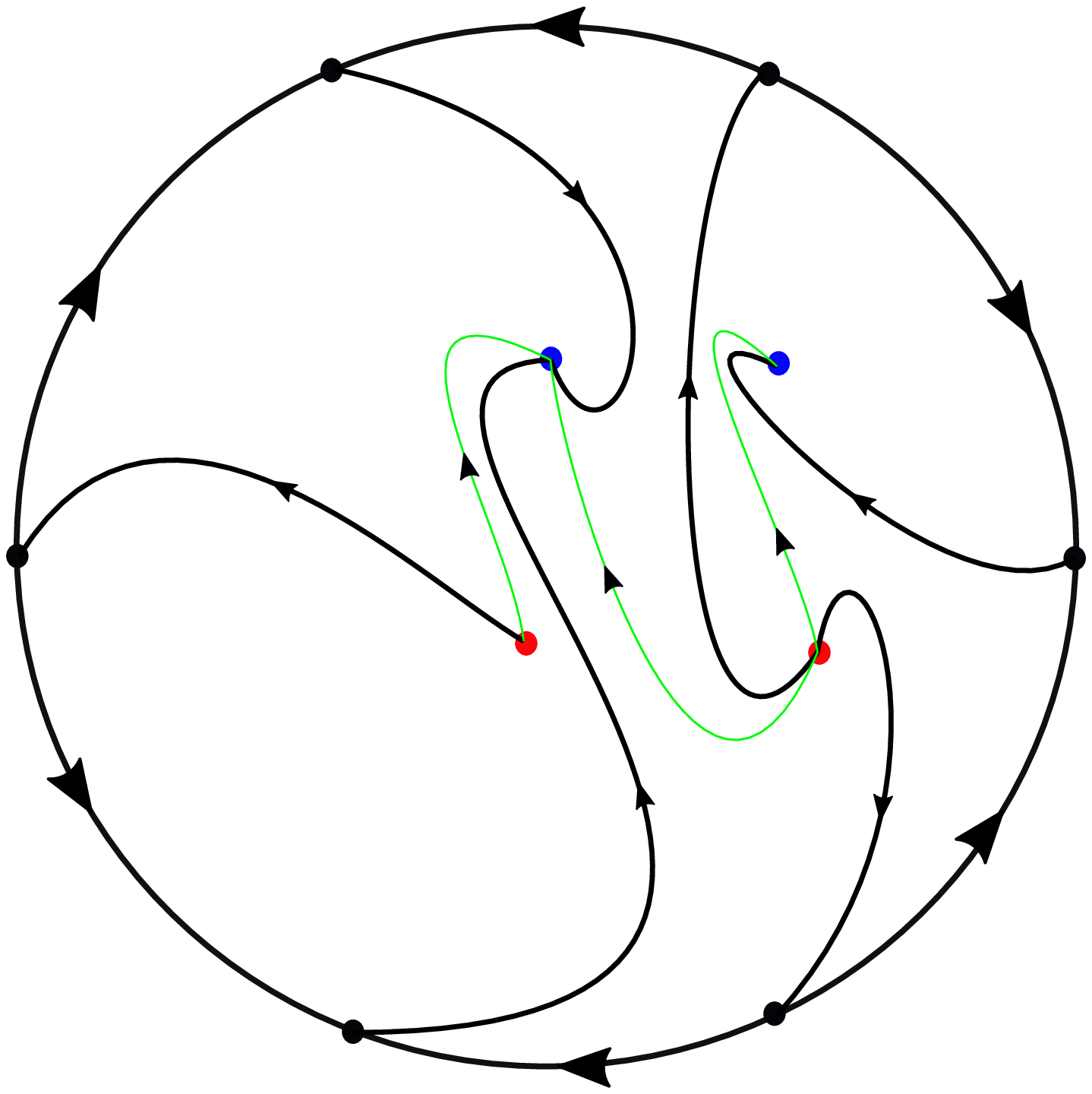} 
	\put(50,-20){\textbf{Q27)}}
	\put(40,35){$p_{1}$}
	\put(72,30){$p_{2}$}
	\put(40,75){$p_{3}$}
	\put(75,70){$p_{4}$}
\end{overpic}

\vspace{1cm}

\begin{overpic}[scale=0.25]
	{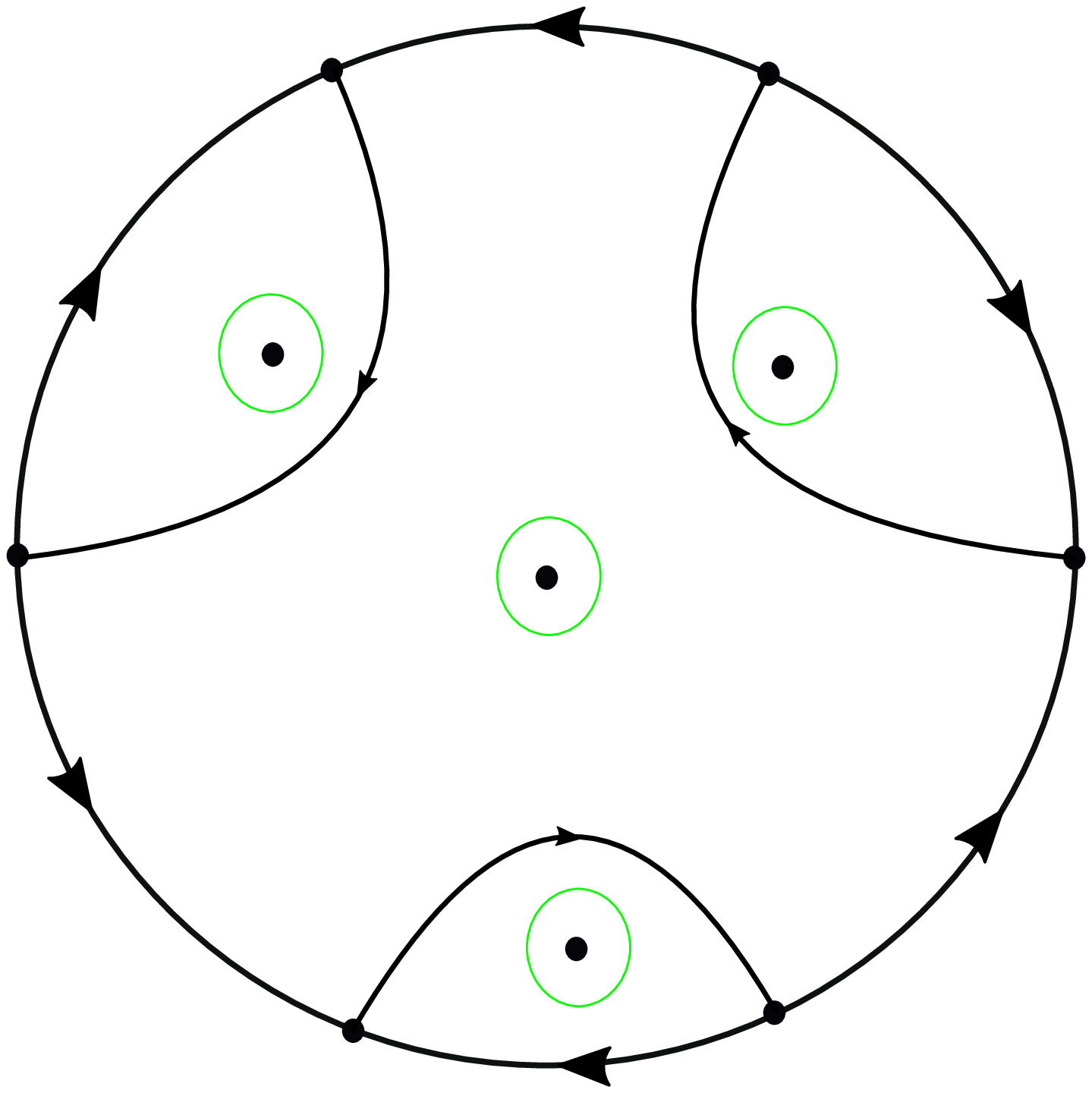} 
	\put(50,-20){\textbf{Q28)}}
	\put(40,35){$p_{1}$}
	\put(40,10){$p_{2}$}
	\put(10,70){$p_{3}$}
	\put(70,80){$p_{4}$}
\end{overpic}
\begin{overpic}[scale=0.25]
	{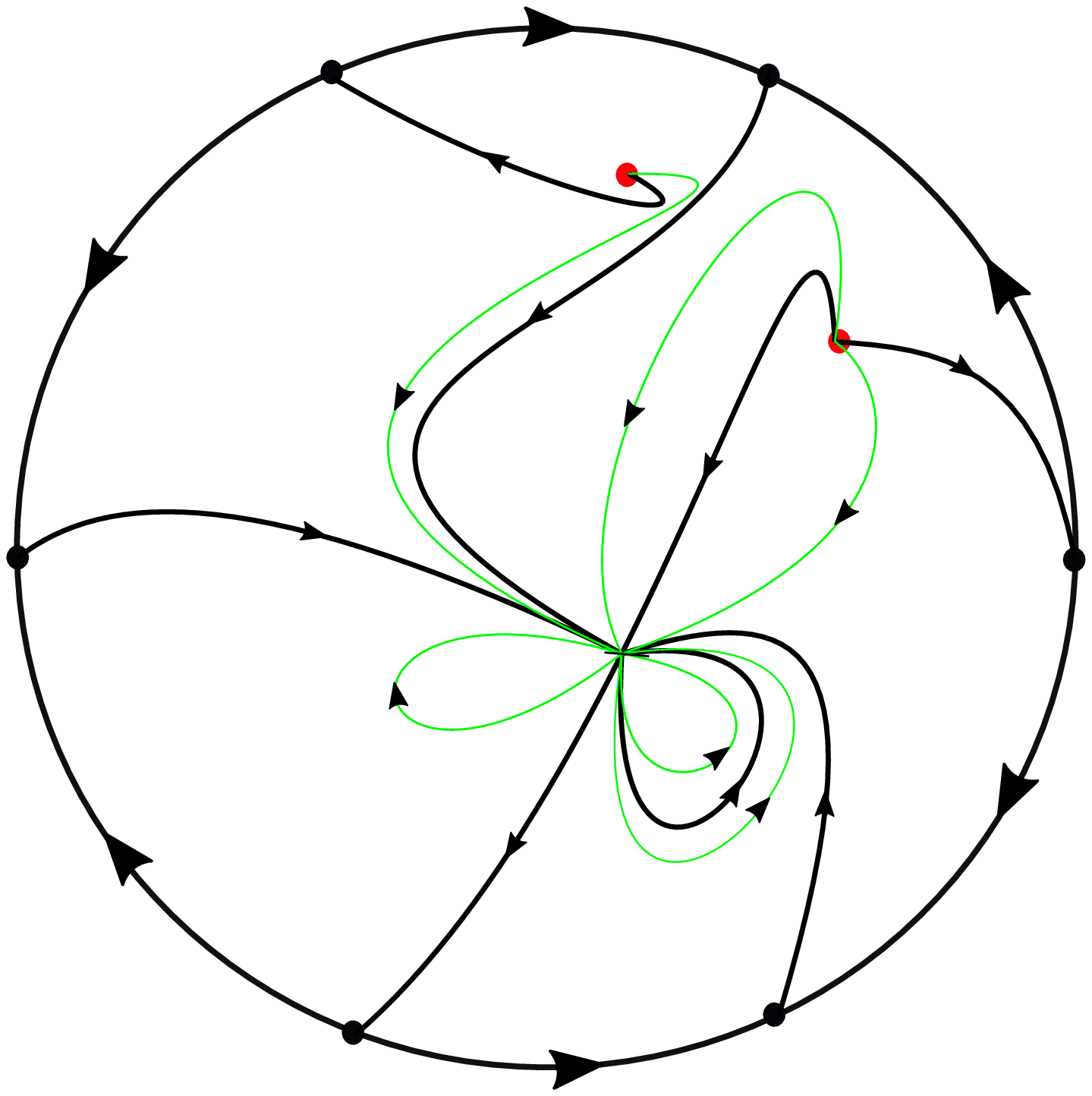} 
	\put(50,-20){\textbf{Q29)}}
	\put(50,90){$p_{1}$}
	\put(82,75){$p_{2}$}
\end{overpic}
\vspace{1cm}

\begin{proof}
	Following the same scheme of Theorem $\ref{teocubphase}$, to prove this result is enough to show that, among all configurations, some of them are no-realizable and the remaining is obtained. All possible configurations are given by:
\begin{multicols}{2}
	\begin{enumerate}[(a)]
	 \item 3 centers and 1 node/focus,
	\item 3 nodes and 1 center/focus,
	\item 1 quadruple,
	\item 4 centers,
	\item 4 node, 
	\item 4 foci, 
	\item 1 node, 2 foci and 1 center, 
	\item 2 foci and 1 double, 
	\item 3 foci and 1 center,
	\item 2 centers and 1 double,
	\item 1 focus and 2 double,
	\item 1 double and 2 node,
	\item 1 center and 2 double,
	\item 1 node and 2 double, 
	\item 2 node and 1 double, 
	\item 1 center, 1 double and 1 focus,
	\item 1 center, 1 double and 1 node,
	\item 1 double, 1 focus and 1 node,
	\item 2 foci and 2 centers,
	\item 1 node, 1 center and 2 foci,
	\item 2 nodes, 1 center and 1 focus, 
	\item 4 foci.
	\end{enumerate}
\end{multicols}	

%\begin{enumerate}[$(a)$]
%    \item 3 centers and 1 node/focus
%    \item 3 nodes and 1 center/focus
%	\item 1 quadruple: 
%	\item 4 centers: 
%	\item 4 node: 
%	\item 4 foci: 
%	\item 1 node, 2 foci and 1 center: 
%	\item 2 foci and 1 double: 
%	\item 3 foci and 1 center:
%	\item 2 centers and 1 double:
%	\item 1 focus and 2 double: 
%	\item 1 double and 2 node:
%	\item 1 center and 2 double: 
%	\item 1 node and 2 double: 
%	\item 2 node and 1 double: 
%	\item 1 center, 1 double and 1 focus:
%	\item 1 center, 1 double and 1 node: 
%	\item 1 double, 1 focus and 1 node: 
%	\item 2 foci and 2 centers: 
%	\item 1 node, 1 center and 2 foci:
%	\item 2 nodes, 1 center and 1 focus: 
%	\item 4 focus 
%\end{enumerate}

From the Proposition \ref{teoGas1} it is impossible to obtain items $(a)$ and $(b)$. All the remaining cases are realizable and we are going to show examples.

For the case $(c)$, using Proposition \ref{sver1}, we obtain Figure $Q1$. The items $(d)$ and $(e)$ were analysed in \cite{AlvGasPro,GAP} and correspond to Figures $Q2$, $Q4$ and $Q28$.

For the case $(f)$, using Remark $\ref{remakfoci}$, we have 4 possibilities that correspond to Figures $Q22-Q27$. According \cite{AlvGasPro},  the phase $Q25$ to $Q27$ are obtained  considering the system $\dot{z}=(-1+2i)z(z-3)(z-2i)(z-(\eta+2i))$. If  $\eta=\eta^{\pm}=(-5\pm \sqrt{89})/2$, we have a center. If $\eta \in (\eta^{-},\eta^{+})$ we have a repelling focus and otherwise an attracting focus. The  phase portraits $Q25$ to $Q27$ are obtained taking $\eta=2, \eta \approx 2.2$ and $\eta=3$ respectively.

For the case $(g)$ we have that  $p_{2}$ is an attracting node, $p_{3}$ is a center and $p_{1}$ and $p_{4}$ are repelling foci. $p_{2}$  is the 
$\omega$--limit of the orbits coming from $I_{4}$, $I_{6}$, $p_{1}$, and $p_{4}$. In the same way, since $p_{1}$ and $p_{4}$ are repelling, their respective $\omega$--limit are $I_{3}$ and $I_{5}$.  See Figure $Q6$.

For the case $(h)$ we have that $p_{2}$ is an attracting focus and $p_{1}$ is a repelling focus. Moreover, we have one double point.  The orbits coming from $p_{1}$ have as $\omega$--limit $I_{3}$, $I_{1}$, the double point, and the point $p_{2}$. The orbit coming from $I_{2}$ has  $p_{2}$ as $\omega$--limit. See Figure $Q7$.

For the case $(i)$ we have three foci,  the points $p_{2}$ and $p_{3}$ are attracting, and the point $p_{1}$ is repelling. Moreover, we  also have a center $p_{4}$. This center  is contained in a region limited by the orbit coming from $I_{3}$ and arriving at $I_{4}$. Since the points $p_{2}$ and $p_{3}$ are attracting, they are $\omega$--limit for the orbits coming from $I_{2}$, $I_{6}$ and the orbits coming from $p_{1}$. Furthermore, the orbits coming from $p_{1}$ also have $I_{1}$ and $I_{5}$ as $\omega$--limit. See Figure $Q8$.

For the case $(j)$ we have two centers $p_{1}$ and $p_{2}$ and one double point. The centers are contained in the regions limited by the orbits coming from $I_{4}$ and arriving at $I_{5}$ and the orbits coming from $I_{2}$ and arriving at $I_{1}$, respectively.  See  Figure $Q9$.

For the case $(k)$ we have one repelling focus $p_{1}$ and two double points. On the other hand, since $p_{1}$ is a repelling  focus, the $\omega$--limit of the orbits coming from $p_{1}$ only can be the double point and $I_{3}$. See Figure $Q10$.

For  the case $(l)$ we have an attracting node $p_{1}$, a repelling node  $p_{2}$  and one double point. The orbits coming from the double point of the left side have  the point $p_{1}$ and $I_{5}$ as $\omega$--limit. The point $p_{1}$ is also $\omega$--limit of the orbits coming from $I_{6}$. For the right side of the double point, we have that $p_{2}$ is repelling with $\omega$--limit being $I_{3}$ and the double point. See Figure $Q11$.

For the case $(m)$ we have one center $p_{1}$ and two double points. The center is contained in a region limited by the orbits coming from $I_{2}$ and arriving at $I_{1}$. See Figure $Q12$.

For the case $(n)$ we have one attracting node $p_{1}$ and two double points. We have 4 orbits coming from double points with $\omega$--limit being $I_{1}$, $p_{1}$,  $I_{5}$ and $I_{3}$. And we also have the double point being $\omega$--limit of the orbits coming from $I_{2}$ and $I_{4}$. Lastly, we can see that $p_{1}$ is $\omega$--limit of the orbits coming from $I_{6}$. See Figure $Q13$.

For the case $(o)$ we have  two repelling foci $p_{1}$ and $p_{2}$ which are the $\omega$--limit of the orbits coming from $I_{1}$ and $I_{3}$ respectively. Moreover, the 
$\omega$--limit of the orbits coming from  $I_{2}, I_{4}, I_{5}$, and $I_{6}$ is the double point. See Figure $Q29$.

For the case $(p)$, we have a center $p_{2}$, an attracting focus $p_1$, and a double point.  $p_{1}$ is  the $\omega$--limit of the orbits coming from  the double point, and of the orbits coming from $I_{2}$. See Figure $Q15$.

For the case $(q)$ we have a center  $p_{1}$, an attracting node $p_{2}$ and a double point. $p_{2}$ is the $\omega$--limit of the orbits coming from double point, and of the orbits coming from $I_{2}$. See Figure $Q16$.

For the case $(r)$ we have  a repelling focus $p_{1}$,  an attracting node $p_{2}$ and  a double point. The orbits coming from $p_{1}$ have $\omega$--limit in the double point, $I_{3}$, $I_{5}$, and $p_{2}$. Lastly, $p_{2}$ is $\omega$--limit of the orbits coming from $I_{4}$. See Figure $Q17$.

For the case $(s)$ we have two foci. The point $p_{3}$ is an attracting focus and the point $p_{4}$ is a repelling focus. Furthermore, we also have two center $p_{1}$ and $p_{4}$. Each center is contained in a limited region formed for a separatrix coming $I_{3}$ and arriving at $I_{4}$ and coming from $I_{6}$ and arriving $I_{1}$. The orbits coming from $p_{4}$ have $\omega$--limit in $I_{5}$ and $p_{2}$. The orbits coming from $I_{2}$ are arriving at $p_{2}$. See Figure $Q18$.  

For the case $(t)$ we have an attracting node $p_{2}$, a center $p_{1}$, an attracting focus $p_{3}$ and  a repelling focus $p_{4}$. The orbits coming from $p_{4}$ have, $p_{2}$, $p_{3}$, $I_{1}$, and $I_{3}$ as $\omega$--limit. On the other way, the only possibility for the orbits coming from $I_{2}$ and $I_{6}$ is arriving at $p_{3}$ and $p_{2}$ respectively. This case is Figure $Q20$.  

For the case $(u)$  we have an attracting node  $p_{2}$, a repelling node  $p_{3}$, a center $p_{1}$ and a repelling focus $p_{4}$.  $p_{2}$ is the $\omega$--limit from the orbits coming from $p_{3}$, $p_{4}$, $I_{2}$, and $I_{4}$. On the other hand, $I_{1}$ and $I_{3}$ are $\omega$--limit of the orbits coming from $p_{3}$ and $p_{4}$ respectively . This case is Figure $Q21$.  
\end{proof}

\begin{remark}
	To obtain the phase portrait in Figures $(a)$ to $(t)$, we take the following systems. 
	\begin{enumerate}[1)]
		\item 1 quadruple: $\dot{z}=z^4.$
		\item 4 centers: $\dot{z}=z(z-i)(z-2i)(z-3i).$
		\item 4 foci: $\dot{z}=z(z-(1-3i))(z-(2-2i))(z-(3-3i)).$
		\item 4 node: $\dot{z}=z(z-1)(z-2)(z-3).$
	    \item 4 nodes triangle: $\dot{z}=z(z-(3+2i))(z-5/3)(z-(3-2i)).$
	    \item 1 node, 2 foci and 1 center: $\dot{z}=z(z-(1-3i))(z-(2-2i))(z-i).$
	    \item 2 foci and 1 double: $\dot{z}=z^2(z-(2-2i)(z-(3-i)).$
		\item 3 foci and 1 center: $\dot{z}=z(z-(1-3i))(z-2i)(z-3).$
		\item 2 centers and 1 double: $\dot{z}=z^2(z-2i)(z+i).$
		\item 1 focus and 2 double: $\dot{z}=z^3(z-(3+i)).$
		\item 1 double and 2 node: $\dot{z}=z^2(z-1)(z+1).$
		\item 1 center and 2 double: $\dot{z}=z^3(z-i).$
		\item 1 node and 2 double: $\dot{z}=z^3(z+1).$
		\item 2 foci and 1 double: $\dot{z}=z^2(z-(1-3i))(z-(2-2i)).$
		\item 1 center, 1 double and 1 focus: $\dot{z}=z^2(z-(1+3i))(z-1).$
		\item 1 center, 1 double and 1 node: $\dot{z}=z^2(z-(1-i))(z-1).$
		\item 1 double, 1 focus and 1 node: $\dot{z}=z^2(z-(2+2i))(z-(2+i)).$
		\item 2 foci and 2 centers: $\dot{z}=z(z-(1+i))(z-(3/5+3i))(z-(3-2i)).$
		\item 1 node, 1 focus and 1 double: $\dot{z}=z^2(z-(1-3i))(z-(-1/3-2i)).$
		\item 1 node, 1 center and 2 foci: $\dot{z}=z(z-(2+2i))(z-(4-2i))(z-(3-i)).$
		\item 2 nodes, 1 center and 1 focus: $\dot{z}=z(z-(3-i))(z-(39/25-52/25i))(z-(507/125-169/125i)).$
		\item 4 foci colinear: $\dot{z}=(-1+3i)z(z-1)(z-4)(z-8).$

       \item 4 foci triangle: $\dot{z}=(-1+3i)z(z-(1+3i))(z-2)(z-(3+12i)).$
%       	\begin{equation*}
%\left\{\begin{array}{ll}
%	\dot{x} & =-x^4 - 12x^3y + 6x^2y^2 + 12xy^3 - y^4 + 51x^3 + 9x^2y - 153xy^2 - 3y^3 - 128x^2 \\& \phantom{=} + 252xy + 128y^2 + 60x - 240y, \\  
%	\dot{y} & = 3x^4 - 4x^3y - 18x^2y^2 + 4xy^3 + 3y^4 - 3x^3 + 153x^2y + 9xy^2 - 51y^3 -126x^2 \\&
%	\phantom{=}-256xy+126y^2+240x+60y.
%\end{array}
%\right.
%	\end{equation*}
	\item 4 foci border: $\dot{z}=(-1+i)z(z-1)(z-2)(z-(3+12i)).$
%	\begin{equation*}
%	\left\{\begin{array}{ll}
%	\dot{x} & = -x^4 - 4x^3y + 6x^2y^2 + 4xy^3 - y^4 + 10x^3 + 6x^2y - 30xy^2 - 2y^3 - 23x^2 \\&\phantom{=}+ 2xy + 23y^2 + 14x - 2y,\\ 
%	\dot{y} & =x^4 - 4x^3y - 6x^2y^2 + 4xy^3 + y^4 - 2x^3 + 30x^2y + 6xy^2 - 10y^3 - x^2 - 46xy \\& \phantom{=}+ y^2 + 2x + 14y.
%	\end{array}
%	\right.
%	\end{equation*}
%	
\item 4 foci Quadrilateral: $\dot{z}=(-1+2i)z(z-3)(z-2i)(z-(2+2i)).$
%\begin{equation*}
%\left\{\begin{array}{ll}
%\dot{x} & =-x^4 - 8x^3y + 6x^2y^2 + 8xy^3 - y^4 + 13x^3 + 18x^2y - 39xy^2 - 6y^3 - 34x^2 \\& \phantom{=}+ 24xy + 34y^2 + 12x - 36y, \\ 
%\dot{y} & = 2x^4 - 4x^3y - 12x^2y^2 + 4xy^3 + 2y^4 - 6x^3 + 39x^2y + 18xy^2 - 13y^3 - 12x^2 \\& \phantom{=} - 68xy + 12y^2 + 36x + 12y.
%\end{array}
%\right.
%\end{equation*}	
	\item 4 foci Quadrilateral: $\dot{z}=(-1+2i)z(z-3)(z-2i)(z-(22/10+2i)).$
%	\begin{equation*}
%	\left\{\begin{array}{ll}
%	\dot{x} & =14.4x + 35.4y^2 + 8y^3x - 8x^3y + 6y^2x^2 + 19.2x^2y - 35.4x^2 + 22.4xy \\& \phantom{=}- 39.6y^2x - 37.2y - x^4 - y^4 - 6.4y^3 + 13.2x^3, \\ 
%	\dot{y} & = -70.8xy - 4x^3y + 4y^3x + 19.2y^2x + 39.6x^2y - 12y^2x^2 - 13.2y^3 - 6.4x^3 \\& \phantom{=}- 11.2x^2 + 37.2x + 11.2y^2 + 14.4y + 2y^4 + 2x^4.
%	\end{array}
%	\right.
%	\end{equation*}	
	
	\item 4 foci Quadrilateral: $\dot{z}=(-1+2i)z(z-3)(z-2i)(z-(3+2i)).$
%	\begin{equation*}
%	\left\{\begin{array}{ll}
%	\dot{x} & =-x^4-8 x^3 y+6 y^2 x^2+8 y^3 x-y^4+14 x^3+24 x^2 y-42 y^2 x-8 y^3-41x^2\\& \phantom{=}+16 x y+41 y^2+24 x-42 y, \\ 
%	\dot{y} & = 2x^4 - 4x^3y - 12x^2y^2 + 4xy^3 + 2y^4 - 8x^3 + 42x^2y + 24xy^2 - 14y^3 - 8x^2 \\& \phantom{=} - 82xy + 8y^2 + 42x + 24y.
%	\end{array}
%	\right.
%	\end{equation*}	
	\item 4 centers: $\dot{z}=z(z^3-i/3).$
\end{enumerate}
\end{remark}

\begin{remark}
	The examples given in phase portraits $\textbf{Q22)}-\textbf{Q27)}$ was given in \cite{AlvGasPro} and Figure $28)$ was given in \cite{GAP}.
\end{remark}

\section{Phase portrait of the family $\dot{z}=\dfrac{1}{f(z)}$}\label{sec6}

In this Section, we study the phase portraits of the family $\dot{z}=\dfrac{1}{f(z)}$, with $f(z)$ polynomial of degree $2,3,$ and $4$.

\begin{proposition}\label{teo1pq}
	Let $f(z)$ be a rational function, i.e., $f(z)=\dfrac{P(z)}{Q(z)},$ where $P(z)=a_{n}z^{n}+\cdots+a_{0}$ and $Q(z)=b_{m}z^{m}+\cdots+b_{0}$ are polynomials in $z$ of degree $n$ and $m$ respectively. Let $c$ be the residue of $g(z)=\dfrac{Q(1/z)}{z^{2}P(1/z)}$ at $z=0$. Then, there exists $R>0$ such that the corresponding equation $\dot{z}=f(z)$ is conformally conjugated, in $\mathbb{C}\backslash \overline{\mathbb{D}(0,R)}$, to
	
\begin{enumerate}[(a)]
	\item $\dot{z}=(1/z)^{m-n}+c(1/z)^{2(m-n)+1}$, if $n<m+1$,
	\item $\dot{z}=(a_{n}/b_{m})z$, if $n=m+1$,
	\item $\dot{z}=z^{2}$, if $n=m+2$,
	\item $\dot{z}=(z)^{n-m}$, if $n>m+2$.
\end{enumerate}	
\end{proposition}
See \cite{GXG} for a proof.\\

\begin{proposition}\label{phaseinfinity}
Consider the equations given by $\dot{z}_{1}=f_{1}(z)^{-1}$, $\dot{z}_{2}=f_{3}(z)^{-1}$ and $\dot{z}_{4}=f_{4}(z)^{-1},$ where $f_{1}(z),f_{2}(z),f_{3}(z)$ are polynomials of degree 2, 3 and 4, respectively. Then, in a neighborhood of infinity, the phase portrait of each equations above is conformally conjugated to the phase portrait in a neighborhood of $\dot{z}=(1/z)^2$, $\dot{z}=(1/z)^3$ and $\dot{z}=(1/z)^4$, respectively.

\begin{center}
\begin{figure}[h]	
\begin{overpic}[scale=0.2]
	{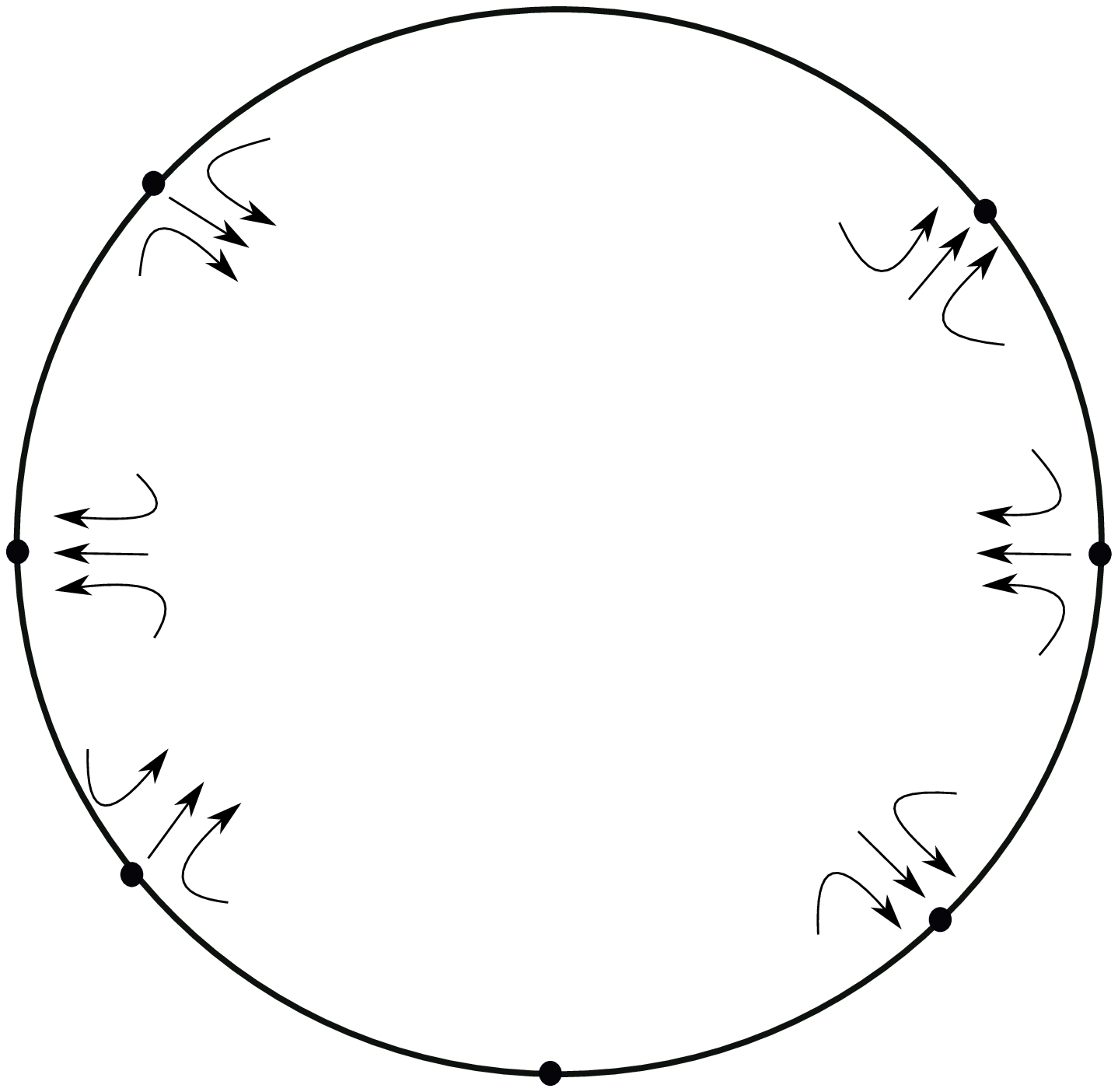} 
\end{overpic}
\hspace{1cm}
\begin{overpic}[scale=0.2]
	{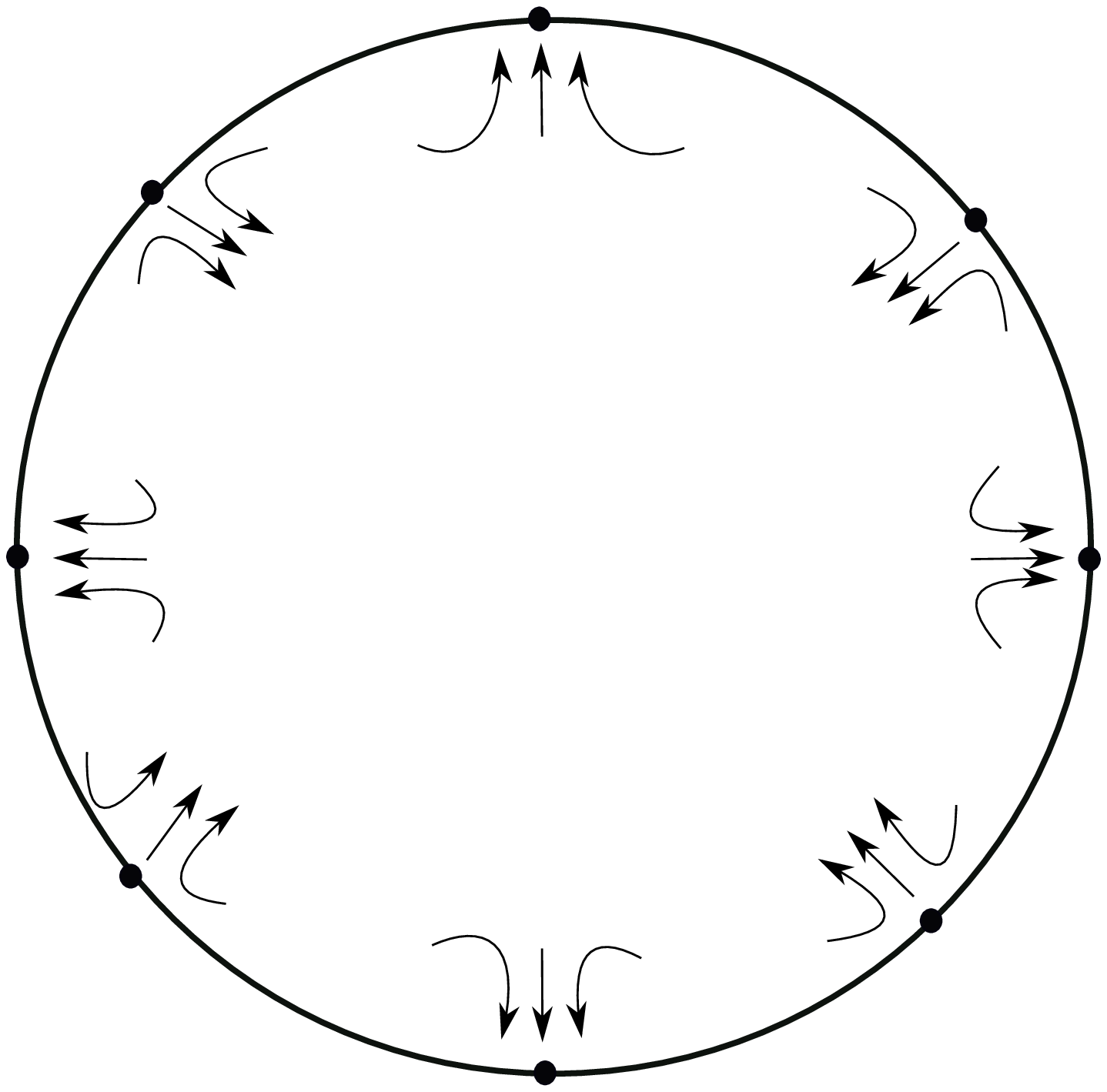} 
\end{overpic}
\hspace{1cm}
\begin{overpic}[scale=0.2]
	{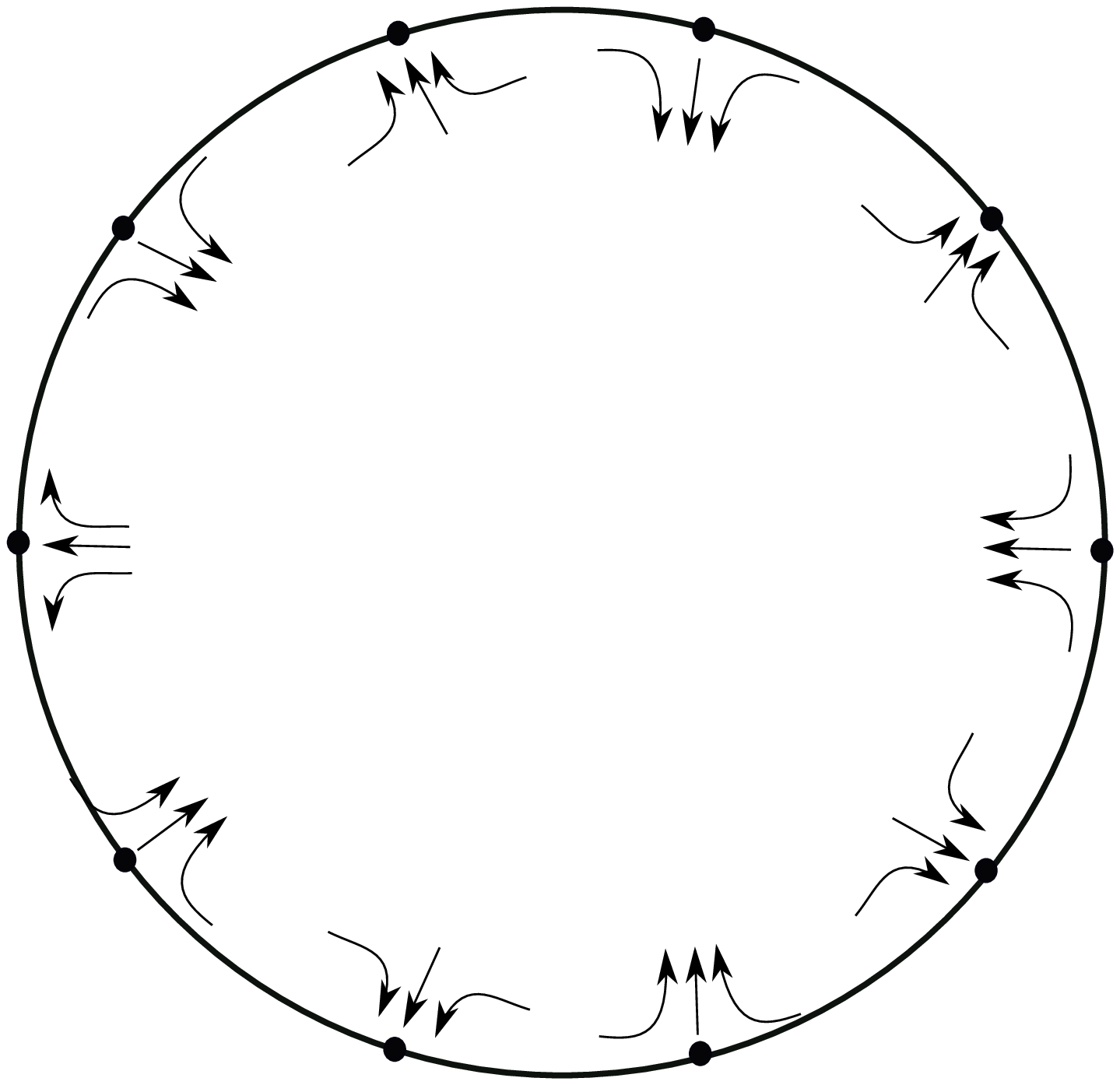} 
\end{overpic}
\caption{Dynamic near infinity of $\dot{z}=\dfrac{1}{z^2}$, $\dot{z}=\dfrac{1}{z^3}$ and $\dot{z}=\dfrac{1}{z^4}$ respectively.}
\end{figure}
\end{center}
\end{proposition}	

\begin{proof}
	Applying Proposition \ref{teo1pq}, item $(a)$, we get $n=0$ and $m=2$ for $\dot{z}_{1}$, $n=0$ and $m=3$, for $\dot{z}_{2}$ and $n=0$ and $m=4$ for $\dot{z}_{3}$. Then, we have $\dot{z}=(1/z)^{2}+c_{1}(1/z)^{5}$, $\dot{z}=(1/z)^{3}+c_{2}(1/z)^{7}$ and $\dot{z}=(1/z)^{4}+c_{3}(1/z)^{9}$. As we have $c_{1}=c_{2}=c_{3}=0$, we obtain $\dot{z}=(1/z)^{2}$, $\dot{z}=(1/z)^{3}$ and $\dot{z}=(1/z)^{4}$ respectively. 
\end{proof}

In this section, we  study  the phase portrait of the systems

\begin{equation}\label{eqs2}
\dot{z}=\dfrac{1}{z(z-A_{1})},
\end{equation}
\begin{equation}\label{eqs3}
 \dot{z}=\dfrac{1}{z(z-A_{1})(z-A_{2})},
\end{equation}
and 
\begin{equation}\label{eqs4}
\dot{z}=\dfrac{1}{z(z-{A}_{1})(z-{A}_{2})(z-{A}_{3})}.
\end{equation}

Applying the change of variables $w=\frac{1}{z}$ in $\dot{z}=\frac{1}{z^2}$, $\dot{z}=\frac{1}{z^3}$ and $\dot{z}=\frac{1}{z^4}$, we obtain

\begin{equation*}
\dot{w}=-\dfrac{1}{z^2}\dot{z}=-\dfrac{1}{z^2}\frac{1}{z^2}= -\dfrac{1}{z^4}=-w^4,
\end{equation*}

\begin{equation*}
\dot{w}=-\dfrac{1}{z^2}\dot{z}=-\dfrac{1}{z^2}\frac{1}{z^3}= -\dfrac{1}{z^5}=-w^5,
\end{equation*}

\begin{equation*}
\dot{w}=-\dfrac{1}{z^2}\dot{z}=-\dfrac{1}{z^2}\frac{1}{z^4}= -\dfrac{1}{z^5}=-w^6,
\end{equation*}
respectively. 

We know that the phase portrait of $\dot{w}=-w^5$ is topologically equivalent in a neighborhood of the origin of $w=0$, which is equivalent to $z=\infty$, to that of Figure 6.(1). Applying a blow-up of $w=0$, and a simple inversion we obtain Figure \ref{figblowup}.

\begin{center}
	\begin{figure}[h]\label{figblowup}
\begin{overpic}[scale=0.35]
	{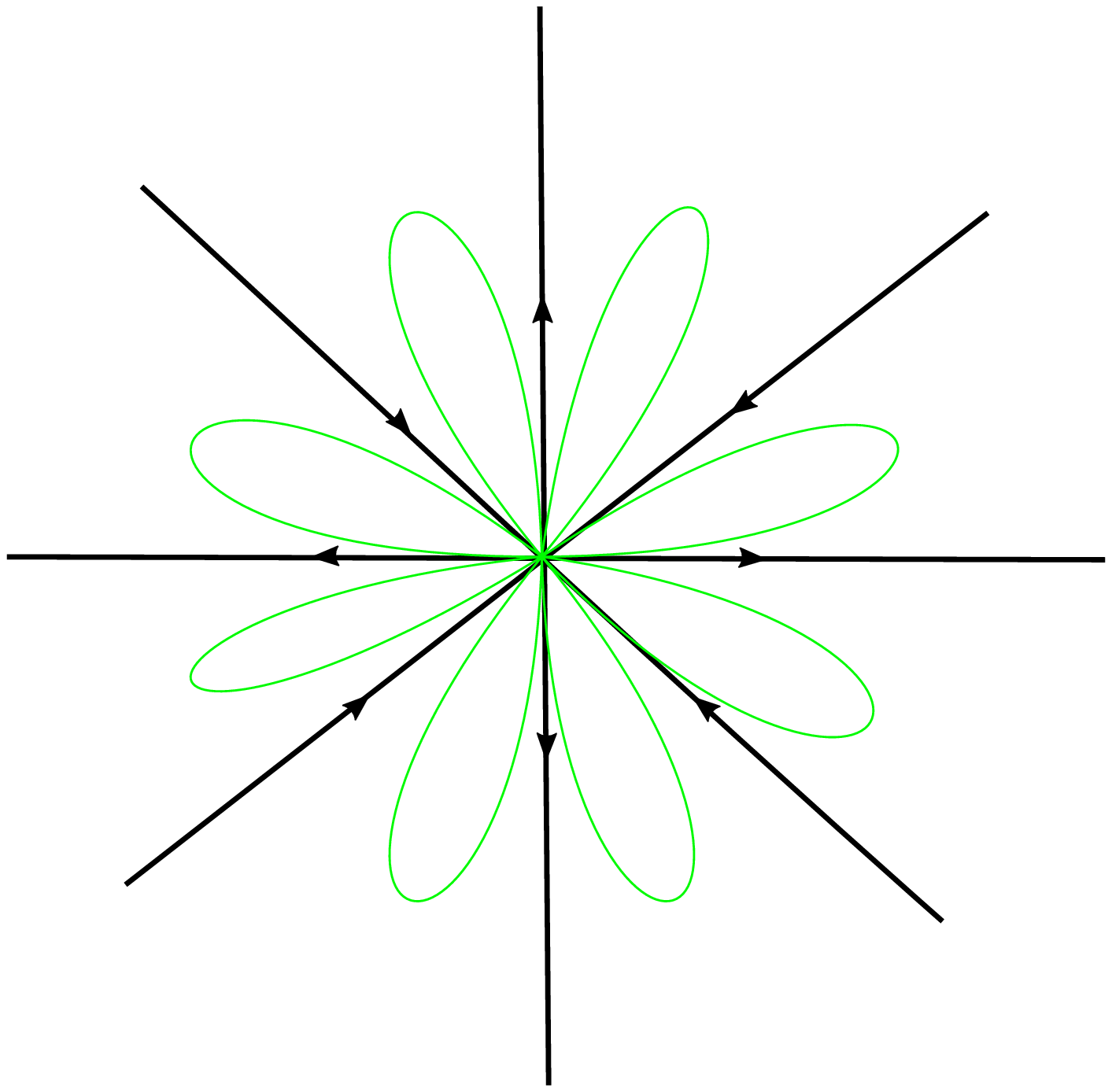} 
	\put(65,-20){(1)}
\end{overpic}
\hspace{1cm}
\begin{overpic}[scale=0.35]
	{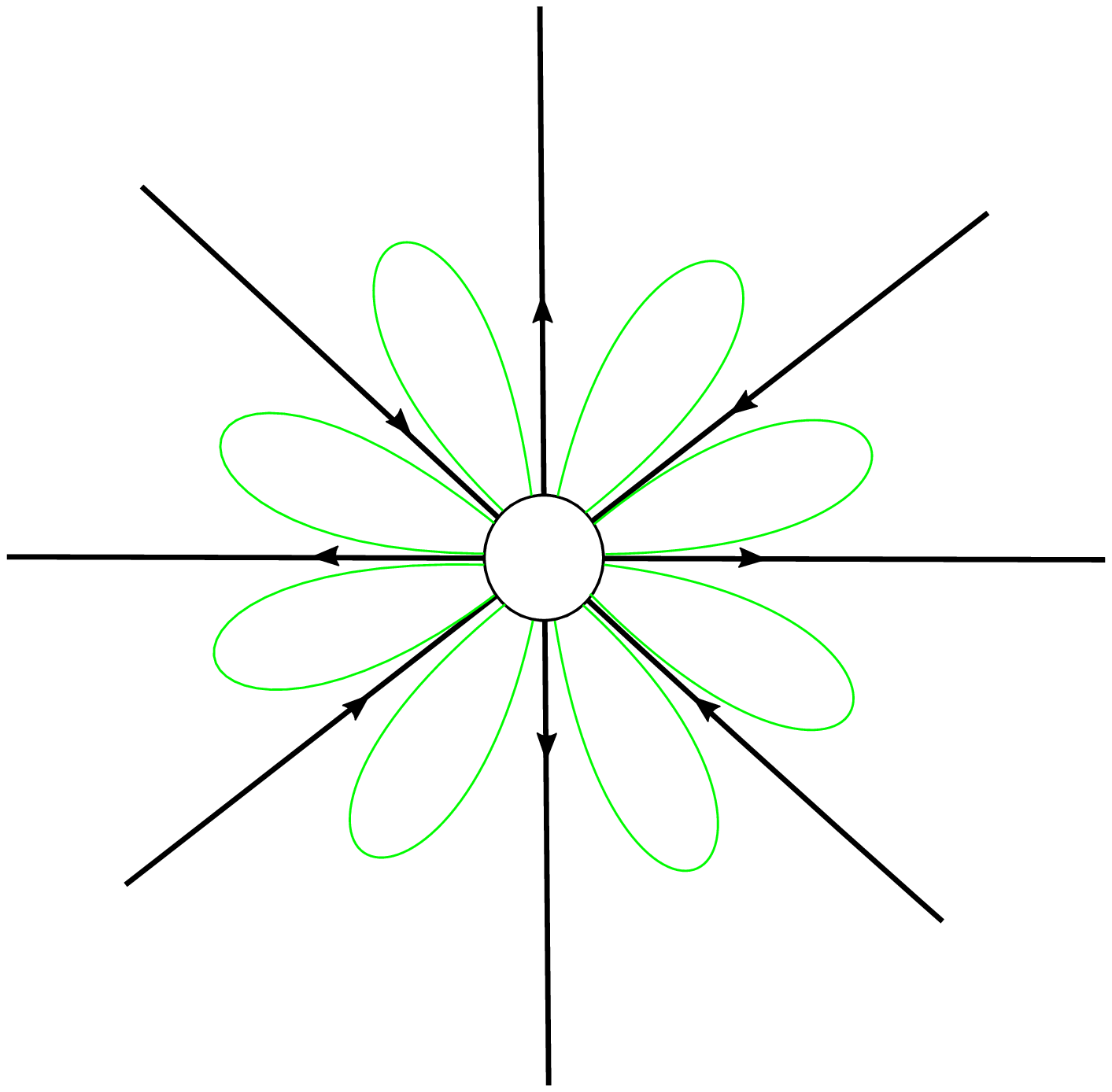} 
	\put(65,-20){(2)}
\end{overpic}
\vspace{0.5cm}
\caption{Topological phase portrait of $w=-w^5$ at $w=0$  and blow-up of the origin.}
\end{figure}
\end{center}

Now, we study the phase portrait in the whole Poincar\'e disc. The next theorems provides the phase portrait of the systems \eqref{eqs3} and \eqref{eqs4}. 

\begin{theorem}\label{teoquadinv}
	For the the system $\dot{z}=\dfrac{1}{z(z-A_{1})}$, we have, through topological equivalence,  the following phase portraits.\end{theorem}

\begin{overpic}[scale=0.25]
	{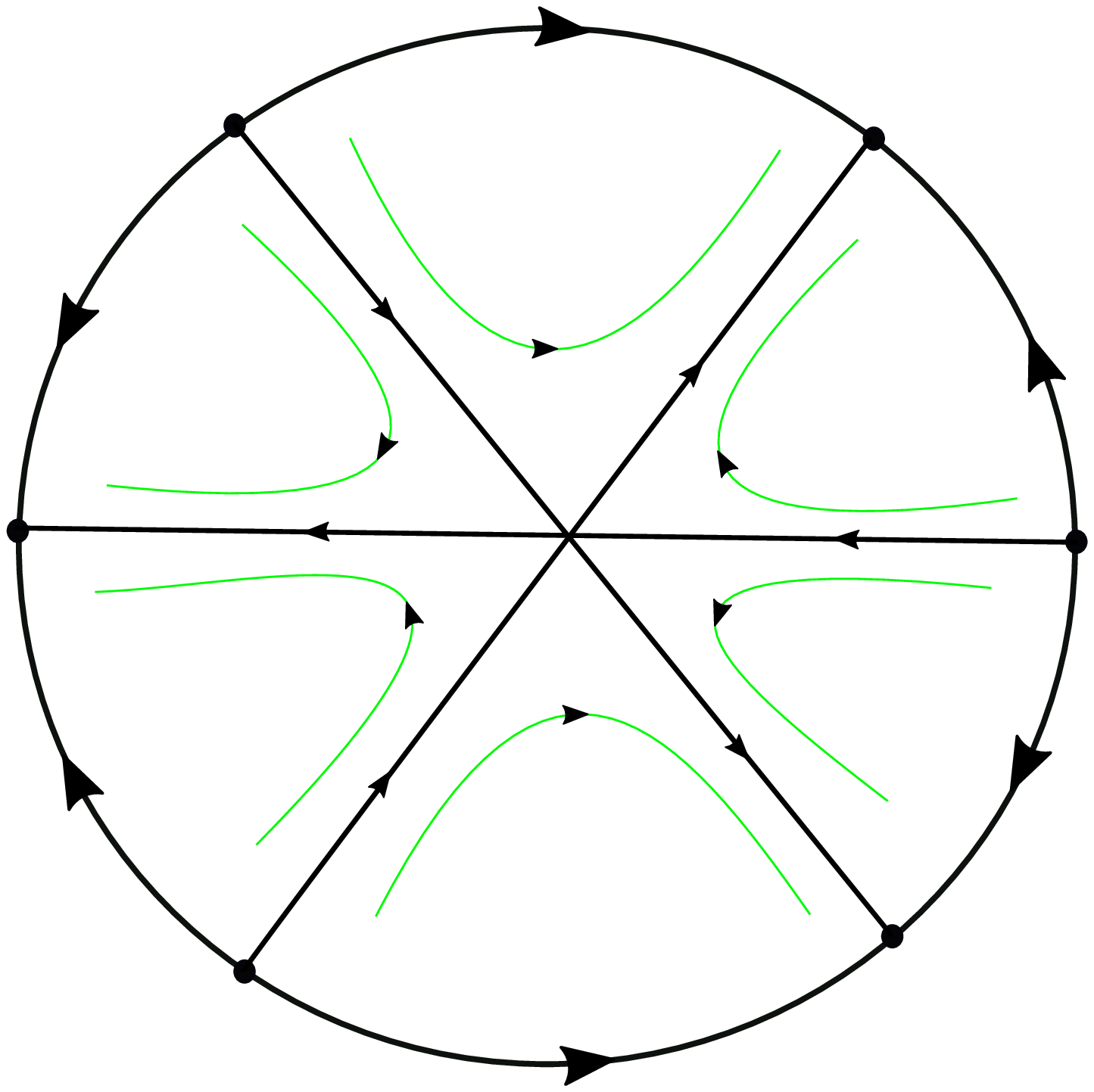} 
		\put(50,-20){\textbf{S1)}}
\end{overpic}
\begin{overpic}[scale=0.25]
	{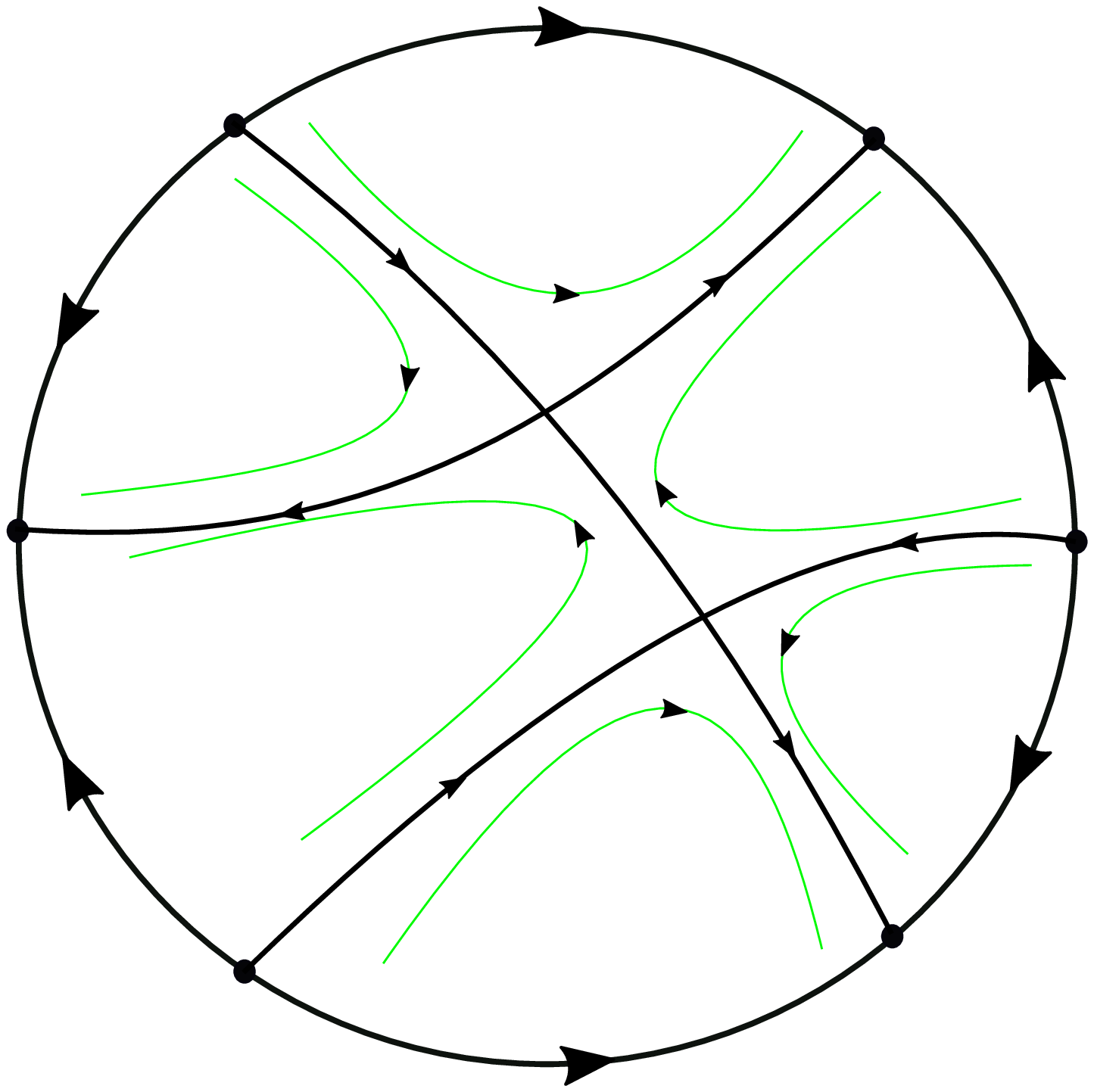} 
		\put(50,-20){\textbf{S2)}}
\end{overpic}
\begin{overpic}[scale=0.25]
	{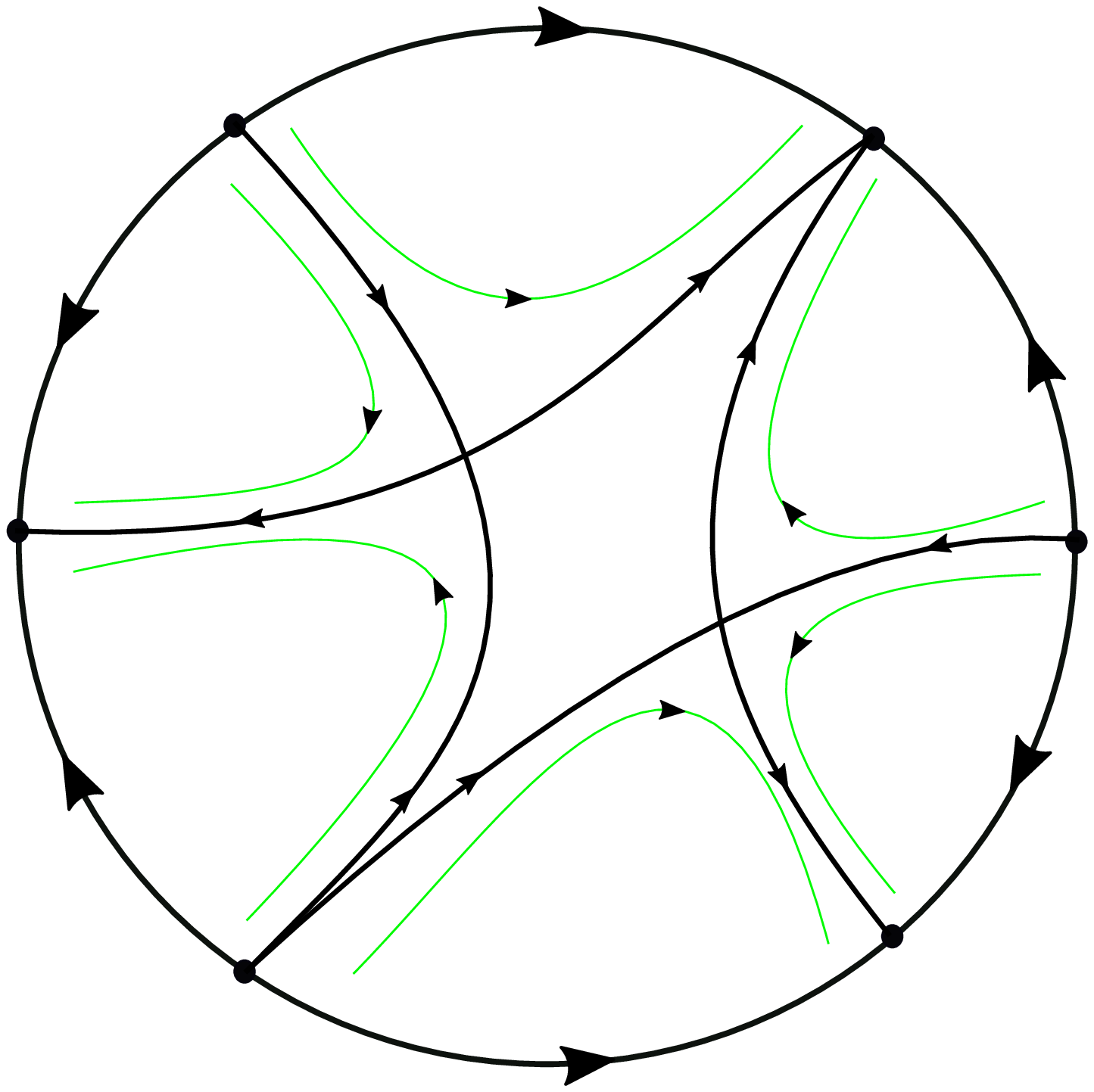} 
		\put(50,-20){\textbf{S3)}}
\end{overpic}

\vspace{1cm}

\begin{proof}
This Theorem can also be found in \cite{Alba} and, for completeness sake we will prove it below. Firstly, let us consider two cases. If $A_{1}=0$ then the origin is a pole of order 2 and by Proposition \ref{sver1} we have Figure $\textbf{S1)}$. If $A_{1}\neq0$, then $A_{1}$ and the origin are distinct poles of order 1. By Proposition \ref{phaseinfinity} we know the phase portrait near infinity, so it must be given by Figure $\textbf{S2)}$ or $\textbf{S3)}$ depending on the separatrices connection. Moreover, each case is realizable.
\end{proof}

\begin{theorem}\label{teocubinv}
	For the system $\dot{z}=\dfrac{1}{z(z-A_{1})(z-A2)}$, we have, through topological equivalence,  the following phase portraits.\end{theorem}

\begin{center}
	\begin{overpic}[scale=0.25]
		{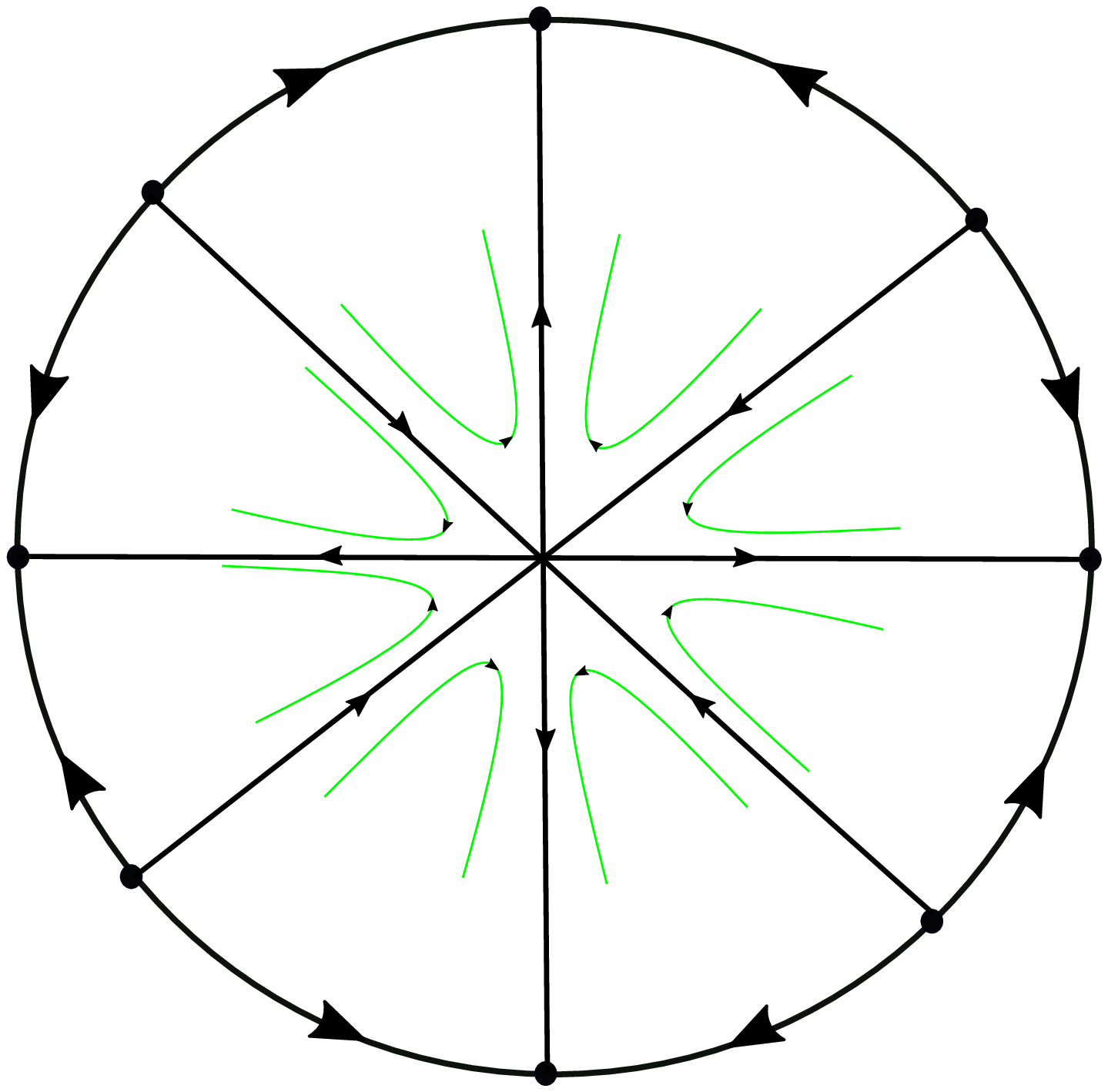} 
		\put(50,-20){\textbf{Sc1)}}
	\end{overpic}
\hspace{1cm}
	\begin{overpic}[scale=0.25]
		{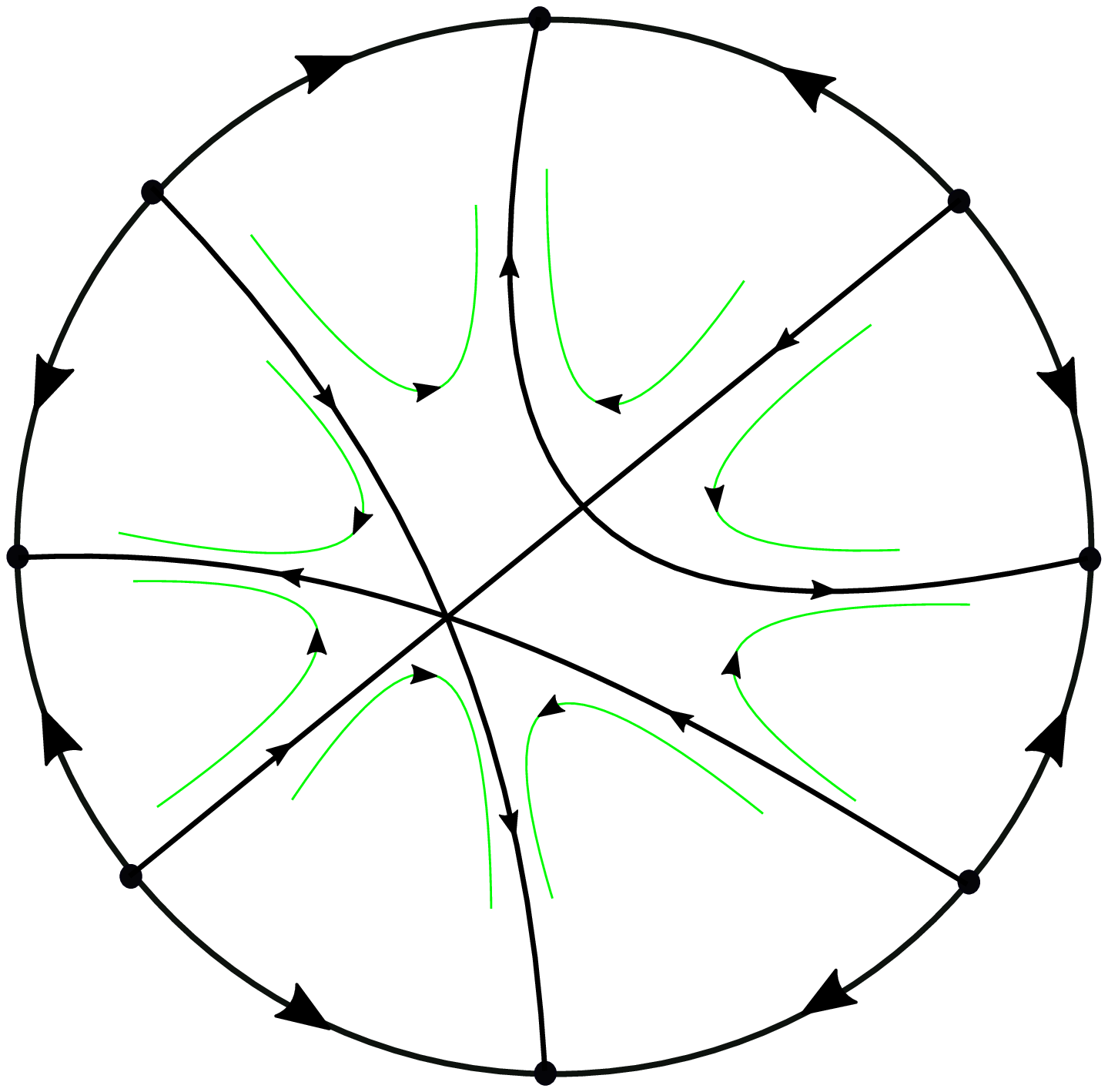} 
		\put(50,-20){\textbf{Sc2)}}
	\end{overpic}
\end{center}

\vspace{1cm}

\begin{center}
    \begin{overpic}[scale=0.25]
	{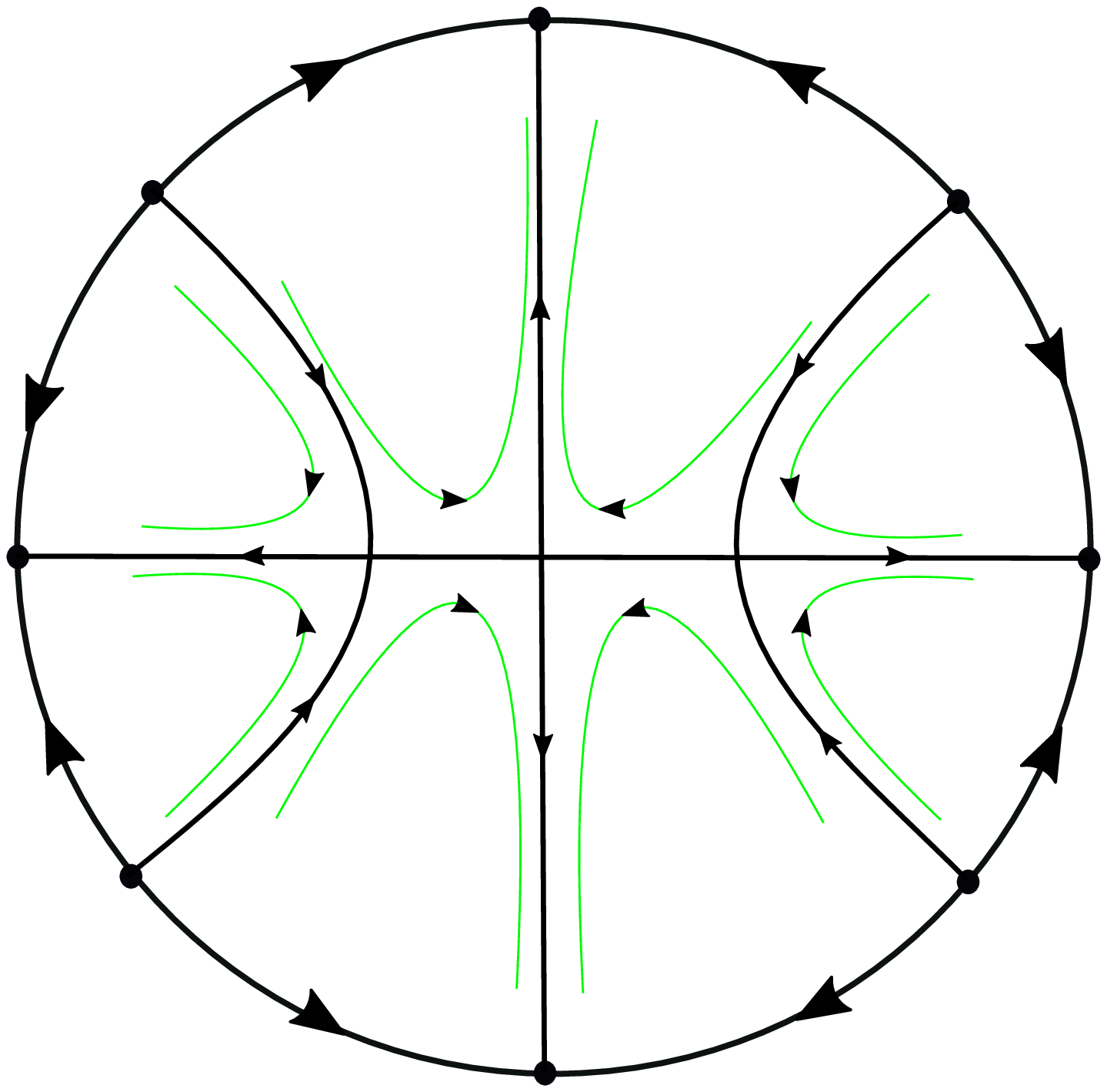} 
	\put(50,-20){\textbf{Sc3)}}
    \end{overpic}
\hspace{1cm}
    \begin{overpic}[scale=0.25]
	{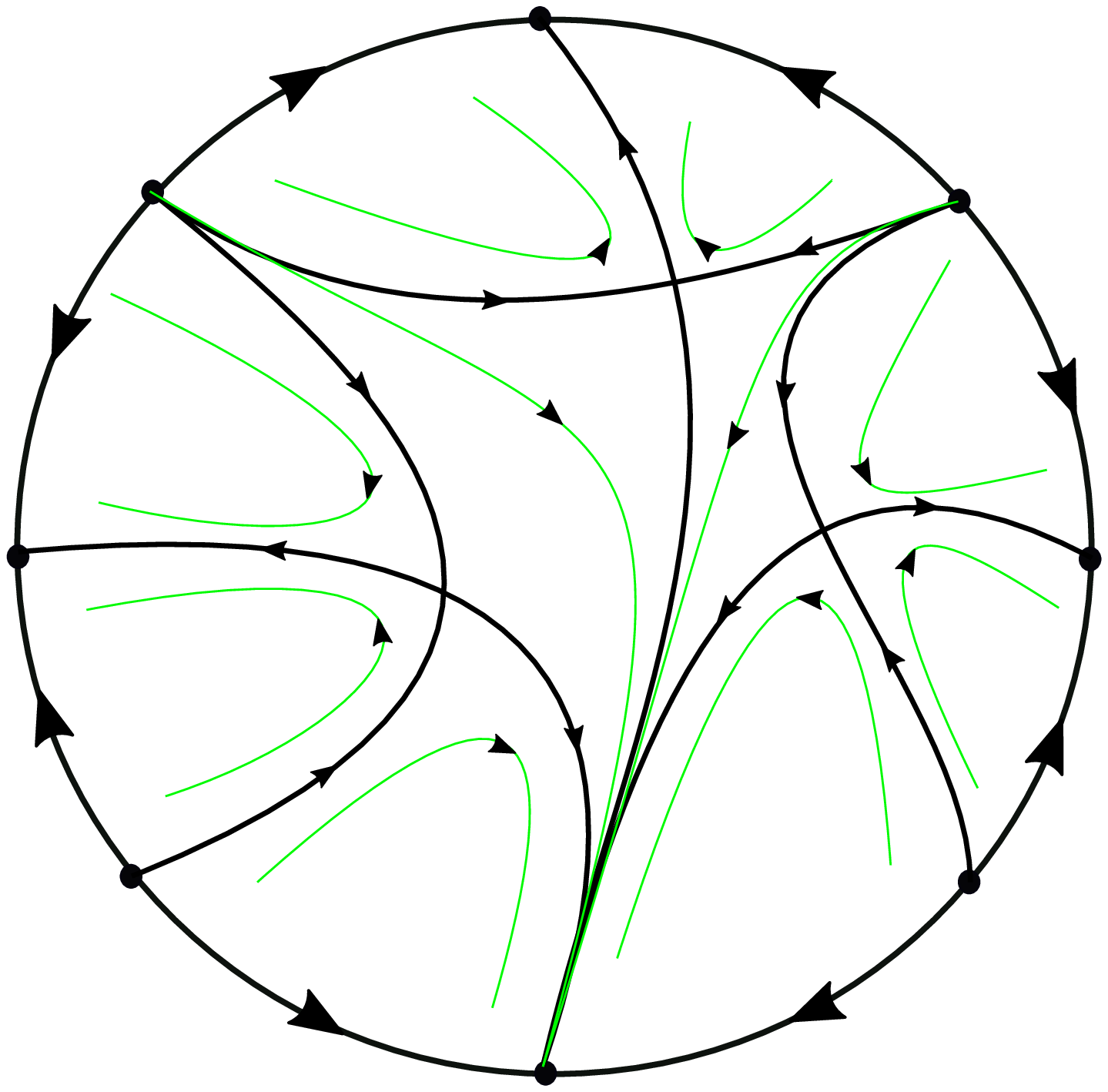} 
	\put(50,-20){\textbf{Sc4)}}
    \end{overpic}	
\end{center}

\vspace{0.5cm}

\begin{proof}
	 If $A_{1}=A_{2}=0$, then the origin is a pole of order 3 and by Proposition \ref{sver1} we have Figure $Sc1$. For the remaining cases, the phase portraits
depend on separatrices connections. If $A_{1}=A_{2}\neq 0$, we have two poles, one of order 1 and other of order 2. Then, we have one possibility to  connect and it is realizable, see figure $Sc2$. For the last case, we have $A_{1}\neq A_{2}$ and both distinct of zero. In this case we have 3 polos of order one and we have 2 possibilities to connect and they are realizable, see figures $Sc3$ and $Sc4$.
	
To obtain the phase portrait in Figures $Sc1$ to $Sc4$, we take the following systems.

\begin{multicols}{2}
	\begin{enumerate}
	\item $\dot{z}=\dfrac{1}{z^{3}}$. \item $\dot{z}=\dfrac{1}{z(z-1)^2}$.
	\vspace{0.2cm}
	\item $\dot{z}=\dfrac{1}{z(z-1)(z-2)}$.
	\vspace{0.2cm}
	\item $\dot{z}=\dfrac{1}{z(z-i)(z-2)}$.
	\end{enumerate}
\end{multicols}
\end{proof}

\begin{theorem}\label{teoquartinv}
	For the system $\dot{z}=\dfrac{1}{z(z-{A}_{1})(z-{A}_{2})(z-{A}_{3})}$, we have, through topological equivalence,  the following phase portraits.\end{theorem}

\begin{overpic}[scale=0.25]
	{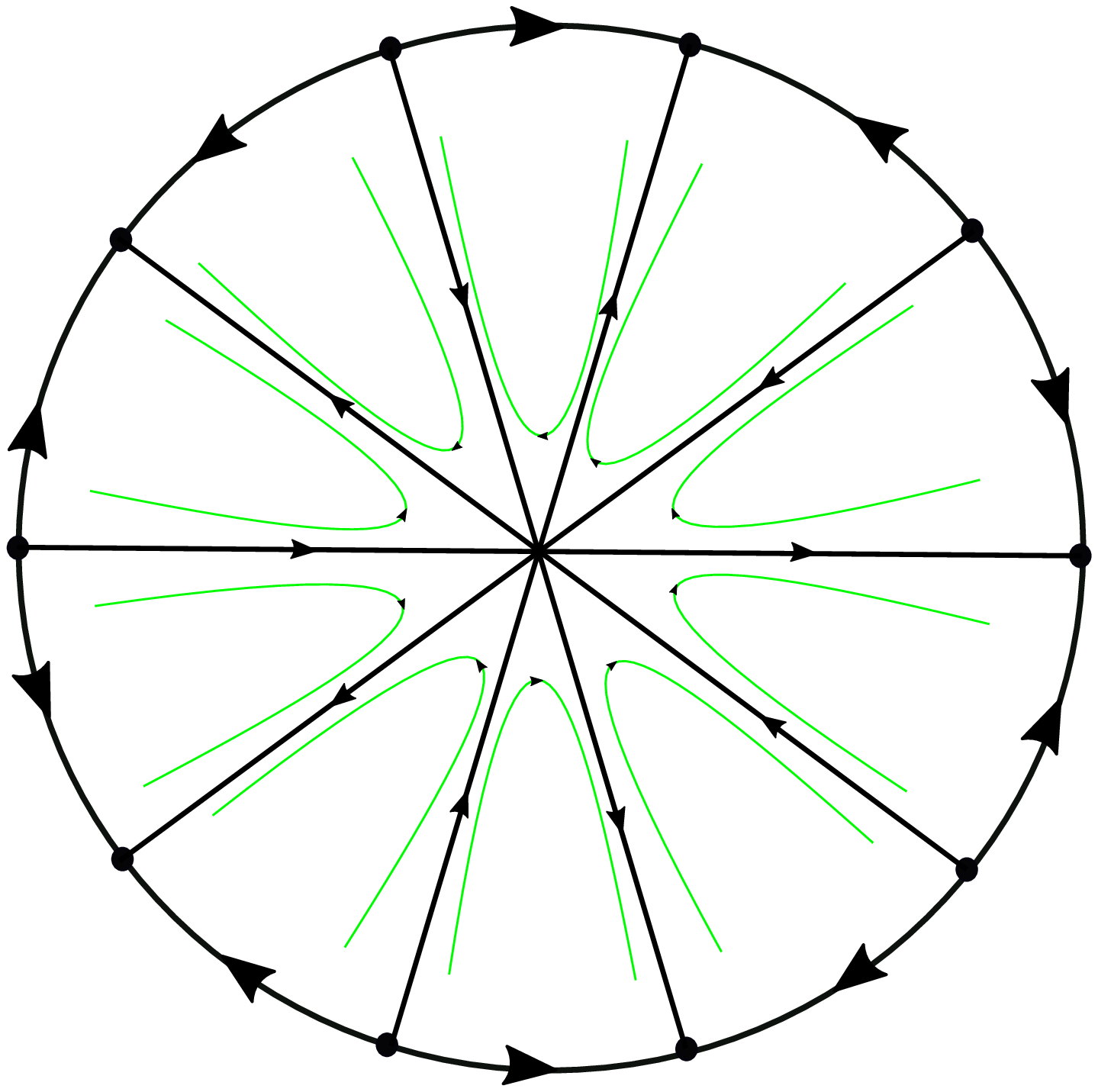} 
	\put(50,-20){\textbf{Sq1)}}
\end{overpic}
\begin{overpic}[scale=0.25]
	{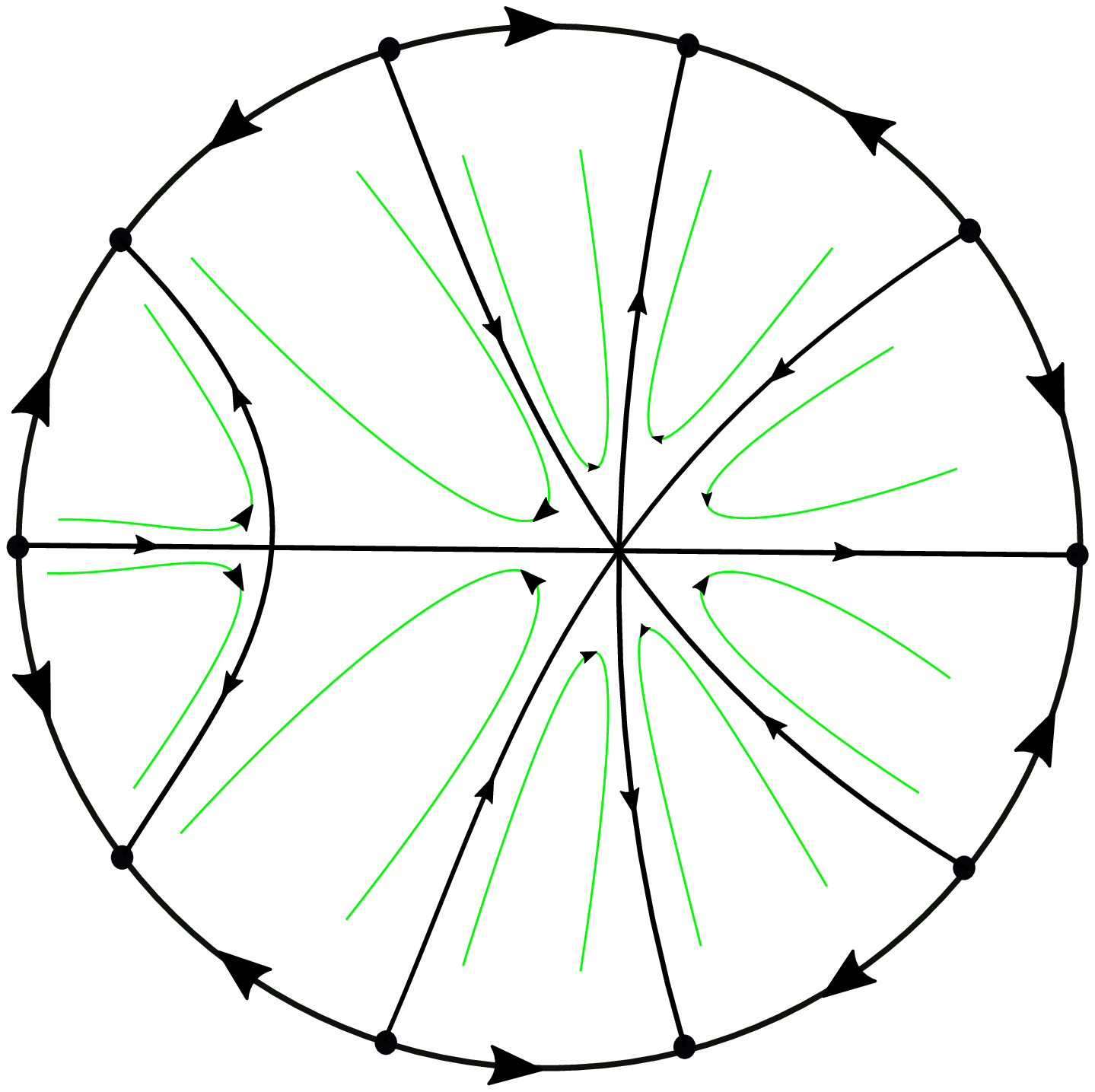} 
	\put(50,-20){\textbf{Sq2)}}
\end{overpic}
\begin{overpic}[scale=0.25]
	{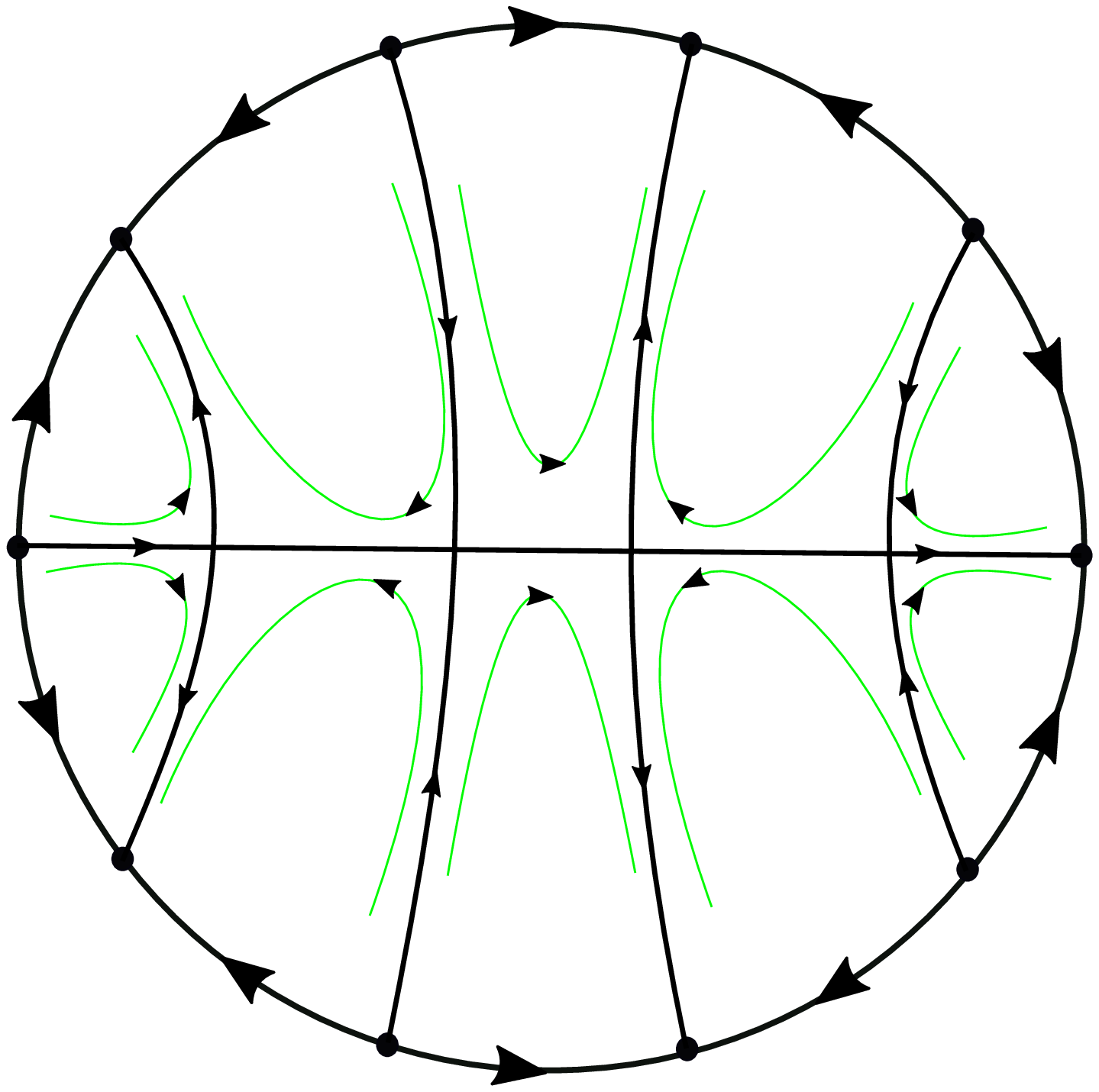} 
	\put(50,-20){\textbf{Sq3)}}
\end{overpic}

\vspace{1cm}

\begin{overpic}[scale=0.25]
	{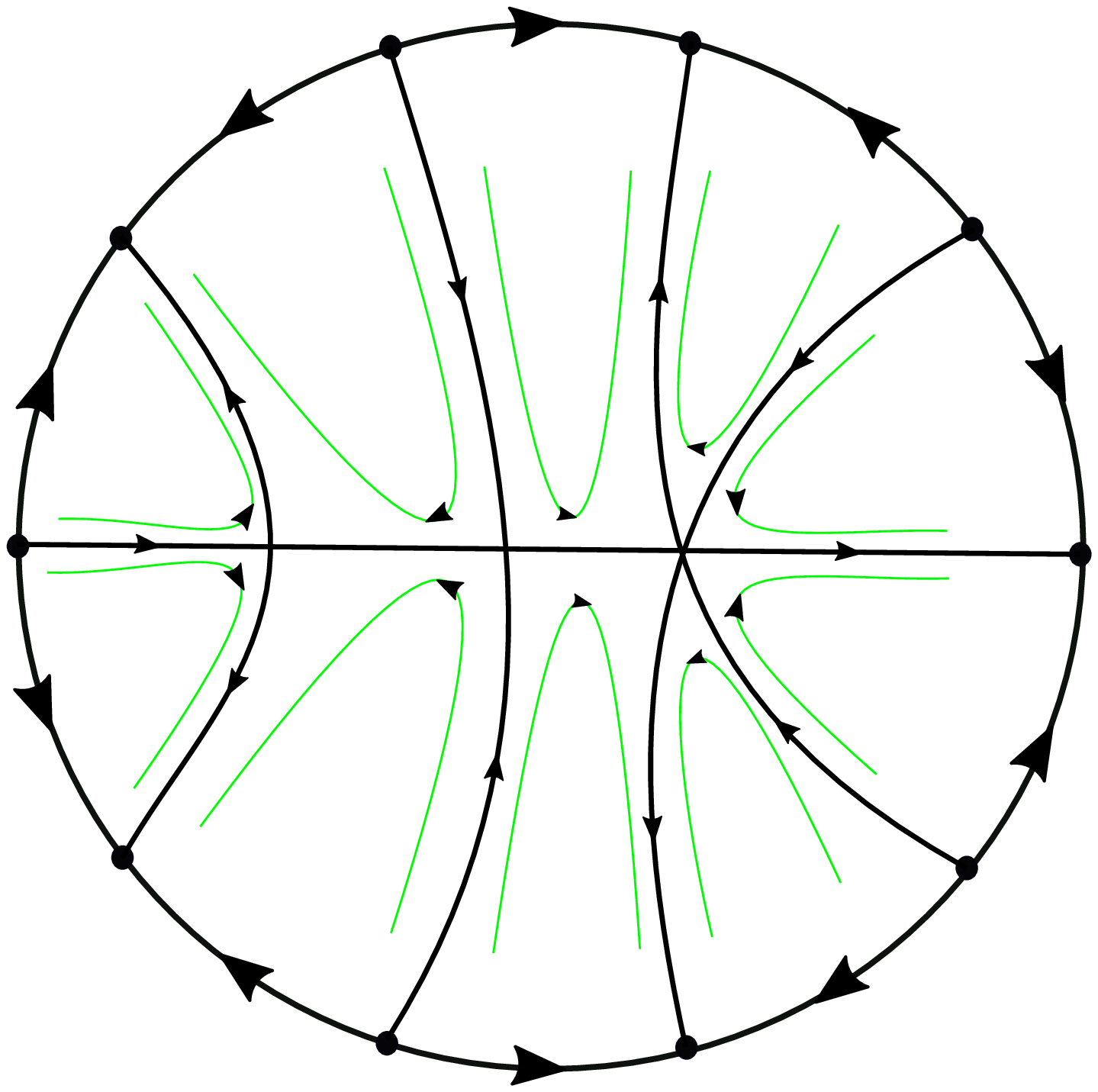} 
	\put(50,-20){\textbf{Sq4)}}
\end{overpic}
\begin{overpic}[scale=0.25]
	{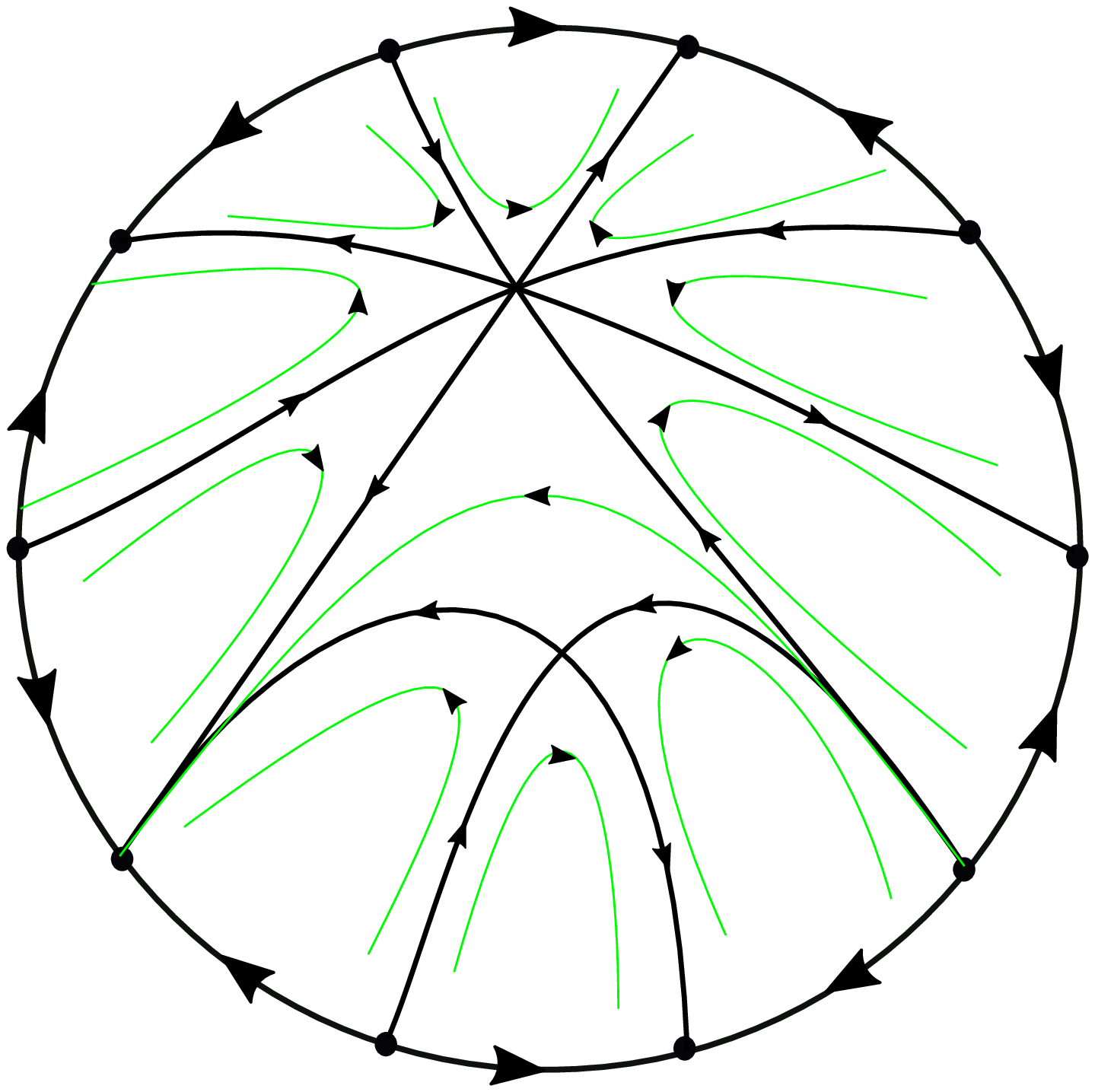} 
	\put(50,-20){\textbf{Sq5)}}
\end{overpic}
\begin{overpic}[scale=0.25]
	{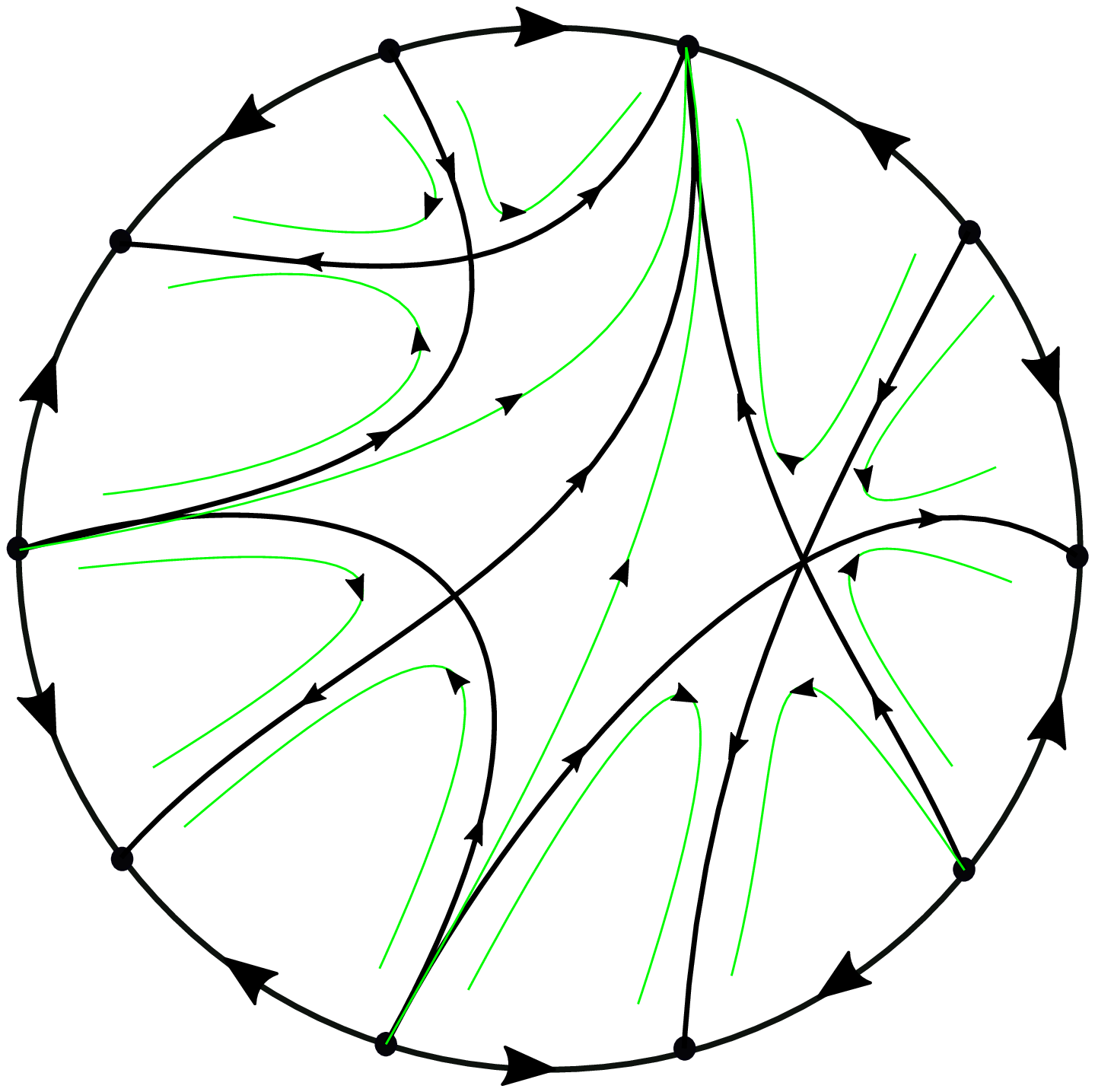} 
	\put(50,-20){\textbf{Sq6)}}
\end{overpic}

\vspace{1cm}

\begin{overpic}[scale=0.25]
	{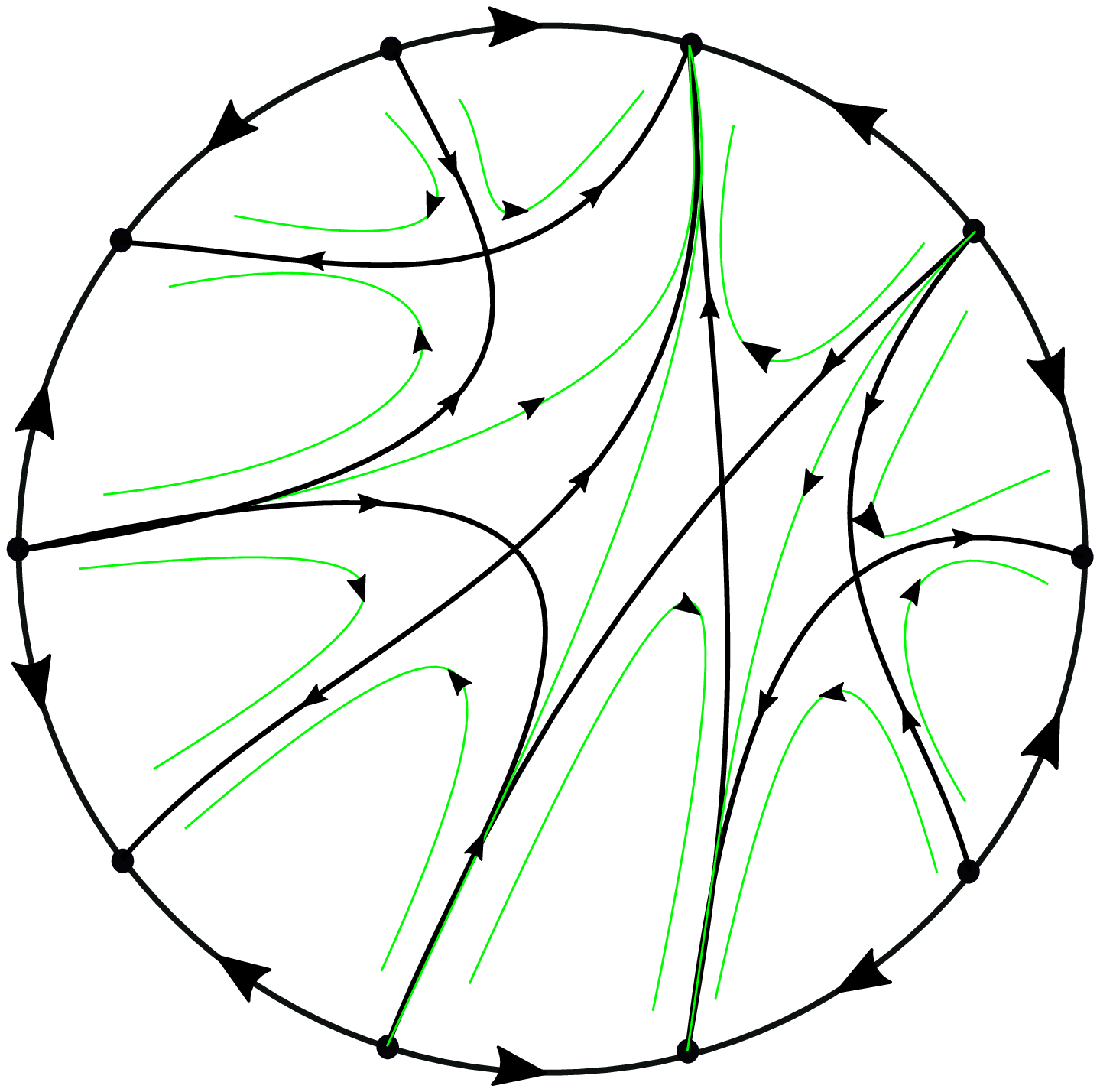} 
	\put(50,-20){\textbf{Sq7)}}
\end{overpic}
\begin{overpic}[scale=0.25]
	{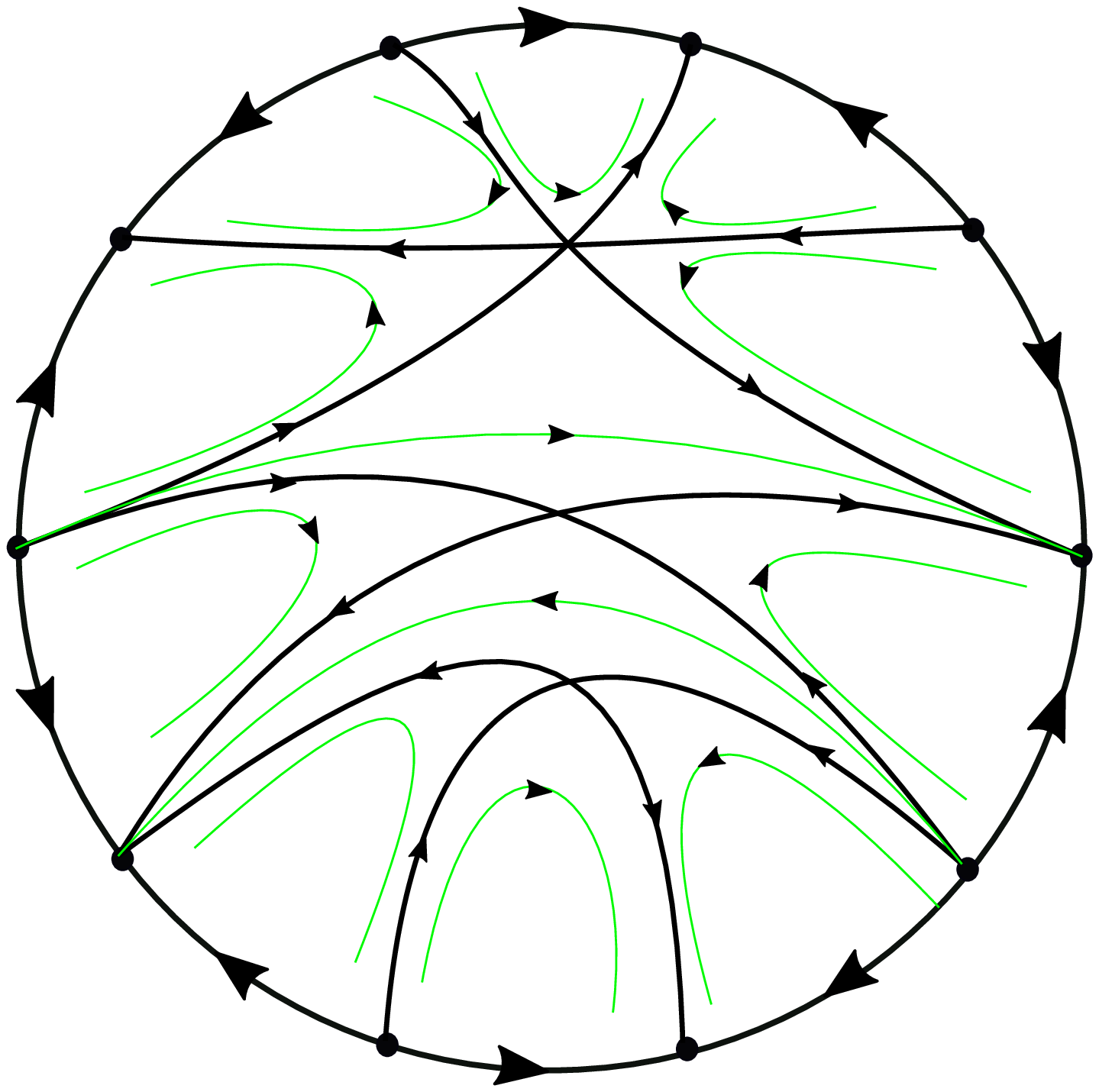} 
	\put(50,-20){\textbf{Sq8)}}
\end{overpic}
\begin{overpic}[scale=0.25]
	{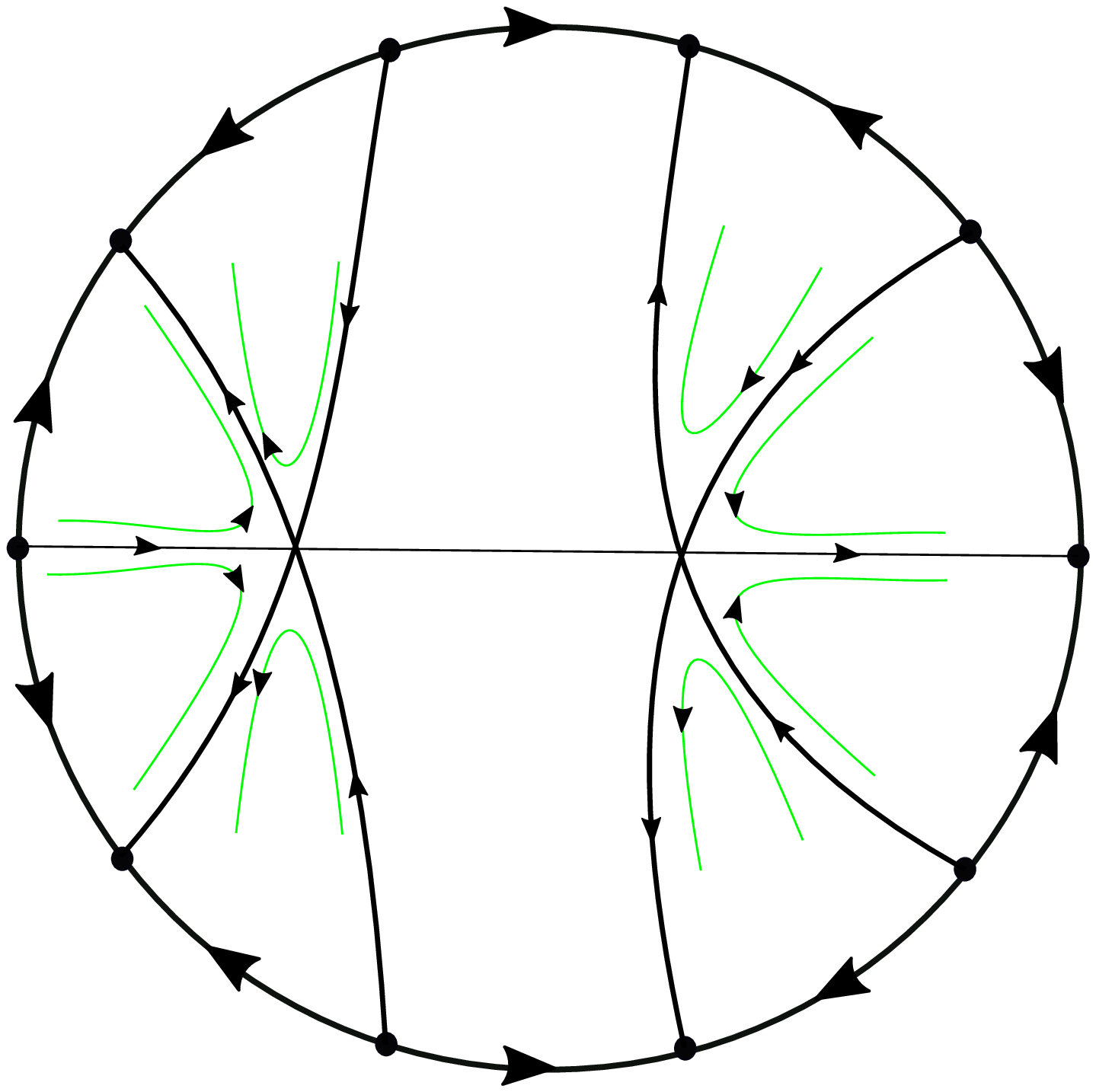} 
	\put(50,-20){\textbf{Sq9)}}
\end{overpic}

\vspace{1cm}

\begin{overpic}[scale=0.25]
	{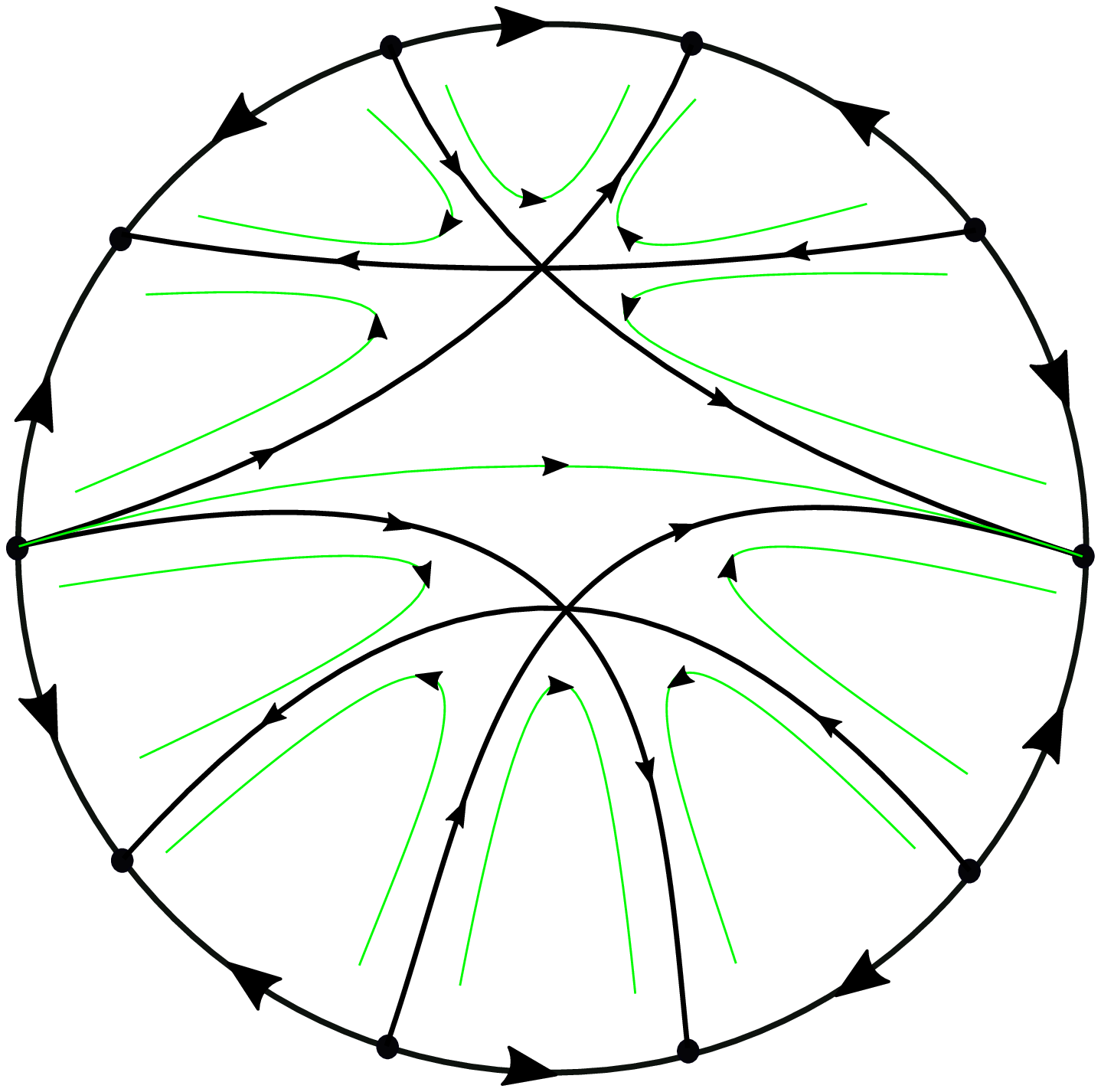} 
	\put(50,-20){\textbf{Sq10)}}
\end{overpic}
\begin{overpic}[scale=0.25]
	{fig11sq.eps} 
	\put(50,-20){\textbf{Sq11)}}
\end{overpic}

\vspace{1cm}

\begin{proof}
		For this system, we have to consider four cases. 
\begin{multicols}{2}
	\begin{enumerate}
		\item $A_{1}=A_{2}=A_{3}=0$.
		\item $A_{2}=A_{3}=0$.
		\item $A_{3}=0$.
		\item $A_{i}\neq 0, i=1\cdots3$.
	\end{enumerate}
\end{multicols}

In case $(1)$ the origin is a pole of order 4 and by Proposition \ref{sver1} we have Figure $Sq1$. For $(2)$, we have two poles of order two. Then the possibilities to the phase portrait is given by Figures $Sq2$, $Sq5$, $Sq9$ and $Sq10$ . Analyzing the case  
$(3)$, we have three polos, being one of order two and two poles of order one. This case corresponds to the phase portraits $Sq4$, $Sq6$, $Sq8$ and $Sq11$. For the last case $(4)$, we have four poles of order one and it corresponds to Figures $3)$ and $Sq7$.

To obtain the phase portrait in Figures $Sq1$ to $Sq11$, we take the following systems:
	
\begin{multicols}{2}
	\begin{enumerate}
		\item $\dot{z}=\dfrac{1}{z^{4}}$.
		\item $\dot{z}=\dfrac{1}{z(z-3)^3}$.
		\item $\dot{z}=\dfrac{1}{z(z-1)(z-2)(z-3)}$.
		\item $\dot{z}=\dfrac{1}{z(z-1)(z-2)^2}$.
		\item $\dot{z}=\dfrac{1}{z(z-i)^3}$.
		\item $\dot{z}=\dfrac{1}{z(z-i)(z-2)^2}$.
		\item $\dot{z}=\dfrac{1}{z(z-1)(z-2)(z-i)}$.
		\item $\dot{z}=\dfrac{1}{z(z-i)(z-2i)^2}$.
		\item $\dot{z}=\dfrac{1}{z^2(z-2)^2}$.
		\item $\dot{z}=\dfrac{1}{z^2(z-i)^2}$.
		\item $\dot{z}=\dfrac{1}{z(z-3)(z-2)^2}$.
	\end{enumerate}
\end{multicols}

\end{proof}

\section{Quadratic, Cubic, and Quartic Conjugate Systems}\label{sec9}

\noindent
\vspace{0.25cm}

In this section we will study the local dynamics of  polynomial conjugate systems $\overline{f_{j}(z)}$, $j=2,3,$ and $4$ where 

\begin{equation*}
\left.\begin{array}{lll}
f_2(z) &= & A_0+A_1z+A_2z^2, \\
f_3(z) &= & B_0+B_1z+B_2z^2+B_3z^3,\\
f_4(z) &= & C_0+C_1z+C_2z^2+C_3z^3+C_4z^4.
\end{array}\right.
\end{equation*}
As seen above, to know local dynamics of $\overline{f_{j}(z)}, j=2,\cdots,4$ it is enough calculate their respective primitives. In this case, the primitives are given by

\begin{equation*}
\left.\begin{array}{lll}
F_2(z) &= & A_0z+A_1\dfrac{z^2}{2}+A_2\dfrac{z^3}{3}, \\
F_3(z) &= & B_0z+B_1\dfrac{z^2}{2}+B_2\dfrac{z^3}{3}+B_3\dfrac{z^4}{4}, \\
F_4(z) &= & C_0z+C_1\dfrac{z^2}{2}+C_2\dfrac{z^3}{3}+C_3\dfrac{z^4}{4}+C_4\dfrac{z^5}{5}.
\end{array}\right.
\end{equation*}
In order to simplify writing, let us denote $A_k=B_k=C_k=a_k+b_ki$, $k=0,...,4$. As we have,

%\begin{equation}
%\left\{\begin{array}{l} 
%\dot{x}=a_2x^2 - a_2y^2 - 2b_2xy + a_1x - b_1y + a_0\\ 
%\dot{y}=-2a_2xy - b_2x^2 + b_2y^2 - a_1y - b_1x - b_0
%\end{array}\right.
%\end{equation}

\begin{equation*}
F_{j}'(z)=\phi^j_x+i\psi^j_x=u_j+iv_j=f_{j}(z),
\end{equation*}
$j=2\cdots4$, in cartesian coordinates it follows that
\begin{equation*}
\left\{\begin{array}{lll}
\phi^2(x,y) & = & 1/3a_2x^3 - a_2y^2x - b_2x^2y + 1/2a_1x^2 - b_1yx + a_0x + 1/3b_2y^3  \\
& & -1/2a_1y^2 - b_0y,\\
\psi^2(x,y) & = & a_2x^2y - 1/3a_2y^3 + 1/3b_2x^3 - b_2xy^2 + a_1xy + 1/2b_1x^2 \\
& & - 1/2b_1y^2 + a_0y + b_0x.
\end{array}\right.
\end{equation*}

\begin{equation*}
\left\{\begin{array}{lll}
\phi^3(x,y) & = &a_3\dfrac{x^4}{4}+a_3\dfrac{y^4}{4}-3a_3\dfrac{x^2y^2}{2}-b_3x^3y+b_ 3xy^3+ a_2\dfrac{x^3}{3}+b_2\dfrac{y^3}{3}\\ & &-a_2y^2x-b_2x^2y
+a_1\dfrac{x^2}{2} -a_1\dfrac{y^2}{2} -b_1xy+a_0x-b_0y,\\
\psi^3(x,y) & = &b_3\dfrac{x^4}{4}+b_3\dfrac{y^4}{4}+a_3x^3y-a_3xy^3-3b_3\dfrac{x^2y^2}{2}+
+a_2x^2y-a_2\dfrac{y^3}{3}\\
& & +b_2\dfrac{x^3}{3}-b_2xy^2 +b_1\dfrac{x^2}{2}-b_1\dfrac{y^2}{2}+a_1xy+a_0y+b_0x.
\end{array}\right.
\end{equation*}

\begin{equation*}
\left\{ \begin{array}{lll} 
\phi^4(x,y)&=&1/5a_4x^5 - 2a_4x^3y^2 + a_4y^4x - b_4x^4y + 2b_4x^2y^3 + 1/4a_3x^4\\ 
& & - 3/2a_3x^2y^2 
- b_3x^3y + b_3y^3x + 1/3a_2x^3 - a_2y^2x - b_2x^2y\\ 
& & + 1/2a_1x^2 - b_1yx + a_0x  - 1/5b_4y^5 + 1/4a_3y^4+ 1/3b_2y^3\\ 
& &  - 1/2a_1y^2 - b_0y,\\ 
\psi^4(x,y)&=&a_4x^4y - 2a_4x^2y^3 + 1/5a_4y^5 + 1/5b_4x^5 - 2b_4x^3y^2 + b_4xy^4 \\ 
& &+ a_3x^3y-a_3xy^3 + 1/4b_3x^4 - 3/2b_3x^2y^2 + 1/4b_3y^4 + a_2x^2y\\ 
& & - 1/3a_2y^3 + 1/3b_2x^3- b_2xy^2 + a_1xy + 1/2b_1x^2 - 1/2b_1y^2\\ 
& & + a_0y + b_0x.
\end{array}\right.
\end{equation*}

Let us consider now the conjugate of $A_{1}z,A_{2}z^2, A_{3}z^3$ and $A_{4}z^4$. As we get the explicit equations $\psi^j,j=2\cdots4$, it is enough to show the level curves of $\psi$ to obtain the local dynamic. Moreover, following the reasoning of  Poincaré Compactification, section \ref{sec3}, it is easy to show that the infinity of $A_{1}z,A_{2}z^2, A_{3}z^3$ and $A_{4}z^4$ are node points ($4,6,8$ and $10$  nodes points respectively) with alternating stability and considering the Poincaré Disk, diametrically opposite points are connected by a separatrix. As seen above, the finite singular points are the saddle type. Then, these facts allow us to make the local study and the global phase portraits of conjugates of  $A_{1}z,A_{2}z^2, A_{3}z^3$ and $A_{4}z^4$. Note that there is a relation between the degree of $z$ and the number of sectors. That is, the number of sectors around the origin is given by $2n+2$, where $n$ is degree of $z^{n}$. 

\begin{figure}[h]
\hspace{0.7cm}
\begin{overpic}[scale=0.30]
	{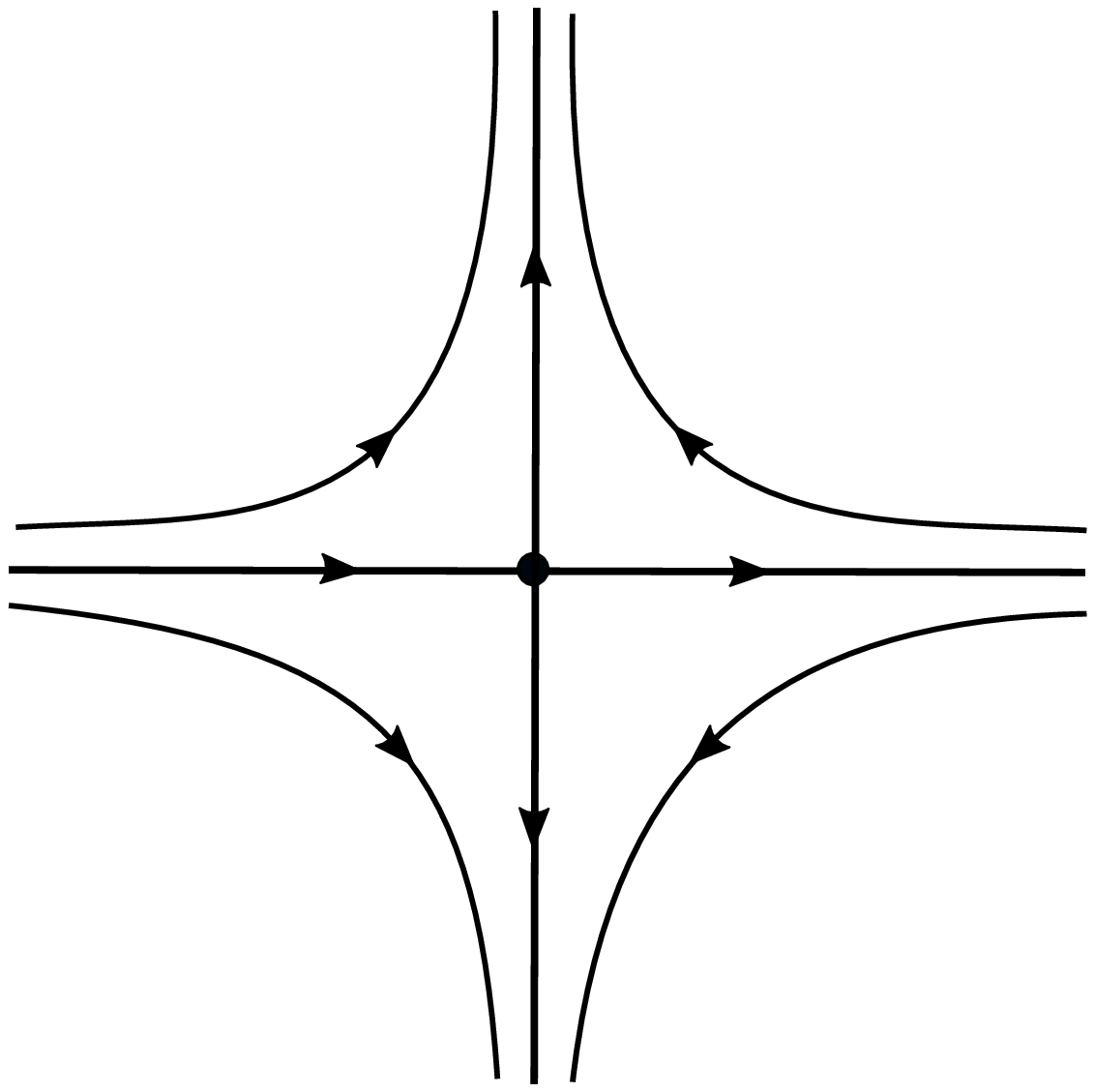} 
\end{overpic}
\hspace{1cm}
\begin{overpic}[scale=0.30]
	{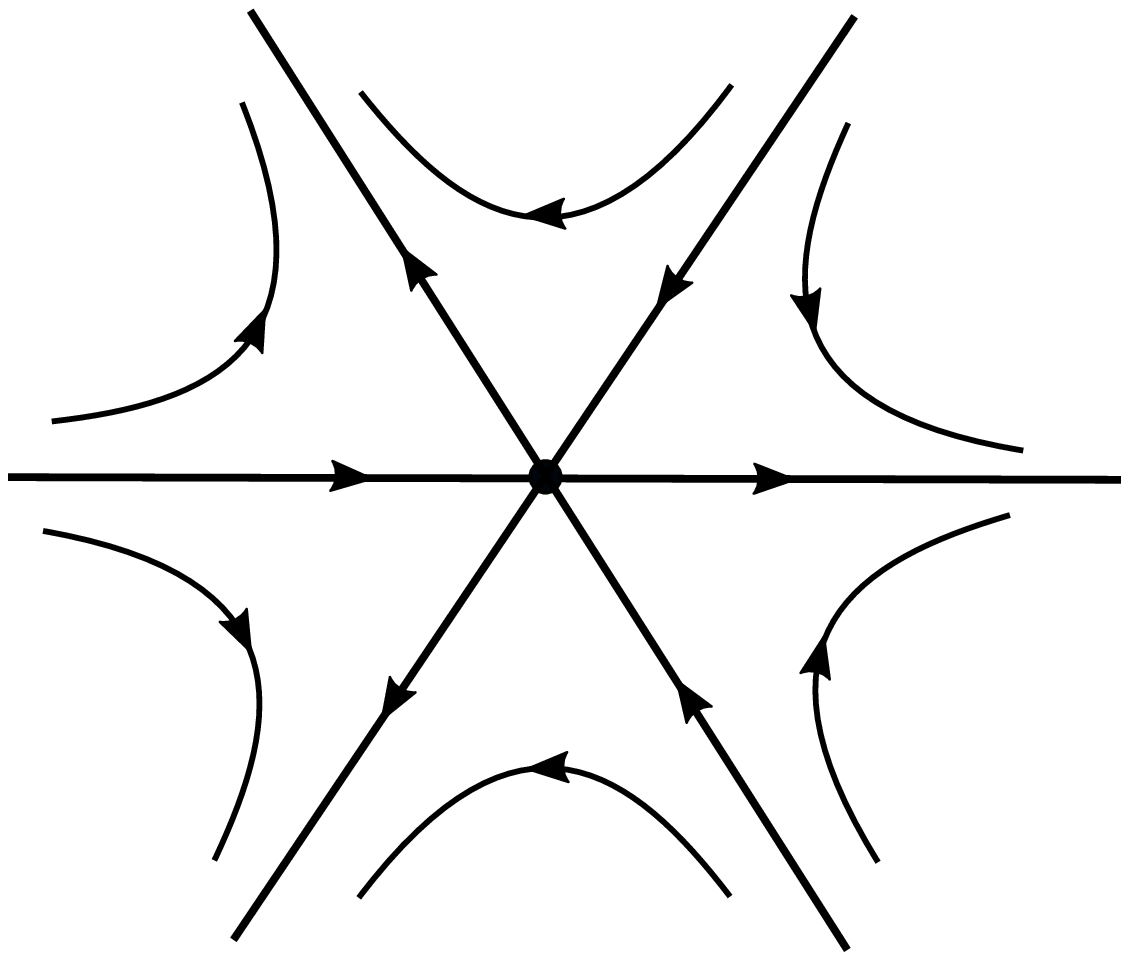} 
\end{overpic}
\\ \vspace{1cm}
\begin{overpic}[scale=0.35]
	{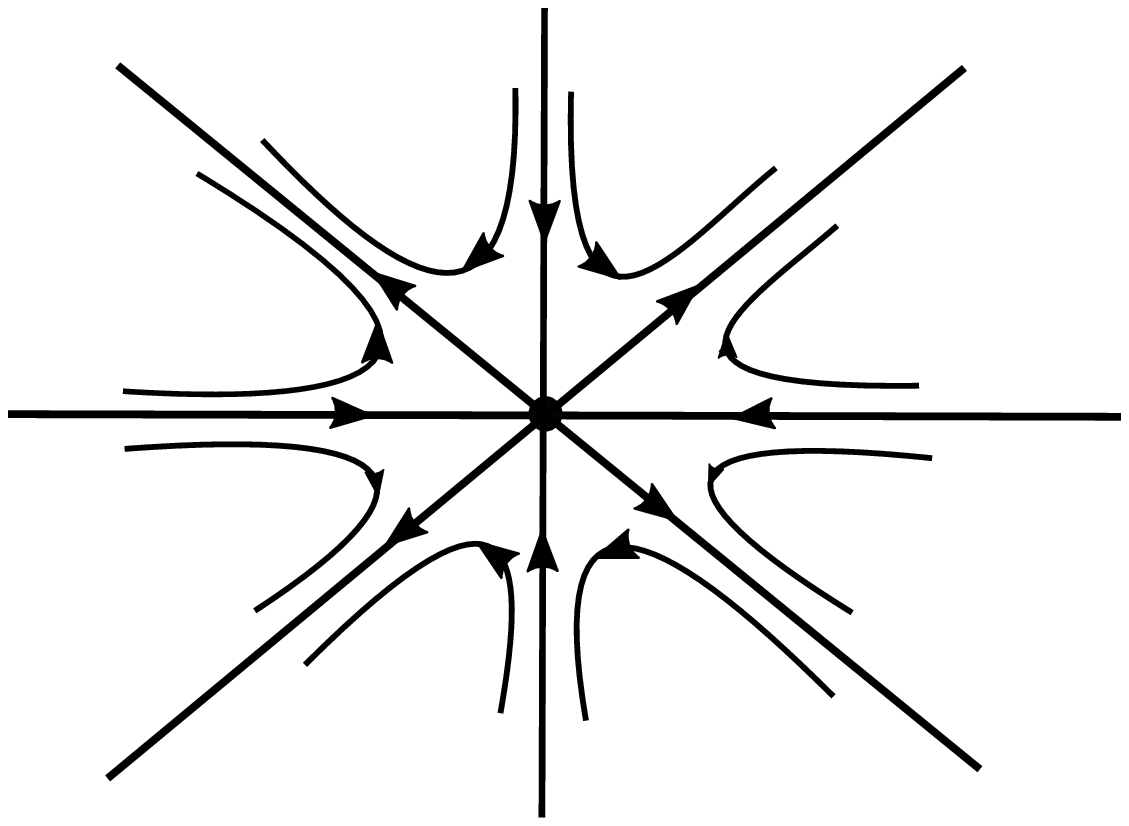} 
\end{overpic}
\hspace{0.8cm}
\begin{overpic}[scale=0.35]
	{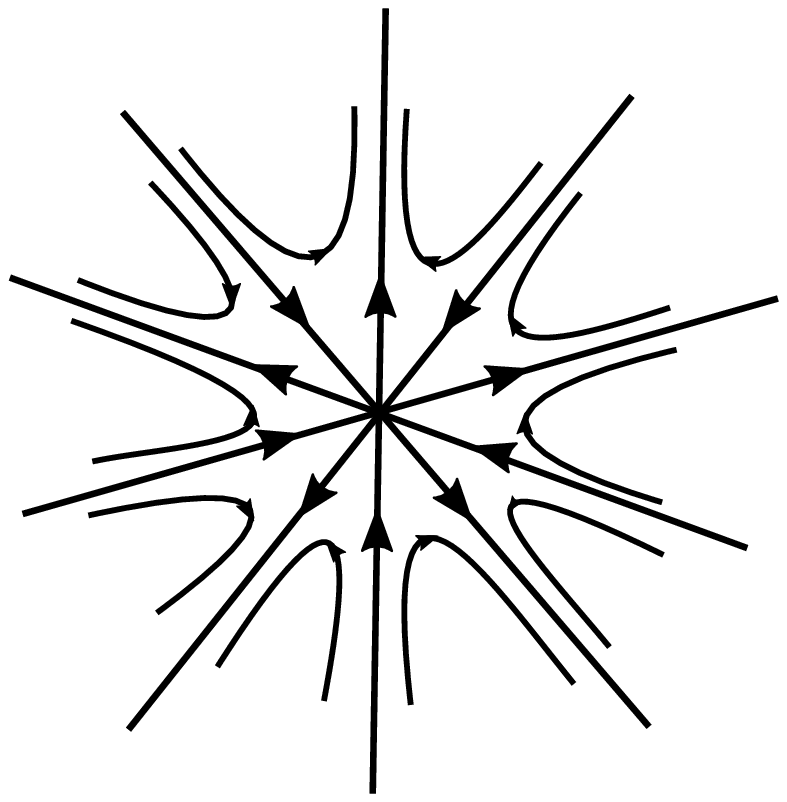} 
\end{overpic}
\caption{Phase Portraits $\dot{z}=\overline{A_{1}z}$, $\dot{z}=\overline{A_{2}z^2}$,  $\dot{z}=\overline{A_{3}z^3}$ and $\dot{z}=\overline{A_{4}z^4}$.}
\end{figure}

\begin{center}
\begin{figure}[h]
\begin{overpic}[scale=0.25]
	{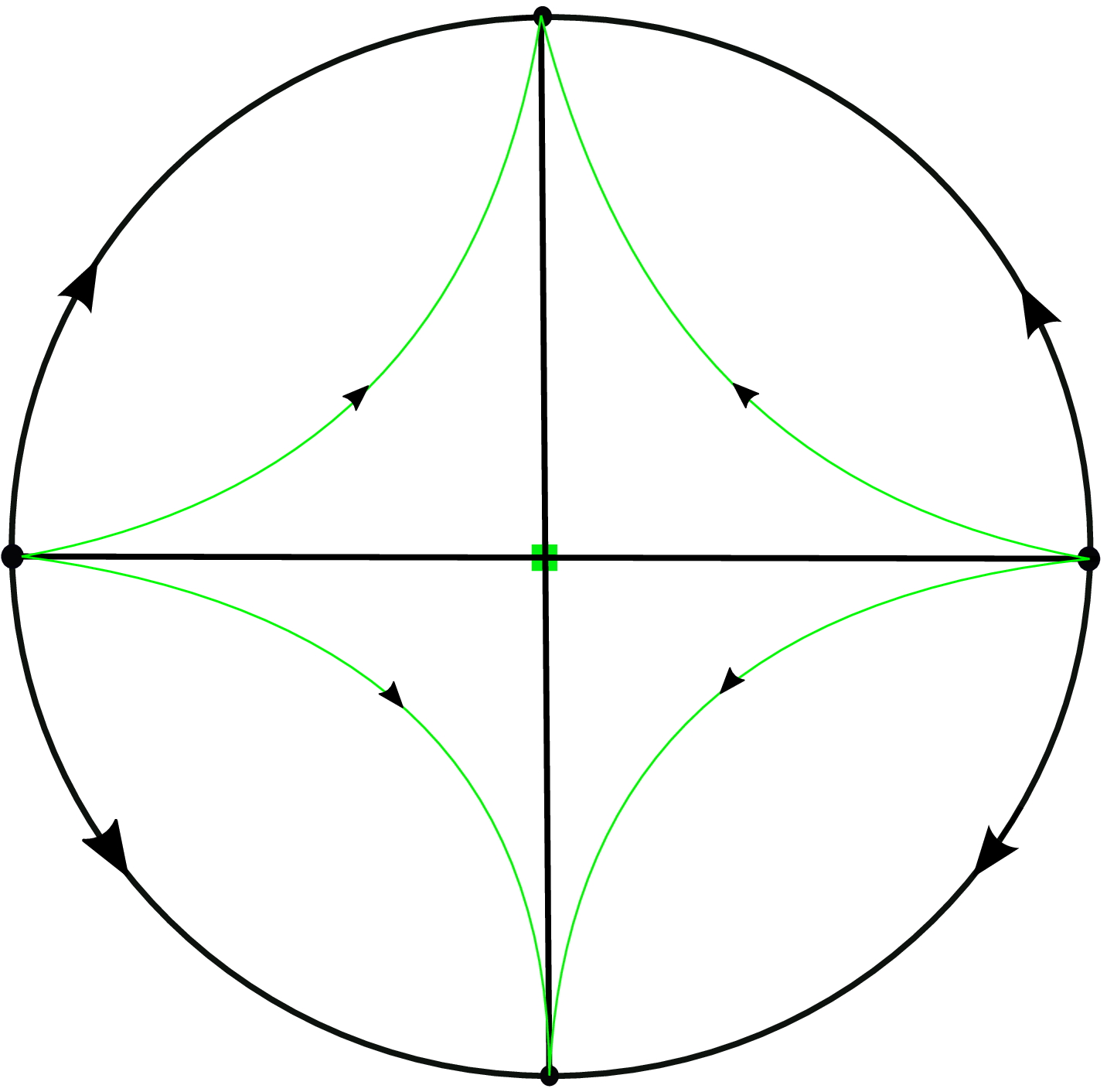} 
\end{overpic}
\hspace{1cm}
\begin{overpic}[scale=0.25]
	{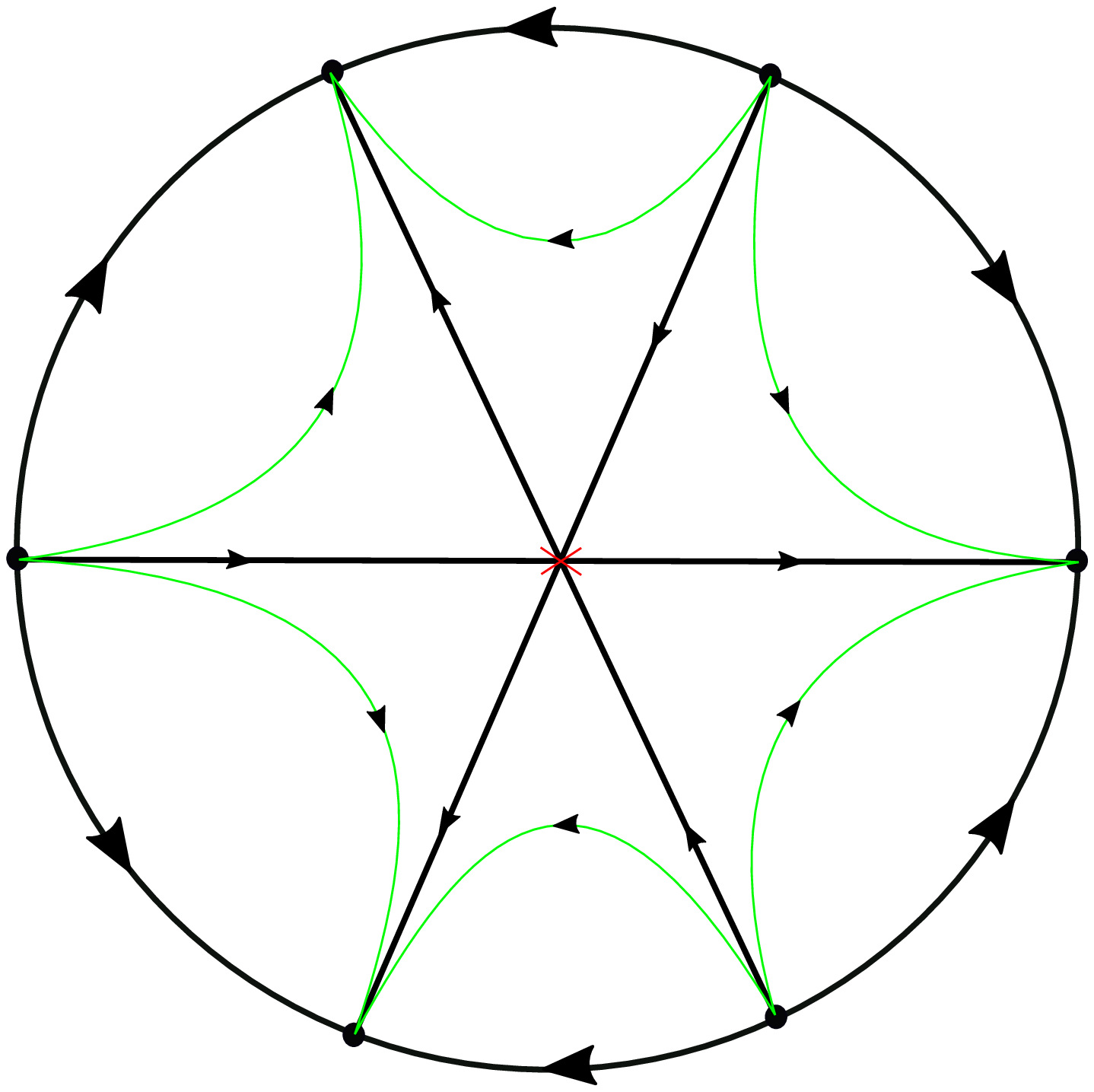} 
\end{overpic}
\caption{Phase Portraits of $\dot{z}=\overline{(1+2i)z}$ and $\dot{z}=\overline{(1+4i)z^2}$.}
\end{figure}
\end{center}

\begin{center}
\begin{figure}[h]
\begin{overpic}[scale=0.25]
	{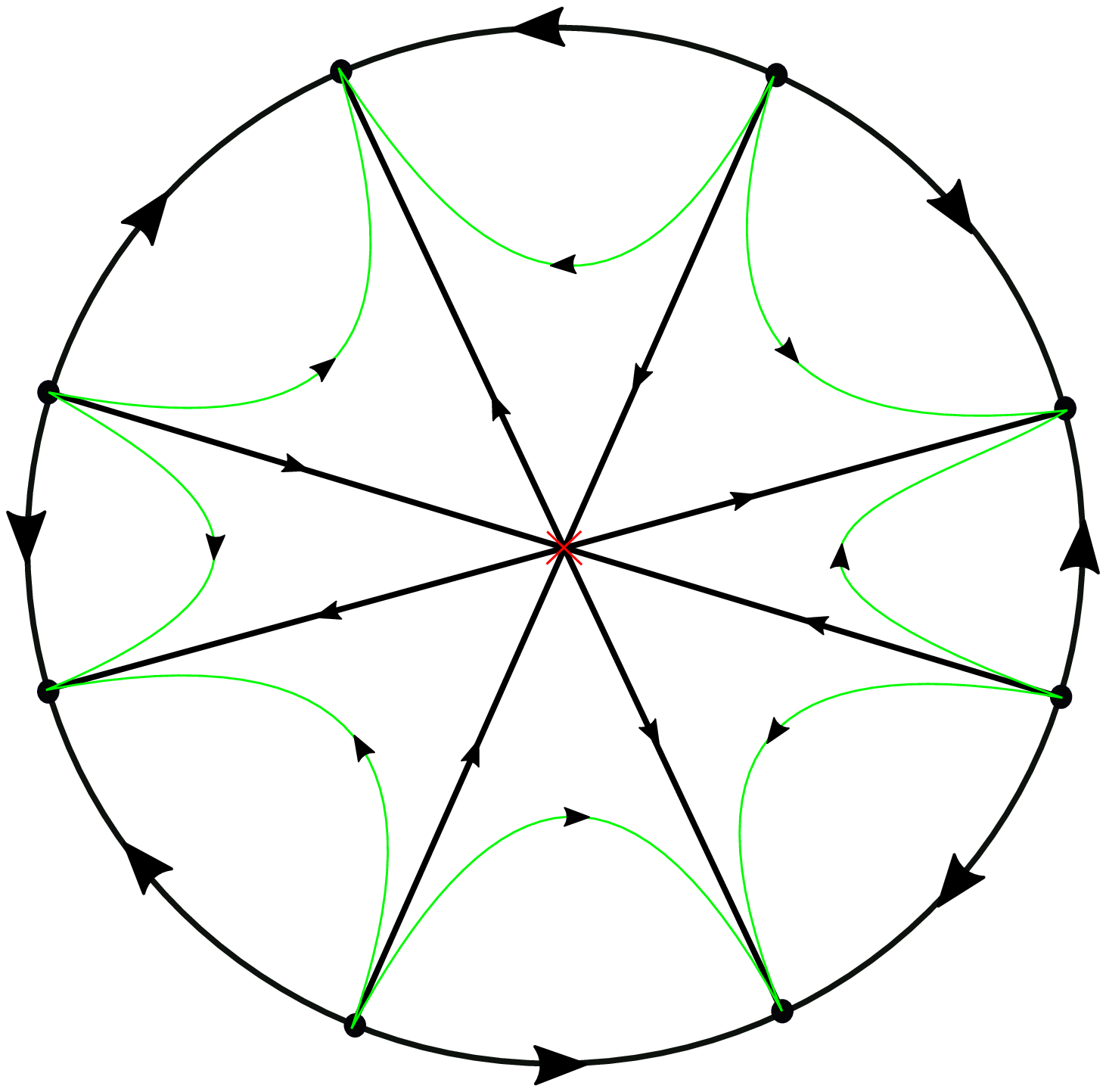} 
\end{overpic}
\hspace{1cm}
\begin{overpic}[scale=0.25]
	{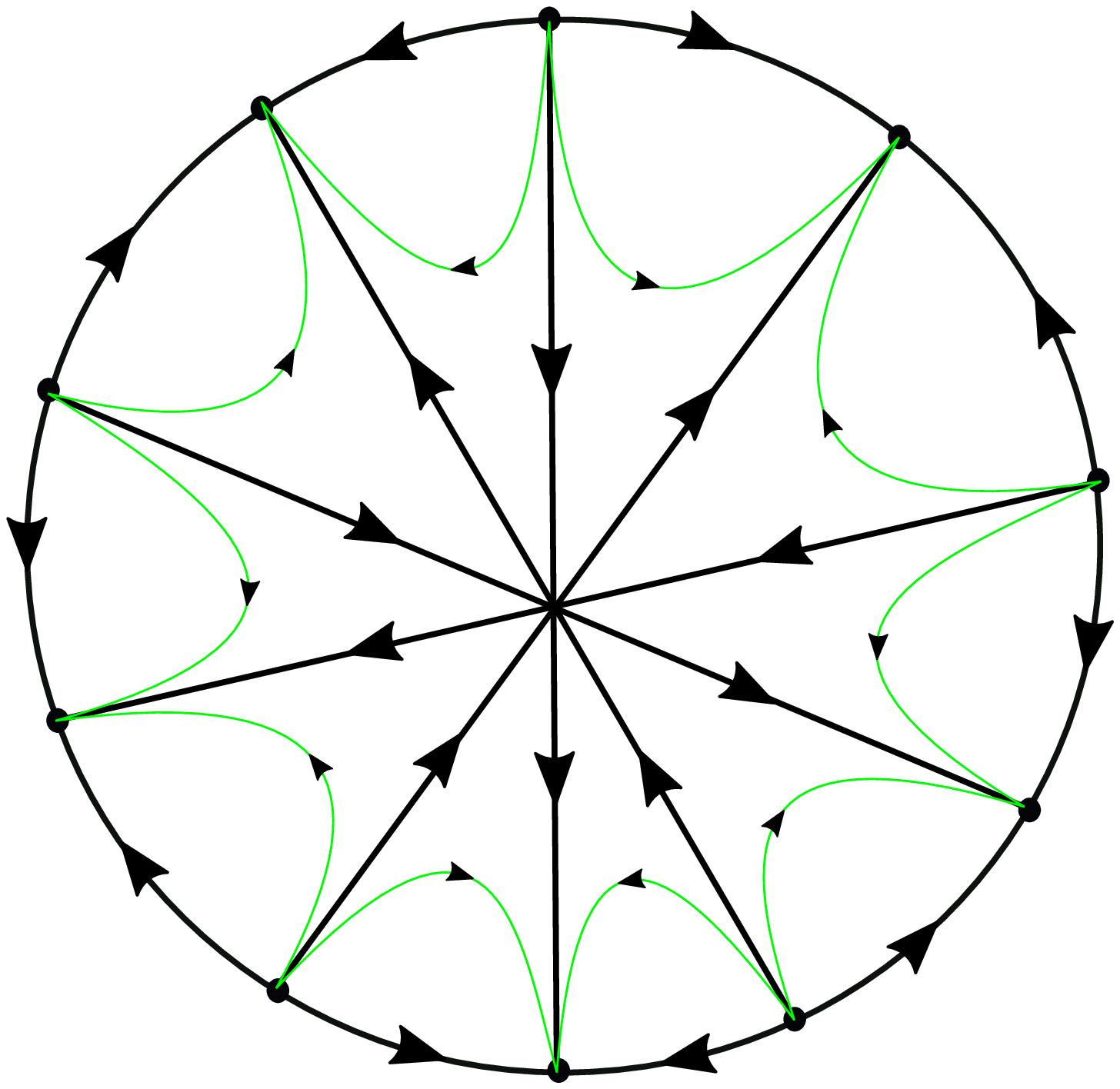} 
\end{overpic}
\caption{Phase Portraits of $\dot{z}=\overline{(1+i)z^3}$ and  $\dot{z}=\overline{z^4}$.}
\end{figure}
\end{center}

\begin{remark}
	We notice that the local dynamics of  $\overline{z}^{2}$ and  of $\overline{z^{2}/(1+z)}$ are topologically equivalent. In fact, this also occur to the conjugate of $z^{3}$ and the conjugate of $z^{3}/(1+z^2)$, and the conjugate of $z^4$ and conjugate of $z^{4}/(1+z^3)$.  In all cases, the local dynamics is formed by $2n+2$ sectors, where $n$ is the degree of $z$, and the origin is of saddle type. This phenomenon does not occur for $z^{n}$ and $z^{n}/(1+z^{n-1})$, that is, it is a peculiarity of conjugate systems. Naturally, the local dynamics being topological equivalent does not implies the same phase portrait in all Poincaré Disk. This fact can be checked if we take the systems $(1+4i)z^{2}$ and $z^2/(1+z)$. In both cases, the local dynamics around the origin is formed by 6 sectors and the origin is saddle type. However, as we can seen in the Figure 12, the phase portrait in all Poincaré disk is different.
\end{remark}

\begin{center}
\begin{figure}[h]	
\begin{overpic}[scale=0.25]
	{fluidos2.eps} 
\end{overpic}
\hspace{1cm}
\begin{overpic}[scale=0.25]
	{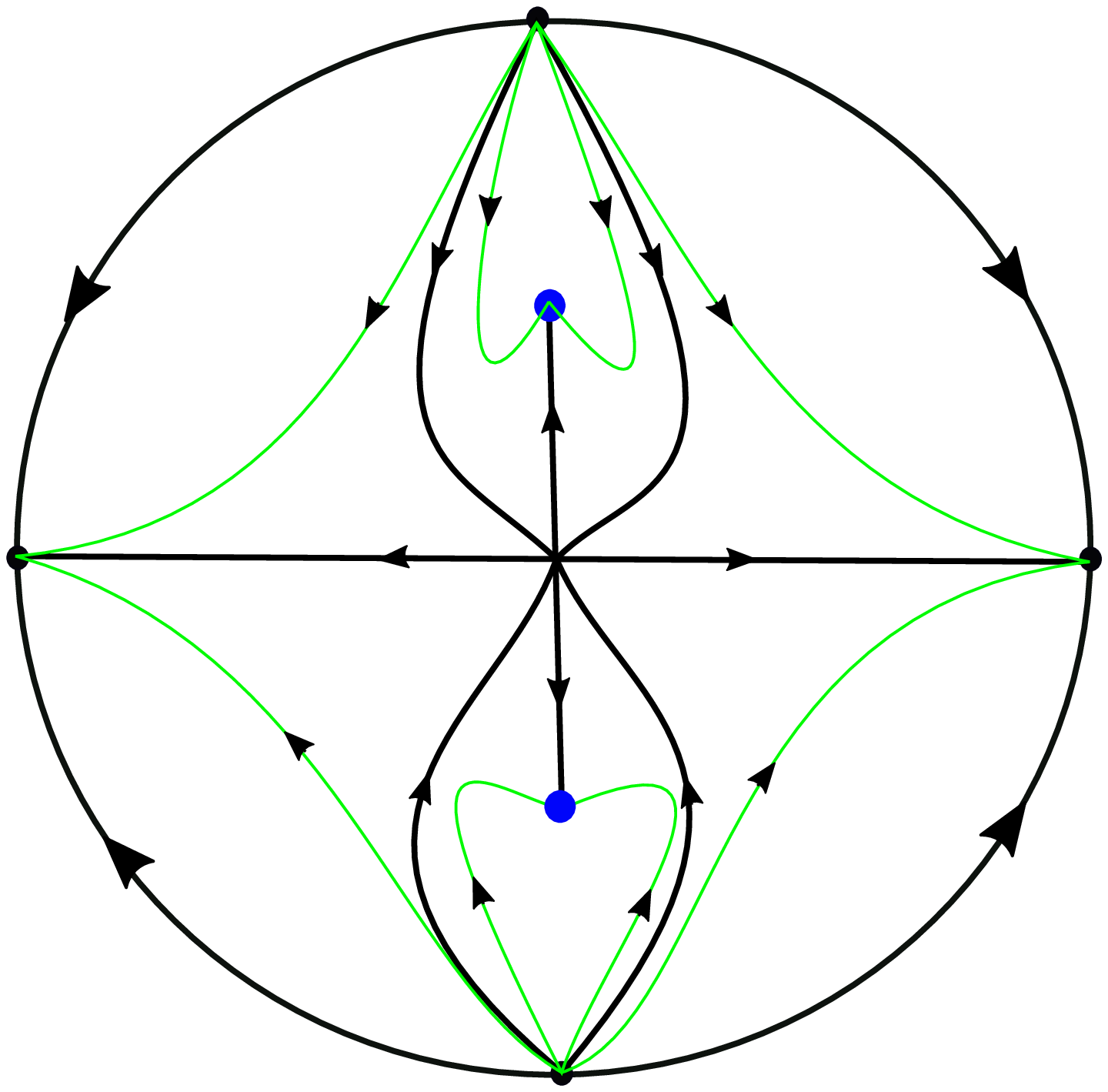} 
\end{overpic}
\caption{Phase Portraits $\dot{z}=\overline{z^2}$ and $\dot{z}=\overline{z^2/(1+z)}$.}
\end{figure}
\end{center}

Now, let us consider equations

\begin{equation*}
\left.\begin{array}{lll}
p_2(z) &= & A_0+A_1z+A_2z^2, \\
p_3(z) &= & B_0+B_1z+B_2z^2+B_3z^3,\\
p_4(z) &= & C_0+C_1z+C_2z^2+C_3z^3+C_4z^4.
\end{array}\right.
\end{equation*}
with $A_0=B_0=C_0=0$, that is, let us consider equations with the form

\begin{equation*}
\left.\begin{array}{lll}
f_2(z) &= & A_1z+A_2z^2, \\
f_3(z) &= & B_1z+B_2z^2+B_3z^3,\\
f_4(z) &= & C_1z+C_2z^2+C_3z^3+C_4z^4.
\end{array}\right.
\end{equation*}

Compactifying on $f_{2}$,$f_{3}$ and $f_{4}$, we obtain that the infinity are nodes  ($4,6,8$ and $10$ node points respectively) with alternating stability. Moreover,  $f_{2}$,$f_{3}$ and $f_{4}$ have $2,3$ and $4$ saddle points. Below, we show one example of phase portrait for each $f_{j}(z)$.

\begin{center}
\begin{figure}[h]	
\begin{overpic}[scale=0.2]
	{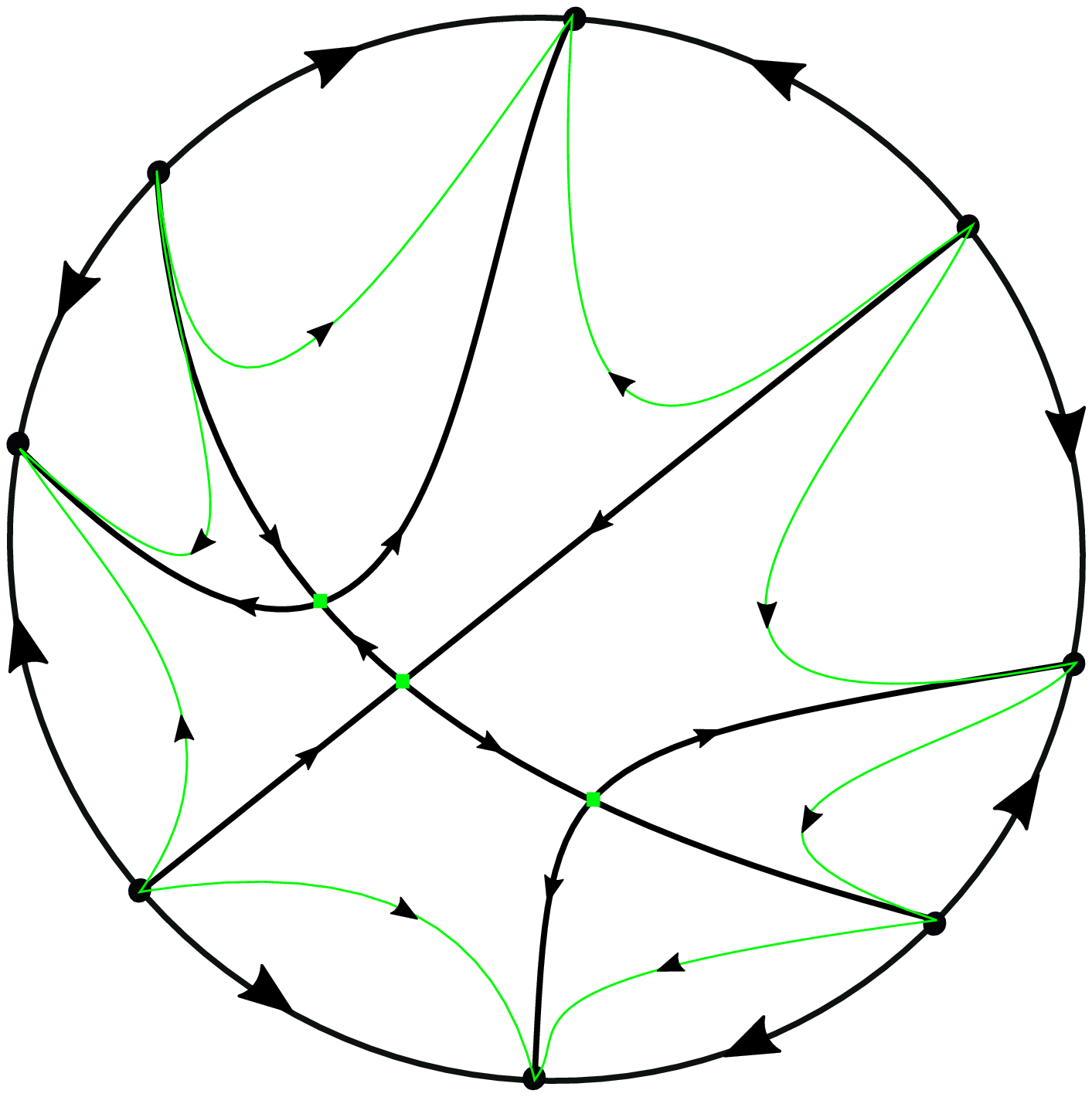} 
\end{overpic}
\hspace{1cm}
\begin{overpic}[scale=0.2]
	{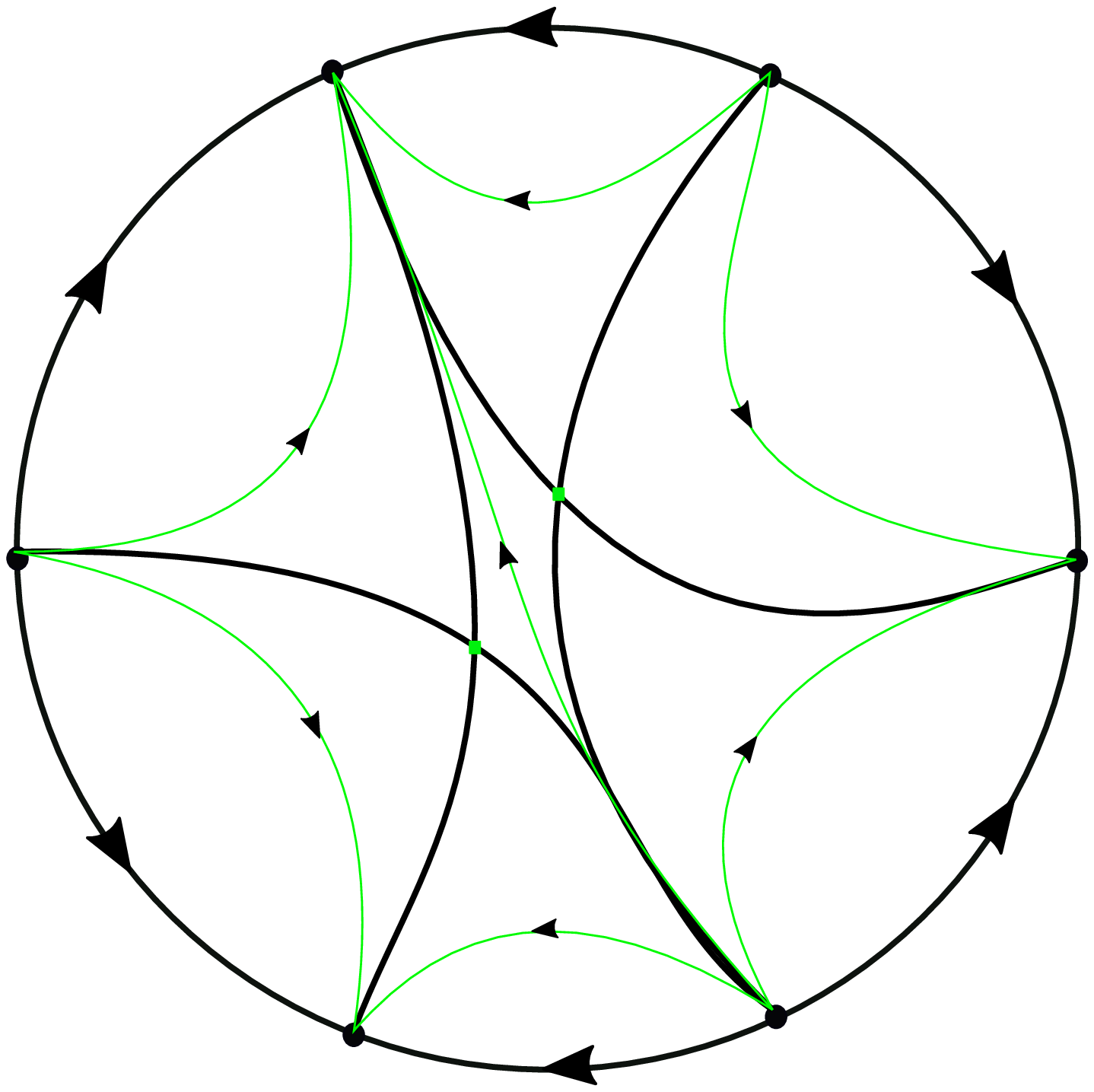} 
\end{overpic}
\hspace{1cm}
\begin{overpic}[scale=0.2]
{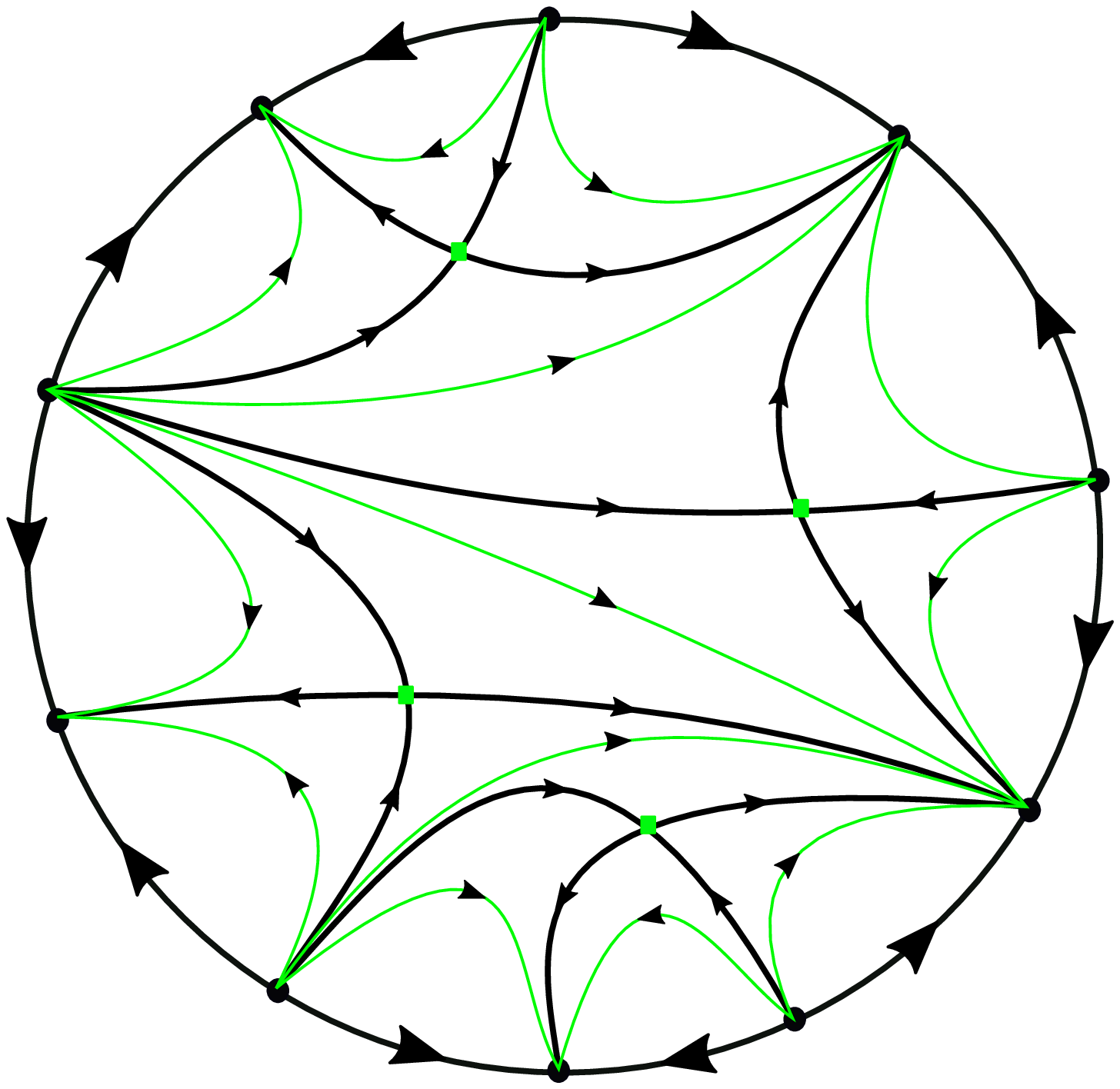} 
\end{overpic}
\caption{Phase Portraits of $\dot{z}=\overline{(1+i)z+(3+4i)z^2}$, $\dot{z}=\overline{(1+3i+z)(2+2i+z)(3+i+z)}$, and $\dot{z}=\overline{(1+i)z+(2-i)z^2+iz^3+2z^4}$, resp.}
\end{figure}
\end{center}

\section{Moebius Systems}\label{sec10}
Moebius transformations form a class of conformal maps that have a very strong geometric appeal and that have surprising properties.
A Moebius transformation is a map of the form $T(z)=\frac{Az+B}{Cz+D}$ satisfying that $AD-BC\neq 0$.

The derivative of a Moebius transformation is given by $T'(z)=\frac{AD-BC}{(Cz+D)^2}$. They are invertible and their inverse is still a Moebius transformation. 
Moreover they can be obtained by composing translations $z\mapsto z+K$, rotations $z\mapsto \alpha z, |\alpha|=1$, inversion $z\mapsto \frac{1}{z}$ and homotheties $z\mapsto \alpha z, |\alpha|\neq1,0 $.

The most natural way to consider $T$ is as a function of the plane extending into itself $T:\overline{\C}\rightarrow \overline{\C}$ defining $T\left(-\frac{D}{C}\right) =\infty$ and $T(\infty)=\frac{A}{C}.$

Two important properties of Moebius transformations are as follows:
\begin{itemize}
	\item Given three distinct points $z_1,z_2,z_3$ in $\overline{\C}$ and another three distinct points in $w_1,w_2,w_3$ in $\overline{\C}$ there is only one Moebius transformation satisfying that $T(z_j)=w_j$, $j=1,2,3$.
	\item The $\mathcal{F}$ family, formed by the circles of $\overline{\C}$ is preserved by Moebius $T$ transformations. From the point of view of the plane, this means that the image of circles and straight lines by Moebius transformations is still circles and straight lines.
\end{itemize}

In this section we describe the phase portrait of holomorphic systems of the type 
\begin{equation}\label{eqMobius} \dot{z}=\frac{Az+B}{Cz+D},\quad AD-BC\neq0.\end{equation}

Equivalently, writing $z=x+iy$, $T=u+iv$, $A=a_1+a_2i$, $B=b_1+b_2i$, $C=c_1+c_2i$ and $D=d_1+d_2i$ we get the planar system
\[\dot{u}=\frac{u_{20}x^2+u_{10}x+u_{02}y^2+u_{01}y+u_{00}}{q(x,y)},\quad \dot{v}=\frac{v_{20}x^2+v_{10}x+v_{02}y^2+v_{01}y+v_{00}}{q(x,y)},\]
with
\[q(x,y)=(c_1^2+c_2^2)x^2+(2c_1d_1+2c_2d_2)x+(c_1^2+c_2^2)y^2+(2c_1d_2-2c_2d_1)y+d_2^2+d_1^2\]
and 
\begin{equation}
\begin{array}{lll}
	&u_{20}=a_1c_1+a_2c_2       &v_{20}=-a_1c_2+a_2c_1\\
	&u_{10}=a_1d_1+a_2d_2+b_1c_1+b_2c_2 &v_{10}=-a_1d_2+a_2d_1-b_1c_2+b_2c_1,\\
	&u_{02}=a_1c_1+a_2c_2& v_{02}=-a_1c_2+a_2c_1,\\
	&u_{01}=a_1d_2-a_2d_1-b_1c_2+b_2c_1 & v_{01}=a_1d_1+a_2d_2-b_1c_1+b_2c_2\\
	&u_{00}=b_2d_2+b_1d_1 & _{00}=-b_1d_2+b_2d_1\end{array}
\end{equation}

\begin{proposition}The equilibrium points and poles of the Moebius systems \eqref{eqMobius} are described in the following table. \\
\begin{center}
		\begin{tabular}{|r|r|r|r|c|l|}
		\hline
		A&B&C&D&Equilibrium&Pole\\ \hline
		$0$ & $\neq0$ &$\neq0$ & $0$ &-&$z_0=0$\\ \hline
		$0$ & $\neq0$ &$\neq0$ & $\neq0$ &-&$z_0=-\frac{D}{C}$\\ \hline
		$\neq0$ & $0$ &$0$ & $\neq0$ &$z_0=0$&-\\ \hline
		$\neq0$ & $0$ &$\neq0$ & $\neq0$ &$z_0=0$&$z_0=-\frac{D}{C}$\\ \hline
		$\neq0$ & $\neq0$ &$\neq0$ & $0$ &$z_0=-\frac{B}{A}$&$z_0=0$\\ \hline
		$\neq0$ & $\neq0$ &$0$ & $\neq0$ &$z_0=-\frac{B}{A}$&-\\ \hline
		$\neq0$ & $\neq0$ &$\neq0$ & $\neq0$ &$z_0=-\frac{B}{A}$&$z_0=-\frac{C}{D}$\\ \hline\end{tabular}
\end{center}
\end{proposition}

\begin{proposition}
	The Moebius system \eqref{eqMobius} is conformally conjugated to $f(z)=\frac{A^2}{AD-BC}z$ in a neighborhood
	of the equilibrium $z_0=-\frac{B}{A}$ and conformally conjugated to $f(z)=\frac{1}{z}$ in a neighborhood
	of the pole $z_0=-\frac{C}{D}$.
\end{proposition}
\begin{proof}
	In fact, it follows of the derivative of the Moebius map $T$ at $-\frac{B}{A}$ which is given by 
\[T'\left(-\frac{B}{A}\right)=\frac{A^2}{AD-BC}\]
and from the limit 
\[\lim_{z\rightarrow-\frac{D}{C}}\left(z+\frac{D}{C}\right)\frac{Az+B}{Cz+D}=\frac{BC-AD}{C}\neq0.\]
\end{proof} 
\begin{proposition}\label{proptransmoebius}
	Let $\dot{z} = T(z)$ be the Moebius system  \eqref{eqMobius} with $A\neq0$.
\begin{enumerate}[a)]
	\item 	Its trajectories are contained in the level curves of 
	\[H(x,y)=\Im\left({\frac{ACz+(-BC+AD)\log(B+Az)}{A^2}}\right).\]
	\item For the  Moebius system, we have, through topological equivalence,  the following phase portraits.
\begin{center}
	\begin{figure}[h]	
\begin{overpic}[scale=0.35]
	{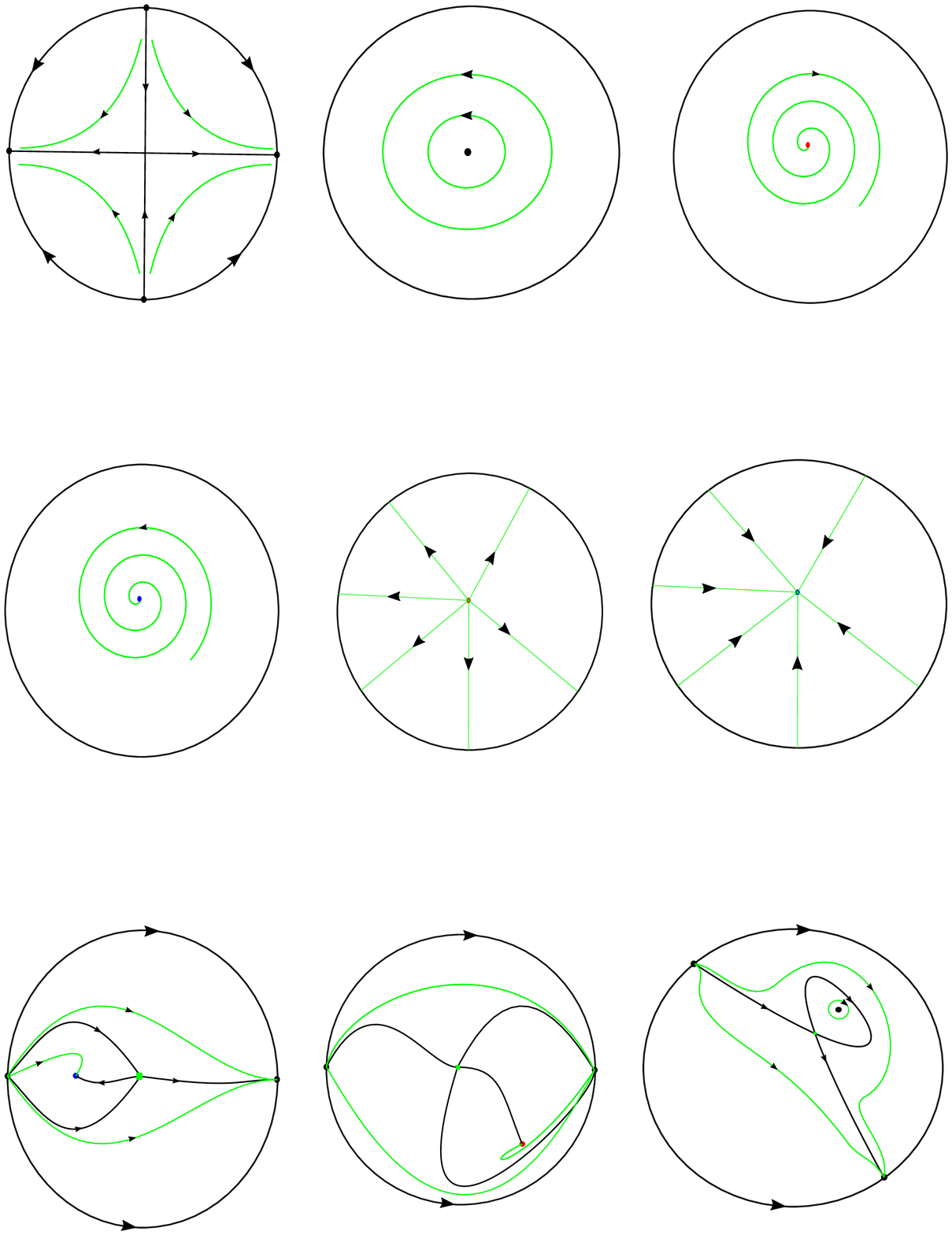} 
	\put(20,170){M1)}
	\put(85,170){M2)}
	\put(155,170){M3)}
	\put(20,80){M4)}
	\put(85,80){M5)}
	\put(155,80){M6)}
	\put(20,-10){M7)}
	\put(85,-10){M8)}
	\put(155,-10){M9)}
\end{overpic}
\vspace{0.2cm}
\caption{Topological phase portrait of Moebius System}
	\end{figure}
\end{center}
\end{enumerate}
\end{proposition}

\begin{proof}
	\begin{enumerate}[a)]
		\item It follows from the primitive
\[\int\dfrac{1}{T(z)}dz=\int\dfrac{Cz+D}{Az+B}dz={\dfrac{ACz+(-BC+AD)\log(B+Az)}{A^2}}.\]
		\item  We have the following cases:
	\begin{itemize}
		\item $A=0$ and $B\neq 0$. In this case, the system $\dot{z}=\dfrac{B}{Cz}$ is conformally conjugated with $\dot{z}=\dfrac{1}{z}$. Then, we have a simple pole singularity at $z=0$. Using Poincar\'{e} Compactification, we obtain 4 noddle points at the infinity with alternated stability. This case corresponds the Figure 18-M1).
		\item $C=0$ and $D\neq 0$. In this case, the system $\dot{z}=\dfrac{Az+B}{D}$ is conformally conjugated with $\dot{z}=\eta z$, $\eta\in\mathbb{C}$. If $\Re(\eta)=0$ and $\Im(\eta)\neq0$, we obtain the Figure 18-M2). If $\Re(\eta)\neq0$ and $\Im(\eta)\neq0$, then we have the Figure 18-M3) and the Figure 18-M4). If $\Re(\eta)\neq0$ and $\Im(\eta)=0$, we obtain the Figure 18-M5) and the Figure 18-M6).
		\item $A\neq 0$, $B\neq 0$, and $C\neq 0$. We have a simple pole at $z=-\frac{D}{C}$ and  \textbf{i)} one center or \textbf{ii)} one focus/node. According \cite{GXG}, we have two nodes at infinity with alternated stability. Then, the possibilities for the phase portrait are given by  Figures 18-M7), 18-M8) and Figure 18-M9). All this cases are realizable and we can obtain with the following systems: 
	\begin{itemize}
		\item $\dot{z}=\frac{(3/13 - 2i/13)((1 + i)z - 3/2 + i)}{z}$ - case M7),
		\item $\dot{z}= \frac{(3/13 - 2i/13)((1 + i)z - 13/3 + i)}{z}$ - case M8),
		\item $\dot{z}=\frac{(12/37 + 2i/37)((1 - 2i)z - 27/14 + i)}{z}$ - case M9).
	\end{itemize}	
\end{itemize}
\end{enumerate}
\end{proof}

\section {Acknowledgments}

This article was possible thanks to the scholarship granted from the Brazilian Federal Agency for Support and Evaluation of Graduate Education (CAPES), in the scope of the Program CAPES-Print, process number 88887.310463/2018-00, International Cooperation Project number 88881.310741/2018-01. The second author is also partially supported by the grant nº 2019/10269-3, S\~ao Paulo Research Foundation (FAPESP).

The authors Luiz Fernando Gouveia  and Gabriel Rond\'on are supported by the grant 2020/04717-0 S\~ao Paulo Research Foundation (FAPESP) and   2020/06708-9 S\~ao Paulo Research Foundation  (FAPESP) respectively.

\section{Appendix}

\subsection{Lyapunov Constants}
	Consider a differential system of the kind
	\begin{equation}\label{a1}
	\dot{x}=\beta y+\varphi(x,y), \quad \dot{y}=-\beta x+\psi(x,y)
	\end{equation}
	with $\beta>0$ and $\varphi,\psi$ analytical functions which vanish together with their first derivatives at the point $(0,0).$
	
	Consider the polar change $x=r\cos\theta$, $y=r\sin\theta$. Thus system \eqref{a1} becomes
	\begin{equation}\label{a2}
	\dot{r}=P(r,\theta), \quad \dot{\theta}=Q(r,\theta).
	\end{equation}
	The phase portarit of \eqref{a2} is composed by the graphics of 
	\[r=f_{\rho}(\theta), \quad f_{\rho}(0)=\rho, \] 
	solutions of 
	\begin{equation}\label{a3}
	\frac{dr}{d\theta}=R(r,\theta)=\frac{P(r,\theta)}{Q(r,\theta)}=R_1(\theta)r+R_2(\theta)r^2+...
	\end{equation}
	with 
	\[R_k(\theta)=\frac{1}{k!}\frac{\partial^kR(r,\theta)}{\partial r^k}|_{r=0}.\]
	The series expansion of $f_{\rho}(\theta)$ is denoted by
	\[f_{\rho}(\theta)=u_1(\theta)\rho+u_2(\theta)\rho^2+...\]
	Put $r=f_{\rho}(\theta)$ in \eqref{a3} to get
	\[u_1'(\theta)\rho+u_2'(\theta)\rho^2+...=R_1(\theta)(u_1(\theta)\rho+...)+R_2(\theta)(u_1(\theta)\rho+...)^2+...\]
	The returning map $\pi:[0,+\infty)\rightarrow\R$ is $\pi(\rho)=f_{\rho}(2\pi)-\rho$. The Lyapunov values are given by
	\[V_k=\frac{\pi^{(k)}(0)}{k!}\]
	
	If there exists $n\in\N$ such that 
	\[\pi'(0)=...=\pi^{(n-1)}(0)=0,\pi^n(0)\neq0 \]
	then $n$ is odd. In this case we say that $\frac{n-1}{2}$ is the multiplicity of the focus.\\
	
		If $(0,0)$ is a multiple focus of multiplicity $k>1$ of \eqref{a1} then 
		\begin{itemize}
			\item [(a)] There exist $\e_0>0$ and $\delta_0>0$ such that for any system $\delta_0$--close to system
			\eqref{a1} has at most $k$ limit cycles in a $\e_0$--neighborhood of $(0,0)$.
			\item [(b)] For any $\e<\e_0, \delta<\delta_0$ and $1\leq s\leq k$ there exists a system which is $\delta$--close
			to \eqref{a1} and has precisely $s$ limit cycles in a $\e$--neighborhood of $(0,0)$.
		\end{itemize}

	We can compute the Lyapunov values using the following algorithm.\\
	1) $R$\\
	2) $r1=diff(R,r)$\\
	3) $R1=subs(r=0,R1)$\\
	4) $r2=diff(r1,r)$\\
	5) $R2=\frac{1}{2}subs(r=0,r2)$\\
	6) $r3=diff(r2,r)$\\
	7) $R3=\frac{1}{6}subs(r=0,r3)$\\
8)	$r4= diff(r3, r)$\\
9) $	R4= (1/24)subs(r = 0, r4)$\\
10) $	r5 = diff(r4, r)$\\
11)	$R5 = (1/120)subs(r = 0, r5)$\\
12)	$r6 = diff(r5, r)$\\
13) $	R6 = (1/720)subs(r = 0, r6)$\\
14) $	r7 = diff(r6, r)$\\
15) $	R7 = (1/5040)subs(r = 0, r7)$\\
	16) $U2=int(R2,\theta=0..k)$\\
	17) $V2=subs(k=2\pi,U2)$\\
	18) $u2=subs(k=\theta,U2)$\\
	19) $U3=int(2U2R2+R3,\theta=0..k)$\\
	20) $V3=subs(k=2\pi,U3)$\\
21) $	u3 = subs(k = u, U3)$\\
22) $	V4 = simplify(subs(k = 2\pi, U4))$\\
23) $	u4 = subs(k = u, U4)$\\
24) $	U5 = int(R2(2u2u3+2u4)+R3(3u2^2+3u3)+4R4u2+R5, u = 0 .. k)$\\
25) $	V5 = simplify(subs(k = 2\pi, U5))$\\
26) $	u5 = subs(k = u, U5)$\\
27) $	U6 = int(R2(2u2u4+u3^2+2u5)+R3(u2^3+6u2u3+3u4)+R4(6u2^2+4u3)+5R5u2+R6, u = 0 .. k)$\\
28) $	V6 = simplify(subs(k = 2\pi, U6))$\\
29) $  V7 := simplify(subs(k = 2\pi, U7))$\\

\subsection{Poincar\'{e} Compactification} Now we present the formulas concerning the Poincar\'{e}
	compactification for a  polynomial differential system with degree $n$ in
	$\mathbb{R}^2$.  More precisely we consider the
	 polynomial differential system
	\begin{equation*}
	\dot{x}= u(x,y), \quad \dot{y}= v(x,y). 
	\end{equation*}
	This polynomial system is extended to an analytic system on a closed
	disk of radius one, whose interior is diffeomorphic to $\R^2$ and
	its boundary, the 1--dimensional circle $\s^1$; plays the role of
	the infinity. This closed ball is denoted by $\D^1$. We consider 4 open charts on $\s^1$
	\begin{itemize}
		\item[(a)]$U_1=\{(x,y):x>0\}$ and $V_1=\{(x,y):x<0\}$,
		\item[(b)]$U_2=\{(x,y):y>0\}$ and $V_2=\{(x,y)):y<0\}$,
	\end{itemize}
	The phase portrait on $U_1$ is the central projection of the phase
	portrait of the system
	\begin{equation}\label{u1}
	\dot{s}=w^n(-su+v),\quad  \dot{w}=w^n(-wu)
	\end{equation}
	where $s$ is the coordinate of the tangent line
	$T\s^1_{(1,0)}$ at $(1,0)\in\s^1$, and  $u,v$  are
	evaluated at $(1/w,s/w)$. Moreover we consider $w=0$.
	
	The flow on $U_2$ is determined by the system
	\begin{equation}\label{u2}
	\dot{s}=w^n(-sv+u), \quad \dot{w}=w^n(-wv)
	\end{equation} where $s$ is the coordinate of the tangent plane
	$T\s^1_{(0,1)}$ at $(0,1)\in\s^1$, and  $u,v$  are
	evaluated at $(s/w,1/w)$. Moreover we consider $w=0$.

	The expression for the extend differential system  in the local
	chart $V_{i},$ $i=1,2$ is the same as in $ U_{i}$ multiplied by
	$(-1).$

\bibliographystyle{acm}
\bibliography{bibliotese}
\end{document}